\documentclass[a4paper,reqno]{amsart}
\usepackage{amsfonts}
\usepackage{amssymb}
\usepackage{amsmath}
\numberwithin{equation}{section}
\usepackage{cite}
\usepackage{fancyhdr}
\usepackage{xcolor}
\usepackage{enumitem}
\usepackage{setspace}
\usepackage{amsmath}
\usepackage{makecell}
\usepackage{appendix}
\usepackage{array}
\usepackage[utf8]{inputenc}
\usepackage{bm}
\makeatletter
\newtheorem{thm}{Theorem}[section]
\newtheorem{lemma}{Lemma}[section]
\newtheorem{prop}{Proposition}[section]

\newtheorem{cond}{Condition}[section]
\newtheoremstyle{myremark}
  {}{}                
  {\normalfont}       
  {}                  
  {\bfseries}         
  {.}                 
  { }                 
  {}                  

\theoremstyle{myremark}
\newtheorem{Rmk}{Remark}[section]
\theoremstyle{plain}
\newtheorem{coro}{Corollary}[section]

\newcommand{\ML}{\mathcal{L}}

\newcommand{\WG}{\Vec{W}_{\boldsymbol{{\Gamma}}}}
\newcommand{\WTG}{\Vec{W}_{\boldsymbol{\Tilde{\Gamma}}}}
\newcommand{\MLWG}{\mathcal{L}_{W_{\boldsymbol{\Gamma}}}}

\newcommand{\TPG}{{P}_{\widetilde{\Gamma}}}
\newcommand{\PG}{{P}_{\Gamma}}

\newcommand{\MI}{\mathcal{I}}

\newcommand{\ME}{\mathcal{E}}

\newcommand{\lelan}{\left\langle}
\newcommand{\rilan}{\right\rangle}

\newcommand{\rmnum}[1]{\romannumeral #1}
\newcommand{\Rmnum}[1]
{\expandafter\@slowromancap\romannumeral #1@}
\makeatother

\title[ Multi-bubble solutions for four dimensional wave equation]{Construction of multi-bubble solutions for the energy-critical wave equation in dimension four}
\author[J. Gu]{Jingyuan Gu}
\address{School of Mathematical Sciences,
University of Science and Technology of China, Hefei 230026, Anhui, China}
\email{gjy4869@mail.ustc.edu.cn}
\author[J. Jendrej]{Jacek Jendrej}
\address{Institut de Math\'{e}matiques de Jussieu, Sorbonne Universit\'{e}, Universit\'{e} Paris Cit\'{e},
4~place Jussieu, 75005 Paris, France
}
\email{jendrej@imj-prg.fr}
 \author[L.Zhao]{Lifeng Zhao}
\address{School of Mathematical Sciences,
University of Science and Technology of China, Hefei 230026, Anhui, China}
\email{zhaolf@ustc.edu.cn}	
\begin{document}
\maketitle
\begin{abstract}
For any \(N\geq 2\), we construct a global solution of the energy-critical
focusing wave equation in dimension four which blows up in infinite time at
\(N\) prescribed points \(z_1,\ldots,z_N\in \mathbb R^4\), provided that the
points form one orbit under a finite group of orthogonal symmetries. We denote by $c:=2\sum_{j\ne k}|z_j-z_k|^{-2}>0$
the corresponding interaction coefficient, which is independent of \(k\). The
common concentration scale satisfies
\[
\log\frac{1}{\lambda(t)}
=
\left(\frac{9c}{4}\right)^{1/3}t^{2/3}+O(t^{1/3})
\qquad \text{as } t\to+\infty .
\]
This concentration rate comes from a genuinely four-dimensional effect: the
borderline decay of the ground state makes the interaction between different
bubbles enter the leading order parameter dynamics.
\end{abstract}

    	\tableofcontents

\section{Introduction}
\subsection{Setting of the problem}
We consider the energy-critical focusing wave equation in dimension four
\begin{equation}\label{NLW 4}
    \partial_{tt}u(t,x)-\Delta u(t,x)=f(u(t,x))\text{,  }\text{  }\text{ }t\in \mathbb R\text{, }x\in \mathbb R^{4},
\end{equation}
where $f(u):=|u|^2u$. Let \(F(u)=\frac14 |u|^4\). The conserved energy is
\[
E(u,\partial_tu)
=
\int_{\mathbb R^4}
\left(
\frac12|\partial_tu|^2+\frac12|\nabla u|^2-F(u)
\right)\,dx,
\]
which is well-defined on the energy space $\dot H^1(\mathbb R^4)\times L^2(\mathbb R^4).$ Writing \(\boldsymbol{u}=(u,\partial_tu)\), the equation can be viewed as the
Hamiltonian flow
\[
\partial_t\boldsymbol{u}=J\circ DE(\vec u),
\qquad
J=\begin{pmatrix}0&1\\ -1&0\end{pmatrix}.
\]
It is invariant under the energy-critical scaling
\[
u(t,x)\mapsto \lambda^{-1}u(t/\lambda,x/\lambda).
\]

Stationary solutions are obtained from the elliptic equation
\[
-\Delta W=f(W).
\]
We recall the explicit positive ground state
\[
W(x)=\left(1+\frac{|x|^2}{8}\right)^{-1}.
\]
Up to scaling and translation, \(W\) is the unique positive finite-energy
solution of this elliptic equation. For \(\lambda>0\) and \(z\in\mathbb R^4\),
we set
\[
W_{\lambda,z}(x):=\frac1\lambda W\left(\frac{x-z}{\lambda}\right),
\qquad
\boldsymbol{W}_{\lambda,z}:=(W_{\lambda,z},0).
\]
The ground state \(W\) is not only a stationary solution, but also a variational
threshold object. It achieves the sharp constant in the critical Sobolev
inequality \cite{Aubin,Talenti}, and its rescalings and translations describe
the basic lack of compactness of the energy-critical problem. Together with the
sign and Lorentz invariances, they generate the soliton family of the equation.

The dynamical role of \(W\) is reflected in the threshold theory for the
focusing energy-critical wave equation. Below the ground state threshold, Kenig
and Merle \cite{KenigMerle08} established the scattering/blow-up dichotomy. At and
above this threshold, much richer non-scattering dynamics may occur. In
particular, one may have type II solutions, whose energy norm remains bounded while
one or several scales concentrate. The leading coherent objects in such dynamics
are precisely rescaled copies of the ground state.

The construction of one-bubble type II solutions for the energy-critical wave
equation has been studied in several dimensions. In dimension three, \cite{DonningerKreiger,KriegerSchlag,KriegerSchlagTataru} constructed type II blow-up
solutions with prescribed blow-up rates. In dimension four, Hillairet and
Raphaël \cite{HR} constructed smooth type II blow-up solutions with a logarithmically
corrected rate. In dimension five, Jendrej \cite{Jtype2d5} constructed type II blow-up
solutions by a different modulation approach. These works describe mechanisms
where the leading dynamics is essentially generated by a single concentrating
copy of the ground state. A complementary line of work concerns the dynamics near a single copy of the
ground state. Invariant or center-stable manifolds near the soliton family were
constructed by \cite{BeceanuCenterStable} and \cite{KriegerSchlagStableManifold}. Threshold and near-threshold dynamics were studied by
\cite{DuyckaertsKenigMerleUniversality,DM,DuyckaertsKenigMerleClassification,
KNSGlobal,KNSCenterStable}, while \cite{KenigMendelson} addressed related questions
by probabilistic methods. These works provide a detailed
description of the one-bubble regime.

The works mentioned above have revealed a variety of dynamical behaviors associated with a single concentrating or stationary bubble. A natural next step is to investigate configurations involving several bubbles. Besides their intrinsic interest, multi-bubble solutions are closely related to  the soliton resolution
program. In this picture, a bounded energy solution is expected to decompose, 
asymptotically and modulo radiation, into a finite sum of decoupled coherent
objects generated from the ground state.  Schematically,
\[
\vec u(t)=\vec v_L(t)+\sum_{j=1}^J \iota_j \vec Q_j(t)+o_{\mathcal E}(1),
\]
where \(\vec v_L\) is a solution to the linear wave equation, and each $\vec Q_j(t)$ belongs to the soliton family generated by $\boldsymbol{W}$ through
scaling, translation and Lorentz transformations.  The different coherent objects are asymptotically decoupled. For the focusing energy-critical wave equation, such
decompositions have been proved in several radial or asymptotic settings by
\cite{CDKM 2022, CDKM 2022.1, CDKM 2023, DuyckaertsKenigMartelMerle, DuyckaertsKenigMerleSolitonResolution, JL, DuyckaertsJiaKenigMerle}. 

Constructive results exhibit particular scenarios predicted by this picture. In
radial settings, the multi-bubble dynamics usually takes the form of a bubble
tree: several bubbles concentrate at the same spatial point but at separated
scales. For the energy-critical wave equation, Jendrej \cite{Jtb} constructed two-bubble
solutions of this type. Similar bubble-tree dynamics has also been studied in
critical geometric models, especially for equivariant wave maps, see \cite{HwangKimWaveMapsBubbleTower, JendrejKriegerWaveMapsBubbleTree, KriegerPalaciosWaveMapsBubbleTrees}.

In nonradial settings, there are two related but different types of constructive
results. The first one concerns moving multi-solitons. In dimension five,
Martel and Merle constructed multi-solitons for the energy-critical wave
equation \cite{MartelMerle5D}, where the solitons separate by translation and velocity parameters.
They also proved the inelasticity of two-soliton collisions in the same setting, see \cite{MartelMerleInelasticity}. More recently, \cite{MartelMerleAnyParameters} constructed five-dimensional multi-solitons
with general parameters, including Lorentz velocities, scales and translations. Recent three-dimensional constructions of Kadar provide related examples where
the slow decay of the ground state produces strong interactions in nonradial
multi-soliton dynamics \cite{Kadar3DMultiSolitons,KadarNoQuantization}. The second one, closer to the present paper, concerns fixed-point bubbling.
Jendrej and Martel constructed infinite-time multi-bubble solutions in
dimension five concentrating at arbitrary prescribed fixed points \cite{JM}. In that
construction the scales are polynomial and depend on the configuration of the
points.

\subsection{Main result}
The present paper treats the fixed-point multi-bubble problem in
dimension four, which is different from the five dimensional case in an essential way: the bubbles exhibit a common stretched-exponential concentration rate, due to the fact that the interaction between distinct bubbles enters the leading order scale dynamics. We state the theorem for configurations which are, up to translation, given by a single finite symmetry orbit.
\begin{thm}\label{multi-bubble solution}
Let \(G\subset O(4)\) be a finite subgroup, and let
\[
Z=\{z_1,\ldots,z_N\}\subset \mathbb R^4
\]
be one \(G\)-orbit, with \(N\ge2\) and with the points \(z_k\) distinct. Thus,
for every \(\rho\in G\), there is a permutation \(\pi_\rho\) such that
\[
\rho z_k=z_{\pi_\rho(k)} .
\]
Assume moreover that the action of \(G\) on \(Z\) is transitive. Then the
quantity
\begin{equation*}
    \sum_{j\ne k}\frac1{|z_j-z_k|^2}
\end{equation*}
is independent of \(k\). We define
\[
c:=2\sum_{j\ne k}\frac1{|z_j-z_k|^2}>0,
\qquad
\xi:=\left(\frac{9c}{4}\right)^{1/3}.
\]

Then there exists a \(G\)-invariant solution
\[
(u,\partial_tu):[0,\infty)\to \dot H^1(\mathbb R^4)\times L^2(\mathbb R^4)
\]
of \((1.1)\), and a positive \(C^1\) function \(\lambda(t)\), such that, as
\(t\to\infty\),
\begin{equation*}
\left|\log\frac1{\lambda(t)}-\xi t^{2/3}\right|
\lesssim t^{1/3},    
\end{equation*}
and
\[
\left\|
u(t)-\sum_{k=1}^N
\frac1{\lambda(t)}
W\left(\frac{\cdot-z_k}{\lambda(t)}\right)
\right\|_{\dot H^1(\mathbb R^4)}
+
\|\partial_tu(t)\|_{L^2(\mathbb R^4)}
\to 0 .
\]
\end{thm}

\begin{Rmk}[The role of the symmetry assumption]
The symmetry assumption should be understood as the natural geometric setting
for the common-scale dynamics constructed in this paper. In dimension four, the
interaction between distinct bubbles enters the leading scaling equation.
Formally, each bubble sees an interaction strength determined by its distances
to all other concentration points. Thus, if all bubbles are to concentrate with
the same leading scale, they should see the same leading interaction strength.
A point-transitive configuration provides a canonical way to ensure this: all
concentration points are geometrically equivalent.

We carry out the construction in the corresponding \(G\)-invariant class. This
class is invariant under the flow, and the uniqueness of the modulation
decomposition then forces the modulation parameters to be equivariant. In
particular,
\[
\lambda_1(t)=\cdots=\lambda_N(t).
\]
By translation invariance of the equation, the symmetry center of the configuration need not be the origin. For instance, any two distinct points are covered: after translating their midpoint to the origin, they form an antipodal pair. Other examples include the vertices of a regular simplex and the vertices of a vertex-transitive regular polytope in \(\mathbb R^4\), after translating the configuration to its symmetry center.

This should be contrasted with the five-dimensional multi-bubble construction,
where the concentration scales are polynomial and the leading constants
associated with different bubbles can be adjusted according to the
configuration. In the present four-dimensional common-scale regime, the leading
exponential rate is tied to the common interaction strength, and the
point-transitive symmetry is the natural way to keep this leading rate common
for all bubbles.
\end{Rmk}
\begin{Rmk}[Interaction-driven scale law]
A main feature of the theorem is that the multi-bubble interaction is part of
the leading order dynamics. The constant \(c\) in Theorem~1.1 is obtained by
projecting the interaction terms onto the scaling direction, and the resulting
coefficient \(\xi\) determines the common concentration rate. Thus the leading
mechanism is genuinely multi-bubble, rather than a perturbation of a one-bubble
dynamics.

The estimate is formulated at the logarithmic level,
\[
\log\frac1{\lambda(t)}=\xi t^{2/3}+O(t^{1/3}),
\]
and we do not determine the multiplicative normalization of \(\lambda(t)\).
This is the natural level of precision for the construction, which controls the
logarithmic scale and a corrected normalized velocity rather than a fixed
reference scale.
\end{Rmk}

\begin{Rmk}[The borderline nature of dimension four]\label{The borderline nature of dimension four}
Dimension four is borderline for the scaling direction. Indeed, in dimension
\(d\), the ground state satisfies
\[
W(r)\sim r^{-(d-2)},\qquad \Lambda W(r)\sim r^{-(d-2)}
\]
as \(r\to\infty\). Hence
\[
\int^\infty |\Lambda W(r)|^2 r^{d-1}\,dr
\sim
\int^\infty r^{3-d}\,dr.
\]
Thus the scaling direction belongs to \(L^2\) at infinity for \(d\ge5\), while
in dimension four it has exactly a logarithmic divergence. This is the
borderline regime between the \(L^2\)-integrable scaling mode in higher
dimensions and the stronger non-integrability in lower dimensions.
\end{Rmk}

The constructive results mentioned above are complemented by rigidity results
near multi-bubble configurations. Such results ask to what extent the asymptotic signs, scales, and modulation parameters are forced once a solution remains near a prescribed
multi-bubble regime. For the energy-critical wave equation,  \cite{JendrejNonexistenceOppositeSigns}  proved a
nonexistence result for radial two-bubbles with opposite signs. More recently, a rigidity result was obtained in \cite{JendrejZhangZhao} for pure multi-bubble solutions of Jendrej--Martel in dimension five, showing that the concentration rates are uniquely determined by the bubble interactions. A closely
related energy-critical wave-type model is given by equivariant wave maps,
where the coherent objects are harmonic map bubbles and bubble-tree dynamics
arise naturally. In this setting, \cite{JendrejLawrieExpansion,JendrejLawrieUniqueness } studied threshold two-bubble dynamics and later proved refined asymptotics and uniqueness of the two-bubble solution. Thus, beyond existence, multi-bubble dynamics often exhibit a rigid asymptotic structure.

Multi-bubble and multi-soliton constructions have also been developed in other
dispersive and geometric models. In the mass-critical nonlinear Schrödinger
equation, \cite{Mkps} constructed solutions blowing up at a prescribed finite number
of points. Multi-soliton constructions for gKdV, NLS and nonlinear
Klein--Gordon equations were developed in works such as
\cite{CombetMartel,CMM,CoteMunoz,J2bNLS}. In critical geometric and parabolic problems,
bubble and bubble-tree dynamics also appear for wave maps, Yang--Mills
equations and energy-critical heat flows; see for instance
\cite{RaphaelRodnianski,Schweyer,DelPinoMussoWei,CortazarDelPinoMusso,DelPinoMussoWei3DInfiniteHeat,KriegerSchlagTataruWavemap,KriegerSchlagTataruYangmils}. These works form a broader
background for the construction of solutions whose leading dynamics is governed
by several coherent structures.
\subsection{Strategy of the proof}

We now describe the strategy of the proof. We first give an outline of the argument, introducing the main objects used in the construction. The proof is divided into four stages.

\medskip
\noindent\emph{Step 1. Formal dynamics.}
We start from a formal multi-bubble ansatz
\[
u(t)\simeq \sum_{k=1}^N \frac{1}{\lambda_k}W\left(\frac{x-\boldsymbol{y}_k(t)}{\lambda_k(t)}\right).
\]
Since the scaling direction is not square-integrable in dimension four, the
velocity component in the scaling direction has to be truncated. Thus the formal
velocity is taken in the form
\[
\partial_t u(t)
\simeq
\sum_{k=1}^N
\left(
\frac{b_k(t)}{\lambda_k^2(t)}(\Lambda W)\left(\frac{x-\boldsymbol{y}_k(t)}{\lambda_k(t)}\right)\chi\left(\frac{x}{t}\right)
+
\frac{\boldsymbol{v}_k(t)}{\lambda_k(t)}\cdot(\nabla W)\left(\frac{x-\boldsymbol{y}_k(t)}{\lambda_k(t)}\right)
\right).
\]
Where the $\chi$ is a standard smooth cut-off function satisfying 
\begin{equation*}
    \chi(x)=1\quad \text{ when }|x|\leq 1, \text{ and }\quad \chi(x)=0\quad \text{ when }|x|\geq 2.
\end{equation*}
Projecting the equation onto the scaling and translation directions gives the
leading order system for the parameters. In particular,
\[
\lambda_k'\simeq -b_k,\qquad \boldsymbol{y}_k'\simeq-\boldsymbol{v}_k
\]
and at the leading order, one obtains schematically
\[
b_k'\log \left(\frac{t}{\lambda_k}\right)\sim\kappa\sum_{j\ne k}\frac{\lambda_j}{|z_k-z_j|^2}.
\]
This is the formal origin of the condition on the points \(z_1,\ldots,z_N\): it is
the compatibility condition which allows all bubbles to follow the same
concentration rate.

\medskip
\noindent\emph{Step 2. Modulation and corrected parameters.}
We then pass from the formal ansatz to an actual modulation decomposition. On a
suitable time interval we write
\[
\vec u(t)=\vec W_{\Gamma(t)}+\vec g(t),
\]
where
\[
\Gamma(t)=(\lambda_1,b_1,\boldsymbol{y}_1,\boldsymbol{v}_1,\ldots,\lambda_N,b_N,\boldsymbol{y}_N,\boldsymbol{v}_N).
\]
The parameters are fixed by orthogonality conditions chosen so that the coercivity
estimates for the energy functionals can be applied. These conditions give
modulation estimates for \(\lambda,b,\boldsymbol{y},\boldsymbol{v}\).

However, the orthogonality conditions are not perfectly aligned with the formal
equations suggested by the ansatz. In particular, the coefficient \(b_k\) obtained
from the coercive decomposition does not directly satisfy a modulation equation
with the accuracy required later. To compensate for this mismatch, we introduce
corrected coefficients \(\beta_k\) and \(\boldsymbol{s}_k\). They are chosen so that the
projection of \(\dot g\) onto the truncated scaling and translation directions is
absorbed into the velocity variable. This yields a corrected modulation system for
\[
(\lambda,\beta,\boldsymbol{y},\boldsymbol{s}),
\]
which retains the leading dynamics obtained in the formal computation while being
compatible with the chosen orthogonality conditions. 

This modulation analysis provides two sets of estimates. The original parameters
\((\lambda,b,\boldsymbol{y},\boldsymbol{v})\) satisfy the bounds needed for the coercivity and the second
order energy estimate, while the corrected parameters \((\lambda,\beta,\boldsymbol{y},\boldsymbol{s})\)
satisfy a more accurate modulation system reflecting the formal dynamics.

\medskip
\noindent\emph{Step 3. Refined approximate solutions.}
The subsequent energy estimates are performed at two levels, and the two levels
require different parameter systems. We therefore construct two refined
approximate solutions. The correction \(\TPG\), associated with
\((\lambda,\beta,\boldsymbol{y},\boldsymbol{s})\), is used in the first order energy estimate. The approximate solution is then given by
\begin{equation*}
    \tilde{\phi}=\sum_{k=1}^N \frac{1}{\lambda_k}W\left(\frac{x-\boldsymbol{y}_k(t)}{\lambda_k(t)}\right)+\TPG.
\end{equation*}
The second correction
\(P_\Gamma\), associated with \((\lambda,b,\boldsymbol{y},\boldsymbol{v})\), is used in the second order
energy estimate and the corresponding approximate solution is denoted as
\begin{equation*}
    \phi=\sum_{k=1}^N \frac{1}{\lambda_k}W\left(\frac{x-\boldsymbol{y}_k(t)}{\lambda_k(t)}\right)+\PG.
\end{equation*}

These corrections are designed to cancel the leading residuals generated by the
modulation equations and by the interaction of different bubbles. We then estimate
the residuals obtained by substituting the two approximate solutions into the
equation. These estimates provide the error bounds needed in the energy argument.

\medskip
\noindent\emph{Step 4. Energy estimates and bootstrap closure.}
We then prove the energy estimates. The first order estimate is carried out for
the decomposition associated with $\tilde{\phi}$. In the present
argument this estimate is simplified by the second order control: the residual
bounds and the second order energy estimate provide the extra information needed to
close the first order energy.

The main part is the second order energy estimate for the decomposition associated with
$\phi$. We introduce
\[
\mathcal I_2
=
\frac12\langle \ML_\phi\dot g,\dot g\rangle
+
\frac12\langle \ML_\phi^2h,h\rangle .
\]
To close the estimate, this energy has to be supplemented by localized virial
corrections. The construction of these correction terms is the key point: their time derivatives produce the positive local control needed in the second order estimate. Thus we define
\[
\mathcal H_2
=
\mathcal I_2+\sum_{k=1}^N\mathcal J_k .
\]
The first and second order energy estimates, together with the modulation and coercivity estimates, improve the bootstrap bounds. The remaining finite-dimensional exit parameters are controlled by a standard
topological argument, see for instance \cite{wazewski}.

\medskip
\subsection{Main ingredients of the proof} We now explain the main difficulties specific to the four-dimensional multi-bubble setting. There are two main features behind the argument. The first one is the strength of
the multi-bubble interaction. In dimension four, the tails of different bubbles
interact strongly enough to enter the leading order modulation system. This leads to the stretched exponential concentration rate in the final dynamics, and also explains the compatibility condition imposed on the points $z_1,\ldots,z_N$.  This feature has already been incorporated in the formal
dynamics described above.

The second feature, which is the main analytic difficulty in the proof, is the slow
decay of the ground state and of the scaling direction. More precisely,
\[
W(r)\sim \frac{8}{r^2},\qquad
\Lambda W(r)\sim -\frac{8}{r^2},\qquad r\to\infty,
\]
and hence
\[
\int_1^R |W(r)|^2r^3\,dr\sim \log R,\qquad
\int_1^R |\Lambda W(r)|^2r^3\,dr\sim \log R.
\]
Thus both \(W\) and \(\Lambda W\) have borderline \(L^2\) tails. The rest of this
discussion explains how this slow decay affects the ansatz, the orthogonality
conditions, and the energy estimates.

\medskip
\noindent\emph{1. The growing cutoff and the need for a second order energy.}
Since the scaling direction is not square-integrable, the scaling component in
the velocity ansatz has to be truncated. The truncation radius cannot be fixed.
Indeed, both the first and the second order energy estimates contain terms
coming from the difference between the actual time derivative of the remainder
and the velocity variable chosen in the ansatz. These terms force us to estimate
the defect
\[
\ML_k((\Lambda W)\chi_R)_k .
\]
Without the cutoff, the leading contribution would vanish because
\[
\ML_k\Lambda_kW_k=0.
\]
With a cutoff, however,
\[
\ML(\Lambda W\chi_R)=[\ML,\chi_R]\Lambda W.
\]
If \(R\) is fixed, this commutator defect is not small enough. We therefore let
the truncation radius grow and use the cutoff \(\chi_t=\chi(\cdot/t)\). The
choice of the radius \(t\) is a convenient normalization, rather than a light-cone
localization. What matters is that the commutator defect produced by the
truncated scaling direction becomes lower order in the energy estimates, especially
in the terms measuring the mismatch between \(\partial_t g\) and the velocity
variable in the ansatz. The precise normalization of the growing radius is not
expected to affect the leading modulation dynamics.

The same borderline behavior also explains why a purely first order energy
argument is insufficient. The truncated scaling direction has logarithmic
\(L^2\)-size, and the first order energy does not fully exploit the cancellation
\(\ML_k\Lambda_kW_k=0\). The second order energy is introduced precisely to use
this cancellation at the level of the linearized operator. In particular, near the
\(k\)-th bubble one has
\[
\ML_\phi\simeq \ML_k,\qquad
\ML_k\Lambda_kW_k=0,\qquad
\ML_k\nabla_kW_k=0.
\]
Thus the role of the second order energy is not merely to control a stronger
norm, but to make the kernel cancellation available in the energy estimate.

\medskip
\noindent\emph{2. The choice of orthogonality conditions and corrected parameters.}
The modulation parameters are fixed by imposing orthogonality conditions on the
remainder. In particular, for the velocity part we impose orthogonality to the
truncated kernel directions, schematically
\[
\left\langle \dot g,\frac1{\lambda_k}(\Lambda_kW_k)\chi_k\right\rangle=0,
\qquad
\left\langle \dot g,\frac1{\lambda_k}(\nabla_kW_k)\chi_k\right\rangle=0 .
\]
Here $\chi_k:=\chi((\cdot-\boldsymbol{y}_k)/\lambda_k M)$, where $M$ is a large and fixed number. These orthogonality conditions are chosen
for two reasons. First, they are compatible with the coercivity estimates, and
therefore allow us to pass from estimates on the energy functionals to estimates
on the corresponding energy norms. Second, the defects created by applying the linearized operator to these truncated kernel directions are compatible with the later modulation and energy estimates, after the relevant normalization, they become small when \(M\) is chosen large. This property will be used in the second order energy
argument.

The price of this choice is that the orthogonality conditions are not perfectly
aligned with the formal modulation equations obtained from the ansatz. The
coefficient \(b_k\), \(\boldsymbol{v}_k\) fixed by the decomposition is the correct one for the coercive
orthogonality conditions, but it does not directly satisfy a modulation equation
with sufficient accuracy. To recover a sharper modulation system, we introduce
nearby corrected coefficients \(\beta_k\), and similarly \(\boldsymbol{s}_k\) for the translation
directions. The corrected coefficients remain close to the original ones, while
the pair \((\lambda,\beta)\) satisfies the refined scaling equation suggested by
the formal dynamics.

This correction does not change the orthogonality conditions imposed on
\(\dot g\). Rather, it leads to a corrected velocity variable, of the form
\[
\dot h
=
\dot g
+
\sum_k\frac{b_k-\beta_k}{\lambda_k}(\Lambda_kW_k)\chi_t
-
\sum_k\frac{v_k-s_k}{\lambda_k}\cdot\nabla_kW_k .
\]
The difference between \(\dot h\) and \(\dot g\) is controlled in \(L^2\) by the
energy bounds, so this correction is suitable for the first order estimate.

However, the same correction cannot be inserted into the second order energy.
Indeed, the scaling part of the correction has size
\[
\left\|
\frac{\beta_k-b_k}{\lambda_k}(\Lambda_kW_k)\chi_t
\right\|_{\dot H^1}
\sim
\frac{|\beta_k-b_k|}{\lambda_k},
\]
which is much larger than the scale allowed in the second order bootstrap.
Therefore the first order estimate uses the corrected velocity variable, while
the second order estimate is carried out with the original parameters
\((\lambda,b,\boldsymbol{y},\boldsymbol{v})\) and the original remainder \((g,\dot g)\).

\medskip
\noindent\emph{3. Rough modulation derivatives and the virial matching.}
The price of using the original parameters in the second order estimate is that
only rough bounds are available for \(b'\) and \(v'\). These bounds are not strong
enough by themselves. A large part of the corresponding terms is nevertheless
harmless, because the rough derivatives are coupled with the kernel directions
\(\Lambda W\) and \(\nabla W\), and are therefore eliminated by the second order
structure described above.

There remain some terms which are not removed directly by this cancellation. These
terms mainly arise from the time derivative of the refined correction
\(P_\Gamma\). Here the large cutoff \(M\) in the orthogonality conditions becomes
useful. After differentiating the orthogonality conditions, the corresponding
defect terms carry a small coefficient depending on \(M\). Schematically, if
\(\mathcal N(t)\) denotes the positive local quantity controlled in the second
order estimate, the remaining bad contribution has the form
\[
-o_M(1)\mathcal N(t),
\qquad o_M(1)\to0\quad\text{as }M\to\infty .
\]
On the other hand, the localized virial correction produces a positive local
contribution
\[
+c\mathcal N(t),\qquad c>0 .
\]
Choosing \(M\) sufficiently large allows this positive term to absorb the
\(o_M(1)\mathcal N(t)\) loss. This gives the differential inequality for the modified second order energy, and hence provides the main estimate needed to close the bootstrap.

\medskip
\noindent\emph{4. The localized virial operators.}
The localized virial correction is a standard ingredient in the multi-bubble
energy method. The delicate point here is its second order implementation in dimension four. The commutator has to provide a positive local term at the level
of the second order energy, while the four-dimensional \(\dot H^2\) Hardy
inequality is borderline and carries a logarithmic loss. For this reason the
radial weight \(q\) has to be chosen carefully. The resulting operators \(A_k\)
and \(\underline A_k\) yield the positive local term used in the second order
energy estimate.
\medskip

\medskip
\textbf{Organization of the paper.}
Section~2 derives the formal modulation system and the leading scale law. Section~3
collects the coercivity estimates. Section~4 establishes the modulation estimates
and the bootstrap framework. Section~5 constructs the refined approximate
solutions and proves the corresponding residual estimates. Section~6 proves the
first and second order energy estimates, including the localized virial correction.
Section~7 closes the bootstrap argument and proves the main theorem.

\subsection{Notations}
We first introduce the basic operators and spectral data associated with the
ground state. Set
\begin{equation*}
    \Lambda f=f+x\cdot \nabla f\text{ , and }\underline{\Lambda}f=2f+x\cdot \nabla f.
\end{equation*}
The linearized operator around \(W\) is
\begin{equation*}
    \ML g:=-\Delta g-f'(W)g=-\Delta g-3W^2 g,
\end{equation*} 
with quadratic form
\[
\langle g,\mathcal Lg\rangle
=
\int_{\mathbb R^4}\bigl(|\nabla g|^2-f'(W)g^2\bigr)\,dx .
\]
For \(\lambda>0\), we also write
\begin{equation*}
    \ML_{\lambda}g:=-\Delta g-f'(W_{\lambda})g=-\Delta g-3\frac{1}{\lambda^2}W^2(\frac{\cdot}{\lambda}) g.
\end{equation*}
 And we use the notation $\ML_k$ to denote $-\Delta g-f'(W_{\lambda_k,y_k})$.  It was shown in \cite{DM} that $\ML$ has a unique negative simple eigenvalue which we denote by $-\nu^2<0$. And the associated eigenfunction is denoted by $Y$. By elliptic regularity $Y$ is smooth and by Agmon estimates it decays exponentially.

Let $N\geq 2$ and $z_1,...,z_N$ be $N$ points of $\mathbb R^4$ distinct two by two. Set
\begin{equation*}
    d:=\frac{1}{2}\min_{j\neq k}|z_j-z_k|>0\text{ , and }\Vec{z}=(z_1,...,z_N).
\end{equation*}
The modulation parameters are denoted by
\begin{equation*}
    \Vec{\lambda}=(\lambda_1,...,\lambda_N)\in (0,\infty)^N\text{ , }\Vec{b}=(b_1,...,b_N)\in \mathbb R^N,
\end{equation*}
\begin{equation*}
  \boldsymbol{y}=(\boldsymbol{y}_1,...,\boldsymbol{y}_N)\in  (\mathbb R^4)^N\text{ , }\boldsymbol{v}=(\boldsymbol{v}_1,...,\boldsymbol{v}_N)\in  (\mathbb R^4)^N .
\end{equation*}
and denote $\Gamma=(\lambda_1,b_1,y_1,v_1...,\lambda_N,b_N,y_N,v_N)$.

For any profile \(F\) constructed from \(W\), we use the convention
\[
F_k:=F_{\lambda_k}(\cdot-\boldsymbol{y}_k)
=\frac1{\lambda_k}F\left(\frac{\cdot-\boldsymbol{y}_k}{\lambda_k}\right).
\]
When an operator carries the subscript \(k\), it is applied before scaling and
translation. Thus
\begin{equation*}
    \begin{aligned}
        W_k&:=W_{\lambda_k}(\cdot-\boldsymbol{y}_k)=\frac{1}{\lambda_k}W\Big(\frac{\cdot-\boldsymbol{y}_k}{\lambda_k}\Big),\\
        \nabla_kW_k&:=(\nabla W)_{\lambda_k}(\cdot-\boldsymbol{y}_k)=\frac{1}{\lambda_k}\nabla W\Big(\frac{\cdot-\boldsymbol{y}_k}{\lambda_k}\Big),
    \end{aligned}
\end{equation*}
and similarly
\begin{equation*}
\begin{aligned}
    \Lambda_k W_k&:=(\Lambda W)_{\lambda_k}(\cdot-\boldsymbol{y}_k)\text{ , }\Delta_k\Lambda_k W_k :=(\Delta \Lambda W)_{\lambda_k}(\cdot-\boldsymbol{y}_k),\\
   \underline{\Lambda}_k\Lambda_kW_k&:=(\underline{\Lambda} \Lambda W)_{\lambda_k}(\cdot-\boldsymbol{y}_k) \text{ and }\Delta_k\nabla_k W_k:=(\Delta \nabla W)_{\lambda_k}(\cdot-\boldsymbol{y}_k).\\
\end{aligned}
\end{equation*}

Throughout the paper, \(\chi\) denotes a smooth radial cut-off satisfying
\[
\chi(x)=1\quad\text{for }|x|\le1,\qquad
\chi(x)=0\quad\text{for }|x|\ge2.
\]
For \(M>0\), \(1\le k\le N\), and \(t>0\), set
\[
\chi_M:=\chi\left(\frac{\cdot}{M}\right),\qquad
\chi_k:=\chi\left(\frac{\cdot-\boldsymbol{y}_k}{\lambda_kM}\right),\qquad
\chi_t:=\chi\left(\frac{\cdot}{t}\right).
\]

We also denote
\begin{equation*}
    Y_k:=\frac{1}{\lambda_k}Y\left(\frac{x-\boldsymbol{y}_k}{\lambda_k}\right),
\end{equation*}
and define the localized stable and unstable directions by
\begin{equation}\label{Z_k and Y_k}
    \boldsymbol{y}_k^{\pm}=(\nu^{-1}Y_k,\pm \lambda_k^{-1}Y_k),\qquad\Vec{Z}_k^{\pm}=\frac{1}{2}\lambda_k^{-1}(\nu\lambda_k^{-1}Y_k,\pm Y_k),
\end{equation}
so that
\begin{equation*}
    \lelan \boldsymbol{y}_k^{\pm},\Vec{Z}_k^{\pm} \rilan=1,\qquad\lelan \boldsymbol{y}_k^{\pm}, \Vec{Z}_k^{\mp} \rilan=0.
\end{equation*}

The multi-bubble profiles are
\begin{equation*}
    W_{\boldsymbol{\Gamma}}=\sum_k W_k,\qquad \Vec{W}_{\boldsymbol{\Gamma}}=\sum_k\Vec{W}_k,
\end{equation*}
where
\begin{equation*}
  \Vec{W}_k=\left(W_k,\frac{b_k}{\lambda_k}(\Lambda_k W_k)\chi_{t}+\frac{\boldsymbol{v}_k\cdot\nabla_kW_k}{\lambda_k}\right).  
\end{equation*}
We also denote the linearized operator near the multi-bubble by
\begin{equation*}
    \ML_{W_{\boldsymbol{\Gamma}}}g:=\left(-\Delta -3\left(\sum_{k}W_k\right)^2\right) g.
\end{equation*}

All inner products \(\langle\cdot,\cdot\rangle\) are taken in \(L^2(\mathbb R^4)\),
unless otherwise specified. For \(\vec g=(g,\dot g)\), define
\[
\|\vec g\|_{\mathcal E}:=\|\vec g\|_{\dot H^1\times L^2},
\qquad
\|\vec g\|_{\mathcal E_2}:=\|g\|_{\dot H^2}
+\|\mathcal L_{W_\Gamma}g\|_{L^2}
+\|\dot g\|_{\dot H^1}.
\]
Finally, the space used in the construction is
\[
X:=(\dot H^1\cap\dot H^2)(\mathbb R^4)
\times (L^2\cap\dot H^1)(\mathbb R^4).
\]

\textbf{Acknowledgments}
    		J. Jendrej was supported by the ERC project INSOLIT (No. 101117126). L. Zhao was supported by National Natural Science Foundation of China (No. 12271497 and No. 12341102).
\section{Formal computation}\label{A formal computation}
We begin with a formal computation which explains the leading modulation
system. At this stage we write the ansatz with local parameters attached to
each bubble. The parameter \(\lambda_k(t)>0\) denotes the scale of the \(k\)-th
bubble, \(\boldsymbol y_k(t)\in\mathbb R^4\) denotes its center, \(b_k(t)\) is
the scaling velocity parameter, and \(\boldsymbol v_k(t)\in\mathbb R^4\) is the
translation velocity parameter. We set the ansatz
\begin{equation}\label{basic ansatz 0}
    \boldsymbol{u}_0:=(U^{(0)},U^{(1)})
\end{equation}
where
\[
U^{(0)}
:=
\sum_{k=1}^N
\frac1{\lambda_k}
W\left(\frac{x-\boldsymbol y_k}{\lambda_k}\right),
\]
and
\[
U^{(1)}
:=
\sum_{k=1}^N
\left[
\frac{b_k}{\lambda_k^2}
(\Lambda W)\left(\frac{x-\boldsymbol y_k}{\lambda_k}\right)
\chi\left(\frac{x}{t}\right)
+
\frac{\boldsymbol v_k}{\lambda_k^2}\cdot
(\nabla W)\left(\frac{x-\boldsymbol y_k}{\lambda_k}\right)
\right].
\]
Here \(\chi\) is a fixed smooth cut-off with $\chi(x)=1$ when $|x|\leq 1$ and $\chi(x)=0$ when $|x|\geq 2$. The truncation in the scaling
component is needed because \(\Lambda W\notin L^2(\mathbb R^4)\). Formally,
\(U^{(1)}\) should be compared with \(\partial_t U^{(0)}\). Since
\[
\partial_t
\left[
\frac1{\lambda_k}
W\left(\frac{x-\boldsymbol y_k}{\lambda_k}\right)
\right]
=
-\frac{\lambda_k'}{\lambda_k^2}
(\Lambda W)\left(\frac{x-\boldsymbol y_k}{\lambda_k}\right)
-
\frac{\boldsymbol y_k'}{\lambda_k^2}\cdot
(\nabla W)\left(\frac{x-\boldsymbol y_k}{\lambda_k}\right),
\]
this suggests the relations
\[
b_k\simeq-\lambda_k',
\qquad
\boldsymbol v_k\simeq-\boldsymbol y_k'.
\] 
The actual construction will later be performed in the symmetry class of the
configuration. More precisely, we introduce the following notation.
\begin{cond}[Symmetric configuration]\label{sym condition}
There exists a finite subgroup \(G\subset O(4)\) which acts transitively on
\[
Z=\{z_1,\ldots,z_N\}.
\]
For \(\rho\in G\), we denote by \(\pi_\rho\) the permutation determined by
\[
\rho z_k=z_{\pi_\rho(k)}.
\]
The construction is carried out in the \(G\)-invariant class. In this class the modulation parameters are restricted to the equivariant
modulation manifold:
\[
\lambda_{\pi_\rho(k)}=\lambda_k,\qquad b_{\pi_\rho(k)}=b_k,
\]
and
\[
\boldsymbol y_{\pi_\rho(k)}=\rho\boldsymbol y_k,\qquad
\boldsymbol v_{\pi_\rho(k)}=\rho\boldsymbol v_k.
\]
In particular, by transitivity,
\begin{equation*}
 \lambda_1=\cdots=\lambda_N=:\lambda,
\qquad
b_1=\cdots=b_N=:b.   
\end{equation*}   
\end{cond}

The translation parameters are not equal as vectors; rather, the whole
\(N\)-tuples \((\boldsymbol y_1,\ldots,\boldsymbol y_N)\) and
\((\boldsymbol v_1,\ldots,\boldsymbol v_N)\) are equivariant under the group
action. In the following formal computation we keep the indices \(k\) in the
local projection formulas, and only afterwards impose this symmetric reduction.
Next, inserting (\ref{basic ansatz 0}) into equation (\ref{NLW 4}), on the one hand, a direct computation formally gives the second order time derivative terms as
\begin{equation}\label{basic parameter left hand}
\begin{aligned}
    \partial_{tt}u=\frac{d}{dt}U^{(1)}=&\sum_{k=1}^N(b_k'(t)(\Lambda W)_{\underline{\lambda_k,y_k}}\chi\left(\frac{x}{t}\right) + \frac{b_k(t)^2}{\lambda(t)}(\underline{\Lambda}(\Lambda W))_{\underline{\lambda_k,y_k}} )\\
    &-\sum_{k=1}^N b_k(t)\frac{x}{t^2}\chi'\left(\frac{x}{t}\right)(\Lambda W)_{\underline{\lambda_k,y_k}}+\sum_{k=1}^N \sum_{j=1}^4 b_k(t)\frac{v_{k,j}}{t}\partial_j\chi\left(\frac{x}{t}\right)(\Lambda W)_{\underline{\lambda_k,y_k}}\\
    &+\sum_{k=1}^N \sum_{j=1}^4 b_k(t)\frac{v_{k,j}}{\lambda_k}\chi\left(\frac{x}{t}\right)(\partial_j\Lambda W)_{\underline{\lambda_k,y_k}}
    +\sum_{k=1}^N\sum_{j=1}^4\frac{v'_{k,j}(t)\partial_jW(\frac{\cdot-\boldsymbol{y}_k}{\lambda_k})}{\lambda_k^2}\\&+\sum_{k=1}^N\sum_{j=1}^4\frac{b_kv_{k,j}(\underline{\Lambda} \partial_j W)\left(\frac{\cdot-\boldsymbol{y}_k}{\lambda_k}\right)}{\lambda_k^3}
   -\sum_{k=1}^N\sum_{i,j=1}^4\frac{v_{k,j}(t)v_{k,i}(t)\partial_j\partial_iW(\frac{\cdot-\boldsymbol{y}_k}{\lambda_k})}{\lambda_k^3}.
\end{aligned}
\end{equation}
On the other hand, the remaining term in equation (\ref{NLW 4}) can be expressed by
\begin{equation}\label{basic paramter right hand}
    \Delta u+f(u)=\Delta U^{(0)}+f(U^{(0)}).
\end{equation}
To derive the equation that the parameters satisfy, we first separately take the inner products of (\ref{basic parameter left hand}) and (\ref{basic paramter right hand}), respectively with $\frac{1}{\lambda_k}\Lambda W(\frac{\cdot-\boldsymbol{y}_k}{\lambda_k})$. The leading terms are 
\begin{equation*}
   \begin{aligned}
    &\lelan \partial_{tt}u,\frac{1}{\lambda_k}\Lambda W(\frac{\cdot-\boldsymbol{y}_k}{\lambda_k}) \rilan\\
    \sim&
    b_k'(t)\lambda_k(t)\lelan \Lambda W, \Lambda W\chi (\frac{\cdot \lambda_k}{t}) \rilan+\sum_{j\neq k}b_j'(t)\lambda_k(t)\lelan \Lambda W(\frac{\cdot-\boldsymbol{y}_k}{\lambda_k}), \frac{1}{\lambda_j}(\Lambda W)(\frac{\cdot-\boldsymbol{y}_j}{\lambda_j})\chi_t  \rilan
    \\&+b_k(t)^2\lelan \underline{\Lambda}\Lambda W,\Lambda W \rilan
     -\frac{b_k(t)\lambda_k(t)}{t^2}\lelan \Lambda W,\Lambda W \chi'\rilan-\sum_{j=1}^4v_{k,j}^2\lelan \partial_{jj}W,\Lambda W \rilan,
    \end{aligned}
\end{equation*}
and
\begin{equation}
    \lelan \Delta u+f(u),\frac{1}{\lambda_k}\Lambda_k W_k \rilan\sim \lelan \sum_{j\neq k}f'(W_k)W_j,\frac{1}{\lambda_k}\Lambda_k W_k  \rilan= \sum_{j\neq k}\frac{8\lambda_j}{|z_k-z_j|^2}\lelan 3W^2,\Lambda W \rilan+O(\lambda_j^3).\label{est:W^2 Lambda W}
\end{equation}
A direct computation using the tail \(\Lambda W(r)=-8r^{-2}+O(r^{-4})\) gives
\begin{equation}
\label{eq:inner-prod-1}
    \lelan \Lambda W, \Lambda W\chi (\frac{\cdot \lambda}{t}) \rilan=128\pi^2\left(
\log\frac{1}{\lambda}
\right)
+
O\left(\log t\right),
\end{equation}
\begin{equation}
\label{eq:inner-prod-2}
\begin{aligned}
   \lelan \frac{1}{\lambda^2}\Lambda W(\frac{\cdot-\boldsymbol{y}_k}{\lambda}), \frac{1}{\lambda^2}(\Lambda W)(\frac{\cdot-\boldsymbol{y}_j}{\lambda})\chi_t  \rilan =128\pi^2\left(
\log\frac{t}{|\boldsymbol{y}_k-\boldsymbol{y}_j|}
\right)
+O\left(1\right)
\end{aligned}
\end{equation}
and
\begin{equation*}
  \lelan \underline{\Lambda}\Lambda W,\Lambda W \rilan=64\pi^2.  
\end{equation*}
Moreover, we denote 
\begin{equation*}
    \Xi_0 (t):=\frac{1}{128\pi^2}\left(\lelan \Lambda W, \Lambda W\chi (\frac{\cdot \lambda}{t}) \rilan+\sum_{j\neq k} \lelan \frac{1}{\lambda^2}\Lambda W(\frac{\cdot-\boldsymbol{y}_k}{\lambda}), \frac{1}{\lambda^2}(\Lambda W)(\frac{\cdot-\boldsymbol{y}_j}{\lambda})\chi_t  \rilan\right).
\end{equation*}
In particular, the computation (\ref{eq:inner-prod-1}), (\ref{eq:inner-prod-2}) and the symmetry condition \ref{sym condition} yield
\begin{equation}
\Xi_0(t)
=
\log\frac1{\lambda(t)}+O(\log t).
\label{Xi-alpha}
\end{equation}
Furthermore, we derive the following parameter system
\begin{equation}\label{equation of lambda and b}
    \begin{cases}
        \lambda'(t)=-b(t)\\
         \Xi_0(t)b'(t)+\frac{1}{2}\frac{b(t)^2}{\lambda(t)}=B_k(\boldsymbol{\lambda},\boldsymbol{y}),
    \end{cases}
\end{equation}
where
\begin{equation*}
    B_k(\boldsymbol
{\lambda},\boldsymbol{y})=\kappa \sum_{j\neq k}\left\{\lambda|\boldsymbol{y}_j-\boldsymbol y_k|^{-2}\right\}\text{ , }\kappa=\frac{24}{128\pi^2}\lelan W^2,\Lambda W \rilan=-2.
\end{equation*}
Indeed, we will prove later that $|\boldsymbol{y}_k-\boldsymbol{z}_k|\ll \lambda^{\frac{3}{2}}$. 
Define \(c>0\) by
\[
c=2\sum_{j\neq k}|\boldsymbol{z}_j-\boldsymbol{z}_k|^{-2} .
\]
Then by condition \ref{sym condition}, $B_k(\boldsymbol{\lambda}(t))$ is independent of $k$ and satisfies $B_k(\boldsymbol{\lambda},\boldsymbol{y})=-c\lambda+o(\lambda^{\frac{3}{2}})$.  With these notations we derive the following lemma.
\begin{lemma}[Formal scale law]\label{lem:formal-scale-law}
Assume that, after the symmetric reduction, the leading scale equation takes the form
\begin{equation*}
    \Xi_0(t)\lambda''-\frac12\frac{(\lambda')^2}{\lambda}
    =
    c\lambda,
\end{equation*}
where $c$ and $\Xi$ is given as above. Set
\[
    \alpha(t):=\log\frac1{\lambda(t)}.
\]
Then the formal leading scale is
\begin{equation}\label{eq:alpha-formal-law}
    \alpha(t)=\xi t^{2/3}+O(\log t),
    \qquad
    \xi:=\left(\frac{9c}{4}\right)^{1/3}.
\end{equation}
\end{lemma}
\begin{proof}
Since \(\lambda=e^{-\alpha}\), we have
\[
    \frac{\lambda''}{\lambda}=(\alpha')^2-\alpha'',
    \qquad
    \frac{(\lambda')^2}{\lambda^2}=(\alpha')^2.
\]
Thus
\[
    \Xi_0(t)\big((\alpha')^2-\alpha''\big)
    -\frac12(\alpha')^2
    =
    c.
\]
The leading balance is
\[
    \alpha(\alpha')^2\simeq c,
\]
which gives
\[
    \alpha(t)\sim \xi t^{2/3},
    \qquad
    \xi=\left(\frac{9c}{4}\right)^{1/3}.
\]
Combining this with the asymptotics of $\Xi$ (\ref{Xi-alpha}) gives (\ref{eq:alpha-formal-law}).
\end{proof}
Next we turn to sketch the system of $(\boldsymbol{y}_k,\boldsymbol{v}_k)$. To achieve this,
 taking the inner product with $\frac{1}{\lambda_k}\partial_jW(\frac{\cdot-\boldsymbol{y}_k}{\lambda_k})$  yields
\begin{equation*}
   \begin{aligned}
    &\lelan \partial_{tt}u,\frac{1}{\lambda_k}\partial_jW(\frac{\cdot-\boldsymbol{y}_k}{\lambda_k}) \rilan\\
    \sim& 
     \frac{b_kv_{k,j}}{t}\lelan\partial_j\chi\left(\frac{x}{t}\right)(\Lambda W)_{\underline{\lambda_k,y_k}},\frac{1}{\lambda_k}\partial_jW(\frac{\cdot-\boldsymbol{y}_k}{\lambda_k}\rilan\\
     &+{b_kv_{k,j}}\lelan\chi\left(\frac{x}{t}\right)(\partial_j\Lambda W)_{\underline{\lambda_k,y_k}},\frac{1}{\lambda_k}\partial_jW(\frac{\cdot-\boldsymbol{y}_k}{\lambda_k}\rilan+(v_{k,j}'(t)\lambda_k(t))\lelan\partial_j W,\partial_j W \rilan.
    \end{aligned}
\end{equation*}
Here we used the symmetry property and the fact $\lelan \underline{\Lambda}(\nabla W), \nabla W   \rilan=0$. Indeed, we can also regard the term 
\begin{equation*}
     b_k(t)\frac{v_{k,j}}{t}\lelan\partial_j\chi\left(\frac{x}{t}\right)(\Lambda W)_{\underline{\lambda_k,y_k}},\frac{1}{\lambda_k}\partial_jW(\frac{\cdot-\boldsymbol{y}_k}{\lambda_k}\rilan\sim \frac{b_k v_k}{t}
\end{equation*}
as a small perturbation. Furthermore, notice that $\partial_j\Lambda f= \underline{\Lambda}\partial_j f$, the first term on the last line above is bounded by 
\begin{equation*}
    \int_{|x-\boldsymbol{y}_k|\geq t}\left(\frac{\lambda_k}{
    |x-\boldsymbol{y}_k|^3}\right)^2dx\lesssim \frac{\lambda_k^2}{t^2},
\end{equation*}
also  negligible. Consequently, we arrive at 
\begin{equation*}
  \lelan \partial_{tt}u,\frac{1}{\lambda_k}\partial_jW(\frac{\cdot-\boldsymbol{y}_k}{\lambda_k}) \rilan
    \sim  (v_{k,j}'(t)\lambda_k(t))\lelan\partial_j W,\partial_j W \rilan.
\end{equation*}
While on the right-hand side, we have
\begin{equation*}
     \lelan \Delta u+f(u),\frac{1}{\lambda_k}\partial_jW(\frac{\cdot-\boldsymbol{y}_k}{\lambda_k}) \rilan \sim \lelan \sum_{i\neq k}3W_{\lambda_k,y_k}^2W_{\lambda_i,y_i},\frac{1}{\lambda_k}\partial_jW(\frac{\cdot-\boldsymbol{y}_k}{\lambda_k}) \rilan\sim  \lambda_k^2\lambda_i.
\end{equation*}

Therefore, we get the system
\begin{equation}\label{system for y v}
    \begin{cases}
        {y}_{k,j}'(t)=-{v}_{k,j}(t)\\
        {v}_{k,j}'(t)=D_{k,j}(\boldsymbol{\lambda}(t)),
    \end{cases}
\end{equation}
where we denote
\begin{equation*}
    D_{k,j}(\boldsymbol{\lambda})=192 \sum_{i\neq k}\left\{\lambda_i\lambda_k\frac{z_{k,j}-z_{i,j}}{|z_k-z_i|^4}\right\}
\end{equation*}
$|\boldsymbol{v}_k'|\sim \lambda^2$, $|\boldsymbol{v}_k|\sim t^{\frac{1}{3}}\lambda^2$. For notational convenience, we also denote
\begin{equation*}
    \boldsymbol{D}_{k}(\boldsymbol{\lambda})=192 \sum_{i\neq k}\left\{\lambda_i\lambda_k\frac{\Vec{z}_k-\Vec{z}_i}{|z_k-z_i|^4}\right\}.
\end{equation*}

\section{Coercivity estimate}
This section is devoted to establishing coercivity estimates for $\lelan \ML g,g \rilan$ and $\lelan \ML g,\ML g \rilan$. These estimates will serve as the foundation for the first and second order energy estimates developed later. We begin with the estimate for $\lelan  \ML g,g\rilan$.
\subsection{Coercivity estimate for $\ML $}
We shall use the standard spectral coercivity of the linearized operator $\ML$ around \(W\), stated in the following lemma.
\begin{lemma}[\cite{DM}, Proposition 5.5]\label{DM coercivity} The operator $\ML$ on $L^2$ with domain $H^2$ is a self-adjoint operator with essential spectrum $[0,+\infty)$, no positive eigenvalue and only one negative eigenvalue $-\nu^2$. Furthermore, if we denote
\begin{equation*}
G_{\perp}:=\left\{f\in \dot{H}^1,\int_{\mathbb R^4}Yf =\int_{\mathbb R^4}\nabla f\cdot \nabla \Lambda W=\int_{\mathbb R^4}\nabla f\cdot\nabla \partial_{x_1} W=...=\int_{\mathbb R^4}\nabla f\cdot\nabla \partial_{x_4} W=0\right\}.
\end{equation*}
Then there exists $c>0$ such that for any $g\in G_{\perp}$, $ \lelan \ML g, g \rilan\geq c\left\|\nabla g\right\|_{L^2}^2.$
\end{lemma}
This is the standard coercivity estimate for the linearized operator around the ground state; see \cite{DM}. Using this we derive the following proposition.

\begin{prop}[Coercivity estimate for $\ML$]\label{coercivity in our orthogonal for L} There exists $\eta>0$ such that for any $g \in \dot{H}^1(\mathbb R^4)$ we have
\begin{equation}\label{Lg g coercivity order 1}
\begin{aligned}
       \int_{\mathbb R^4}(|\nabla g|^2-f'(W)g^2)dx\geq& \eta \left\|\nabla g\right\|_{L^2}^2-\{(\nu^2+1)\lelan Y,g \rilan^2\\
       &+\lelan (\Delta \Lambda W), g \rilan^2+|\lelan (\Delta \nabla W),g \rilan|^2\}.
\end{aligned}
\end{equation}
\end{prop}
\begin{proof}

 This follows from the standard coercivity of \(\ML\). By the standard spectral coercivity of \(\ML\), after adding the negative direction \(Y\), the quadratic form associated with 
\[
\widetilde \ML=\ML+\nu^2\langle Y,\cdot\rangle Y
\]
is nonnegative and its kernel is 
\(\operatorname{span}\{\Lambda W,\partial_{x_1}W,\ldots,\partial_{x_4}W\}\).
 
 Indeed, it suffices to check that the functionals
\[
g\mapsto \langle Y,g\rangle,\quad
g\mapsto \langle\Delta\Lambda W,g\rangle,\quad
g\mapsto \langle\Delta\nabla W,g\rangle
\]
are non-degenerate on the space
\(\operatorname{span}\{Y,\Lambda W,\nabla W\}\).
By symmetry the corresponding \(6\times6\) matrix is triangular, and its diagonal entries are nonzero. Hence the usual compactness-contradiction argument gives the claim.   
\end{proof}

\begin{Rmk}\label{eta irrelevant to choice of M}
We note that $\Delta \Lambda W\notin \dot{H}^{-2}(\mathbb R^4)$. To preserve the same structure in the second order coercivity estimate which will be introduced later, we need to rewrite (\ref{Lg g coercivity order 1}). Since $\Delta \Lambda W\in \dot{H}^{-1}(\mathbb R^4)$, there exists $M_0>0$, such that for any $M\geq M_0$ and $g\in\dot{H}^1(\mathbb R^4)$, it holds
\begin{equation*}\label{Lg g coercivity order 1 neo}
\begin{aligned}
       \int_{\mathbb R^4}(|\nabla g|^2-f'(W)g^2)dx\geq& \frac{\eta}{2} \left\|\nabla g\right\|_{L^2}^2-((\nu^2+1)\lelan Y,g \rilan^2\\
       &+\lelan (\Delta \Lambda W)\chi_M, g \rilan^2+|\lelan (\Delta \nabla W),g \rilan|^2).
\end{aligned}
\end{equation*} 
\end{Rmk}
The orthogonal conditions on the right-hand-side of (\ref{Lg g coercivity order 1}) are closely related to the first order energy estimate. 
We now derive an analogous estimate in a different functional framework in order to carry out the second-order energy analysis. 
\begin{coro}\label{L dot g dot g coercivity}
There exists $\eta>0$  and $M\geq 1$, such that for any $g \in \dot{H}^1(\mathbb R^4)\cap L^2(\mathbb R^4)$, we have
\begin{equation*}\label{coro of corcervity}
\begin{aligned}
  \int_{\mathbb R^4}(|\nabla g|^2-f'(W)g^2)dx\geq& \eta \left\|\nabla g\right\|_{L^2}^2-((\nu^2+1)\lelan Y,g \rilan^2\\
     &+\lelan (\Lambda W)\chi_{M}, g \rilan^2+|\lelan (\nabla W)\chi_{M},g \rilan|^2 ).   
\end{aligned}
\end{equation*}  
\end{coro}
\begin{proof}
The proof is identical to that of Proposition \ref{coercivity in our orthogonal for L}. The only point is that, for \(M\) sufficiently large,
\[
\langle(\Lambda W)\chi_M,\Lambda W\rangle\neq0,\qquad
\langle(\partial_jW)\chi_M,\partial_jW\rangle\neq0.
\]
By symmetry the corresponding matrix is triangular up to the harmless entry
\(\langle(\Lambda W)\chi_M,Y\rangle\), hence non-singular.
\end{proof}
 Next, we turn to the multiple potential case. For $\lambda ,\mu>0$ and $x,y\in \mathbb R^4$, we denote
\begin{equation*}
    \delta ((\lambda,x),(\mu,y)):=|\log(\frac{\lambda}{\mu})|+\frac{|x-y|}{\lambda}.
\end{equation*}
We say that two sequences $(\lambda_n,x_n)$ and $(\mu_n,y_n)$ are orthogonal if
\begin{equation*}
    \lim_{n\to \infty}\delta ((\lambda_n,x_n),(\mu_n,y_n))=\infty.
\end{equation*}
Let $K\geq 1$; in what follows $\sum_k$ denotes $\sum_{k=1}^K$. For $(\lambda^{(k)},x^{(k)})\in (0,\infty)\times \mathbb R^4$, we use the notation
\begin{equation*}
    W^{(k)}(x):=(\lambda^{(k)})^{-1}W((x-x^{(k)})/\lambda^{(k)})
\end{equation*}
and similarly for other functions. Especially, for $M>0$, we denote
\begin{equation*}
   \chi_M^{(k)}:=\chi((x-x_k)/\lambda^{(k)}M).
\end{equation*}
With these notations, we have the following coercivity lemma near multi-bubble.
\begin{lemma}\label{near multi bubble coercivity}
    There exist $\eta>0$ and $M\geq 1$ such that the following holds. Let $(\lambda^{(k)},x^{(k)})\in (0,\infty)\times \mathbb R^4$ for $k=1,...,K$ satisfy 
    \begin{equation*}
        \left\|U-\sum_k W^{(k)}\right\|_{\dot{H}^1}\leq \eta.
    \end{equation*}
   Then for any $g\in \dot{H}^1(\mathbb R^4)\cap \dot{H}^2(\mathbb R^4)$ and $h\in \dot{H}^1(\mathbb R^4)\cap L^2(\mathbb R^4)$
\begin{equation*}
\begin{aligned}
    \int_{\mathbb R^4}(|\nabla g|^2-f'(U)g^2)dx\geq& \eta \left\|\nabla g\right\|_{L^2}^2 -\sum_k\Big\{(\nu^2+1)\lelan  (\lambda^{(k)})^{-2}Y^{(k)},g \rilan^2 \\
    &+\lelan (\lambda^{(k)})^{-2}(\Delta \Lambda W)^{(k)}\chi_M^{(k)}, g \rilan^2+|\lelan (\lambda^{(k)})^{-2}(\Delta \nabla W)^{(k)},g \rilan|^2\Big\},
\end{aligned}
\end{equation*}
 and 
 \begin{equation*}
 \begin{aligned}
 \int_{\mathbb R^4}(|\nabla h|^2-f'(U)h^2)dx\geq& \eta \left\|\nabla h\right\|_{L^2}^2 -\sum_k\Big\{(\nu^2+1)\lelan  (\lambda^{(k)})^{-2}Y^{(k)},h \rilan^2 
    \\&+\lelan (\lambda^{(k)})^{-2}( \Lambda W)^{(k)}\chi_M^{(k)}, h \rilan^2+|\lelan (\lambda^{(k)})^{-2}( \nabla W)^{(k)}\chi_M^{(k)},h \rilan|^2\Big\}.
\end{aligned}
\end{equation*}
\end{lemma}
\begin{proof}
The proof follows from the standard profile decomposition argument and the one-bubble coercivity estimates above, exactly as in the multi-bubble coercivity lemma in \cite{JM}
We omit the details.
\end{proof}
Finally, as a complement, we also introduce the truncated coercivity lemma for $\lelan \ML g,g\rilan$, which is stated as follows.
\begin{lemma}\label{Lemma 9 of 5D multi}
  For any $\eta>0$ there exists $R=R(\eta)>0$ such that for all $g\in \dot{H}^1(\mathbb R^4)$,
\begin{equation*}
    \int_{|x|\leq R}|\nabla g|^2dx-\int_{\mathbb R^4}f'(W)g^2dx\geq -\eta \left\|\nabla g\right\|_{L^2}^2-\nu^2\lelan Y,g \rilan^2.
\end{equation*}
\end{lemma}
Lemma \ref{Lemma 9 of 5D multi} can be proved as in the corresponding truncated coercivity lemma in \cite{JM}. We will use it to handle the correction term in first order energy estimates.

\subsection{Coercivity estimate for $\ML^2$ }

In our proof, the second order energy will play a necessary role. Therefore, we also need to establish the coercivity estimate for $\int_{\mathbb R^4}(\ML_k g)^2dx$. As the beginning, we observe a rough computation yields the following lemma.
\begin{lemma}\label{coercivity proposition for second order operator} For any $g\in \dot{H}^1\cap \dot{H}^2$ and $\lambda_k>0$, we have
    \begin{equation*}\label{L g^2 coercivity}
    \left\|g\right\|_{\dot{H}^2}^2+ \int_{\mathbb R^4}W^4g^2dx+\int_{\mathbb R^4}|\nabla W|^2g^2dx\lesssim  \left\|\ML g\right\|_{L^2}^2+\left\|g\right\|_{\dot{H}^1}^2,
    \end{equation*}
    and
\begin{equation*}\label{Lk g^2 coercivity}
  \left\|g\right\|_{\dot{H}^2}^2+ \int_{\mathbb R^4}W_k^4g^2dx+\int_{\mathbb R^4}|\nabla (W_k)|^2g^2dx\lesssim  \left\|\ML_k g\right\|_{L^2}^2+\frac{1}{\lambda_k^2}\left\|g\right\|_{\dot{H}^1}^2.  
\end{equation*}    
\end{lemma}
\begin{proof}
A direct computation yields that
   \begin{equation*}
   \begin{aligned}
       &\int_{\mathbb R^4}(-\Delta g-f'(W)g)^2dx\\
       =&\int_{\mathbb R^4}(\Delta g)^2dx+6\int_{\mathbb R^4}W^2g\Delta gdx + \int_{\mathbb R^4}9W^4g^2dx\\
       =& \int_{\mathbb R^4}(\Delta g)^2dx+ \int_{\mathbb R^4}9W^4g^2dx-6\int_{\mathbb R^4}W^2|\nabla g|^2dx-12\int_{\mathbb R^4}gW(\nabla g\cdot \nabla W) dx
   \end{aligned}
    \end{equation*}
For the last term, integration by parts yields 
\begin{equation*}\begin{aligned}
     -12\int_{\mathbb R^4}gW(\nabla g\cdot \nabla W) dx =&6 \int_{\mathbb R^4}|\nabla W|^2g^2dx+6 \int_{\mathbb R^4}W(\Delta W)g^2\\=&6 \int_{\mathbb R^4}|\nabla W|^2g^2dx+6 \int_{\mathbb R^4}W^4g^2.
\end{aligned}
\end{equation*}
Plugging this back we have 
\begin{equation}\label{subcoercivity for L g^2}
\begin{aligned}
   &\int_{\mathbb R^4}(-\Delta g-f'(W)g)^2dx\\=& \int_{\mathbb R^4}(\Delta g)^2dx+ 9\int_{\mathbb R^4}W^4g^2dx
   +6 \int_{\mathbb R^4}|\nabla W|^2g^2dx+6 \int_{\mathbb R^4}W^4g^2-6\int_{\mathbb R^4}W^2|\nabla g|^2dx
\end{aligned}
\end{equation}
Using the pointwise decay of $W$, this term is bounded below by 
\begin{equation*}
   \int_{\mathbb R^4}(\Delta g)^2dx+ \int_{\mathbb R^4}W^4g^2dx+\int_{\mathbb R^4}|\nabla W|^2g^2dx-C\int_{\mathbb R^4}|\nabla g|^2dx, 
\end{equation*}
and as a result it holds
\begin{equation*}
 \int_{\mathbb R^4}(\Delta g)^2dx+ \int_{\mathbb R^4}W^4g^2dx+\int_{\mathbb R^4}|\nabla W|^2g^2dx\lesssim  \left\|\ML g\right\|_{L^2}^2+\left\|g\right\|_{\dot{H}^1}^2.
\end{equation*}
For any $\lambda_k>0$, the following inequality directly yields
\begin{equation}\label{Lg norm and dot 1 norm together}
 \int_{\mathbb R^4}(\Delta g)^2dx+ \int_{\mathbb R^4}W_k^4g^2dx+\int_{\mathbb R^4}|\nabla (W_k)|^2g^2dx\lesssim  \left\|\ML_k g\right\|_{L^2}^2+\frac{1}{\lambda_k^2}\left\|g\right\|_{\dot{H}^1}^2.
\end{equation}
Moreover, using the Riesz transform,  the left hand side of (\ref{Lg norm and dot 1 norm together}) can be refined as 
\begin{equation*}\label{Lg norm and dot 1 norm together 2.0}
 \int_{\mathbb R^4}\sum_{1\leq i,j\leq 4}|\partial_i\partial_j g|^2dx+ \int_{\mathbb R^4}W_k^4g^2dx+\int_{\mathbb R^4}|\nabla (W_k)|^2g^2dx\lesssim  \left\|\ML_k g\right\|_{L^2}^2+\frac{1}{\lambda_k^2}\left\|g\right\|_{\dot{H}^1}^2,
\end{equation*}
which completes the proof.
\end{proof}
 However, due to the bootstrap assumption which will be introduced later. Lemma \ref{coercivity proposition for second order operator} is not always sufficient. To fix this, similar to Proposition \ref{coercivity in our orthogonal for L}, the following coercivity property holds for $\lelan \ML g,\ML g \rilan$.

\begin{prop}[Coercivity estimate with $\lelan \ML g,\ML g \rilan$]\label{Coercivity estimate with LG LG} Let $M\geq 1$ be a large constant, then there exists $\eta(M)>0$ such that for any $g \in \dot{H}^1(\mathbb R^4)\cap \dot{H}^2(\mathbb R^4)$, we have
\begin{equation}\label{second order coercivity for L with H 2}
\begin{aligned}
 \int_{\mathbb R^4}(\ML g)^2dx \geq& \eta \left(\left\|g\right\|_{\dot{H}^2}^2+\int_{\mathbb R^4}\frac{|\nabla g|^2}{|x|^2}+\frac{g^2}{|x|^4(1+|\log |x||)^2}dx\right)\\&- \lelan (\Delta\Lambda W)\chi_M, g \rilan^2-|\lelan (\Delta \nabla W),g \rilan|^2.    
\end{aligned}
\end{equation}   
\end{prop}
\begin{proof}
First we claim it only suffices to verify 
\begin{equation}\label{second order coercivity for L}
\begin{aligned}
 \int_{\mathbb R^4}(\ML g)^2dx \geq& \eta \left(\int_{\mathbb R^4}\frac{|\nabla g|^2}{|x|^2}+\frac{g^2}{|x|^4(1+|\log |x||)^2}dx\right)\\&- \lelan (\Delta\Lambda W)\chi_M, g \rilan^2-|\lelan (\Delta \nabla W),g \rilan|^2.     
\end{aligned}
\end{equation}
Once (\ref{second order coercivity for L}) is established, via computation (\ref{subcoercivity for L g^2}) and the fact 
\begin{equation*}
    \int_{\mathbb R^4}W^2|\nabla g|^2dx=\int_{\mathbb R^4}\frac{|\nabla g|^2}{(1+\frac{x^2}{8})^2}dx\leq \int_{\mathbb R^4}\frac{64|\nabla g|^2}{|x|^2}dx,
\end{equation*}
immediately we have
\begin{equation*}
  \begin{aligned}
 \left(1+\frac{\eta}{128}\right)\int_{\mathbb R^4}(\ML g)^2dx \geq& \frac{\eta}{128}\left\|g\right\|_{\dot{H}^2}^2+ \eta \left(\int_{\mathbb R^4}\frac{|\nabla g|^2}{2|x|^2}+\frac{g^2}{|x|^4(1+|\log |x||)^2}dx\right)\\&- \lelan (\Delta\Lambda W)\chi_M, g \rilan^2-|\lelan (\Delta \nabla W),g \rilan|^2.    
\end{aligned}  
\end{equation*}
Hence, (\ref{second order coercivity for L with H 2}) holds by taking $\eta$ sufficiently small. Now in order to prove (\ref{second order coercivity for L}), again from the computation (\ref{subcoercivity for L g^2}), we obtain
\begin{equation}\label{sub coercivity proper}
\begin{aligned}
    \int_{\mathbb   R^4}(\ML g)^2dx\geq  c\left(\int_{\mathbb R^4}(\Delta g)^2+\frac{g^2}{1+|x|^6}dx\right)-\frac{1}{c}\left(\int_{\mathbb R^4}\frac{|\nabla g|^2}{1+|x|^4}+\frac{|g|^2}{1+|x|^8}dx\right).
\end{aligned}
\end{equation}
Now we argue by contradiction, assuming the proposition is false. Then there exists a sequence $g_n\in \dot{H}^1\cap \dot{H}^2$ such that for $n=1,2...$,
\begin{equation}\label{contra assum 1 for L g^2}
 \int_{\mathbb R^4}\frac{|\nabla g_n|^2}{|x|^2}+\frac{g_n^2}{|x|^4(1+|\log |x||)^2}dx=1   
\end{equation}
 and    
\begin{equation}\label{contra assum 2 for L g^2}
  \int_{\mathbb R^4}(\ML g_n)^2dx  \leq \frac{1}{n}\text{ , }\lelan (\Delta\Lambda W)\chi_M, g_n\rilan=\lelan (\Delta\partial_j W),g_n \rilan_{j=1,2,3,4}=0.
\end{equation}
Then by semicontinuity of the norm, a subsequence of $g_n$ weakly converges to a solution $g_{\infty}\in H_{loc}^1$ of $\ML g_{\infty}=0$. The solution $g_{\infty}$ is smooth away from the origin and hence explicit integration of the ODE and the regularity assumption at the origin $g_{\infty}\in H_{loc}^1$ indicate that  
\begin{equation*}
    g_{\infty}=\alpha_0\Lambda W+\sum_{j=1}^4\alpha_j (\partial_jW)
\end{equation*}
for some $\alpha_i\in \mathbb R$, $i=0,1,2,3,4$. On the one hand, (\ref{contra assum 1 for L g^2}) together with the local compactness of Sobolev embeddings ensure that, up to a subsequence,
\begin{equation*}
    \int_{\mathbb R^4}\frac{|\nabla g_n|^2}{1+|x|^4}dx+ \int_{\mathbb R^4}\frac{| g_n|^2}{1+|x|^8}dx\to \int_{\mathbb R^4}\frac{|\nabla g_{\infty}|^2}{1+|x|^4}dx+ \int_{\mathbb R^4}\frac{| g_{\infty}|^2}{1+|x|^8}dx
\end{equation*}
and 
\begin{equation*}
 \lelan (\Delta\Lambda W)\chi_M, g_n\rilan\to \lelan (\Delta\Lambda W)\chi_M, g_{\infty}\rilan \text{ , }  \lelan (\Delta\partial_j W),g_n \rilan\to\lelan (\Delta\partial_j W),g_{\infty} \rilan,
\end{equation*}
thanks to the $\chi$ localization. Therefore, using symmetry property, we conclude
\begin{equation*}
    \alpha_0 \lelan \Lambda W, (\Delta \Lambda W)\chi_M \rilan=\alpha_i \lelan \partial_i W, (\Delta\partial_i W)\rilan_{i=1,2,3,4}=0,
\end{equation*}
and thus $\alpha_i=0$, $i=0,1,2,3,4$. On the other hand, the subcoercivity (\ref{sub coercivity proper}) and  (\ref{contra assum 2 for L g^2}) together yield
\begin{equation*}
  \int_{\mathbb R^4}\frac{|\nabla g_n|^2}{1+|x|^4}dx+ \int_{\mathbb R^4}\frac{| g_n|^2}{1+|x|^8}dx\geq    c\left(\int_{\mathbb R^4}(\Delta g_n)^2+\frac{g_n^2}{1+|x|^6}dx\right).
\end{equation*}
Applying (\ref{Hardy 1}), (\ref{Hardy 2}) and (\ref{contra assum 1 for L g^2}) to the right hand side above gives
\begin{equation*}
 \left(\int_{\mathbb R^4}(\Delta g_n)^2+\frac{g_n^2}{1+|x|^6}dx\right)\geq C \left( \int_{\mathbb R^4}\frac{|\nabla g_n|^2}{|x|^2}+\frac{g_n^2}{|x|^4(1+|\log |x||)^2}dx\right) =C>0.
\end{equation*}
As a result the below bound $ \int_{\mathbb R^4}\frac{|\nabla g_n|^2}{1+|x|^4}dx+ \int_{\mathbb R^4}\frac{| g_n|^2}{1+|x|^8}dx\geq C>0$ and furthermore,
\begin{equation*}
\begin{aligned}
   &\alpha_0^2\left( \int_{\mathbb R^4}\frac{|\nabla \Lambda W|^2}{1+|x|^4}dx+ \int_{\mathbb R^4}\frac{| \Lambda W|^2}{1+|x|^8}dx\right)+\sum_{i=1}^4\alpha_i^2 \left(\int_{\mathbb R^4}\frac{|\nabla \partial_i  W|^2}{1+|x|^4}dx+ \int_{\mathbb R^4}\frac{|  \partial_i  W|^2}{1+|x|^8}dx\right) \\
   \geq & \int_{\mathbb R^4}\frac{|\nabla g_{\infty}|^2}{1+|x|^4}dx+ \int_{\mathbb R^4}\frac{| g_{\infty}|^2}{1+|x|^8}dx\geq C>0,
\end{aligned}
\end{equation*}
which leads to a contradiction. This concludes the proof of the proposition.
\end{proof}

In application, we always use the scaling version of Proposition \ref{Coercivity estimate with LG LG}, which is listed as below.
\begin{coro}\label{coercivity est with LKG LKG} Let $M\geq 1$ be a large constant, then  there exists $\eta>0$ such that for any $\lambda>0$  and $g \in \dot{H}^1(\mathbb R^4)\cap \dot{H}^2(\mathbb R^4)$, we have
\begin{equation}\label{second order coercivity for L scaling verison}
\begin{aligned}
\int_{\mathbb R^4}(\ML_{\lambda} g)^2dx \geq& \eta \left(\left\|g\right\|_{\dot{H}^2}^2+\int_{\mathbb R^4}\frac{|\nabla g|^2}{|x|^2}+\frac{g^2}{|x|^4(1+|\log |x/\lambda||)^2}dx\right)\\
&-\frac{1}{\lambda^6} \left\{\lelan((\Delta\Lambda W)\chi_M)_{\lambda}, g \rilan^2+|\lelan (\Delta\nabla W)_
{\lambda},g \rilan|^2\right\}.     
\end{aligned}
\end{equation}    
\end{coro}

\section{Modulation and estimate of the parameters}
In this section, we construct solutions to (\ref{NLW 4}) of the form  
\begin{equation}\label{g(t) equation}
    \Vec{u}(t)=\Vec{W}_{\boldsymbol{\Gamma}(t)}+\Vec{g}(t)
\end{equation}
with $\left\|\Vec{g}(t)\right\|_{\mathcal{E}}\ll 1$ on suitable time intervals. Let $T_0>1$ be a sufficiently large constant. For any $T\geq T_0$, we consider initial data prescribed at time $t=T$, whose precise definition will be given in Lemma \ref{choice of u(T)} below, and we study the corresponding solution $\Vec{u}(t)$. 

We aim to control this solution on a backward time interval. To  this end, we introduce bootstrap estimates (see (\ref{bootstap a for energy})-(\ref{bootstrap a for stable})) and define
\begin{equation*}
\begin{aligned}
   T_* :=& \inf \{ t \in [T_0, T] \ : \ \vec u \text{ is defined on } [t,T]
\text{ and admits}\\&\text{ a decomposition of the form (\ref{g(t) equation}) with }
(\Gamma, \vec g) \text{ satisfying } (\ref{bootstap a for energy})-(\ref{bootstrap a for stable}) \}.
\end{aligned}  
\end{equation*}
{The choice of the time-dependent $\mathcal{C}^1$ parameter vector $\boldsymbol{\Gamma}(t)$ will ensure the orthogonality conditions
\begin{equation}\label{orthogonal condition 1}
   \lelan \frac{1}{\lambda_k}(\Delta_k\Lambda_k W_k) \chi_k,g \rilan=0\text{ , } \lelan\frac{1}{\lambda_k}\left(\Delta_k\nabla_k W_k \right),g\rilan=0,
\end{equation}
and 
\begin{equation}\label{orthogonal condition 2}
 \lelan {\frac{1}{\lambda_k}(\Lambda_k W_k)\chi_{k}},\dot{g}\rilan=0\text{ , }\lelan \frac{1}{\lambda_k}(\nabla_kW_k)\chi_k,\dot{g} \rilan=0.    
\end{equation}
Recall that $\chi_k:=\chi((\cdot-\boldsymbol{y}_k)/\lambda_k M)$}. The $M$ here is a large constant which will be fixed in the energy estimate later. This can be achieved by implicit function theorem and will be proved in step 0 of Lemma \ref{close the bootstrap for parameters system}. For the stable/unstable direction we denote
\begin{equation*}
    a_k^{\pm}:=\lelan \Vec{Z}_k^{\pm},\Vec{g} \rilan.
\end{equation*}
Next, we construct well-prepared initial conditions at $t=T\geq T_0$ with a family of free parameters {$(\omega_0,\omega_1)$}  related to instabilities. 
\begin{lemma}\label{choice of u(T)}
    For any $T>T_0$ and any $(\omega_0,\omega_1)\in [-1,1]^2$, there exists a data $\Vec{u}(T)=\Vec{u}[T,(\omega_0,\omega_1)]\in X$ such that
    \begin{equation*}
        \Vec{u}(T)=\Vec{W}_{\boldsymbol{\Gamma}(T)}+\Vec{g}(T),
    \end{equation*}
    with $\Gamma (T)$ defined by 
    \begin{equation*}\label{initial data for lambda b y v}
    \begin{aligned}
         \lambda_k(T)=\lambda(T)=\exp \left(-\xi T^{\frac{2}{3}}-\omega_0 T^{\frac{1}{3}}\right),&\quad 
       b_k(T)=b(T)=\lambda(T)\left(\frac{2\xi}{3}T^{-\frac{1}{3}}-\frac{\omega_0}{3}T^{-\frac{2}{3}}\right),\\
      \text{and }\quad\boldsymbol{y}_k(T)=\boldsymbol{z}_k,&\quad\boldsymbol{v}_k(T)={0}.
    \end{aligned}
    \end{equation*}
    and $\Vec{g}(T)$ is G-invariant satisfies (\ref{orthogonal condition 1}), (\ref{orthogonal condition 2}). For all $k=1,...,N$, it holds
    \begin{equation}\label{initial data of energy}
         \left\|\Vec{g}(T)\right\|_{\ME}\lesssim T^{\frac{1}{4}}\lambda_k(T)^2,\qquad\left\|\Vec{g}(T)\right\|_{\ME_2}\lesssim  T^{\frac{1}{4}}\lambda_k(T),
    \end{equation}
    \begin{equation}\label{initial data of stable unstable}
       \lelan \Vec{Z}_k^+(T),\Vec{g}(T) \rilan=0\text{ , }\lelan \Vec{Z}_k^-(T),\Vec{g}(T) \rilan=T^{\frac{1}{4}}\lambda_k(T)^2\omega_1,
    \end{equation}
    where $\Vec{Z}_k^{\pm}(T)$ are defined as in (\ref{Z_k and Y_k}) for $\boldsymbol{\Gamma}=\boldsymbol{\Gamma}(T)$. Moreover, $\Vec{u}(T)$ is continuous in $X$ with respect to $T$ and $(\omega_0,\omega_1)$. 
\end{lemma}
\begin{proof}
    For $\boldsymbol{\Gamma}=\boldsymbol{\Gamma}(T)$, we consider $\Vec{g}(T)=(g(T),\dot{g}(T))$ of the form
    \begin{equation}\label{definition of g T}
    \begin{aligned}
        \Vec{g}(T)=\sum_k\Bigg\{& b_k^+\boldsymbol{y}_k^++b_k^-\boldsymbol{y}_k^- +\frac{((\boldsymbol{c}_k\cdot \nabla_k)W_k,0)}{\lelan\Delta\partial_{x_1}W,\partial_{x_1}W\rilan}+\frac{d_k(\Lambda_kW_k,0)}{\lelan\Delta\Lambda W,\Lambda W\chi_M\rilan}
        \\&+\frac{(0,(\boldsymbol{e}_k\cdot\nabla_k)W_k)}{\lambda_k\lelan (\partial_{x_1} W)\chi_M, \partial_{x_1} W\rilan}+\frac{f_k(0,(\Lambda_kW_k)\chi_k)}{\lambda_k\lelan \Lambda W\chi_M,\Lambda W\chi_M \rilan} \Bigg\}  
    \end{aligned}
    \end{equation}
Consider the linear map $\Phi:(\mathbb R^{12})^N\to (\mathbb R^{12})^N$ defined as follows:
\begin{equation*}
    \begin{aligned}
        &\Phi(( b_k^+,b_k^-,\boldsymbol{c}_k,d_k,\boldsymbol{e}_k,f_k )_{k=1,...,N})\\
        =\Big(&\lelan \Vec{Z}_k^+,\Vec{g} \rilan,\lelan\Vec{Z}_k^-,\Vec{g} \rilan,\lambda_k^{-2}\lelan (\Delta_k\nabla_k W_k) ,g  \rilan,\lambda_k^{-2}\lelan (\Delta_k\Lambda_k W_k) \chi_k,g \rilan,\\ &\lambda_k^{-1}\lelan (\nabla_k W_k)\chi_k,\dot{g} \rilan,\lambda_k^{-1}\lelan (\Lambda_k W_k) \chi_k,\dot{g} \rilan\Big)_{k=1,...,N}
    \end{aligned}
\end{equation*}
    For $T$ large enough, one verifies that the interaction between different bubbles is negligible for large $T$ and  $G-$invariant , which implies that $\Phi$ is a small perturbation of a block-diagonal matrix $diag_N(A)$ where the $12\times 12$ matrix $A$ can be expressed as
    \begin{equation*}
       A= I+(m_{ij})_{12\times 12}.
    \end{equation*}
    Here all the $m_{ij}$, $1\leq i,j\leq 12$ vanish,  except
    \begin{equation*}
    \begin{aligned}
      m_{1,7}=m_{2,7}=\lelan Y, (\Delta \Lambda W)\chi_M\rilan\text{ , }&m_{1,12}=-m_{2,12}= \lelan Y, ( \Lambda W)\chi_M\rilan;\\
      \text{ and }&m_{12,1}=-m_{12,2}=\lelan Y, ( \Lambda W)\chi_M \rilan    .
    \end{aligned}   
    \end{equation*}
    Then due to the exponential decay of $Y$ and $\lelan Y,\Lambda W \rilan=0$, we obtain that $\Phi$ is invertible. Moreover, in order to impose the conditions (\ref{initial data of stable unstable}), we choose the coefficients by inverting the linear map $\Phi$.
\begin{equation*}
    \left|\Phi^{-1}((0,T^{\frac{1}{4}}\lambda_k(T)^2\omega_1,0,...,0)_{k=1,...,N})\right|\lesssim T^{\frac{1}{4}}\lambda_k(T)^2,
\end{equation*}
    and so  $\left\|\Vec{g}(T)\right\|_{\dot{H}^1\times L^2}\lesssim T^{\frac{1}{4}}\lambda_k(T)^2$. Notice that the target vector is \(G\)-invariant. Since the profiles in (\ref{definition of g T}) and the
orthogonality directions are equivariant, the map \(\Phi\) commutes with the
action of \(G\) on the coefficient vector. By the uniqueness of
\(\Phi^{-1}\), the coefficients chosen above are \(G\)-invariant, and therefore
the correction \(\vec g(T)\) defined by (\ref{definition of g T}) is \(G\)-invariant. Furthermore, from the scaling property and the definition of $\Vec{g}(T)$ (formula (\ref{definition of g T})), the second order energy $\left\|\vec g(T)\right\|_{\ME_2}$ is bounded by
    \begin{equation*}
    \begin{aligned}
      \left\|\vec{g}(T)\right\|_{\ME_2}\lesssim& \sum_k\Bigg\{(|b_k^+|+|b_k^-|)\left(\left\|Y_k\right\|_{\dot{H}^2}+\left\|\frac{1}{\lambda_k}Y_k\right\|_{\dot{H}^1}\right)\\
      &+|\boldsymbol{c}_k|\frac{\left\|\nabla_k W_k\right\|_{\dot{H}^2}+\left\|(W_{\boldsymbol{\Gamma}}^2-W_k^2)\nabla_k W_k\right\|_{L^2}}{\lelan\Delta\partial_{x_1}W,\partial_{x_1}W\rilan}+|d_k|\frac{\left\|\Lambda_kW_k\right\|_{\dot{H}^2}+\left\|(W_{\boldsymbol{\Gamma}}^2-W_k^2)\Lambda_kW_k\right\|_{L^2}}{\lelan\Delta\Lambda W,\Lambda W\chi_M\rilan}\\
      &+|\boldsymbol{e}_k|\frac{\left\|\nabla_k W_k \right\|_{\dot{H}^1}}{\lambda_k\lelan \partial_{x_1} W\chi(\lambda_k\cdot), \partial_{x_1} W\rilan } +|f_k|\frac{\left\|(\Lambda_kW_k)\chi_k)\right\|_{\dot{H}^1}}{\lambda_k\lelan \Lambda W\chi(\lambda_k,\cdot),\Lambda W\chi(\lambda_k,\cdot) \rilan}\Bigg\}\\
      \lesssim& \frac{1}{\lambda_k}\left|\Phi^{-1}((0,T^{\frac{1}{4}}\lambda_k(T)^2\omega_1,0,...,0)_{k=1,...,N})\right|\lesssim T^{\frac{1}{4}}\lambda_k(T).
    \end{aligned}
    \end{equation*}
    which completes the choice of initial data.
\end{proof}
 Let \(\vec u(t)\) be the solution corresponding to the terminal datum given by
Lemma \ref{choice of u(T)}. Since the equation is invariant under the action of \(G\), and since
the terminal datum is \(G\)-invariant, uniqueness of the Cauchy problem implies
that
\[
\vec u(t,\rho x)=\vec u(t,x),\qquad \rho\in G,
\]
on its interval of existence. This together with the orthogonal conditions (\ref{orthogonal condition 1}), (\ref{orthogonal condition 2}) gives the existence and the symmetry of modulation parameters.

\begin{lemma}[Symmetry of the modulation parameters]\label{lem:sym-mod}
Assume that the \(G\)-invariant solution \(\vec u(t)\) admits a decomposition
of the form (\ref{g(t) equation}), with \(\boldsymbol\Gamma(t)\) chosen by the orthogonality conditions
(\ref{orthogonal condition 1}), (\ref{orthogonal condition 2}). Then the modulation parameters satisfy the equivariance
relations of Condition \ref{sym condition}. In particular,
\[
\lambda_1=\cdots=\lambda_N=:\lambda,\qquad
b_1=\cdots=b_N=:b,
\]
and
\[
\boldsymbol y_{\pi_\rho(k)}=\rho\boldsymbol y_k,\qquad
\boldsymbol v_{\pi_\rho(k)}=\rho\boldsymbol v_k.
\]
Moreover, the remainder \(\vec g\) is \(G\)-invariant, and
\[
 a_k^{\pm}:=\lelan \Vec{Z}_k^{\pm},\Vec{g} \rilan
\]
is independent of $k$.
\end{lemma}
\begin{proof}
Apply any \(\rho\in G\) to the decomposition (\ref{g(t) equation}). Since \(\vec u\) is
\(G\)-invariant and the profiles, cutoffs, and orthogonality directions are
equivariant, this gives another decomposition of the same solution satisfying
the same orthogonality conditions. By the local uniqueness in the implicit
function theorem for (\ref{orthogonal condition 1})--(\ref{orthogonal condition 2}), the transformed parameters must
coincide with the original ones. This gives the equivariance relations for
\(\boldsymbol\Gamma(t)\), hence the common scale and scaling velocity by transitivity.
The \(G\)-invariance of \(\vec g\) follows by subtracting the invariant
multi-bubble profile from \(\vec u\). Finally, the equivariance of
\(\vec Z_k^\pm\) gives \(a_{\pi_\rho(k)}^\pm=a_k^\pm\), so \(a_k^\pm\) is
independent of \(k\).
\end{proof}
We shall therefore write
\[
a^\pm:=a_k^\pm,
\]
while keeping the index \(k\) in local estimates when it is useful to indicate
the bubble under consideration.

We now introduce the quantities used in the bootstrap estimates. First, we introduce corrections $\beta_k$ and $\boldsymbol{s}_k$. Let $\boldsymbol{\beta}:=(\beta_1,...,\beta_N)$ and satisfies linear system
\begin{equation*}
  {\lelan \frac{1}{\lambda_k}(\Lambda_k W_k)\chi_{t},\sum_{j=1}^N \frac{\beta_j-b_j}{\lambda_j}(\Lambda_j W_j)\chi_{t} \rilan}  ={\lelan \dot{g},\frac{1}{\lambda_k}(\Lambda_k W_k)\chi_{t} \rilan},
\end{equation*}
for all $k\in \left\{1,...,N\right\}$.  Since
\[
\left\langle 
\frac1{\lambda_k}(\Lambda_kW_k)\chi_t,
\frac1{\lambda_k}(\Lambda_kW_k)\chi_t
\right\rangle
\sim \log(t/\lambda)\sim t^{2/3},
\]
whereas the off-diagonal terms are \(O(\log t)\), the coefficient matrix is
diagonally dominant and \(\beta\) is well-defined. Moreover, the system is
\(G\)-equivariant and the right-hand side is \(G\)-invariant. By uniqueness of
the solution of this linear system,
\[
\beta_{\pi_\rho(k)}=\beta_k,\qquad \rho\in G.
\]
Thus, by transitivity,
\[
\beta_1=\cdots=\beta_N=:\beta.
\]
 Furthermore, the continuity of $\boldsymbol{\lambda},\boldsymbol{b}$ and $\dot{g}$ yields $\boldsymbol{\beta}\in \mathcal{C}(I_{\max})$. The translation correction \(\boldsymbol s=(\boldsymbol s_1,\ldots,\boldsymbol s_N)\)
is defined analogously by
\[
\left\langle 
\frac1{\lambda_k}\nabla W_k,
\sum_{j=1}^N\frac{\boldsymbol s_j-\boldsymbol v_j}{\lambda_j}\cdot\nabla_j W_j
\right\rangle
=
\left\langle 
\dot g,\frac1{\lambda_k}\nabla W_k
\right\rangle .
\]
The same argument shows that \(\boldsymbol s\) is well-defined and equivariant:
\[
\boldsymbol s_{\pi_\rho(k)}=\rho\boldsymbol s_k,\qquad \rho\in G.
\]
In fact, \(\beta_k\) and \(\boldsymbol s_k\) can be regarded as corrected
versions of \(b_k\) and \(\boldsymbol v_k\). With these parameters we define 
\begin{equation*}
    \dot{h}:=\dot{g}+\sum_{j=1}^N\frac{b_j-\beta_j}{\lambda_j}(\Lambda_j W_j)\chi_t-\sum_{j=1}^N\frac{(\boldsymbol{v}_j-\boldsymbol{s}_j)}{\lambda_j}\cdot\nabla_j W_j.
\end{equation*}
And furthermore we have
\begin{equation*}
    \Vec{u}=(g,\dot{h})+\Vec{W}_{\boldsymbol{\Tilde{\Gamma}}},
\end{equation*}
where 
\begin{equation*}
    \Vec{W}_{\boldsymbol{\Tilde{\Gamma}}}=\sum_{j=1}^N\left(W_j,\frac{\beta_j}{\lambda_j}(\Lambda_j W_j)\chi_{t}+\frac{\boldsymbol{s}_j\cdot\nabla_jW_j}{\lambda_j}\right).
\end{equation*}
By the equivariance of \(\beta,\boldsymbol s\) and of the original modulation
parameters, both \(\dot h\) and \(\Vec{W}_{\boldsymbol{\Tilde{\Gamma}}}\) are \(G\)-invariant.
We also denote 
\begin{equation*}
    \Tilde{a}_k^{\pm}:=\lelan \Vec{Z}_k^{\pm}, (g,\dot{h}) \rilan,
\end{equation*}
which indicates
\begin{equation*}
    \Tilde{a}_k^{\pm}=a_k^{\pm}+\lelan \Vec{Z}_k^{\pm}, (0, \sum_{j=1}^N\frac{b_j-\beta_j}{\lambda_j}(\Lambda_j W_j)\chi_t+\sum_{j=1}^N\frac{(\boldsymbol{v}_j-\boldsymbol{s}_j)}{\lambda_j}\cdot\nabla_j W_j) \rilan.
\end{equation*}
Since \((g,h)\) is \(G\)-invariant and the localized stable/unstable directions
are equivariant,
\[
\widetilde a_{\pi_\rho(k)}^\pm=\widetilde a_k^\pm.
\]
Hence \(\widetilde a_k^\pm\) is independent of \(k\); we write
\[
\widetilde a^\pm:=\widetilde a_k^\pm.
\]
We now introduce the scalar quantities used to formulate the scale bootstrap. As in Section \ref{A formal computation}, we set
\[
\alpha(t):=\log\frac1{\lambda(t)},
\]
and 
\begin{equation*}
    \Xi (t):=\frac{1}{128\pi^2}\left(\lelan \Lambda W\chi \left(\frac{\cdot \lambda}{t}\right), \Lambda W\chi \left(\frac{\cdot \lambda}{t}\right) \rilan+\sum_{j\neq k} \lelan\Lambda W({\cdot-\boldsymbol{y}_k})\chi \left(\frac{\cdot \lambda}{t}\right), \Lambda W({\cdot-\boldsymbol{y}_j})\chi \left(\frac{\cdot \lambda}{t}\right) \rilan\right).
\end{equation*}
By the
equivariance of \((\boldsymbol y_1,\ldots,\boldsymbol y_N)\) and (\ref{eq:inner-prod-1})--(\ref{eq:inner-prod-2}), this quantity is
independent of \(k\), and 
\begin{equation}\label{est of Xi}
  \Xi(t)=\alpha(t)+O(\log t).  
\end{equation}
Finally, we set
\[
m(t):=\frac{\beta(t)}{\lambda(t)},
\qquad
\mu(t):=\left(\frac{c}{\Xi(t)}\right)^{1/2},
\]
where \(c>0\) is the constant defined in Section 2. The bootstrap will control \(m-\mu\), rather than comparing \(\lambda(t)\) with
a fixed reference profile up to a multiplicative constant.

Next we introduce the following bootstrap estimates
\begin{equation}\label{bootstap a for energy}
      \left\|\Vec{g}(t)\right\|_{\ME}\leq t^{-\frac{1}{6}}\lambda(t), 
\end{equation}
\begin{equation}\label{bootstap a for 2nd energy}
        \left\|\Vec{g}(t)\right\|_{\ME_2}\leq t^{\frac{1}{2}} \lambda(t), 
\end{equation}
\begin{equation}\label{bootstrap a for lambda}
   \left|\alpha(t)-\xi t^{\frac{2}{3}}\right|\leq t^{\frac{1}{3}}, 
\end{equation}
\begin{equation}\label{bootstrap a for b}
    \left|m(t)-\mu(t)\right|\leq Ct^{-\frac{2}{3}},
\end{equation}
\begin{equation}\label{bootstrap a for y}
   |\boldsymbol{y}_k(t)-\boldsymbol{z}_k|\leq t^{\frac{5}{2}}\lambda(t)^2, 
\end{equation}
\begin{equation}\label{bootstrap a for v}
  |\boldsymbol{v}_k(t)|\leq t^{\frac{7}{6}}\lambda(t)^2,  
\end{equation}
\begin{equation}\label{bootstrap a for unstable}
   (a^+(t))^2\leq  t^{\frac{1}{2}}\lambda(t)^4.
\end{equation}
\begin{equation}\label{bootstrap a for stable}
 t^{-\frac{1}{2}}\lambda(t)^{-4}(a^-(t))^2\leq  1
\end{equation}
Here $C>0$ is a constant. Let $\Vec{u}\in \mathcal{C}(I_{\max};X)$, where $I_{\max}\ni T$, be the maximal solution of (\ref{NLW 4}) corresponding to any data $\Vec{u}(T)$ as given by Lemma \ref{choice of u(T)}. Since $\Vec{u}(T)\in X$, the persistence of the regularity indicates that $\Vec{u}\in \mathcal{C}(I_{\max};X)$. Such regularity will allow energy computations without density argument. We record a consequence of the scale bootstrap which will be used later. From (\ref{est of Xi}) and (\ref{bootstrap a for lambda}), \[
\Xi(t)=\alpha(t)+O(\log t)
      =\xi t^{\frac23}+\bigl(\alpha(t)-\xi t^{\frac23}\bigr)+O(\log t).
\]
Using $(1+x)^{-\frac12}=1-\frac12x+O(x^2)$ we have
\begin{equation*}
\begin{aligned}
   \mu(t)
=\left(\frac{c}{\Xi(t)}\right)^{\frac12} 
=\left(\frac{c}{\xi}\right)^{\frac12}t^{-\frac13}
\left(
1-\frac12
\frac{\alpha(t)-\xi t^{\frac23}+O(\log t)}
{\xi t^{\frac23}}
+O(t^{-\frac23})
\right).
\end{aligned}
\end{equation*}
Since \(c=\frac49\xi^3\), this gives
\begin{equation}\label{rough of mu}
    \mu(t)
=
\frac{2\xi}{3}t^{-\frac13}
-
\frac{\alpha(t)-\xi t^{\frac23}}{3t}
+
O(t^{-1}\log t)=\frac{2\xi}{3}t^{-\frac13}+O(t^{-\frac{2}{3}}).
\end{equation}
In particular, $\mu(t)\sim t^{-\frac13}.$

In the rest of this section, we will prove the following lemma.
\begin{lemma}\label{close the bootstrap for parameters system} It holds $T_0\leq T_*\leq T$ and if $T_*>T_0$ then

(\rmnum{1}) Equality is reached at $t=T_*$ in at least one of the inequalities (\ref{bootstap a for energy})-(\ref{bootstrap a for stable}).

(\rmnum{2}) On $[T_*,T]$, it holds   
\begin{equation}\label{lambda'+b}
   |{\lambda}'(t)+{b}(t)|\lesssim t^{\frac{1}{2}}\lambda(t)^2, 
\end{equation}
\begin{equation}\label{y'+v}
   |\boldsymbol{y}'_k(t)+\boldsymbol{v}_k(t)|\lesssim t^{\frac{1}{2}}\lambda(t)^2,  
\end{equation}
\begin{equation}\label{b and beta}
    |\beta(t)-b(t)|\lesssim (t^{-\frac{5}{6}}\sqrt{\log t})\lambda(t),
\end{equation}
\begin{equation}\label{beta'+=to}
   \left| \Xi(t)\beta'(t)+\frac{1}{2}\frac{\beta(t)^2}{\lambda(t)}+c\lambda(t)\right|\lesssim  (t^{-\frac{7}{6}}\sqrt{\log t})\lambda(t),
\end{equation}
\begin{equation}\label{rough est on b'}
    |b'(t)|\lesssim t^{\frac{1}{2}}\lambda(t),
\end{equation}
\begin{equation}\label{v and s}
    |\boldsymbol{v}_k(t)-\boldsymbol{s}_k(t)|\lesssim t\lambda(t)^2,
\end{equation}
\begin{equation}\label{s'+=to}
 \left| -{s}_{k,j}'(t)+D_{k,j}(\boldsymbol{\lambda},\boldsymbol{y}) \right|\lesssim t^{\frac{2}{3}}\lambda(t)^2 ,  
\end{equation}
\begin{equation}\label{rough est on v'}
    |\boldsymbol{v}_k'(t)|\lesssim t^{\frac{1}{2}}\lambda(t),
\end{equation}
\begin{equation}\label{a and tilde a}
    |a^{\pm}(t)-\Tilde{a}^{\pm}(t)|\lesssim (t^{-\frac{5}{6}}\sqrt{\log t})\lambda(t)^3,
\end{equation}
\begin{equation}\label{a' and tilde a'}
   |(a^{\pm}(t))'-(\Tilde{a}^{\pm}(t))'|\lesssim t^{\frac{1}{2}}\lambda(t)^3,
\end{equation}
\begin{equation}\label{a' also satisfies}
    \left|({a}^{\pm}(t))'\mp \nu \lambda^{-1}{a}^{\pm}(t)\right|\lesssim \lambda(t).
\end{equation}
\end{lemma}

\begin{Rmk}
Here the estimate on $\boldsymbol{v}_k'$ appears rather rough at first sight and may seem insufficient to control the parameter $\boldsymbol{v}_k$. However, this does not cause any difficulty.  Indeed, the terms involving $\boldsymbol{v}_k'$ are always coupled with $\nabla_k W_k$, which lies in the kernel of the linearized operator. As a consequence, this contribution is sufficiently small for our purposes despite the lack of precision in the bound on $\boldsymbol{v}_k'$.   
\end{Rmk}

\begin{proof}\textbf{Step 0. Proof of the continuity.}
At $t=T$, Lemma \ref{choice of u(T)} provides suitable initial data for decomposition (\ref{g(t) equation}) which meets all the inequalities of (\ref{bootstap a for energy})-(\ref{bootstrap a for stable}). Here we check (\ref{bootstrap a for b}), which does not follow directly at $T$ from Lemma \ref{choice of u(T)}. By (\ref{rough of mu}) and the initial data of $\alpha(T)$
, we obtain
\[
\begin{aligned}
\mu(T)=
\frac{2\xi}{3}T^{-\frac13}
-\frac{\omega_0}{3}T^{-\frac23}
+O(T^{-1}\log T).
\end{aligned}
\]
On the other hand, Lemma \ref{choice of u(T)} gives
\[
\frac{b(T)}{\lambda(T)}
=
\frac{2\xi}{3}T^{-\frac13}
-\frac{\omega_0}{3}T^{-\frac23}.
\]
Therefore
\[
\left|\frac{b(T)}{\lambda(T)}-\mu(T)\right|
\lesssim T^{-1}\log T
=o(T^{-\frac23}).
\]
It remains only to replace \(b\) by the corrected coefficient \(\beta\).
By the definition of \(\beta\), the diagonal dominance of the corresponding
linear system, and the estimate on \(\dot g(T)\) from Lemma \ref{choice of u(T)}, we have
\[
|\beta(T)-b(T)|\lesssim \left\|\dot{g}(T)\right\|_{L^2}T^{-\frac{1}{3}}\lesssim \lambda(T)^2.
\]
Hence
\begin{equation}\label{m T- mu T}
 |m(T)-\mu(T)|
=
\left|\frac{\beta(T)}{\lambda(T)}-\mu(T)\right|=O(T^{-1}\log T)
=o(T^{-\frac23}),   
\end{equation}
and so the bootstrap estimate (\ref{bootstrap a for b}) is satisfied at \(t=T\), after increasing
\(T_0\) if necessary.

By the local Cauchy theory for (\ref{NLW 4}), it is clear that if a solution $\Vec{u}$ satisfies (\ref{g(t) equation}) with (\ref{bootstap a for energy})-(\ref{bootstrap a for stable}) on some interval $[t,T]$, then the solution $\Vec{u}$ also exists on $[t-\tau,T]$, for some $\tau>0$.

Next in order to decompose $\Vec{u}(t)$ for $t<T$, we denote function $\boldsymbol{F}$ as 
\begin{equation*}
\begin{aligned}
   \boldsymbol{F}(t,\boldsymbol{\Gamma}(t)):=\Bigg(  \lelan \frac{1}{\lambda_k}(\Delta_k\Lambda_k W_k),g \rilan&\text{ , } \lelan\frac{1}{\lambda_k}\left(\Delta_k\nabla_k W_k \right)\chi_k,g\rilan,\\
   \lelan {\frac{1}{\lambda_k}(\Lambda_k W_k)\chi_{k}},\dot{g}\rilan&\text{ , }\lelan \frac{1}{\lambda_k}(\nabla_kW_k)\chi_k,\dot{g} \rilan\Bigg)_{k=1,...,N}.
\end{aligned}
\end{equation*}
It is easy to verify that $\boldsymbol{F}$ is continuous in $t$ and locally Lipschitz in $\boldsymbol{\Gamma}$. Now consider the differential  system $\boldsymbol{D\Gamma}'(t)=\boldsymbol{F}(t,\boldsymbol{\Gamma}(t))$. The matrix $\boldsymbol{D}$ is a perturbation of the block matrix $diag_N(D_0)\in \mathbb R^{10N\times10N}$, where 
\begin{equation*}
    D_0=diag(\left\|\nabla\Lambda W\right\|_{L^2}^2,(\left\|\nabla \partial_{x_j}W\right\|_{L^2}^2)_{j=1,...,4},\left\|\Lambda W \chi\right\|_{L^2}^2,(\left\|\partial_{x_j} W\right\|_{L^2}^2)_{j=1,..,4}).
\end{equation*}
From the definition of multi-bubble and the control of $\Vec{g}$ (which is (\ref{bootstap a for energy})), the perturbation is sufficiently small. Therefore, (\rmnum{1}) follows from Cauchy-Lipschitz theorem and continuity arguments.

To prove the estimates in (\rmnum{2}) hold on the whole interval $[T_*,T]$, first we need to rewrite the equation of $\Vec{g}(t)$ as
\begin{equation}\label{formal dt g}
    \partial_t\Vec{g}=J\circ DE(\Vec{W}_{\boldsymbol{\Gamma}}+\Vec{g})-\boldsymbol{\lambda}'\partial_{\boldsymbol{\lambda}}\Vec{W}_{\boldsymbol{\Gamma}}-\boldsymbol{b}'\partial_{\boldsymbol{b}}\Vec{W}_{\boldsymbol{\Gamma}}-\boldsymbol{y}'\partial_{\boldsymbol{y}}\Vec{W}_{\boldsymbol{\Gamma}}-\boldsymbol{v}'\partial_{\boldsymbol{v}}\Vec{W}_{\boldsymbol{\Gamma}}.
\end{equation}
To be specific, $\partial_t g$ satisfies
\begin{equation}\label{g-dot g}
    \partial_t g=\dot{g}+\sum_{j=1}^N \lambda_j^{-1}(\lambda_j'\Lambda_jW_j+b_j(\Lambda_j W_j)\chi_{t})+\sum_{j=1}^N \lambda_j^{-1}(\boldsymbol{y}_j'+\boldsymbol{v}_j)\cdot\nabla_j W_j,
\end{equation} 
while $\partial_t \dot{g}$ satisfies
\begin{equation}\label{partial g dot}
    \begin{aligned}
        \partial_t \dot{g}=&\Delta g+f(W_{\boldsymbol{\Gamma}}+g)-\sum_{j=1}^Nf(W_j)+\sum_{j=1}^N\frac{\lambda_j'b_j}{\lambda_j^2} (\underline{\Lambda}_j\Lambda_jW_j)\chi_{t}\\
        &-\sum_{j=1}^N\frac{b_j'}{\lambda_j}(\Lambda_jW_j\chi_{t})+\sum_
        {j=1}^N\frac{b_j|x-\boldsymbol{y}_j|}{\lambda_jt^2}(\Lambda_jW_j\chi'_{t})+\sum_{j=1}^N\frac{b_j}{\lambda_j^2}(\boldsymbol{y}_j'\cdot\nabla)(\Lambda_j W_j\chi_{t})\\
         & -\sum_{j=1}^N\sum_{l=1}^4\frac{v'_{j,l}(\partial_lW)(\frac{\cdot-\boldsymbol{y}_j}{\lambda_j})}{\lambda_j^2}-\sum_{j=1}^N\sum_{l=1}^4\frac{\lambda'_jv_{j,l}(\underline{\Lambda} \partial_l W)\left(\frac{\cdot-\boldsymbol{y}_j}{\lambda_j}\right)}{\lambda_j^3}\\
   &+\sum_{j=1}^N\sum_{l,m=1}^4\frac{v_{j,l}y'_{j,m}(\partial_l\partial_mW)(\frac{\cdot-\boldsymbol{y}_j}{\lambda_j})}{\lambda_j^3}.
    \end{aligned}
\end{equation}

\textbf{Step 1. Estimate of $\lambda'+b$ (\ref{lambda'+b}).} We derive the equation for $\lambda_k'+b_k$ by taking time derivative on the orthogonal condition $\lelan \frac{1}{\lambda_k}(\Delta_k\Lambda_k W_k) \chi_k,g \rilan$. From (\ref{g-dot g}), it holds
\begin{equation}\label{orth with Lambda W}
    \begin{aligned}
        0=&\frac{d}{dt}\lelan \frac{1}{\lambda_k}(\Delta_k\Lambda_k W_k) \chi_k,g \rilan\\
        =
        &-\lelan \frac{\lambda_k'}{\lambda_k^3}\left(\underline{\Lambda}[(\Delta\Lambda W\right)\chi_M])\left(\frac{x-\boldsymbol{y}_k}{\lambda_k}\right),g \rilan-\lelan \sum_{i=1}^4\frac{(y_k')_i}{\lambda_k^3}(\partial_i[(\Delta \Lambda W)\chi_M])\left(\frac{x-\boldsymbol{y}_k}{\lambda_k}\right),g \rilan\\
        &+\lelan \frac{1}{\lambda_k}(\Delta_k\Lambda_kW_k) \chi_k,\dot{g}+\sum_{j=1}^N \lambda_j^{-1}(\lambda_j'\Lambda_jW_j+b_j(\Lambda_j W_j)\chi_{t})+\sum_{j=1}^N\lambda_j^{-1}(\boldsymbol{y}_j'+\boldsymbol{v}_j)\cdot\nabla_j W_j\rilan.
    \end{aligned}
\end{equation}
We start from the first line of (\ref{orth with Lambda W}). Using Cauchy inequality and Sobolev inequality gives
\begin{equation*}
    \begin{aligned}
       &\left|\lelan \frac{\lambda_k'}{\lambda_k^3}\left(\underline{\Lambda}[(\Delta\Lambda W\right)\chi_M])\left(\frac{x-\boldsymbol{y}_k}{\lambda_k}\right),g \rilan\right| \\
       \lesssim &|\lambda_k'|\left\|\lambda_k^{-3}\left(\underline{\Lambda}[(\Delta\Lambda W\right)\chi_M])\left(\frac{x-\boldsymbol{y}_k}{\lambda_k}\right)\right\|_{L^{\frac{4}{3}}}\left\|g\right\|_{\dot{H}^1}\lesssim |\lambda_k'| \left\|g\right\|_{\dot{H}^1}.
    \end{aligned}
\end{equation*}
Similarly, the second term on the right hand side of (\ref{orth with Lambda W}) is controlled by  
\begin{equation*}
    \begin{aligned}
      &\left|\lelan \sum_{i=1}^4\frac{(y_k')_i}{\lambda_k^3}(\partial_i[(\Delta \Lambda W)\chi_M])\left(\frac{x-\boldsymbol{y}_k}{\lambda_k}\right),g \rilan\right|\\
      &\lesssim |\boldsymbol{y}_k'|\left\|\partial_i[(\Delta\Lambda W)\chi_M]\right\|_{L^{\frac{4}{3}}}\left\|g\right\|_{\dot{H}^1}\lesssim |\boldsymbol{y}_k'|\left\|g\right\|_{\dot{H}^1}
    \end{aligned}
\end{equation*}
 Next, we check all the terms on the last line of (\ref{orth with Lambda W}). The term involving $\dot{g}$ is estimated by Cauchy inequality:
\begin{equation*}
  \left|\lelan  \frac{1}{\lambda_k}(\Delta_k\Lambda_k W_k) \chi_k, \dot{g} \rilan \right|\lesssim \left\|\frac{1}{\lambda_k}(\Delta_k\Lambda_k W_k) \chi_k\right\|_{L^2} \left\|\dot{g}\right\|_{L^2}\lesssim  \left\|\dot{g}\right\|_{L^2}.
\end{equation*}
Moreover, by symmetry one has $\lelan (\Delta_k\Lambda_kW_k)\chi_k,(y_k')_i \partial_i W_k \rilan=0$. Hence, separating the contribution with $j=k$ and using the symmetry identity above, we obtain
\begin{equation*}
    \begin{aligned}
        &\lelan \frac{1}{\lambda_k}(\Delta_k\Lambda_kW_k )\chi_k,\sum_{j=1}^N \lambda_j^{-1}\left(\lambda_j'\Lambda_jW_j+b_j(\Lambda_j W_j)\chi_{t}\right)+\sum_{j=1}^N\frac{\boldsymbol{y}_j'}{\lambda_j}\cdot\nabla_j W_j \rilan\\
        =&(\lambda_k'+b_k)\lelan (\Delta \Lambda W)\chi_M,\Lambda W \rilan+\lelan \frac{1}{\lambda_k}(\Delta_k\Lambda_kW_k )\chi_k,\sum_{j\neq k} \frac{1}{\lambda_j}(\lambda_j'\Lambda_jW_j+b_j(\Lambda_j W_j)\chi_{t}+\boldsymbol{y}_j'\cdot\nabla_j W_j )\rilan.
    \end{aligned}
\end{equation*}
The interaction term in here is estimated by the pointwise decay of stationary $W$:
\begin{equation*}
    \begin{aligned}
      &\left|\lelan \frac{1}{\lambda_k}(\Delta_k\Lambda_kW_k )\chi_k,\sum_{j\neq k} \frac{1}{\lambda_j}(\lambda_j'\Lambda_jW_j+b_j(\Lambda_j W_j)\chi_{t}+\boldsymbol{y}_j'\cdot\nabla_j W_j )\rilan\right|\\
      \lesssim&\sum_{j\neq k}\left(\frac{\lambda_k^2\log\left(\frac{1}{\lambda_k}\right)}{|\boldsymbol{y}_j-\boldsymbol{y}_k|^2}(|\lambda_j'|+|b_j|)+\frac{\lambda_j \lambda_k^2\log\left(\frac{1}{\lambda_k}\right)}{|\boldsymbol{y}_j-\boldsymbol{y}_k|^3}|\boldsymbol{y}_j'|\right).
    \end{aligned}
\end{equation*}
Collecting the previous estimates yields
\begin{equation}\label{ortho 1 modu get for b and lambda'}
\begin{aligned}
      |\lambda_k'+b_k|\lesssim& (1+|\lambda_k'|+|\boldsymbol{y}_k'|)\left\|\Vec{g}\right\|_{\ME}+(|\lambda_k'|+|b_k|)\lambda_k^2
    \\&\sum_{j\neq k}\frac{\lambda_k^2\log\left(\frac{1}{\lambda_k}\right)}{|\boldsymbol{y}_k-\boldsymbol{y}_j|^2}(|\lambda_j'|+|b_j|)+\sum_{j\neq k}|y_k'|\frac{\lambda_j \lambda_k^2\log\left(\frac{1}{\lambda_k}\right)}{|\boldsymbol{y}_k-\boldsymbol{y}_j|^3}.
\end{aligned}
\end{equation}
An analogous but better estimate can be derived in the norm adapted to the second order energy. For this, it is enough to treat the terms involving $g$ and $\dot g$. We first consider the contribution of $g$.  By Cauchy's inequality and let $\epsilon>0$ be a small constant, it can be re-estimated as 
\begin{equation*}
     \begin{aligned}
       &\left|\lelan \frac{\lambda_k'}{\lambda_k^3}\left(\underline{\Lambda}[(\Delta\Lambda W\right)\chi_M])\left(\frac{x-\boldsymbol{y}_k}{\lambda_k}\right),g \rilan\right|\\ 
     \lesssim&\frac{|\lambda_k'|}{\lambda_k^{1+\epsilon}}{\left\|\left({\underline{\Lambda}[(\Delta\Lambda W))\chi_M]}\ |\cdot|^{2-\epsilon}\right)\left(\frac{x-\boldsymbol{y}_k}{\lambda_k}\right)\right\|_{L^2}\left\|\frac{g}{|x-\boldsymbol{y}_k|^{2-\epsilon}}\right\|_{L^2}}.
    \end{aligned} 
\end{equation*}
Here we allow for a small loss $\epsilon > 0$, which arises from the critical Hardy and Sobolev inequalities in dimension four and only affects lower order terms. Given the pointwise decay $\Delta \Lambda W(x)\sim \lelan x\rilan^{-6}$, the first $L^2$ norm is bounded by 
\begin{equation*}
    \left\|\left({\underline{\Lambda}[(\Delta\Lambda W))\chi_M]}\ |\cdot|^{2-\epsilon}\right)\left(\frac{x-\boldsymbol{y}_k}{\lambda_k}\right)\right\|_{L^2}\lesssim \left(\int_{\lambda_k}^{\infty}\left(\frac{\lambda_k^{4+\epsilon}}{r^{4+\epsilon}}\right)^2r^3dr\right)^{\frac{1}{2}}\lesssim \lambda_k^2.
\end{equation*}
While the second $L^2$ norm, from Hardy's inequality, is dominated by
\begin{equation}\label{hardy 2-episllon}
    \left\|\frac{g}{|x-\boldsymbol{y}_k|^{2-\epsilon}}\right\|_{L^2}\lesssim \left\|g\right\|_{\dot{H}^2}+\left\|g\right\|_{\dot{H}^1}.
\end{equation}
Combining these estimates together gives
\begin{equation}\label{orth 1.1 re}
   \left|\lelan \frac{\lambda_k'}{\lambda_k^3}\left(\underline{\Lambda}[(\Delta\Lambda W\right)\chi_M])\left(\frac{x-\boldsymbol{y}_k}{\lambda_k}\right),g \rilan\right|\lesssim {|\lambda_k'|}{\lambda_k^{1-\epsilon}} (\left\|g\right\|_{\dot{H}^2}+\left\|g\right\|_{\dot{H}^1}).
\end{equation}
Using an identical method, we also derive
\begin{equation*}
   \begin{aligned}
      \left|\lelan \sum_{i=1}^4\frac{(y_k')_i}{\lambda_k^3}(\partial_i[(\Delta \Lambda W)\chi_M])\left(\frac{x-\boldsymbol{y}_k}{\lambda_k}\right),g \rilan\right|\lesssim |\boldsymbol{y}_k'|\lambda_k^{1-\epsilon}(\left\|g\right\|_{\dot{H}^2}+\left\|g\right\|_{\dot{H}^1}).
    \end{aligned}   
\end{equation*}
Next, for terms involving $\dot{g}$, applying Cauchy inequality and Sobolev embedding we obtain
\begin{equation*}
    \left|\lelan  \frac{1}{\lambda_k}(\Delta_k\Lambda_k W_k) \chi_k, \dot{g} \rilan\right| \lesssim \left\|\frac{1}{\lambda_k}(\Delta_k\Lambda_k W_k) \chi_k\right\|_{L^{\frac{4}{3}}}\left\|\dot{g}\right\|_{\dot{H}^1} \lesssim \lambda_k \left\|\dot{g}\right\|_{\dot{H}^1}.
\end{equation*}
Therefore, combining these estimates together we arrive at  
\begin{equation*}
\begin{aligned}
    |\lambda_k'+b_k|\lesssim& {\lambda_k^{1-\epsilon}}(\lambda_k^{\epsilon}+|\lambda_k'|+|\boldsymbol{y}_k'|)(\left\|\Vec{g}\right\|_{
    \ME_2}+\left\|\Vec{g}\right\|_{\ME})
    \\
    &+(|\lambda_k'|+|b_k|)\lambda_k^2+\sum_{j\neq k}\frac{\lambda_k^2\log\left(\frac{1}{\lambda_k}\right)}{|\boldsymbol{y}_k-\boldsymbol{y}_j|^2}(|\lambda_j'|+|b_j|)+\sum_{j\neq k}|\boldsymbol{y}_k'|\frac{\lambda_j \lambda_k^2\log\left(\frac{1}{\lambda_k}\right)}{|\boldsymbol{y}_k-\boldsymbol{y}_j|^3}.  
\end{aligned}
\end{equation*}
This completes the proof of (\ref{lambda'+b}). We note that the corresponding estimate at the level of the second order energy yields an improved control compared with (\ref{ortho 1 modu get for b and lambda'}). 

\textbf{Step 2. Estimate of $\boldsymbol{y}_k'+\boldsymbol{v}_k$ (\ref{y'+v}).} Analogously, for the orthogonal condition involving the translation, taking the time derivative we have
\begin{equation*}
\begin{aligned}
  0=&\frac{d}{dt}\lelan\frac{1}{\lambda_k}\left(\Delta_k\nabla_k W_k \right),g\rilan\\
  =& -\lelan \frac{\lambda_k'}{\lambda_k^3}\left(\underline{\Lambda}(\Delta\nabla W\right))\left(\frac{x-\boldsymbol{y}_k}{\lambda_k}\right),g \rilan-\lelan \sum_{i=1}^4\frac{(y_k')_i}{\lambda_k^3}\partial_i(\Delta\nabla W)\left(\frac{x-\boldsymbol{y}_k}{\lambda_k}\right), g\rilan\\
  &+\lelan \frac{1}{\lambda_k}\left(\Delta_k\nabla_k W_k \right),\dot{g}+\sum_{j=1}^N \lambda_j^{-1}(\lambda_j'\Lambda_jW_j+b_j(\Lambda_j W_j)\chi_{t})+\sum_{j=1}^N\lambda_j^{-1}(\boldsymbol{y}_j'+\boldsymbol{v}_j)\cdot\nabla_j W_j\rilan .
\end{aligned}
\end{equation*}
Again applying Sobolev inequality to the first term gives
\begin{equation*}
    \left|\lelan \frac{\lambda_k'}{\lambda_k^3}\left(\underline{\Lambda}(\Delta\nabla W\right))\left(\frac{x-\boldsymbol{y}_k}{\lambda_k}\right),g \rilan\right|\lesssim |\lambda_k'|\left\|\underline{\Lambda}(\Delta\nabla W)\right\|_{L^{\frac{4}{3}}}\left\|g\right\|_{\dot{H}^1}\lesssim |\lambda_k'|\left\|g\right\|_{\dot{H}^1}.
\end{equation*}
When using second order energy norm, this estimate can be refined as
\begin{equation*}
\begin{aligned}
   &\left|\lelan \frac{\lambda_k'}{\lambda_k^3}\left(\underline{\Lambda}(\Delta\nabla W\right))\left(\frac{x-\boldsymbol{y}_k}{\lambda_k}\right),g \rilan\right|\\
   \lesssim& \frac{|\lambda_k'|}{\lambda_k^{1+\epsilon}}\left\|\left(\underline{\Lambda}(\Delta\nabla W)|\cdot|^{2-\epsilon}\right)\left(\frac{x-\boldsymbol{y}_k}{\lambda_k}\right)\right\|_{L^2}\left\|\frac{g}{|x-\boldsymbol{y}_k|^{2-\epsilon}}\right\|_{L^2}.
\end{aligned}    
\end{equation*}
Follow the analysis of (\ref{orth 1.1 re}), this term is furthermore dominated by 
\begin{equation*}
    \frac{|\lambda_k'|}{\lambda_k^{1+\epsilon}}\left(\int_{\lambda_k}^{\infty}\left(\frac{\lambda_k^{3+\epsilon}}{r^{3+\epsilon}}\right)^2r^3dr\right)^{\frac{1}{2}}(\left\|g\right\|_{\dot{H}^2}+\left\|g\right\|_{\dot{H}^1})\lesssim |\lambda_k'|\lambda_k^{1-\epsilon}(\left\|g\right\|_{\dot{H}^2}+\left\|g\right\|_{\dot{H}^1}).
\end{equation*}
For the second term, identically we obtain
\begin{equation*}
    \begin{aligned}
     \left|\lelan \sum_{i=1}^4\frac{(y_k')_i}{\lambda_k^3}\partial_i(\Delta\nabla W)\left(\frac{x-\boldsymbol{y}_k}{\lambda_k}\right), g\rilan\right|\lesssim |\boldsymbol{y}_k'|\lambda_k^{1-\epsilon}(\left\|g\right\|_{\dot{H}^2}+\left\|g\right\|_{\dot{H}^1}).
    \end{aligned}
\end{equation*}
Next, we estimate the term involving $\dot{g}$.  By Cauchy's inequality, it holds
\begin{equation*}
    \lelan \frac{1}{\lambda_k}\left(\Delta_k\nabla_k W_k \right),\dot{g}\rilan\lesssim \left\|\dot{g}\right\|_{L^2}\text{ ,  and } \lelan \frac{1}{\lambda_k}\left(\Delta_k\nabla_k W_k \right),\dot{g}\rilan\lesssim \lambda_k\left\|\dot{g}\right\|_{\dot{H}^1}.
\end{equation*}
Finally for the remaining part, by symmetry $\lelan \Delta \nabla W, \Lambda W  \rilan=0$, the summation equals to 
\begin{equation*}
\begin{aligned}
     &\lelan \frac{1}{\lambda_k}\left(\Delta_k\nabla_k W_k \right),\sum_{j=1}^N \lambda_j^{-1}(\lambda_j'\Lambda_jW_j+b_j(\Lambda_j W_j)\chi_{t})+\sum_{j=1}^N\lambda_j^{-1}(\boldsymbol{y}_j'+\boldsymbol{v}_j)\cdot\nabla_j W_j\rilan \\
     =&(\boldsymbol{y}_k'+\boldsymbol{v}_k)\left\|\nabla W\right\|_{\dot{H}^1}^2+\lelan \frac{1}{\lambda_k}\left(\Delta_k\nabla_k W_k \right),\sum_{j\neq k}\lambda_j^{-1}((\boldsymbol{y}_j'+\boldsymbol{v}_j)\cdot\nabla_j W_j+\lambda_j'\Lambda_jW_j+b_j(\Lambda_j W_j)\chi_{t})\rilan.
\end{aligned}
\end{equation*}
The diagonal $j=k$ term contributes to the leading term. For the off-diagonal $j\neq k$ case, since $\Delta \nabla W\sim \lelan r \rilan^{-7}$, in the near $k$-bubble region $|x-\boldsymbol{y}_k|\leq d$, it holds
\begin{equation*}
    \begin{aligned}
 &\left| \lelan \frac{1}{\lambda_k}\left(\Delta_k\nabla_k W_k \right),\sum_{j\neq k}\lambda_j^{-1}((\boldsymbol{y}_j'+\boldsymbol{v}_j)\cdot\nabla_j W_j+\lambda_j'\Lambda_jW_j+b_j(\Lambda_j W_j)\chi_{t})\rilan\right|\\
\lesssim& \sum_{j\neq k} \int_{\lambda_k}^{d}\frac{1}{\lambda_k^2}\Big(\frac{\lambda_k}{r}\Big)^7\frac{\lambda_j|\boldsymbol{y}_j'+\boldsymbol{v}_j|}{|\boldsymbol{y}_k-\boldsymbol{y}_j|^3}r^3dr+\sum_{j\neq k} \int_{\lambda_k}^{\infty}\frac{1}{\lambda_k^2}\Big(\frac{\lambda_k}{r}\Big)^7\frac{|b_j|+|\lambda_j'|}{|\boldsymbol{y}_k-\boldsymbol{y}_j|^2}r^3dr\\
 \lesssim&\sum_{j\neq k}\frac{|\boldsymbol{y}_j'+\boldsymbol{v}_j|\lambda_k^2\lambda_j}{|\boldsymbol{y}_j-\boldsymbol{y}_k|^3}+\frac{(|b_j|+|\lambda_j'|)\lambda_k^2}{|\boldsymbol{y}_j-\boldsymbol{y}_k|^2}.
    \end{aligned}
\end{equation*}
Here we assume that $d\leq \inf_{j\neq k}\frac{|\boldsymbol{y}_j-\boldsymbol{y}_k|}{4}$. In the region near $j$-bubble $|x-\boldsymbol{y}_j|\leq d$, similarly we have
\begin{equation*}
    \begin{aligned}
         &\left| \lelan \frac{1}{\lambda_k}\left(\Delta_k\nabla_k W_k \right),\sum_{j\neq k}\lambda_j^{-1}((\boldsymbol{y}_j'+\boldsymbol{v}_j)\cdot\nabla_j W_j+\lambda_j'\Lambda_jW_j+b_j(\Lambda_j W_j)\chi_{t})\rilan\right|\\
         \lesssim&\sum_{j\neq k}\int_{\lambda_j}^{d}\frac{\lambda_k^5}{|\boldsymbol{y}_k-\boldsymbol{y}_j|^7}\frac{|\boldsymbol{y}_j'+\boldsymbol{v}_j|}{\lambda_j^2}\left(\frac{\lambda_j}{r}\right)^3r^3dr+\sum_{j\neq k}\int_{\lambda_j}^{d}\frac{\lambda_k^5}{|\boldsymbol{y}_k-\boldsymbol{y}_j|^7}\frac{|b_j|+|\lambda_j'|}{\lambda_j^2}\left(\frac{\lambda_j}{r}\right)^2r^3dr\\
         \lesssim& \sum_{j\neq k}\frac{\lambda_k^5\lambda_j|\boldsymbol{y}_j'+\boldsymbol{v}_j|}{|\boldsymbol{y}_k-\boldsymbol{y}_j|^7}+\frac{\lambda_k^5(|b_j|+|\lambda_j'|)}{|\boldsymbol{y}_k-\boldsymbol{y}_j|^7}.
    \end{aligned}
\end{equation*}
Finally, in the exterior region $\mathbb R^4\setminus \bigcup_{j=1}^N B(|x-\boldsymbol{y}_j|\leq d)$, the decay of $W$ yields
\begin{equation*}
    \begin{aligned}
         &\left| \lelan \frac{1}{\lambda_k}\left(\Delta_k\nabla_k W_k \right),\sum_{j\neq k}\lambda_j^{-1}((\boldsymbol{y}_j'+\boldsymbol{v}_j)\cdot\nabla_j W_j+\lambda_j'\Lambda_jW_j+b_j(\Lambda_j W_j)\chi_{t})\rilan\right|\\ 
         \lesssim&\int_{d}^{\infty}\frac{\lambda_k^5}{r^7}\frac{\lambda_j}{r^3}\left(\boldsymbol{y}_j'+\boldsymbol{v}_j\right)r^3dr+\int_{d}^{\infty}\frac{\lambda_k^5}{r^7}\frac{|\lambda_j'|+|b_j|}{r^2}r^3dr\lesssim \lambda_k^5.
    \end{aligned}
\end{equation*}
In conclusion, collecting the estimates above gives
\begin{equation*}
    |\boldsymbol{y}_k'+\boldsymbol{v}_k|\lesssim (1+|\lambda_k'|+|\boldsymbol{y}_k'|) \left\|\Vec{g}\right\|_{\ME}+\sum_{j\neq k}\frac{|\boldsymbol{y}_j'+\boldsymbol{v}_j|\lambda_k^2\lambda_j}{|\boldsymbol{y}_j-\boldsymbol{y}_k|^3}+\frac{(|b_j|+|\lambda_j'|)\lambda_k^2}{|\boldsymbol{y}_j-\boldsymbol{y}_k|^2},
\end{equation*}
and
\begin{equation}\label{y'+v 2nd}
\begin{aligned}
   |\boldsymbol{y}_k'+\boldsymbol{v}_k|\lesssim & {\lambda_k^{1-\epsilon}}(\lambda_k^{\epsilon}+|\lambda_k'|+|\boldsymbol{y}_k'|) (\left\|\Vec{g}\right\|_{\ME_2}+\left\|\Vec{g}\right\|_{\ME})\\&+\sum_{j\neq k}\frac{|\boldsymbol{y}_j'+\boldsymbol{v}_j|\lambda_k^2\lambda_j}{|\boldsymbol{y}_j-\boldsymbol{y}_k|^3}+\frac{(|b_j|+|\lambda_j'|)\lambda_k^2}{|\boldsymbol{y}_j-\boldsymbol{y}_k|^2}.  
\end{aligned}
\end{equation}
 Furthermore, (\ref{y'+v}) follows from (\ref{y'+v 2nd}). We note that the corresponding estimate at the level of the second order energy provides a sharper control.

\textbf{Step 3. Control of the correction parameter $\beta$ (\ref{b and beta}) and (\ref{beta'+=to}).} We begin with $(\ref{b and beta})$. In view of the definition of $\beta_k$, it holds
\begin{equation*}
    |b_k-\beta_k|\lesssim\left|\frac{\sum_{j=1}^N\lelan \dot{g},\frac{1}{\lambda_j}(\Lambda_j W_j)\chi_{t} \rilan}{\lelan \frac{1}{\lambda_k}(\Lambda_k W_k)\chi_{t}, \frac{1}{\lambda_k}(\Lambda_kW_k)\chi_{t} \rilan}\right|\lesssim t^{-\frac{2}{3}}\left|\lelan \dot{g},\frac{1}{\lambda_k}(\Lambda_k W_k)\chi_{t} \rilan\right|.
\end{equation*}
Invoking the bootstrap assumptions (\ref{bootstap a for energy}) and (\ref{bootstap a for 2nd energy}), we compute
\begin{equation*}
\begin{aligned}
    \left|\lelan \dot{g},\frac{1}{\lambda_j}(\Lambda_j W_j)\chi_{t} \rilan\right|\lesssim& \left|\int_{0}^{\frac{1}{t}}\dot{g}\frac{1}{\lambda_j}(\Lambda_j W_j)\chi_{t}   \right|+\left|\int_{\frac{1}{t}}^{2t}\dot{g}\frac{1}{\lambda_j}(\Lambda_j W_j)\chi_{t}   \right|\\
    \lesssim& \left\|\dot{g}\right\|_{\dot{H}^1}\left(\int_{\lambda_j}^{\frac{1}{t}}\left(\frac{1}{r^2}\right)^{\frac{4}{3}}r^3dr\right)^{\frac{3}{4}}+\left\|\dot{g}\right\|_{L^2}\left(\int_{\frac{1}{t}}^t\left(\frac{1}{r^2}\right)^2r^3dr\right)^{\frac{1}{2}}\\
    \lesssim& \left\|\dot{g}\right\|_{\dot{H}^1}\frac{1}{t}+\left\|\dot{g}\right\|_{L^2}\sqrt{\log t}\lesssim (t^{-\frac{1}{6}}\sqrt{\log t})\lambda,
\end{aligned}
\end{equation*}
which completes the proof (\ref{b and beta}).  We next turn to (\ref{beta'+=to}). Taking inner product with $\frac{1}{\lambda_k}(\Lambda_k W_k)\chi_{t}$ on both sides of (\ref{partial g dot}) yields
\begin{equation}\label{g and W inner with Lambda W}
    \begin{aligned}
        &\lelan \partial_t \dot g, \frac{1}{\lambda_k}(\Lambda_k W_k)\chi_{t}\rilan-\lelan \Delta g+f(W_{\boldsymbol{\Gamma}}+g)-\sum_{j=1}^Nf(W_j) ,\frac{1}{\lambda_k}(\Lambda_k W_k)\chi_{t}\rilan\\
        =
        &\lelan  \sum_{j=1}^N\left(\frac{\lambda_j'b_j}{\lambda_j^2} (\underline{\Lambda}_j\Lambda_jW_j)\chi_{t}
        -\frac{b_j'}{\lambda_j}(\Lambda_jW_j\chi_{t})+\frac{b_j|x-\boldsymbol{y}_j|}{\lambda_jt^2}(\Lambda_jW_j\chi'_{t})\right),\frac{1}{\lambda_k}(\Lambda_k W_k)\chi_{t}\rilan\\
        &+\lelan \sum_{j=1}^N\frac{b_j}{\lambda_j^2}(y_j'\cdot\nabla)(\Lambda_j W_j\chi_{t})
          -\sum_{j=1}^N\sum_{l=1}^4\frac{v'_{j,l}(\partial_lW)(\frac{\cdot-\boldsymbol{y}_j}{\lambda_j})}{\lambda_j^2} ,\frac{1}{\lambda_k}(\Lambda_k W_k)\chi_{t}\rilan\\
        &-\lelan \sum_{j=1}^N\sum_{l=1}^4\frac{\lambda'_jv_{j,l}(\underline{\Lambda} \partial_l W)\left(\frac{\cdot-\boldsymbol{y}_j}{\lambda_j}\right)}{\lambda_j^3}-\sum_{j=1}^N\sum_{l,m=1}^4\frac{v_{j,l}y'_{j,m}(\partial_l\partial_mW)(\frac{\cdot-\boldsymbol{y}_j}{\lambda_j})}{\lambda_j^3} ,\frac{1}{\lambda_k}(\Lambda_k W_k)\chi_{t}\rilan. 
    \end{aligned}
\end{equation}
Proceeding term by term, we estimate both sides of (\ref{g and W inner with Lambda W}). Starting from the left hand side, the term containing $\partial_t \dot{g}$ can be expanded as
\begin{equation}\label{dot g and Lambda k Lambda W chi t}
  \lelan \partial_t \dot g, \frac{1}{\lambda_k}(\Lambda_k W_k)\chi_{t}\rilan=\frac{d}{dt}\lelan \dot{g},\frac{1}{\lambda_k}(\Lambda_k W_k)\chi_{t} \rilan-\lelan \dot{g}, \frac{d}{dt}\left(\frac{1}{\lambda_k}(\Lambda_k W_k)\chi_{t}\right) \rilan . 
\end{equation}
The first term here, from the definition of $\beta_k$, equals to
\begin{equation}\label{d dt beta k- bk L2 equal to}
\begin{aligned}
    &\frac{d}{dt}\left(\lelan \frac{1}{\lambda_k}(\Lambda_k W_k)\chi_{t},\sum_{j=1}^N \frac{(\beta_j-b_j)}{\lambda_j}(\Lambda_j W_j)\chi_{t} \rilan\right)\\
    =&\lelan \frac{1}{\lambda_k}(\Lambda_k W_k)\chi_{t},\sum_{j=1}^N \frac{(\beta_j'-b_j')}{\lambda_j}(\Lambda_j W_j)\chi_{t} \rilan+\sum_{j=1}^N(\beta_j-b_j)\frac{d}{dt}\left(\lelan \frac{1}{\lambda_k}(\Lambda_k W_k)\chi_{t}, \frac{1}{\lambda_j}(\Lambda_j W_j)\chi_{t} \rilan\right). 
\end{aligned}
\end{equation}
We keep the term $\lelan \frac{1}{\lambda_k}(\Lambda_k W_k)\chi_{t},\sum_{j=1}^N \frac{(\beta_j'-b_j')}{\lambda_j}(\Lambda_j W_j)\chi_{t} \rilan$. For the remaining part, differentiating the inner product gives
\[
\left|
\frac{d}{dt}
\left\langle
\frac1{\lambda_k}(\Lambda_k W_k)\chi_t,
\frac1{\lambda_j}(\Lambda_j W_j)\chi_t
\right\rangle
\right|
\lesssim
\frac1t+\frac{|\lambda_k'|}{\lambda_k}
+\frac{|\lambda_j'|}{\lambda_j}
+|\boldsymbol{y}_k'|+|\boldsymbol{y}_j'|.
\]
Here the contribution of $\partial_t\chi_t$ is of order $t^{-1}$, while the derivatives of the scaling parameters give the terms
$|\lambda_k'|/\lambda_k$ and $|\lambda_j'|/\lambda_j$. By (\ref{rough of mu}), these quantities are bounded by $t^{-1/3}$. Hence
\[
\left|
\frac{d}{dt}
\left\langle
\frac1{\lambda_k}(\Lambda_k W_k)\chi_t,
\frac1{\lambda_j}(\Lambda_j W_j)\chi_t
\right\rangle
\right|
\lesssim t^{-1/3}.
\]
Therefore, from (\ref{b and beta}), the corresponding contribution is controlled by
\begin{equation}\label{beta-b d dt Lambda W}
\begin{aligned}
     &\left|\sum_{j=1}^N(\beta_j-b_j)\frac{d}{dt}\left(\lelan \frac{1}{\lambda_k}(\Lambda_k W_k)\chi_{t}, \frac{1}{\lambda_j}(\Lambda_j W_j)\chi_{t} \rilan\right)\right|\\
     \lesssim& (t^{-\frac{5}{6}}\sqrt{\log t})\lambda t^{-\frac{1}{3}}\lesssim (t^{-\frac{7}{6}}\sqrt{\log t})\lambda.   
\end{aligned}
\end{equation}
 Then for the second term in (\ref{dot g and Lambda k Lambda W chi t}), a direct computation gives 
\begin{equation*}
 \frac{d}{dt}\left(\frac{1}{\lambda_k}(\Lambda_k W_k)\chi_{t} \right)= -\frac{\lambda_k'}{\lambda_k^2}(\underline{\Lambda}_k\Lambda_kW_k)\chi_{t}-\frac{\boldsymbol{y'_k}}{\lambda_k^2}\cdot(\nabla_k\Lambda_kW_k)\chi_{t}+\frac{1}{\lambda_k}(\Lambda_k W_k)(\frac{d}{dt}\chi_{t}). 
\end{equation*}
Notice that $\underline{\Lambda}\Lambda W\in L^{\frac{4}{3}}$, $\nabla \Lambda W\in L^2$, using bootstrap assumption yields
\begin{equation*}
    \left|\lelan \dot{g}, \frac{\lambda_k'}{\lambda_k^2}(\underline{\Lambda}_k\Lambda_kW_k)\chi_{t}\rilan\right|\lesssim \left\|\dot{g}\right\|_{\dot{H}^1}\left\|\frac{\lambda_k'}{\lambda_k^2}(\underline{\Lambda}_k\Lambda_kW_k)\chi_{t}\right\|_{L^{\frac{4}{3}}}\lesssim \left\|\dot{g}\right\|_{\dot{H}^1}|\lambda_k'|\ll \lambda^{\frac{3}{2}},
\end{equation*}
and 
\begin{equation*}
    \left|\lelan  \dot{g}, \frac{\boldsymbol{y'_k}}{\lambda_k^2}\cdot(\nabla_k\Lambda_kW_k)\chi_{t} \rilan\right|\lesssim \left\|\dot{g}\right\|_{L^2}\frac{|\boldsymbol{y}_k'|}{\lambda_k}\ll \lambda^{\frac{3}{2}}.
\end{equation*}
Furthermore,  from (\ref{bootstrap a for y}), the term concerning the derivative of cut-off function is dominated by 
\begin{equation*}
    \left\|\frac{1}{\lambda_k}(\Lambda_k W_k)(\frac{d}{dt}\chi_{t})\right\|_{L^2}\lesssim \left(\int_t^{2t}\left(\frac{1}{r^2t}\right)^2r^3dr\right)^{\frac{1}{2}}\lesssim \frac{1}{t},
\end{equation*}
which, combined with the two inequalities above leads to
\begin{equation}\label{sharp est of b' left 1.2}
\begin{aligned}
   \left|\lelan \dot{g}, \frac{d}{dt}\left(\frac{1}{\lambda_k}(\Lambda_k W_k)\chi_{t}\right) \rilan\right|\lesssim \lambda^{\frac{3}{2}}+\left\|\dot{g}\right\|_{L^2} \left\|\frac{1}{\lambda_k}(\Lambda_k W_k)(\frac{d}{dt}\chi_{t})\right\|_{L^2}\lesssim  (t^{-\frac{7}{6}})\lambda.
\end{aligned}
\end{equation}
Therefore, collecting (\ref{d dt beta k- bk L2 equal to}) , (\ref{beta-b d dt Lambda W}) and (\ref{sharp est of b' left 1.2}) together, the first term of (\ref{g and W inner with Lambda W}) is then estimated by 
\begin{equation}\label{sharp est of b left 1}
     \left|\lelan \partial_t \dot g, \frac{1}{\lambda_k}(\Lambda_k W_k)\chi_{t}\rilan-\lelan \frac{1}{\lambda_k}(\Lambda_k W_k)\chi_{t},\sum_{j=1}^N \frac{(\beta_j'-b_j')}{\lambda_j}(\Lambda_j W_j)\chi_{t} \rilan\right|\lesssim (t^{-\frac{7}{6}}\sqrt{\log t})\lambda.
\end{equation}
Next, we estimate the contribution of $\Delta g+f(W_{\boldsymbol{\Gamma}}+g)-\sum_{j=1}^Nf(W_j)$. It is natural to split this term into
\begin{equation}\label{sharp est of b left 2}
   \begin{aligned}
       &-\lelan \frac{1}{\lambda_k}  (\Lambda_kW_k)\chi_{t}, \Delta g+f(W_{\boldsymbol{\Gamma}}+g)-\sum_{j=1}^Nf(W_j)\rilan\\
       =&-\lelan \frac{1}{\lambda_k} (\Lambda_kW_k)\chi_{t}, \Delta g+f(W_{\boldsymbol{\Gamma}}+g)-f(W_{\boldsymbol{\Gamma}})\rilan-\lelan \frac{1}{\lambda_k} (\Lambda_kW_k)\chi_{t},f(W_{\boldsymbol{\Gamma}})-\sum_{j=1}^Nf(W_j)\rilan.
   \end{aligned} 
\end{equation}
Using the definition of $W_{\boldsymbol{\Gamma}}$ we derive
\begin{equation}\label{sharp est of b left 2.0}
    \begin{aligned}
   &-\lelan \frac{1}{\lambda_k} (\Lambda_kW_k)\chi_{t}, \Delta g+f(W_{\boldsymbol{\Gamma}}+g)-f(W_{\boldsymbol{\Gamma}})\rilan  \\
   =&\lelan \frac{1}{\lambda_k} (\Lambda_kW_k)\chi_{t}, \ML_k g \rilan+\lelan \frac{1}{\lambda_k} (\Lambda_kW_k)\chi_{t}, (3W_{\boldsymbol{\Gamma}}^2-3W_k^2)g+3W_{\boldsymbol{\Gamma}}g^2+g^3 \rilan,
    \end{aligned}
\end{equation}
which can be regarded as the major linear part plus the multi-bubble and nonlinear part. Due to the fact that $\ML_k$ is a self-adjoint operator and using Cauchy inequality, the linear part can be estimated by
\begin{equation}\label{sharp est of b left 2.1.0}
    \begin{aligned}
    \lelan \frac{1}{\lambda_k} (\Lambda_kW_k)\chi_{t}, \ML_k g \rilan=&\lelan \frac{1}{\lambda_k}\ML_k[(\Lambda_kW_k)\chi_{t}],g \rilan
    \lesssim \left\|\frac{1}{\lambda_k}\ML_k[(\Lambda_kW_k)\chi_{t}]\right\|_{L^{\frac{4}{3}}}\left\|g\right\|_{\dot{H}^1}.
    \end{aligned}
\end{equation}
Provided that $\ML_k(\Lambda_k W_k)=0$, the norm $\left\|\frac{1}{\lambda_k}\ML_k[(\Lambda_kW_k)\chi_{t}]\right\|_{L^{\frac{4}{3}}}$ is bounded by
\begin{equation*}
    \begin{aligned}
        \left\|\frac{1}{\lambda_k}\ML_k[(\Lambda_kW_k)\chi_{t}]\right\|_{L^{\frac{4}{3}}}=\left(\int_{t\leq |x-\boldsymbol{y}_k|\leq 2t}\left(\frac{1}{\lambda_k}\nabla(\Lambda_kW_k)\cdot \nabla(\chi_{t})+\frac{1}{\lambda_k}(\Lambda
        _kW_k)\Delta(\chi_{t})\right)^{\frac{4}{3}}dx\right)^{\frac{3}{4}}.
    \end{aligned}
\end{equation*}
Since $|\nabla \Lambda W(x)|\sim \langle x\rangle^{-3}$, $|\nabla (\chi_{t})|\lesssim \frac{1}{t}$ and $|\Delta(\chi_{t})|\lesssim \frac{1}{t^2}$, it holds
\begin{equation*}
    \begin{aligned}
  &\left(\int_{t\leq |x-\boldsymbol{y}_k|\leq 2t}\left(\frac{1}{\lambda_k}\nabla(\Lambda_kW_k)\cdot \nabla(\chi_{t})+\frac{1}{\lambda_k}(\Lambda
        _kW_k)\Delta(\chi_{t})\right)^{\frac{4}{3}}dx\right)^{\frac{3}{4}}\\
  \lesssim&\left(\int_{t}^{2t}\left(\frac{1}{r^3}\cdot\frac{1}{t}\right)^{\frac{4}{3}}r^3dr\right)^{\frac{3}{4}}  +\left(\int_{t}^{2t}\left(\frac{1}{r^2}\cdot\frac{1}{t^2}\right)^{\frac{4}{3}}r^3dr\right)^{\frac{3}{4}}          \lesssim \frac{1}{t}.
    \end{aligned}
\end{equation*}
Plugging this back into (\ref{sharp est of b left 2.1.0}) yields 
\begin{equation*}
   \left|\lelan \frac{1}{\lambda_k} (\Lambda_kW_k)\chi_{t}, \ML_k g \rilan\right|
    \lesssim \frac{1}{t}\left\|g\right\|_{\dot{H}^1} \lesssim t^{-\frac{7}{6}}\lambda,
\end{equation*}
 Returning to (\ref{sharp est of b left 2.0}), it suffices to handle the multi-bubble and nonlinear part. Indeed, the nonlinear interaction, using Cauchy, is bounded by   
\begin{equation*}
    \begin{aligned}
    \lelan \frac{1}{\lambda_k} (\Lambda_kW_k)\chi_{t}, (3W_{\boldsymbol{\Gamma}}^2-3W_k^2)g \rilan \lesssim&\left\|\frac{1}{\lambda_k} (\Lambda_kW_k)\chi_{t}, (3W_{\boldsymbol{\Gamma}}^2-3W_k^2)\right\|_{L^{\frac{4}{3}}}\left\|g\right\|_{\dot{H}^1}\\
    \lesssim&\sum_{j\neq k}\left(\int_{\lambda_k}^t\left(\frac{1}{r^2}\cdot\frac{\lambda_k\lambda_j}{r^2|\boldsymbol{y}_k-\boldsymbol{y}_j|^2}\right)^{\frac{4}{3}}r^3dr\right)^{\frac{3}{4}}\left\|g\right\|_{\dot{H}^1}\\
    \lesssim&\sum_{j\neq k}\frac{\lambda_j}{|\boldsymbol{y}_k-\boldsymbol{y}_j|^2}\left\|g\right\|_{\dot{H}^1}\lesssim t^{-\frac{1}{6}}\lambda^2.
    \end{aligned}
\end{equation*}
Moreover, from (\ref{bootstap a for energy}), (\ref{bootstap a for 2nd energy}) and interpolation inequality, it holds
\begin{equation}\label{h dot 3 2 for g}
    \left\|g\right\|_{\dot{H}^{\frac{3}{2}}}\lesssim \left\|g\right\|_{\dot{H}^2}^{\frac{1}{2}}\left\|g\right\|_{\dot{H}^1}^{\frac{1}{2}}\lesssim t^{\frac{1}{6}}\lambda.
\end{equation}
Applying Cauchy inequality to the higher order nonlinear part then yields
\begin{equation*}
\begin{aligned}
  \left| \lelan \frac{1}{\lambda_k} (\Lambda_kW_k)\chi_{t}, 3W_{\boldsymbol{\Gamma}}g^2+g^3 \rilan \right|\lesssim& \left\|\frac{1}{\lambda_k} (\Lambda_kW_k)\chi_{t}\right\|_{L^{\frac{8}{3}}} \left\|W_{\boldsymbol{\Gamma}}\right\|_{L^{\frac{8}{3}}}\left\|g\right\|_{\dot{H}^{\frac{3}{2}}}^2+\frac{1}{\lambda_k}\left\|g\right\|_{\dot{H}^1}^3\\
  \lesssim& \sum_{j=1}^N\frac{\sqrt{\lambda_j}}{\sqrt{\lambda_k}}\left\|g\right\|_{\dot{H}^{\frac{3}{2}}}^2 +\frac{1}{\lambda_k}\left\|g\right\|_{\dot{H}^1}^3\ll\lambda^{\frac{3}{2}},  
\end{aligned}
\end{equation*}
where $\epsilon>0$ is a small constant as before. Collecting these estimates we arrive at 
\begin{equation*}\label{sharp est of b left 2.1 final}
    \left|\lelan \frac{1}{\lambda_k} (\Lambda_kW_k)\chi_{t}, \Delta g+f(W_{\boldsymbol{\Gamma}}+g)-f(W_{\boldsymbol{\Gamma}})\rilan\right|\lesssim t^{-\frac{7}{6}}\lambda,
\end{equation*}
which completes the control of the first term of (\ref{sharp est of b left 2}). To tackle the second term, in view of (\ref{est:W^2 Lambda W}),  we get
\begin{equation*}\label{sharp est of b' left 2.2}
\begin{aligned}
   &\lelan \frac{1}{\lambda_k} (\Lambda_kW_k)\chi_{t},f(W_{\boldsymbol{\Gamma}})-\sum_{j=1}^Nf(W_j)\rilan\\
   =&-128\pi^2c\lambda+\lelan \frac{1}{\lambda_k}(\Lambda_k W_k)(1-\chi_{t}),\sum_{j\neq k}3W_k^2W_j \rilan+\sum_{j\neq k}\kappa\Big(\frac{ \lambda_j}{| \boldsymbol{y}_k-\boldsymbol y_j|^2}-\frac{ \lambda_j}{|\boldsymbol z_k-\boldsymbol z_j|^2}\Big)\\
   &+\lelan \frac{1}{\lambda_k} (\Lambda_kW_k)\chi_{t}, f(W_{\boldsymbol{\Gamma}})-\sum_{j=1}^Nf(W_j)-\sum_{j\neq k}3W_k^2W_j\rilan+O(\lambda^3).
\end{aligned}
\end{equation*}
The third term appears because \(B_k(\boldsymbol{\lambda},\boldsymbol{y})\) is defined using the fixed centers \(\boldsymbol{y}_j\). The expansion of \(W_j\) around the \(k\)-th bubble gives \(\lambda_j/|\boldsymbol{y}_k-\boldsymbol{y}_j|^2\); adding and subtracting \(\lambda_j/|\boldsymbol{z}_k-\boldsymbol{z}_j|^2\) isolates the leading contribution \(128\pi^2B_k(\boldsymbol{\lambda},\boldsymbol{y})\) and leaves precisely this difference. Applying (\ref{bootstrap a for y}), this term is bounded by
\begin{equation*}
\begin{aligned}
     \left|\sum_{j\neq k}\kappa\Big(\frac{ \lambda_j}{|\boldsymbol{y}_k-\boldsymbol y_j|^2}-\frac{ \lambda_j}{|\boldsymbol z_k-\boldsymbol z_j|^2}\Big)\right|\leq& \sum_{j\neq k}\kappa \lambda_j \left(\frac{1}{|\boldsymbol z_k-\boldsymbol z_j|^2}-\frac{1}{|\boldsymbol z_k-\boldsymbol{z}_j|^2+|\boldsymbol{y}_k-\boldsymbol{z}_k|^2+|\boldsymbol{y}_j-\boldsymbol{z}_j|^2}\right) \\
     \lesssim&\frac{\lambda(|\boldsymbol{y}_k-\boldsymbol{z}_k|^2+|\boldsymbol{y}_j-\boldsymbol{z}_j|^2)}{|\boldsymbol{z}_k-\boldsymbol{z}_j|^2(|\boldsymbol{z}_k-\boldsymbol{z}_j|^2+|\boldsymbol{y}_k-\boldsymbol{z}_k|^2+|\boldsymbol{y}_j-\boldsymbol{z}_j|^2)}\ll \lambda^4,
\end{aligned}
\end{equation*}
which is negligible. For the second term, a computation yields
\begin{equation*}
   \left|\lelan \frac{1}{\lambda_k}(\Lambda_k W_k)(1-\chi_{t}),\sum_{j\neq k}3W_k^2W_j \rilan\right| \lesssim \int_t^{\infty}\frac{1}{r^2}\frac{\lambda_k^2}{r^4}\frac{\lambda_j}{t^2}r^3dr\lesssim \lambda^3,
\end{equation*}
which is also sufficiently small. As for the remaining part, again using the decay of $\Lambda W$ and $W$ we have
\begin{equation*}
    \begin{aligned}
      &\left|\lelan \frac{1}{\lambda_k} (\Lambda_kW_k)\chi_{t}, f(W_{\boldsymbol{\Gamma}})-\sum_{j=1}^Nf(W_j)-\sum_{j\neq k}3W_k^2W_j\rilan\right|\\
      \lesssim& \left|\int\frac{1}{\lambda_k} (\Lambda_kW_k)\chi_{t}W_j^2W_kdx  \right|
      \lesssim \int_{\lambda_k}^t\frac{1}{r^2}\frac{\lambda_j^2}{|z_k-z_j|^4}\frac{\lambda_k}{r^2}r^3dr\\\lesssim& \lambda_k\lambda_j^2\log(t/\lambda_k)\lesssim t\lambda^3.
    \end{aligned}
\end{equation*}
Consequently, collecting these inequalities yields
\begin{equation*}
 \left|\lelan \frac{1}{\lambda_k} (\Lambda_kW_k)\chi_{t},f(W_{\boldsymbol{\Gamma}})-\sum_{j=1}^Nf(W_j)\rilan+128\pi^2c\lambda\right|\lesssim   t\lambda^3. 
\end{equation*}
In conclusion, the left hand side (we denote it as $LHS$) of (\ref{g and W inner with Lambda W}) is estimated by
\begin{equation*}\label{right hand side of key beta'}
    \left|LHS-\lelan \frac{1}{\lambda_k}(\Lambda_k W_k)\chi_{t},\sum_{j=1}^N \frac{(\beta_j'-b_j')}{\lambda_j}(\Lambda_j W_j)\chi_{t} \rilan-128\pi^2c\lambda\right|\lesssim (t^{-\frac{7}{6}}\sqrt{\log t})\lambda,
\end{equation*}
where the leading order of the error term comes from (\ref{beta-b d dt Lambda W}). We emphasize that this left-hand side contains the leading modulation term, which compensates the corresponding $b_k'$ contribution on the right-hand side. We now turn to the remaining terms. One can therefore replace the $b_j$ with $\beta_j$ and leave a small error which can be absorbed by $(t^{-\frac{7}{6}}\sqrt{\log t})\lambda$ . Then we invest the whole summation. The major contribution comes from 
\begin{equation*}
    \begin{aligned}
     &\lelan \left(\frac{\lambda_k'\beta_k}{\lambda_k^2} (\underline{\Lambda}_k\Lambda_kW_k)\chi_{t}
        -\sum_{j=1}^N\frac{\beta_j'}{\lambda_j}(\Lambda_jW_j\chi_{t})+\frac{\beta_k|x-\boldsymbol{y}_k|}{\lambda_kt^2}(\Lambda_kW_k\chi'_{t})\right),\frac{1}{\lambda_k}(\Lambda_k W_k)\chi_{t}\rilan \\
    +&\lelan \sum_{l=1}^4\frac{v_{k,l}(t)y'_{k,l}(t)\partial_l\partial_lW(\frac{\cdot-\boldsymbol{y}_k}{\lambda_k})}{\lambda_k^3},\frac{1}{\lambda_k}(\Lambda_k W_k)\chi_{t}\rilan:=I_{k}.
    \end{aligned}
\end{equation*}
Compare this with the equation of (\ref{equation of lambda and b}) , we claim that 
\begin{equation*}
    \left|I_k+128\pi^2\left(\Xi(t)\beta'(t)+\frac{1}{2}\frac{\beta(t)^2}{\lambda(t)}\right)\right|\lesssim t^{-\frac{4}{3}}\lambda.
\end{equation*}
holds. To prove the claim, recall the computation in Section \ref{A formal computation} and the definition of $\Xi(t)$, it suffices to prove 
\begin{equation*}
  \left|\lelan \left(\frac{\beta_k|x-\boldsymbol{y}_k|}{\lambda_kt^2}(\Lambda_kW_k\chi'_{t})+\sum_{l=1}^4\frac{v_{k,l}(t)y'_{k,l}(t)\partial_l\partial_lW(\frac{\cdot-\boldsymbol{y}_k}{\lambda_k})}{\lambda_k^3}\right),\frac{1}{\lambda_k}(\Lambda_k W_k)\chi_{t}\rilan \right|  \lesssim t^{-\frac{4}{3}}\lambda.
\end{equation*}
This can be achieved by the computation
\begin{equation*}
    \begin{aligned}
        &\left|\lelan \left(\frac{\beta_k|x-\boldsymbol{y}_k|}{\lambda_kt^2}(\Lambda_kW_k\chi'_{t})+\sum_{l=1}^4\frac{v_{k,l}(t)y'_{k,l}(t)\partial_l\partial_lW(\frac{\cdot-\boldsymbol{y}_k}{\lambda_k})}{\lambda_k^3}\right),\frac{1}{\lambda_k}(\Lambda_k W_k)\chi_{t}\rilan \right|\\
        \lesssim&\int_{t}^{2t}\frac{\beta_k}{t}\frac{1}{r^4}r^3dr+\int_{\lambda_k}^t\sum_{l=1}^4\frac{|v_{k,l}y_{k,l}'|\lambda_k}{r^4}\frac{1}{r^2}r^3dr\lesssim t^{-\frac{4}{3}}\lambda.
    \end{aligned}
\end{equation*}
Therefore,  the claim is proved. For the remaining terms, since the symmetry property gives
\begin{equation*}
   \lelan \nabla_k \Lambda_k W_k,\Lambda_kW_k\rilan=\lelan \nabla_k W_k, \Lambda_k W_k \rilan=\lelan (\underline{\Lambda}\partial_l)_k W_k ,\Lambda_k W_k\rilan=\lelan (\partial_l \partial_m)_k W_k,\Lambda_k W_k \rilan=0, 
\end{equation*}
where $l\neq m$. Thus, this part equals to 
\begin{equation}\label{remaining beta sys}
\begin{aligned}
      \sum_{j\neq k}\Bigg\{&\lelan \frac{\lambda_j'b_j}{\lambda_j^2}(\underline{\Lambda}_j\Lambda_jW_j)\chi_t+\frac{b_j|x-\boldsymbol{y}_j|}{\lambda_jt^2}(\Lambda_jW_j\chi'_{t})+\frac{b_j}{\lambda_j^2}(\boldsymbol{y}_j'\cdot \nabla)(\Lambda_j W_j \chi_t), \frac{1}{\lambda_k}(\Lambda_k W_k)\chi_t \rilan\\ 
      -&\lelan\frac{\boldsymbol{v}_j'\cdot \nabla_j W_j}{\lambda_j} +\sum_{l=1}^4\frac{\lambda'_jv_{j,l}(\underline{\Lambda} \partial_l W)\left(\frac{\cdot-\boldsymbol{y}_j}{\lambda_j}\right)}{\lambda_j^3}+\sum_{l,m=1}^4\frac{v_{j,l}y'_{j,m}(\partial_l\partial_mW)(\frac{\cdot-\boldsymbol{y}_j}{\lambda_j})}{\lambda_j^3} ,\frac{1}{\lambda_k}(\Lambda_k W_k)\chi_{t}\rilan\Bigg\}.
\end{aligned}
\end{equation}
To estimate these terms, we divide the space $\mathbb R^4$ into $\Tilde{B}:=\bigcup_{1\leq j\leq N}B(|x-\boldsymbol{y}_j|\leq d)$ and $\mathbb R^4\setminus \widetilde B$, where $d:\inf_{j\neq k}\frac{1}{4}|\boldsymbol{y}_k-\boldsymbol{y}_j|$. We start from the inner region $\Tilde{B}$. First, consider the region  $B(|x-\boldsymbol{y}_k|\leq d)$, the decay of stationary function gives 
\begin{equation*}
(\underline{\Lambda}_j\Lambda_jW_j)\chi_{t}\sim\frac{\lambda_j^3}{|\boldsymbol{y}_k-\boldsymbol{y}_j|^4}\text{ , }\Lambda_jW_j\sim \frac{\lambda_j}{|\boldsymbol{y}_k-\boldsymbol{y}_j|^2}\text{ , }\nabla (\Lambda_j W_j)\sim \frac{\lambda_j^2}{|\boldsymbol{y}_k-\boldsymbol{y}_j|^3},
\end{equation*} 
\begin{equation*}
    \nabla_j W_j\sim \frac{\lambda_j^2}{|\boldsymbol{y}_k-\boldsymbol{y}_j|^3}\text{ , }(\underline{\Lambda}\partial_l W)\left(\frac{\cdot-\boldsymbol{y}_j}{\lambda_j}\right)\sim \frac{\lambda_j^3}{|\boldsymbol{y}_k-\boldsymbol{y}_j|^3}\text{ , }(\partial_l\partial_mW)\left(\frac{\cdot-\boldsymbol{y}_j}{\lambda_j}\right)\sim \frac{\lambda_j^4}{|\boldsymbol{y}_k-\boldsymbol{y}_j|^4}.
\end{equation*}
As a consequence, on this region (\ref{remaining beta sys}) is dominated by
\begin{equation*}
    \begin{aligned}
        &\sum_{j\neq k}\Bigg\{\frac{|\lambda_j'|b_j\lambda_j^3}{\lambda_j^2|\boldsymbol{y}_k-\boldsymbol{y}_j|^4}+\frac{b_j\lambda_j^2}{\lambda_jt|\boldsymbol{y}_k-\boldsymbol{y}_j|}+\frac{b_j|\boldsymbol{y}_j'|\lambda_j}{\lambda_j^2|\boldsymbol{y}_k-\boldsymbol{y}_j|^2}+\frac{|\boldsymbol{v}_j'|\lambda_j^2}{\lambda_j|\boldsymbol{y}_k-\boldsymbol{y}_j|^3}+\frac{|\lambda_j'||\boldsymbol{v}_j|\lambda_j^3}{\lambda_j^3|\boldsymbol{y}_k-\boldsymbol{y}_j|^3}+\frac{|\boldsymbol{y}_j'||\boldsymbol{v}_j|\lambda_j^4}{\lambda_j^3|\boldsymbol{y}_k-\boldsymbol{y}_j|^4}\Bigg\}\\
        \lesssim& \sum_{j\neq k}\Bigg\{|\lambda_j'|b_j\lambda_j+\frac{b_j\lambda_j}{t}+\frac{b_j|\boldsymbol{y_j'}|}{\lambda_j}+{|\boldsymbol{v}_j'|\lambda_j}+|\lambda_j'||\boldsymbol{v}_j|+|\boldsymbol{y}_j'||\boldsymbol{v}_j|\lambda_j\Bigg\}\ll \lambda^{\frac{3}{2}}.
    \end{aligned}
\end{equation*}
Then, on the region $B(|x-\boldsymbol{y}_j|\leq d)$, using the fact that $\frac{1}{\lambda_k}\Lambda_k W_k\sim \frac{1}{|\boldsymbol{y}_k-\boldsymbol{y}_j|^2}$ on this region and adopting the same method as above yields these terms are bounded by
\begin{equation*}
    \begin{aligned}
    \int_{\lambda_j}^d\Bigg(\frac{|\lambda_j'b_j\lambda_j|}{r^4}+\frac{b_j|\boldsymbol{y}_j'|}{\lambda_jr^3}+\frac{|\boldsymbol{v}_j'\lambda_j|}{r^3}+\frac{|\lambda_j'\boldsymbol{v}_j\lambda_j|}{r^3}+\frac{|\boldsymbol{v}_j|| y_j'\lambda_j|}{r^4}   \Bigg)\frac{r^3}{|\boldsymbol{y}_k-\boldsymbol{y}_j|^2}dr \ll\lambda^{\frac{3}{2}}.
    \end{aligned}
\end{equation*}
Finally, on the exterior region $\mathbb R^4\setminus \widetilde B$, with the same approach the integration is dominated by 
\begin{equation*}
    \sum_{j\neq k}\int_{d}^{2t}\left(\frac{|\lambda_j'|b_j\lambda_j}{r^6}+\frac{b_j|\boldsymbol{y}_j'|}{\lambda_jr^5}+\frac{|\boldsymbol{v}_j'\lambda_j|}{r^5}+\frac{|\lambda_j'\boldsymbol{v}_j\lambda_j|}{r^5}+\frac{|\boldsymbol{v}_j|| y_j'\lambda_j|}{r^6}  \right)r^3dr\ll \lambda^{\frac{3}{2}}.
\end{equation*}
Collecting all the inequalities in this step together, it holds
\begin{equation*}
   \left| \Xi(t)\beta'(t)+\frac{1}{2}\frac{\beta(t)^2}{\lambda(t)}+c\lambda\right|\lesssim  (t^{-\frac{7}{6}}\sqrt{\log t})\lambda,
\end{equation*}
on the bootstrap interval, which proves (\ref{beta'+=to}).

\textbf{Step 4. A rough estimate of $b_k'$ (\ref{rough est on b'}).}
In this step we estimate $b_k'$. Taking time derivative of 
$\lelan {\frac{1}{\lambda_k}(\Lambda_k W_k)\chi_{k}},\dot{g}\rilan=0$ and using the equation \eqref{partial g dot} yields
\begin{equation}\label{d dt ortho 3}
    \begin{aligned}
        0=&\frac{d}{dt} \lelan {\frac{1}{\lambda_k}(\Lambda_k W_k)\chi_{k}},\dot{g}\rilan\\
        =&-\lelan \frac{\lambda_k'}{\lambda_k^3}\left(\underline{\Lambda}[\Lambda W \chi]\right)\left(\frac{x-\boldsymbol{y}_k}{\lambda_k}\right),\dot{g} \rilan-\lelan  \sum_{i=1}^4\frac{1}{\lambda_k^3}[\partial_i(\Lambda W\chi)]\left(\frac{x-\boldsymbol{y}_k}{\lambda_k}\right)(y_k')_i,\dot{g}  \rilan\\
        &+\lelan \frac{1}{\lambda_k} (\Lambda_kW_k)\chi_{k}, \Delta g+f(W_{\boldsymbol{\Gamma}}+g)-\sum_{j=1}^Nf(W_j)\rilan\\
        &+\lelan \frac{1}{\lambda_k} (\Lambda_kW_k)\chi_{k}, \sum_{j=1}^N\frac{\lambda_j'b_j}{\lambda_j^2} (\underline{\Lambda}_j\Lambda_jW_j)\chi_{t}\rilan-\lelan \frac{1}{\lambda_k} (\Lambda_kW_k)\chi_{k}, \sum_{j=1}^N\frac{b_j'}{\lambda_j}(\Lambda_jW_j)\chi_{t} \rilan\\
        &+\lelan \frac{1}{\lambda_k} (\Lambda_kW_k)\chi_{k}, \sum_
        {j=1}^N\frac{b_j|x-\boldsymbol{y}_j|}{\lambda_jt^2}((\Lambda_jW_j)\chi_{t}') \rilan+\lelan \frac{1}{\lambda_k} (\Lambda_kW_k)\chi_{k}, \sum_{j=1}^N\frac{b_j}{\lambda_j^2}(y_j'\cdot\nabla)((\Lambda_j W_j)\chi_{t})\rilan\\
        &-\lelan \frac{1}{\lambda_k} (\Lambda_kW_k)\chi_{k}, \sum_{j=1}^N\sum_{l=1}^4\frac{v'_{j,l}(\partial_lW)(\frac{\cdot-\boldsymbol{y}_j}{\lambda_j})}{\lambda_j^2} \rilan -\lelan \frac{1}{\lambda_k} (\Lambda_kW_k)\chi_{k}, \sum_{j=1}^N\sum_{l=1}^4\frac{\lambda'_jv_{j,l}(\underline{\Lambda} \partial_l W)\left(\frac{\cdot-\boldsymbol{y}_j}{\lambda_j}\right)}{\lambda_j^3} \rilan\\
        &+\lelan \frac{1}{\lambda_k} (\Lambda_kW_k)\chi_{k}, \sum_{j=1}^N\sum_{l,m=1}^4\frac{v_{j,l}y'_{j,m}(\partial_l\partial_mW)(\frac{\cdot-\boldsymbol{y}_j}{\lambda_j})}{\lambda_j^3}\rilan.
\end{aligned}
\end{equation}
Most terms in (\ref{d dt ortho 3}) share the same estimates in Step 3. Here we stress the major differences coming from the terms
\begin{equation*}
    \lelan \frac{1}{\lambda_k} (\Lambda_kW_k)\chi_{k}, \ML_k g\rilan\text{ and } \lelan \frac{1}{\lambda_k} (\Lambda_kW_k)\chi_{k}, \sum_{j=1}^N\frac{b_j'}{\lambda_j}(\Lambda_jW_j)\chi_{t} \rilan.
\end{equation*}
Due to the fact that $\ML_k [(\Lambda_k W_k)\chi_k]$ vanishes on $\lambda_k M\leq |x-\boldsymbol{y}_k|\leq 2\lambda_k M $, it holds
\begin{equation*}
    \lelan \frac{1}{\lambda_k} (\Lambda_kW_k)\chi_{k}, \ML_k g\rilan\lesssim \left\|g(t)\right\|_{\ME_2}\left(\int_{\lambda_kM}^{2\lambda_kM}\frac{1}{r^4}r^3dr\right)^{\frac{1}{2}}\lesssim t^{\frac{1}{2}}\lambda.
\end{equation*}
For the other term, using the definition of $\chi_k$ we have
\begin{equation*}
   \lelan \frac{1}{\lambda_k} (\Lambda_kW_k)\chi_{k}, \sum_{j=1}^N\lambda_j^{-1}b_j'((\Lambda_jW_j)\chi_{t}) \rilan=b_k'\lelan \Lambda W, \Lambda W \chi_M\rilan+O\left(\sum_{j\neq k}|b_j'|\lambda_k^2\right)
\end{equation*}
As a consequence, we obtain
\begin{equation*}
    |b'|\lesssim t^{\frac{1}{2}}\lambda.
\end{equation*}
Indeed, since $\lelan \Lambda W, \Lambda W \chi_M\rilan\sim O(\log M)$, (\ref{rough est on b'}) can be refined as 
\begin{equation}\label{slight refine of b'}
    |b_k'|\lesssim \frac{\left\|\ML_k g\right\|_{|x-\boldsymbol{y}_k|\leq 2\lambda_k M }}{\log M}+\lambda.
\end{equation}
\textbf{Step 5. Control of the correction parameter $\boldsymbol{s}_k$  (\ref{v and s}), (\ref{s'+=to}).} We now treat the parameter $\boldsymbol{s}$. The definition and the main structure of the argument follow closely those for $\beta_k$,  and we only describe the parts that require additional care. Recall the definition of $\boldsymbol{s}_k$, we have
\begin{equation*}
    |\boldsymbol{v}_k-\boldsymbol{s}_k|\lesssim\left|\frac{\lelan \dot{g}, \frac{1}{\lambda_k}(\nabla_k W_k) \rilan}{\left\|\nabla W\right\|_{L^2}^2}\right| .
\end{equation*}
By Cauchy inequality, the difference is controlled by the following expression
\begin{equation*}
 \left\|\dot{g}\right\|_{\dot{H}^1}\left\|\frac{1}{\lambda_k}(\nabla_k W_k)\right\|_{\dot{H}^{-1}(|x-\boldsymbol{y}_k|\leq 1)}+\left\|\dot{g}\right\|_{L^2}\left\|\frac{1}{\lambda_k}(\nabla_k W_k)\right\|_{L^2(|x-\boldsymbol{y}_k|\geq 1)}.   
\end{equation*}
By Sobolev and integration, we obtain 
\begin{equation*}
   \left\|\frac{1}{\lambda_k}(\nabla_k W_k)\right\|_{\dot{H}^{-1}(|x-\boldsymbol{y}_k|\leq 1)} \lesssim \lambda_k\|\nabla W\|_{L^\frac43(|\cdot|\leq 1/\lambda_k)}\lesssim \lambda |\log (1/\lambda)|^{\frac{3}{4}}\lesssim t^{\frac{1}{2}}\lambda,
\end{equation*}
and
\begin{equation*}
    \left\|\frac{1}{\lambda_k}(\nabla_k W_k)\right\|_{L^2(|x-\boldsymbol{y}_k|\geq 1)}\lesssim \|\nabla W\|_{L^2(|\cdot|\geq 1/\lambda_k)}\lesssim \lambda.
\end{equation*}
Combining these with the energy bootstrap assumption of $g$ gives
\begin{equation*}
   |\boldsymbol{v}_k-\boldsymbol{s}_k|\lesssim\left\|\dot{g}\right\|_{\dot{H}^1}t^{\frac{1}{2}}\lambda+\left\|\dot{g}\right\|_{L^2}\lambda\lesssim t\lambda^2,
\end{equation*}
which proves (\ref{v and s}). Then we turn to the proof of (\ref{s'+=to}), taking inner product with $\frac{1}{\lambda_k}\nabla_k W_k$ yields
\begin{equation}\label{partial dot g innner with nabla w}
     \begin{aligned}
        &\lelan \partial_t \dot g, \frac{1}{\lambda_k}\nabla_k W_k\rilan-\lelan \Delta g+f(W_{\boldsymbol{\Gamma}}+g)-\sum_{j=1}^Nf(W_j) ,\frac{1}{\lambda_k}\nabla_k W_k\rilan\\
        =
        &\lelan  \sum_{j=1}^N\left(\frac{\lambda_j'b_j}{\lambda_j^2} (\underline{\Lambda}_j\Lambda_jW_j)\chi_{t}
        -\frac{b_j'}{\lambda_j}(\Lambda_jW_j\chi_{t})+\frac{b_j|x-\boldsymbol{y}_j|}{\lambda_jt^2}(\Lambda_jW_j\chi'_{t})\right),\frac{1}{\lambda_k}\nabla_k W_k\rilan\\
        &+\lelan \sum_{j=1}^N\frac{b_j}{\lambda_j^2}(\boldsymbol{y}_j'\cdot\nabla)(\Lambda_j W_j\chi_{t})
          -\sum_{j=1}^N\sum_{l=1}^4\frac{v'_{j,l}(\partial_lW)(\frac{\cdot-\boldsymbol{y}_j}{\lambda_j})}{\lambda_j^2} ,\frac{1}{\lambda_k}\nabla_k W_k\rilan\\
        &-\lelan \sum_{j=1}^N\sum_{l=1}^4\frac{\lambda'_jv_{j,l}(\underline{\Lambda} \partial_l W)\left(\frac{\cdot-\boldsymbol{y}_j}{\lambda_j}\right)}{\lambda_j^3}-\sum_{j=1}^N\sum_{l,m=1}^4\frac{v_{j,l}y'_{j,m}(\partial_l\partial_mW)(\frac{\cdot-\boldsymbol{y}_j}{\lambda_j})}{\lambda_j^3} ,\frac{1}{\lambda_k}\nabla_k W_k\rilan. 
    \end{aligned} 
\end{equation}
We examine term by term in (\ref{partial dot g innner with nabla w}). Starting from the $\partial_t \dot{g}$ term, it holds
\begin{equation*}
    \lelan \partial_t \dot g, \frac{1}{\lambda_k}\nabla_k W_k\rilan=\frac{d}{dt}\lelan \dot g, \frac{1}{\lambda_k}\nabla_k W_k\rilan-\lelan \dot{g}, \frac{d}{dt}\left(\frac{1}{\lambda_k}\nabla_k W_k\right)\rilan.
\end{equation*}
The first term, using the notation $\boldsymbol{s}_k$ and the fact $\lelan \frac{1}{\lambda_k}\nabla_kW_k, \frac{1}{\lambda_k}\nabla_kW_k\rilan=\left\|\nabla W\right\|_{L^2}^2$ is a constant, can be expanded as
\begin{equation*}
\begin{aligned}
&\frac{d}{dt}\lelan \dot g, \frac{1}{\lambda_k}\nabla_k W_k\rilan=\frac{d}{dt}\lelan \frac{1}{\lambda_k}\nabla_kW_k, \sum_{j=1}^N\frac{\boldsymbol{s}_j-\boldsymbol{v}_j}{\lambda_j}\cdot\nabla_j W_j \rilan\\
=&  \lelan \frac{1}{\lambda_k}\nabla_kW_k, \sum_{j=1}^N\frac{\boldsymbol{s}_j'-\boldsymbol{v}_j'}{\lambda_j}\cdot\nabla_j W_j \rilan+\sum_{j\neq k}(\boldsymbol{s}_j-\boldsymbol{v}_j)\frac{d}{dt}\lelan \frac{1}{\lambda_k}\nabla_kW_k,  \frac{1}{\lambda_j}\nabla_jW_j\rilan.  
\end{aligned} 
\end{equation*}
Still we keep the $(\boldsymbol{s}_j'-\boldsymbol{v}_j')$ term. For the second term, using the separation of the centers and the off-diagonal decay estimates, we have
\[
\left|
\frac{d}{dt}
\left\langle
\frac1{\lambda_k}\nabla_k W_k,
\frac1{\lambda_j}\nabla_j W_j
\right\rangle
\right|
\lesssim
\lambda_k\lambda_j
\left(
\frac{|\lambda_k'|}{\lambda_k}
+\frac{|\lambda_j'|}{\lambda_j}
+|y_k'|+|y_j'|
\right)
\lesssim
t^{-\frac{1}{3}}\lambda^2 .
\]
Therefore, by the bootstrap estimates,
\[
\left|
\sum_{j\neq k}(\boldsymbol{s}_j-\boldsymbol{v}_j)
\frac{d}{dt}
\left\langle
\frac1{\lambda_k}\nabla_k W_k,
\frac1{\lambda_j}\nabla_j W_j
\right\rangle
\right|
\lesssim
\sum_{j\neq k}|\boldsymbol{s}_j-\boldsymbol{v}_j|\lambda_k\lambda_j t^{-\frac{1}{3}}
\lesssim \lambda^3.
\] 
Then, for $\lelan \dot{g}, \frac{d}{dt}\left(\frac{1}{\lambda_k}\nabla_k W_k\right)\rilan$ notice that
\begin{equation*}
  -\frac{d}{dt}\left(\frac{1}{\lambda_k}\nabla_k W_k\right)=  \frac{\lambda_k'}{\lambda_k^2}\underline{\Lambda}_k \nabla_k W_k+\frac{\boldsymbol{y}_k'}{\lambda_k^2}\cdot\nabla_k(\nabla_k W_k).
\end{equation*}
Applying the analysis we used to handle $|\boldsymbol{v}_k-\boldsymbol{s}_k|$, we have
\begin{equation*}
\begin{aligned}
    \left| \lelan \dot{g}, \frac{\lambda_k'}{\lambda_k^2}\underline{\Lambda}_k \nabla_k W_k\rilan\right| \lesssim&  t^{-\frac{1}{3}}\left(\left\|\dot{g}\right\|_{\dot{H}^1}\left(\int_{\lambda_k}^1\left(\frac{\lambda_k}{r^3}\right)^{\frac{4}{3}}r^3dr\right)^{\frac{3}{4}}+\left\|\dot{g}\right\|_{L^2}\left(\int_{1}^{\infty}\left(\frac{\lambda_k}{r^3}\right)^2r^3dr\right)^{\frac{1}{2}}\right)\\
    \lesssim& t^{-\frac{1}{3}}\left((t^{\frac{1}{2}}\lambda)(\lambda (\log (1/\lambda))^{\frac{3}{4}})+(t^{-\frac{1}{6}}\lambda)\lambda \right)\lesssim t^{\frac{2}{3}}\lambda^2,
\end{aligned}
\end{equation*}
and 
\begin{equation*}
     \left| \lelan \dot{g}, \frac{\boldsymbol{y}_k'}{\lambda_k^2}\cdot\nabla_k(\nabla_k W_k)\rilan\right|\lesssim {|\boldsymbol{y}_k'|}\left\|\dot{g}\right\|_{\dot{H}^1}\left\|\frac{1}{\lambda_k^3}\sum_{1\leq i,j\leq 4}(\partial_i \partial_j W)\left(\frac{x-{y}_k}{\lambda_k}\right)\right\|_{L^{\frac{4}{3}}} \ll \lambda^{\frac{5}{2}}.
\end{equation*}
Therefore, the first term on the left-hand-side of (\ref{partial dot g innner with nabla w}) can be estimated by
\begin{equation}\label{sharp est for s' left 1}
    \left|\lelan \partial_t \dot g, \frac{1}{\lambda_k}(\nabla_k W_k)\rilan-\lelan \frac{1}{\lambda_k}\nabla_kW_k, \sum_{j=1}^N\frac{\boldsymbol{s}_j'-\boldsymbol{v}_j'}{\lambda_j}\cdot\nabla_j W_j \rilan\right|\lesssim  t^{\frac{2}{3}}\lambda^2.
\end{equation}
Next, we turn to the second term. Using the the fact that $\nabla_k W_k\in \ker \ML_k$ we know 
\begin{equation*}
   \lelan \Delta g+f(W_{\boldsymbol{\Gamma}}+g)-\sum_{j=1}^Nf(W_j) ,\frac{1}{\lambda_k}\nabla_k W_k\rilan= \lelan f(W_{\boldsymbol{\Gamma}}+g)-3W_k^2g-\sum_{j=1}^Nf(W_j) ,\frac{1}{\lambda_k}\nabla_k W_k\rilan.
\end{equation*}
Furthermore, the definition of $f$ gives
\begin{equation*}
\begin{aligned}
    &\lelan f(W_{\boldsymbol{\Gamma}}+g)-3W_k^2g-\sum_{j=1}^Nf(W_j) ,\frac{1}{\lambda_k}\nabla_k W_k\rilan\\
    =&\lelan \sum_{j\neq k}3W_jW_{\Gamma}g+\sum_{j=1}^N 3W_j g^2+g^3,\frac{1}{\lambda_k}\nabla_k W_k\rilan\\
    &\quad+\lelan \sum_{j\neq k}3W_k^2W_j+\sum_{i,j\neq k}W_kW_iW_j+\sum_{i,j,l\neq k}W_iW_jW_l ,\frac{1}{\lambda_k}\nabla_k W_k\rilan\\
    :=&\Rmnum{1}+\Rmnum{2}.
\end{aligned}
\end{equation*}
We estimate $\Rmnum{1}$ first. Using that $W_j\sim\frac{\lambda_j}{|\boldsymbol{y}_k-\boldsymbol{y}_j|^2}$ in the region $|x-\boldsymbol{y}_k|\leq d$ and $\frac{1}{\lambda_k}\nabla_k W_k\sim \frac{\lambda_k}{|\boldsymbol{y}_k-\boldsymbol{y}_j|^3}$ when  $|x-\boldsymbol{y}_j|\leq d$ we have $\left|\lelan \sum_{j\neq k}3W_jW_{\boldsymbol{\Gamma}}g,\frac{1}{\lambda_k}\nabla_k W_k \rilan\right|$ is controlled by 
\begin{equation*}
\begin{aligned}
&\left|\lelan \sum_{j\neq k}3W_jW_{\boldsymbol{\Gamma}}g,\frac{1}{\lambda_k}\nabla_k W_k \rilan\right|\lesssim \sum_{j\neq k}\left\|W_jW_{\boldsymbol{\Gamma}}\frac{1}{\lambda_k}\nabla_kW_k\right\|_{L^{\frac{4}{3}}}\left\|g\right\|_{\dot{H}^1}\\
   \lesssim&\left(\left(\int_{\lambda_k}^d\left(\frac{\lambda_j\lambda_k^2}{r^4|\boldsymbol{y}_k-\boldsymbol{y}_j|^3}\right)^{\frac{4}{3}}r^3dr\right)^{\frac{3}{4}}+\left(\int_{\lambda_j}^d\left(\frac{\lambda_j^2\lambda_k}{r^4|\boldsymbol{y}_k-\boldsymbol{y}_j|^3}\right)^{\frac{4}{3}}r^3dr\right)^{\frac{3}{4}}\right)\left\|g\right\|_{\dot{H}^1}\\  
  \lesssim& (\lambda_j\lambda_k +\lambda_j\lambda_k)\left\|g\right\|_{\dot{H}^1}\lesssim t^{-\frac{1}{6}}\lambda^3,
\end{aligned}
\end{equation*}
where the last two inequalities come from bootstrap assumption (\ref{bootstap a for energy}) and (\ref{bootstrap a for lambda}). On the exterior region, this term can also be proved small in the same way. Then we turn to the quadratic $g$ term. In the summation, when $j=k$, in view of the decay of $W$, it holds 
\begin{equation*}
    \begin{aligned}
       \left| \int_{\mathbb R^4} \frac{1}{\lambda_k}\nabla_kW_k  W_k g^2dx\right|\leq  \int_{\mathbb R^4} \frac{1}{\lambda_k^3}\frac{1}{\lelan \frac{x-\boldsymbol{y}_k}{\lambda_k}\rilan^5} g^2dx 
    \end{aligned}
\end{equation*}
Since when $|x-\boldsymbol{y}_k|\leq \lambda_k$ let $\epsilon>0$ be a small constant, a direct computation shows 
\begin{equation*}
     \frac{1}{\lambda_k^3}\frac{1}{\lelan \frac{x-\boldsymbol{y}_k}{\lambda_k}\rilan^5}\lesssim \frac{1}{\lambda_k^3}\lesssim \frac{\lambda_k^{1-2\epsilon}}{|x-\boldsymbol{y}_k|^{4-2\epsilon}}.
\end{equation*}
When $|x-\boldsymbol{y}_k|\geq \lambda_k$,  it holds
\begin{equation*}
    \left|\frac{1}{\lambda_k^3}\frac{1}{\lelan \frac{|x-\boldsymbol{y}_k|}{\lambda_k}\rilan^5}\right|\lesssim \frac{\lambda_k^2}{|x-\boldsymbol{y}_k|^5}\lesssim \frac{\lambda_k^{1-2\epsilon}}{|x-\boldsymbol{y}_k|^{4-2\epsilon}}.
\end{equation*}
Therefore, using (\ref{hardy 2-episllon}) we arrive at the estimate 
\begin{equation*}
\begin{aligned}
     \left| \int_{\mathbb R^4} \frac{1}{\lambda_k}\nabla_kW_k  W_k g^2dx\right| \lesssim& \int_{\mathbb R^4}\frac{\lambda_k^{1-2\epsilon}g^2}{|x-\boldsymbol{y}_k|^{4-2\epsilon}}dx\lesssim\lambda_k^{1-2\epsilon}\left\|\frac{g}{|x-\boldsymbol{y}_k|^{2-\epsilon}}\right\|_{L^2}^2
     \\\lesssim&\lambda_k^{1-2\epsilon}\left(\left\|g\right\|_{\dot{H}^2}^2+\left\|g\right\|_{\dot{H}^1}^2\right)\ll \lambda^{\frac{5}{2}}.
\end{aligned}
\end{equation*}
  The estimates when $j\neq k$ can be completed with the same approach. Finally, the cubic term is roughly bounded by Cauchy as 
 \begin{equation*}
     \lelan g^3,\frac{1}{\lambda_k}\nabla_k W_k \rilan\lesssim \left\|g\right\|_{\dot{H}^1}^3\left\|\frac{1}{\lambda_k}\nabla_k W_k\right\|_{\dot{H}^1}\lesssim \frac{1}{\lambda_k}\left\|g\right\|_{\dot{H}^1}^3\lesssim t^{-\frac{1}{2}}\lambda^2
 \end{equation*}
 As a result, we conclude that 
\begin{equation}\label{lemma 4 I}
   \left|\Rmnum{1}\right|\lesssim t^{-\frac{1}{2}}\lambda^2. 
\end{equation} 
Next we turn to $\Rmnum{2}$. As in Step 3, the major part comes from $\lelan \sum_{j\neq k}3W_k^2W_j,\frac{1}{\lambda_k}\nabla_k W_k\rilan$. In particular, the argument from (\ref{system for y v}) indicates 
\begin{equation*}
    \left|\lelan \sum_{j\neq k}3W_k^2W_j,\frac{1}{\lambda_k}\nabla_k W_k\rilan-\boldsymbol{D}_k(\boldsymbol{\lambda})\right|\lesssim O(\lambda^4).
\end{equation*}
For the remaining terms, we take $\lelan W_kW_iW_j,\frac{1}{\lambda_k}\nabla_k W_k \rilan$ for instance, where $i,j\neq k$. Using Cauchy's inequality we get
\begin{equation*}
  \left|\lelan W_kW_iW_j,\frac{1}{\lambda_k}\nabla_k W_k \rilan\right|  \leq \left\|W_kW_i\right\|_{L^2}\left\|W_j\frac{1}{\lambda_k}\nabla_k W_k\right\|_{L^2}\lesssim \lambda_k\lambda_i\lambda_j \sqrt{\log\left(\frac{1}{\lambda_k}\right)}\ll \lambda^{\frac{5}{2}}.
\end{equation*}
The other terms can be estimated in the same way and we obtain
\begin{equation}\label{lemma 4 II}
    \left|\Rmnum{2}-\boldsymbol{D}_k(\boldsymbol{\lambda})\right|\lesssim \lambda^{\frac{5}{2}}.
\end{equation}
Collecting (\ref{lemma 4 I}) and (\ref{lemma 4 II}), we conclude that
\begin{equation}\label{sharp est for s' left 2}
    \left|\lelan \Delta g+f(W_{\boldsymbol{\Gamma}}+g)-\sum_{j=1}^Nf(W_j) ,\frac{1}{\lambda_k}\nabla_k W_k\rilan-\boldsymbol{D}_k(\boldsymbol{\lambda})\right|\lesssim t^{-\frac{1}{2}}\lambda^2.
\end{equation}
We have now finished the estimate for terms on the left hand side of (\ref{partial dot g innner with nabla w}). 

Now we consider terms on the right hand side. In order to handle the summation of $j$, as in Step 3, again we split the discussion into the $j=k$ case and $j\neq k$ case. When $j=k$, the symmetry property and the fact $\lelan \underline{\Lambda} (\nabla W), \nabla W\rilan=0$ shows it equals to 
\begin{equation*}
    \begin{aligned}
      &\lelan \frac{b_k}{\lambda_k^2}(\boldsymbol{y}_k'\cdot\nabla)(\Lambda_k W_k\chi_{t})
          -\sum_{l=1}^4\frac{v'_{k,l}(t)\partial_lW(\frac{\cdot-\boldsymbol{y}_k}{\lambda_k})}{\lambda_k^2} ,\frac{1}{\lambda_k}\nabla_k W_k\rilan.
    \end{aligned}
\end{equation*}
By similar in (\ref{system for y v}), combining this term with (\ref{sharp est for s' left 1}) and (\ref{sharp est for s' left 2}) we have the major part of $\boldsymbol{s}_k$ system which is 
\begin{equation*}
 -{s}_{k,j}'+D_{k,j}(\lambda,\Vec{y_k}).   
\end{equation*}
It suffices for us to show the remaining part contributes  small perturbation of the ODE system. All these terms will be controlled by the right-hand side of (\ref{s'+=to}).  To view this, one should check the smallness when $j\neq k$. We stress that the largest term comes from 
\begin{equation*}
\begin{aligned}
    \left|\lelan \frac{b_j'}{\lambda_j}(\Lambda_jW_j\chi_{t}),\frac{1}{\lambda_k}\nabla_k W_k \rilan \right|\lesssim& |b_j'|\left(\int_{\lambda_j}^1\frac{1}{r^2}\cdot\frac{\lambda_k}{|z_k-z_j|^3}r^3dr+\int_{\lambda_k}^1\frac{1}{|z_k-z_j|^2}\cdot \frac{\lambda_k}{r^3}r^3dr\right) \\
    \lesssim& |b_j'||\lambda_k|\lesssim  t^{\frac{1}{2}}\lambda^2,
\end{aligned}
\end{equation*}
where (\ref{rough est on b'}) is used in the last line. As a result, we establish (\ref{s'+=to}), where the leading order comes from (\ref{sharp est for s' left 1}).

\textbf{Step 6. Proof of rough estimate of $\boldsymbol{v}_k'$ (\ref{rough est on v'}).} This step is analogous to the estimate of $b'$ in step 4. We stress that although $v$ is expected to decay much faster than $b$, the resulting bound for $v'$ shares the same rough bound with $b'$.  To see this, taking time derivative of $\lelan \frac{1}{\lambda_k}(\nabla_k W_k)\chi_k,\dot{g} \rilan=0$, we have
\begin{equation}\label{d dt nabla W dot g}
    \begin{aligned}
        0=&\frac{d}{dt}\lelan  \frac{1}{\lambda_k}(\nabla_k W_k)\chi_k,\dot{g} \rilan\\
          =&-\lelan \frac{\lambda_k'}{\lambda_k^3}\left(\underline{\Lambda}[(\nabla W) \chi]\right)\left(\frac{x-\boldsymbol{y}_k}{\lambda_k}\right),\dot{g} \rilan-\lelan  \sum_{i=1}^4\frac{1}{\lambda_k^3}[\partial_i((\nabla W)\chi)]\left(\frac{x-\boldsymbol{y}_k}{\lambda_k}\right)(y_k')_i,\dot{g}  \rilan\\
        &+\lelan  \frac{1}{\lambda_k}(\nabla_k W_k)\chi_k, \Delta g+f(W_{\boldsymbol{\Gamma}}+g)-\sum_{j=1}^Nf(W_j)\rilan\\
        &+\lelan  \frac{1}{\lambda_k}(\nabla_k W_k)\chi_k, \sum_{j=1}^N\frac{\lambda_j'b_j}{\lambda_j^2} (\underline{\Lambda}_j\Lambda_jW_j)\chi_{t}\rilan-\lelan  \frac{1}{\lambda_k}(\nabla_k W_k)\chi_k,  \sum_{j=1}^N\frac{b_j'}{\lambda_j}(\Lambda_jW_j)\chi_{t} \rilan\\
        &+\lelan  \frac{1}{\lambda_k}(\nabla_k W_k)\chi_k, \sum_
        {j=1}^N\frac{b_j|x-\boldsymbol{y}_j|}{\lambda_jt^2}((\Lambda_jW_j)\chi_{t}') \rilan+\lelan  \frac{1}{\lambda_k}(\nabla_k W_k)\chi_k, \sum_{j=1}^N\frac{b_j}{\lambda_j^2}(y_j'\cdot\nabla)((\Lambda_j W_j)\chi_{t})\rilan\\
        &-\lelan \frac{1}{\lambda_k}(\nabla_k W_k)\chi_k, \sum_{j=1}^N\sum_{l=1}^4\frac{v'_{j,l}(\partial_lW)(\frac{\cdot-\boldsymbol{y}_j}{\lambda_j})}{\lambda_j^2} \rilan -\lelan  \frac{1}{\lambda_k}(\nabla_k W_k)\chi_k, \sum_{j=1}^N\sum_{l=1}^4\frac{\lambda'_jv_{j,l}(\underline{\Lambda} \partial_l W)\left(\frac{\cdot-\boldsymbol{y}_j}{\lambda_j}\right)}{\lambda_j^3} \rilan\\
        &+\lelan  \frac{1}{\lambda_k}(\nabla_k W_k)\chi_k, \sum_{j=1}^N\sum_{l,m=1}^4\frac{v_{j,l}y'_{j,m}(\partial_l\partial_mW)(\frac{\cdot-\boldsymbol{y}_j}{\lambda_j})}{\lambda_j^3}\rilan.
    \end{aligned}
\end{equation}
Analogous to Step 4, the major contribution comes from
\begin{equation*}
   \lelan  \frac{1}{\lambda_k}(\nabla_k W_k)\chi_k, \ML_k g\rilan\text{ , and } -\lelan \frac{1}{\lambda_k}(\nabla_k W_k)\chi_k, \sum_{j=1}^N\sum_{l=1}^4\frac{v'_{j,l}(\partial_lW)(\frac{\cdot-\boldsymbol{y}_j}{\lambda_j})}{\lambda_j^2} \rilan.
\end{equation*}
Since $\ML_k[(\nabla_k W_k)\chi_k]$ vanishes except on the region $\lambda_kM\leq |x-\boldsymbol{y}_k|\leq 2\lambda_k M$, using Cauchy inequality, the $\ML_k g$ term is dominated by
\begin{equation*}
\begin{aligned}
    \left|\lelan  \frac{1}{\lambda_k}(\nabla_k W_k)\chi_k, \ML_k g\rilan\right|\lesssim \left\|g(t)\right\|_{\ME_2}\left(\int_{\lambda_kM}^{2\lambda_k M}\left(\frac{\lambda_k}{r^3}\right)^2r^3dr\right)^{\frac{1}{2}}\leq \frac{Ct^{\frac{1}{2}}\lambda}{M}.
    \end{aligned}
\end{equation*}
While for the other term, using the definition of $\chi_k$ and the decay of $\nabla W$ we obtain
\begin{equation*}
 -\lelan \frac{1}{\lambda_k}(\nabla_k W_k)\chi_k, \sum_{j=1}^N\sum_{l=1}^4\frac{v'_{j,l}(\partial_lW)(\frac{\cdot-\boldsymbol{y}_j}{\lambda_j})}{\lambda_j^2} \rilan=-\boldsymbol{v}_k'\lelan (\nabla W)\chi_M,\nabla W\rilan+O(\boldsymbol{v}_k'\lambda_j\lambda_k^2).   
\end{equation*}
Therefore, the step is completed if the remaining part of (\ref{d dt nabla W dot g}) is small enough to be absorbed by the right-hand  side of (\ref{rough est on v'}). As in step 4, most of these terms have already been treated in step 5. Here we emphasize the different terms
\begin{equation*}
    -\lelan  \frac{1}{\lambda_k}(\nabla_k W_k)\chi_k, \sum_{j=1}^N\sum_{l=1}^4\frac{\lambda'_jv_{j,l}(\underline{\Lambda} \partial_l W)\left(\frac{\cdot-\boldsymbol{y}_j}{\lambda_j}\right)}{\lambda_j^3} \rilan.
\end{equation*}
and
\begin{equation*}
 \lelan  \frac{1}{\lambda_k}(\nabla_k W_k)\chi_k, \sum_{j=1}^N\frac{b_j}{\lambda_j^2}(y_j'\cdot\nabla)((\Lambda_j W_j)\chi_{t})\rilan   
\end{equation*}
In step 5, these terms vanish due to $\lelan\nabla W, \underline{\Lambda}(\nabla W)\rilan=0$. The cut-off function here breaks the cancellation, thus we have to check them more carefully. Applying Cauchy and bootstrap assumption (\ref{bootstrap a for b}), (\ref{bootstrap a for v}) and (\ref{lambda'+b}) to the first term yields
\begin{equation*}
    \left|\lelan  \frac{1}{\lambda_k}(\nabla_k W_k)\chi_k, \sum_{j=1}^N\sum_{l=1}^4\frac{\lambda'_jv_{j,l}(\underline{\Lambda} \partial_l W)\left(\frac{\cdot-\boldsymbol{y}_j}{\lambda_j}\right)}{\lambda_j^3} \rilan\right|\lesssim \frac{|\boldsymbol{b}||\boldsymbol{v}|}{|\boldsymbol{\lambda}|}+\frac{|\boldsymbol{\lambda'}+\boldsymbol{b}||\boldsymbol{v}|}{|\boldsymbol{\lambda}|}\ll t^{2}\lambda^2,
\end{equation*}
which is negligible. The second term can be proved to be small in the same way. Consequently, inequality (\ref{rough est on v'}) is established. In particular, as (\ref{slight refine of b'}), the estimate of $v'$ can be refined as
\begin{equation}\label{slight refine of v'}
    |v_k'|\lesssim \frac{\left\|\ML_k g\right\|_{|x-\boldsymbol{y}_k|\leq 2\lambda_k M }}{M}+\lambda.
\end{equation}

\textbf{Step 7. Estimate of $a_k^{\pm}$ (\ref{a and tilde a})--(\ref{a' also satisfies}).} Lastly, we consider the stable/unstable direction. Starting from (\ref{a and tilde a}), the definition of $\Tilde{a}^{\pm}$ leads to
\begin{equation*}
    |a_k^{\pm}-\Tilde{a}_k^{\pm}|=\left|\lelan \Vec{Z}_k^{\pm}, (0, \sum_{j=1}^N\frac{b_j-\beta_j}{\lambda_j}(\Lambda_j W_j)\chi_t+\sum_{j=1}^N\frac{(\boldsymbol{v}_j-\boldsymbol{s}_j)}{\lambda_j}\cdot\nabla_j W_j) \rilan\right|.
\end{equation*}
Recall that the eigenfunction $Y$ is exponential decay and orthogonal to $\Lambda W$ and $\nabla W$. Hence it holds 
\begin{equation*}
    \begin{aligned}
     &\left|\lelan \Vec{Z}_k^{\pm}, (0, \sum_{j=1}^N\frac{b_j-\beta_j}{\lambda_j}(\Lambda_j W_j)\chi_t+\sum_{j=1}^N\frac{(\boldsymbol{v}_j-\boldsymbol{s}_j)}{\lambda_j}\cdot\nabla_j W_j) \rilan\right|\\
     \lesssim& |b_k-\beta_k|\left|\lelan Y,\Lambda W \left(1-\chi\left(\frac{\lambda_k \cdot}{t}\right)\right) \rilan\right|+\left|\lelan \frac{1}{\lambda_k}Y_k,\sum_{j\neq k}\frac{b_j-\beta_j}{\lambda_j}(\Lambda_j W_j)\chi_t+\sum_{j\neq k}\frac{(\boldsymbol{v}_j-\boldsymbol{s}_j)}{\lambda_j}\cdot\nabla_j W_j) \rilan\right|\\
     \lesssim& |b_k-\beta_k|\int_{\frac{t}{\lambda_k}}^{\infty}\frac{1}{r^{10}}r^3dr+\sum_{j\neq k}\left(\frac{|b_j-\beta_j|}{|\boldsymbol{y}_k-\boldsymbol{y}_j|^2}+\frac{|\boldsymbol{v}_j-\boldsymbol{s}_j|\lambda_j}{|\boldsymbol{y}_k-\boldsymbol{y}_j|^3}\right)\lambda_k^2\left\|Y\right\|_{L^1}
     \lesssim (t^{\frac{1}{6}}\sqrt{\log t})\lambda^3,
     \end{aligned}
\end{equation*}
where we used (\ref{b and beta}) and $|Y(x)|\lesssim\frac{1}{|x|^8}$ when $|x|\geq 1$ in the last line. Thus (\ref{a and tilde a}) is proved. Similarly, we can compute
\begin{equation*}
    \left| \frac{d}{dt}a_k^{\pm}-\frac{d}{dt}\tilde{a}_k^{\pm}\right|=\left|\frac{d}{dt}\lelan \Vec{Z}_k^{\pm}, (0, \sum_{j=1}^N\frac{b_j-\beta_j}{\lambda_j}(\Lambda_j W_j)\chi_t+\sum_{j=1}^N\frac{(\boldsymbol{v}_j-\boldsymbol{s}_j)}{\lambda_j}\cdot\nabla_j W_j) \rilan\right|.
\end{equation*}
Since $\Lambda Y$ and $\underline{\Lambda}Y$ are still exponentially decay functions, the leading order term comes from 
\begin{equation*}
 \left|\lelan \frac{1}{\lambda_k}Y_k,\frac{d}{dt}\left(\sum_{j\neq k}\frac{b_j-\beta_j}{\lambda_j}(\Lambda_j W_j)\chi_t+\sum_{j\neq k}\frac{(\boldsymbol{v}_j-\boldsymbol{s}_j)}{\lambda_j}\cdot\nabla_j W_j) \right)\rilan\right|.  
\end{equation*}
From (\ref{rough est on b'}) and (\ref{rough est on v'}), this term is dominated by
\begin{equation*}
    \sum_{j\neq k}\left(\frac{|b_j'|+|\beta_j'|}{|\boldsymbol{y}_k-\boldsymbol{y}_j|^2}+\frac{|\boldsymbol{v}_j'|+|\boldsymbol{s}_j'|\lambda_j}{|\boldsymbol{y}_k-\boldsymbol{y}_j|^3}\right)\lambda_k\left\|Y\right\|_{L^1}
     \lesssim t^{\frac{1}{2}}\lambda^3.
\end{equation*}
Therefore, (\ref{a' and tilde a'}) is proved.

Next, in order to prove (\ref{a' also satisfies}),  we first prove
\begin{equation}\label{a pm '=to}
  \left|(\Tilde{a}_k^{\pm})'\mp \nu \lambda_k^{-1}\Tilde{a}_k^{\pm}\right|\lesssim \lambda.
\end{equation}
Moreover, on account of narration convenience, we denote
\begin{equation*}
    \Vec{g}_c:=(g,\dot{h}).
\end{equation*}
Then, using the definition of $\dot{h}$, the energy norm of $\Vec{g}_c$ can be estimated by
\begin{equation*}
\begin{aligned}
      \left\|\Vec{g}_c\right\|_{\dot{H}^1\times L^2}\lesssim& \left\|\Vec{g}\right\|_{\dot{H^1}\times L^2}+\sum_{j=1}^N|b_j-\beta_j|\left\|\frac{1}{\lambda_j}(\Lambda_j W_j)\chi_t\right\|_{L^2}+\sum_{j=1}^N|\boldsymbol{v}_j-\boldsymbol{s}_j|\left\|\frac{1}{\lambda_j}\nabla_j W_j\right\| \\
      \lesssim& t^{-\frac 16}\lambda+t^{-\frac{5}{6}}\sqrt{\log t} \sqrt{\log (t/\lambda)}\lambda+t\lambda^2\lesssim t^{-\frac 16}\lambda.
\end{aligned}
\end{equation*}
With this,  we compute
\begin{equation*}
    \frac{d}{dt}\Tilde{a}_k^{\pm}=\lelan \partial_t \Vec{Z}_k^{\pm},\Vec{g}_c\rilan+\lelan \Vec{Z}_k^{\pm},\partial_t\Vec{g}_c \rilan.
\end{equation*}
For the first term,  
\begin{equation*}
    \partial_t \Vec{Z}_k^{\pm}=\lambda_k'\partial_{\lambda_k}\Vec{Z}_k^{\pm}+y_k'\cdot\partial_{y_k}\Vec{Z}_k^{\pm}.
\end{equation*}
Then due to the exponential decay of $Y$ and the bootstrap assumptions, we obtain
\begin{equation*}\label{d dt Z inner with gc}
    \left|\lelan \partial_t \Vec{Z}_k^{\pm},\Vec{g}_c \rilan\right|\lesssim \left(\left|\frac{\lambda_k'}{\lambda_k}\right|+\left|\frac{y_k'}{\lambda_k}\right|\right)\left\|\Vec{g}_c\right\|_{\dot{H}^1\times L^2}\lesssim t^{-\frac{1}{2}}\lambda.
\end{equation*}
For the second term, analogous to (\ref{formal dt g}) we have
\begin{equation}\label{unstable qual}
    \begin{aligned}
   \lelan \Vec{Z}_k^{\pm},\partial_t\Vec{g}_c \rilan=&\lelan \Vec{Z}_k^{\pm}, J\circ DE(\WTG+\Vec{g}_c)\rilan+\frac{1}{2}\lelan \frac{1}{\lambda_k}Y_k,\sum_{j=1}^N\frac{\beta_j}{\lambda_j}(\Lambda_jW_j)\left(\frac{|x-\boldsymbol{y}_k|}{t^2}\chi'(\frac{x-\boldsymbol{y}_k}{t})\right) \rilan\\& -\lelan \Vec{Z}_k^{\pm}, \boldsymbol{\lambda}'\partial_{\boldsymbol{\lambda}}\WTG\rilan
   -\lelan \Vec{Z}_k^{\pm}, \boldsymbol{\beta}'\partial_{\boldsymbol{\beta}}\WTG\rilan-\lelan \Vec{Z}_k^{\pm},\boldsymbol{y}'\partial_{\boldsymbol{y}}\WTG \rilan-\lelan \Vec{Z}_k^{\pm}, \boldsymbol{s}'\partial_{\boldsymbol{s}}\WTG\rilan.
    \end{aligned}
\end{equation}
We begin the analysis with the terms on the second line of (\ref{unstable qual}), from the definition, 
\begin{equation*}
   -\boldsymbol{\lambda}'\partial_{\boldsymbol{\lambda}}\WTG =\sum_{j=1}^N\left(\frac{\lambda_j'}{\lambda_j}\Lambda_j W_j,\frac{\beta_j \lambda_j'}{\lambda_j^{2}}(\underline{\Lambda}_j\Lambda_j W_j)\chi_{t}+\frac{\lambda_j'\boldsymbol{s}_j\cdot \underline{\Lambda}_j\nabla_j W_j}{\lambda_j^2}\right).
\end{equation*}
Since $\lelan Y_k, \Lambda_k W_k \rilan=0$, $\left|\lelan \frac{\nu}{\lambda_k^2}Y_k,\Lambda_jW_j \rilan\right|\lesssim \frac{\lambda_j \lambda_k}{|z_j-z_k|^2}$ (when $j\neq k$),  
\begin{equation*}
 \left|\lelan \frac{1}{\lambda_k}Y_k,\sum_{j=1}^N \frac{\beta_j \lambda_j'}{\lambda_j^{2}}(\underline{\Lambda}_j\Lambda_j W_j)\chi_{t}\rilan \right|\lesssim  \left|\frac{\beta_k \lambda_k'}{\lambda_k}\right|+\sum_{j\neq k}\left|\frac{\lambda_k^3\lambda_j^2\beta_j\lambda_j'}{|z_k-z_j|^4}\right|,  
\end{equation*}
and
\begin{equation*}
  \left|\lelan \frac{1}{\lambda_k}Y_k,\sum_{j=1}^N \frac{\lambda_j'\boldsymbol{s}_j\cdot \underline{\Lambda}_j\nabla_j W_j}{\lambda_j^2}\rilan \right|\lesssim  \left|\frac{\lambda_k'\boldsymbol{s}_k}{\lambda_k}\right|+\sum_{j\neq k}\left|\frac{\lambda_k^2\lambda_j'\boldsymbol{s}_j}{|z_k-z_j|}\right|,
\end{equation*}
we arrive at the estimate
\begin{equation}\label{Zk and lambda'partial lambda}
    \left|\lelan \Vec{Z}_k^{\pm}, \boldsymbol{\lambda}'\partial_{\boldsymbol{\lambda}}\WTG\rilan\right|\lesssim t^{-\frac{2}{3}}\lambda.
\end{equation}
Then for $\lelan \Vec{Z}_k^{\pm}, \boldsymbol{\beta}'\partial_{\boldsymbol{\beta}}\WTG\rilan$, the definition of $\WTG$ yields
\begin{equation*}
  -\boldsymbol{\beta}'\partial_{\boldsymbol{\beta}}\WTG=\left(0,\sum_{j=1}^N\frac{\beta_j'}{\lambda_j}(\Lambda_j W_j)\chi_{t}\right).  
\end{equation*}
From bootstrap assumption for $\lambda_k$ and $b_k$, together with the estimates (\ref{b and beta}), (\ref{beta'+=to}) that we obtain for $\beta$, the derivative of $\beta_k$ is controlled by
\begin{equation}\label{rough est of beta'}
    \left|\beta_k'\right|\lesssim \frac{|c\lambda|}{|\Xi(t)|}+\frac{b_k^2+|(b_k-\beta_k)(b_k+\beta_k)|}{\lambda_k|\Xi(t)|}\lesssim t^{-\frac{2}{3}}\lambda.
\end{equation}
Therefore, using Cauchy, decay of $Y$ and  (\ref{rough est of beta'}) we have
\begin{equation}\label{Zk and beta'}
    \left|\lelan \Vec{Z}_k^{\pm}, \boldsymbol{\beta}'\partial_{\boldsymbol{\beta}}\WTG\rilan\right|\lesssim \left|{\beta_k'}\right|\lesssim t^{-\frac{2}{3}}\lambda.
\end{equation}
For the $\boldsymbol{y}$ derivative term, it holds 
\begin{equation*}
\begin{aligned}
     -\boldsymbol{y}'\partial_{\boldsymbol{y}}\WTG=\sum_{j=1}^N\Bigg(&\frac{1}{\lambda_j}(\boldsymbol{y}_j'\cdot \nabla_j)W_j,\frac{\beta_j}{\lambda_j^2}((\boldsymbol{y}_j'\cdot \nabla_j)\Lambda_j W_j)\chi_{t})\\&+\frac{\beta_j}{\lambda_jt}\Lambda_j W_j\left(\boldsymbol{y}'\cdot (\nabla \chi)\left(\frac{\cdot-\boldsymbol{y}_k}{t}\right)\right)+\frac{\boldsymbol{y}_j'\cdot\nabla_j(\boldsymbol{s}_j\cdot\nabla_j W_j )}{\lambda_j^2}\Bigg) .
\end{aligned}
\end{equation*}
Again using $\lelan Y_k, \nabla_k W_k \rilan=0$, $\left|\lelan \frac{\nu}{\lambda_k^2}Y_k,\nabla_jW_j \rilan\right|\lesssim \frac{\lambda_j^2 \lambda_k}{|\boldsymbol{y}_j-\boldsymbol{y}_k|^3}$ (when $j\neq k$),  the exponential decay of $Y$, (\ref{bootstrap a for y}), (\ref{bootstrap a for v}) and (\ref{lambda'+b}), (\ref{y'+v}), we derive the bound
\begin{equation}\label{Zk and y'}
   \left |\lelan \Vec{Z}_k^{\pm},\boldsymbol{y}'\partial_{\boldsymbol{y}}\Vec{W}_{\boldsymbol{\Gamma}} \rilan\right|\lesssim \sum_{j\neq k}\frac{\lambda_j^2\lambda_k}{|\boldsymbol{y}_j-\boldsymbol{y}_k|^3}+ \left|\frac{b_k \boldsymbol{y}_k'}{\lambda_k}\right|+\left|\frac{b_k\boldsymbol{y}_k'}{t}\right|+\left|\frac{\boldsymbol{y}_k'\boldsymbol{s}_k}{\lambda_k}\right|\lesssim t^{\frac{2}{3}}\lambda^2.
\end{equation}
For the last term in (\ref{unstable qual}), the computation
\begin{equation*}
    -\boldsymbol{s}'\partial_{\boldsymbol{s}}\Vec{W}_{\boldsymbol{\Gamma}}=\sum_{j=1}^N(0,\lambda_j^{-1}(\boldsymbol{s}_j'\cdot\nabla_j)W_j).
\end{equation*}
The bootstrap assumption for $\boldsymbol{\lambda}$ and $\boldsymbol{v}$, together with estimate (\ref{s'+=to}) yields
\begin{equation*}
    |s_{k,j}'|\lesssim |D_{k,j}(\boldsymbol{\lambda},\boldsymbol{y})|+t^{\frac{2}{3}}\lambda^2\lesssim t^{\frac{2}{3}}\lambda^2,
\end{equation*}
and therefore,
\begin{equation}\label{Zk and s'}
  \left |\lelan \Vec{Z}_k^{\pm},\boldsymbol{s}'\partial_{\boldsymbol{s}}\WTG \rilan\right|  \lesssim |\boldsymbol{s}_k'|\lesssim t^{\frac{2}{3}}\lambda^2.
\end{equation}
In conclusion, putting (\ref{Zk and lambda'partial lambda}), (\ref{Zk and beta'}), (\ref{Zk and y'}), (\ref{Zk and s'}) together, the summation of the second line of (\ref{unstable qual}) is bounded by $t^{-\frac{2}{3}}\lambda$. Next, we invest the first line. From the estimate of ${\beta}$, we have
\begin{equation*}
    \left|\frac{1}{2}\lelan \frac{1}{\lambda_k}Y_k,\sum_{j=1}^N\frac{\beta_j}{\lambda_j}(\Lambda_jW_j)\left(\frac{|x-\boldsymbol{y}_k|}{t^2}\chi'\left(\frac{x-\boldsymbol{y}_k}{t}\right)\right) \rilan\right|\lesssim \left|\frac{\beta}{t}\right|\lesssim t^{-\frac{4}{3}}\lambda.
\end{equation*}
Now we focus on the term $\lelan \Vec{Z}_k^{\pm}, J\circ DE(\WTG+\Vec{g}_c)\rilan$. In fact, for fixed $k\in\left\{1,...,N\right\}$, we can rewrite $J\circ DE(\WTG+\Vec{g}_c)$ as
\begin{equation}\label{express of jde}
    \begin{aligned}
      J\circ DE(\WTG+\Vec{g}_c)= &\Bigg(\sum_{j=1}^N \frac{\beta_j}{\lambda_j}(\Lambda_j W_j)\chi_{t}+\sum_{j=1}^N\frac{\boldsymbol{s}_j\cdot \nabla_j W_j}{\lambda_j},0\Bigg)\\
        &+\Bigg(0,f(W_{\boldsymbol{\Gamma}}+g)-\sum_{j=1}^Nf(W_k)-f'(W_k)g\Bigg)+J\circ D^2E(W_k)\Vec{g}_c.
    \end{aligned}
\end{equation}
Plugging the definition of $\Vec{Z}_k$ into the expression and compute the inner product we derive
\begin{equation*}
\begin{aligned}
 &\left|\lelan \Vec{Z}_k^{\pm}, \left( \sum_{j=1}^N \frac{\beta_j}{\lambda_j}(\Lambda_j W_j)\chi_{t}+\sum_{j=1}^N\frac{\boldsymbol{s}_j\cdot \nabla_j W_j}{\lambda_j},0\right) \rilan \right|\\
 =&\left|\lelan\frac{\nu}{2}\lambda_k^{-2}Y_k,\sum_{j=1}^N \frac{\beta_j}{\lambda_j}(\Lambda_j W_j)\chi_{t}+\sum_{j=1}^N\frac{\boldsymbol{s}_j\cdot \nabla_j W_j}{\lambda_j}\rilan\right|.    
\end{aligned}
\end{equation*}
Using the fact $\lelan Y_k,\Lambda_kW_k \rilan=\lelan Y_k,\nabla_k W_k \rilan=0$ and applying the analysis deducing (\ref{Zk and lambda'partial lambda}) and (\ref{Zk and y'}) we assert 
\begin{equation*}
\begin{aligned}
  &\left|\lelan\frac{\nu}{2}\lambda_k^{-2}Y_k,\sum_{j=1}^N \frac{\beta_j}{\lambda_j}(\Lambda_j W_j)\chi_{t}+\sum_{j=1}^N\frac{\boldsymbol{s}_j\cdot \nabla_j W_j}{\lambda_j}\rilan\right|\\
  \lesssim& \sum_{j\neq k}\left(\frac{\beta_j\lambda_k\left\|Y\right\|_{L^1}}{|\boldsymbol{y}_j-\boldsymbol{y}_k|^2}+\frac{\lambda_j|\boldsymbol{s}_j|\lambda_k\left\|Y\right\|_{L^1}}{|\boldsymbol{y}_j-\boldsymbol{y}_k|^3}\right)+\left|\lelan \frac{\nu}{2\lambda_k^2}Y_k, \frac{b_k}{\lambda_k}(\Lambda_k W_k)(1-\chi_{t}) \rilan \right|\lesssim \lambda^2.
\end{aligned}
\end{equation*}
Then for the second term on the right hand side of (\ref{express of jde}), the definition of $\Vec{Z}_k$ again gives
\begin{equation*}
\begin{aligned}
     &\left|\lelan \Vec{Z}_k,\Bigg(0,f(W_{\boldsymbol{\Gamma}}+g)-\sum_{j=1}^Nf(W_k)-f'(W_k)g\Bigg) \rilan\right|\\
     =& \left|\lelan \lambda_k^{-1}Y_k, f(W_{\boldsymbol{\Gamma}}+g)-\sum_{j=1}^N f(W_j)-f'(W_k)g \rilan\right|.
\end{aligned}
\end{equation*}
By triangle inequality, this is bounded by 
\begin{equation*}
\begin{aligned}
   &\left|\lelan  \lambda_k^{-1}Y_k, f(W_{\boldsymbol{\Gamma}}+g)- f(W_{\boldsymbol{\Gamma}})-f'(W_{\boldsymbol{\Gamma}})g \rilan\right|\\+&\left|\lelan \lambda_k^{-1}Y_k, f(W_{\boldsymbol{\Gamma}})-\sum_{j=1}^N f(W_j) \rilan\right|+\left|\lelan \lambda_k^{-1}Y_k,(f'(W_{\boldsymbol{\Gamma}})-f'(W_k))g\rilan\right|
\end{aligned}  
\end{equation*}
The first term here contains $g^2$ and thus, using Cauchy and Sobolev,  is controlled by
\begin{equation*}
    \left\| \frac{1}{\lambda_k}Y_k\right\|_{L^2}\left\|W_{\boldsymbol{\Gamma}}\right\|_{\dot{H}^1}\left\|g\right\|_{\dot{H}^1}\left\|g\right\|_{\dot{H}^2}\lesssim t^{\frac{1}{3}}\lambda^2.
\end{equation*}
While for the second and third term, the exponential decay of $Y$ and Cauchy inequality yield
\begin{equation}\label{a' and a largest pertubation}
 \left|\lelan \lambda_k^{-1}Y_k, f(W_{\boldsymbol{\Gamma}})-\sum_{j=1}^N f(W_j) \rilan\right|\lesssim \sum_{j\neq k}\left\|\frac{1}{\lambda_k}Y_k W_j\right\|_{L^2} \left\|W_k\right\|_{L^4}^2\lesssim \lambda,
\end{equation}
and
\begin{equation*}
 \left|\lelan \lambda_k^{-1}Y_k,(f'(W_{\boldsymbol{\Gamma}})-f'(W_k))g\rilan\right|\lesssim  \sum_{j\neq k}\left\|\frac{1}{\lambda_k}Y_k W_j\right\|_{L^2} \left\|W_k\right\|_{L^4}\left\|g\right\|_{L^4}  \lesssim  t^{-\frac{1}{6}}\lambda^2.
\end{equation*}
We emphasize that the main order term on the right hand side of (\ref{a pm '=to}) comes from (\ref{a' and a largest pertubation}).  Lastly, since $\ML Y=-\nu^2 Y$, we have $\lelan \Vec{Z}_k^{\pm},J\circ D^2E(W_k)\Vec{g}_c \rilan=\pm \nu \lambda_k^{-1}\Tilde{a}_k^{\pm}$. As a result we have
\begin{equation*}
    \left|(\Tilde{a}_k^{\pm})'\mp \nu \lambda_k^{-1}\Tilde{a}_k^{\pm}\right|\lesssim \lambda
\end{equation*}
holds on the interval. Furthermore, combining (\ref{a pm '=to}) with (\ref{a and tilde a}), (\ref{a' and tilde a'}) gives
\begin{equation*}
\begin{aligned}
       \left|(a_k^{\pm})'\mp\nu \lambda_k^{-1}a_k^{\pm}\right|\lesssim&    \left|(\Tilde{a}_k^{\pm})'\mp \nu \lambda_k^{-1}\Tilde{a}_k^{\pm}\right|+|(a_k^{\pm})'-(\Tilde{a}_k^{\pm})'|+|\nu \lambda_k^{-1}{a}_k^{\pm}-\nu \lambda_k^{-1}\Tilde{a}_k^{\pm}|\\
       \lesssim&\lambda+t^{\frac{1}{2}}\lambda^2+(t^{-\frac{5}{6}}\log t)\lambda\lesssim \lambda,
\end{aligned}
\end{equation*}
which completes the proof of (\ref{a' also satisfies}).
\end{proof}
The following statement is the main part of the proof of Theorem \ref{multi-bubble solution}. 
\begin{prop}\label{main step} For any $T>T_0$, there exist $(\omega_0,\omega_1)\in [-1,1]^2 $ such that the solution $\Vec{u}$ of (\ref{NLW 4}) with data $\Vec{u}(T)$ given by Lemma \ref{choice of u(T)} satisfies $T_*=T_0$.
    
\end{prop}
\section{Construction of the refined approximate solution}\label{Construction of the refined approximate solution}
\subsection{Construction of the approximate solutions}
We construct refined approximate solutions by adding suitable correction profiles to the leading-order multi-bubble ansatz. These profiles are designed to capture the main contributions arising in the modulation equations and to improve the accuracy of the approximation.

We now introduce the profiles $Q$, $S$, and $Z_{l}$, for $l\in \left\{1,2,3,4\right\}$, as the solutions to the following equations
\begin{equation*}
    \ML Q=-\Lambda W \chi_t(\lambda \cdot)-\frac{\lelan \Lambda W \chi_t(\lambda \cdot),\Lambda W \rilan}{32\pi^2}f'(W),
\end{equation*} 
\begin{equation*}
    \ML S=-(\underline{\Lambda}\Lambda W)\chi_t(\lambda \cdot)+\frac{\lelan (\underline{\Lambda}\Lambda W)\chi_t(\lambda \cdot),\Lambda W \rilan}{\lelan \Lambda W \chi_t(\lambda \cdot),\Lambda W \rilan}\Lambda W \chi_t(\lambda \cdot),
\end{equation*}
\begin{equation*}
    \ML Z_{l}=-(\Lambda \partial_l W+\partial_l W)=-\partial_l \Lambda W,
\end{equation*}
where we stress that $\lelan f'(W),\Lambda W\rilan=-32\pi^2$.

These profiles will serve as the building blocks of the approximate solutions constructed below. The existence and smoothness of these profiles follow from standard ODE theory. We stress that $Q$ and $S$ are radially symmetric.
\begin{lemma}[Pointwise estimates for the correction profiles] For $m=0,1$, the profiles $Q$ and $S$ satisfy the following pointwise estimates:
\begin{equation*}
    \left|\frac{dQ}{dy^{m}}(y)\right|\lesssim\begin{cases}
    \log\left(\frac{2t}{\lambda}\right)|y|^{2-m}&\text{ , when } |y|\leq 1 \\
    \log(\frac{2t}{\lambda |y|})y^{-m}+\frac{|y|^2+(1+\log |y|) \log(2t/\lambda )}{1+|y|^{2+m}}&\text{ , when }1\leq |y|\leq 2t/\lambda,
    \end{cases}    
\end{equation*}
\begin{equation*}
    \left|\frac{d}{dy^{m}}\left(\frac{dQ}{dt}\right)(y)\right|\lesssim\begin{cases}
    \left(\frac{|\lambda'|}{\lambda}+\frac{1}{t}\right)|y|^{2-m}&\text{ , when } |y|\leq 1 \\
    \left(\frac{|\lambda'|}{\lambda}+\frac{1}{t}\right)\frac{|y|^2}{1+|y|^{2+m}}&\text{ , when }1\leq |y|\leq 2t/\lambda,
    \end{cases}    
\end{equation*}
\begin{equation*}
   \left|\frac{dS}{dy^{m}}(y)\right|\lesssim\begin{cases}
    |y|^{2-m}&\text{ , when } |y|\leq 1 \\
   \frac{1+\log y}{1+y^{2+m}}+\frac{\log(t/\lambda y)+1}{\log(t/
    \lambda)y^{m}}&\text{ , when }1\leq |y|\leq 2t/\lambda.
    \end{cases}      
\end{equation*}
\begin{equation*}
   \left|\frac{d}{dy^{m}}\left(\frac{dS}{dt}\right)(y)\right|\lesssim\begin{cases}
    \frac{y^{2-m}}{\log(2t/\lambda)^2}\left(\frac{|\lambda'|}{\lambda}+\frac{1}{t}\right)&\text{ , when } |y|\leq 1 \\
   \left(\frac{|\lambda'|}{\lambda}+\frac{1}{t}\right)\left(\frac{1}{y^{m}\log(t/\lambda)}+\frac{\log(t/\lambda y)}{y^{m}\log(t/\lambda)^2}+\frac{1}{1+y^{2+m}}\right)&\text{ , when }1\leq |y|\leq 2t/\lambda.
    \end{cases}      
\end{equation*}
\end{lemma}
The proof of this lemma relies on elementary ODE arguments and is postponed to Appendix \ref{pointwise est of refined profile}. Next we introduce rescaled profiles
\begin{equation*}
    Q_k:=Q_{\lambda_k}(\cdot-\boldsymbol{y}_k)\text{ , }S_k:=S_{\lambda_k}(\cdot-\boldsymbol{y}_k)\text{ , and }Z_{k,l}=(Z_{l})_{\lambda_k}(\cdot-\boldsymbol{y}_k).
\end{equation*}
Using the pointwise bounds above together with the scaling properties of the profiles, we derive estimates in the $L^2$, $\dot H^1$ and $\dot H^2$ norms. These bounds will be used repeatedly in the construction of the approximate solutions.
\begin{lemma}[Localized norm estimates for the rescaled profiles]\label{Norm estimates for the rescaled profiles}
Let $Q_k$, $S_k$ and $Z_{k,l}$ be defined as above. 
The following estimates
hold locally near the \(k\)-th bubble. Throughout this lemma, all \(L^2\),
\(\dot H^1\), and \(\dot H^2\) norms are taken over the region
\[
B(y_k,2d)=\{x\in\mathbb R^4:\ |x-y_k|\le 2d\}.
\]

\medskip

\noindent\textbf{Estimates for $Q_k$.}
\[
\|Q_k\|_{L^2} \lesssim \frac{\log t}{\lambda(t)},\quad
\|Q_k\|_{\dot H^1} \lesssim \frac{\log t}{\lambda(t)},\quad
\|Q_k\|_{\dot H^2} \lesssim \frac{\log (t/\lambda(t))}{\lambda(t)}.
\]

\medskip

\noindent\textbf{Estimates for $S_k$.}
\begin{equation*}
\begin{aligned}
    \left\|S_k\right\|_{L^2}\lesssim {\frac{1}{\lambda(t)\log(t/\lambda(t))}}\text{ , }\left\|S_k\right\|_{\dot{H}^1}\lesssim {\frac{1}{\log(t/\lambda(t))}}\text{ , }\left\|S_k\right\|_{\dot{H}^2}\lesssim \frac{1}{\lambda(t)}.
\end{aligned}
\end{equation*}

\medskip

\noindent\textbf{Estimates for $Z_{k,l}$.}
\[
\|Z_{k,l}\|_{L^2} \lesssim 1,\quad
\|Z_{k,l}\|_{\dot H^1} \lesssim (\log(2d/\lambda(t)))^{1/2},\quad
\|Z_{k,l}\|_{\dot H^2} \lesssim \frac{1}{\lambda(t)}.
\]
\end{lemma}
\begin{proof}
We begin with $Q_k$. By the pointwise estimate of $Q$, a direct computation yields
\begin{equation*}
\begin{aligned}
       \left\|Q_k\right\|_{L^2(|x-\boldsymbol{y}_k|\leq 2d)}\lesssim& \left(\int_0^{2\lambda_k }\left(\frac{\log(2t/\lambda_k)}{\lambda_k}\frac{s^2}{\lambda_k^2}\right)^2s^3ds\right)^{1/2}\\
       &+\left(\int_{\lambda_k}^{2d}\left(\frac{\log (2t/s)}{\lambda_k}+\frac{\frac{s^2}{\lambda_k^2}+\log (\frac{s}{\lambda_k})\log(2t/\lambda_k)}{\lambda_k(1+(\frac{s}{\lambda_k}))^2}\right)^2s^3ds\right)^{1/2}
       \lesssim {\frac{\log t}{\lambda}},
\end{aligned}
\end{equation*}
and
\begin{equation*}
\begin{aligned}
       \left\|Q_k\right\|_{\dot{H}^1(|x-\boldsymbol{y}_k|\leq 2d)}\lesssim&\left(\int_0^{2\lambda_k }\left(\frac{\log(2t/\lambda_k)}{\lambda_k}\frac{s}{\lambda_k^2}\right)^2s^3ds\right)^{1/2}\\&+\left(\int_{\lambda_k}^{2d}\left(\frac{\log(2t/s)}{\lambda_ks}+\frac{\frac{s^2}{\lambda_k^2}+\log(\frac{s}{\lambda_k})\log(2t/\lambda_k)}{\lambda_k(1+(\frac{s}{\lambda_k})^3)}\right)^2s^3ds\right)^{1/2}
       \lesssim {\frac{\log t}{\lambda}}.
\end{aligned}
\end{equation*}
To estimate the $\dot H^2$ norm of $Q_k$, we use interior elliptic regularity:
\begin{equation*}
\begin{aligned}
       \left\|Q_k\right\|_{\dot{H}^2(|x-\boldsymbol{y}_k|\leq 2d)}\lesssim&\left\|\Delta Q_k\right\|_{L^2(|x-\boldsymbol{y}_k|\leq 3d)}+\left\|Q_k\right\|_{L_2(|x-\boldsymbol{y}_k|\leq 3d)}.
\end{aligned}
\end{equation*}
Since $Q_k$ solves
\begin{equation*}
     -\Delta Q_k-3W_k^2Q_k=-\frac{1}{\lambda_k^3}\left((\Lambda W \chi_t(\lambda_k \cdot))+\frac{\lelan \Lambda W \chi_t(\lambda_k \cdot),\Lambda W \rilan}{32\pi^2}f'(W)\right)\left(\frac{y-\boldsymbol{y}_k}{\lambda_k}\right).
\end{equation*}
{Therefore, $\left\|\Delta Q_k\right\|_{L^2(|x-\boldsymbol{y}_k|\leq 3d)}$ is bounded by }
\begin{equation*}
  3\left\|W_k^2Q_k\right\|_{L^2(|x-\boldsymbol{y}_k|\leq 3d)} +\frac{1}{\lambda_k}\left\|\Lambda W\right\|_{L^2(|y|\leq \frac{3d}{\lambda_k})}+\frac{\lelan \Lambda W \chi_t(\lambda \cdot),\Lambda W \rilan}{32\pi^2\lambda_k}\left\|3W^2\right\|_{L^2(|y|\leq \frac{3d}{\lambda_k})}.
\end{equation*}
The last two terms on the right-hand side are bounded by $\frac{\log(2t/\lambda_k)}{\lambda_k}$. While for the first term, the pointwise estimate of $Q$ gives
\begin{equation*}
    \begin{aligned}
      &\left\|W_k^2Q_k\right\|_{L^2(|x-\boldsymbol{y}_k|\leq 3d)}\lesssim  \left(\int_0^{2\lambda_k }\left(\frac{\log(2t/\lambda_k)}{\lambda_k}\frac{s^2}{\lambda_k^2}\frac{1}{\lambda_k^2}\right)^2s^3ds\right)^{1/2}\\
      &+\left(\int_{2\lambda_k}^{2d}\left(\frac{\log (2t/s)}{\lambda_k}+\frac{\frac{s^2}{\lambda_k^2}+\log (\frac{s}{\lambda_k})\log(1/\lambda_k)}{\lambda_k(1+(\frac{s}{\lambda_k}))^2}\right)^2\left(\frac{\lambda_k^2}{s^4}\right)^2s^3ds\right)^{1/2}
      \lesssim \frac{\log (2t/\lambda)}{\lambda}.
    \end{aligned}
\end{equation*}
As a consequence, it holds the estimate 
\begin{equation*}
    \left\|Q_k\right\|_{\dot{H}^2(|x-\boldsymbol{y}_k|\leq 2d)}\lesssim \left\|\Delta Q_k\right\|_{L^2(|x-\boldsymbol{y}_k|\leq 3d)}\lesssim \frac{\log (2t/\lambda)}{\lambda}.
\end{equation*}
The same argument yields the corresponding bounds for $S_k$:
\begin{equation*}
    \begin{aligned}
        \left\|S_k\right\|_{L^2(|x-\boldsymbol{y}_k|\leq 2d)}\lesssim & \left(\int_0^{2\lambda_k }\left(\frac{s^2}{\lambda_k^2}\right)^2s^3ds\right)^{1/2}+\left(\int_{\lambda_k}^{2d}\left(\frac{1}{\lambda_k}\frac{1+\log\left(\frac{s}{\lambda_k}\right)}{1+(s/\lambda_k)^2}\right)^2s^3ds\right)^{1/2}\\
       &+\left(\int_{\lambda_k}^{2d}\left(\frac{\log(t/s)+1}{\lambda_k\log(t/\lambda_k)}\right)^2s^3ds\right)^{1/2}
       \lesssim {\frac{\log t}{\lambda\log(t/\lambda)}},
    \end{aligned}
\end{equation*}
and
\begin{equation*}
    \begin{aligned}
        \left\|S_k\right\|_{\dot{H}^1(|x-\boldsymbol{y}_k|\leq 2d)}\lesssim & \left(\int_0^{2\lambda_k }\left(\frac{s}{\lambda_k^2}\right)^2s^3ds\right)^{1/2}+\left(\int_{\lambda_k}^{2d}\left(\frac{1}{\lambda_k}\frac{1+\log\left(\frac{s}{\lambda_k}\right)}{1+(s/\lambda_k)^3}\right)^2s^3ds\right)^{1/2}\\
       &+\left(\int_{\lambda_k}^{2d}\left(\frac{\log(t/s)+1}{\lambda_k\log(t/\lambda_k)(s/\lambda_k)}\right)^2s^3ds\right)^{1/2}
       \lesssim {\frac{\log t}{\log(t/\lambda)}}.
    \end{aligned}
\end{equation*}
Furthermore, $\left\|S_k\right\|_{\dot{H}^2}$ is dominated by
\begin{equation*}
    \begin{aligned}
        &\left\|W_k^2S_k\right\|_{L^2(|x-\boldsymbol{y}_k|\leq 3d)} +\frac{1}{\lambda_k}\left\|\underline{\Lambda}\Lambda W\right\|_{L^2}+\frac{\lelan (\underline{\Lambda}\Lambda W)\chi_t(\lambda r),\Lambda W \rilan}{\lelan \Lambda W \chi_t(\lambda \cdot),\Lambda W \rilan\lambda_k}\left\|\Lambda W\right\|_{L^2(|y|\leq \frac{3d}{\lambda_k})}
        \lesssim\frac{1}{\lambda},
    \end{aligned}
\end{equation*}
which completes the estimate for $S_k$ 

We next derive local estimates for the profiles $Z_{k,l}$. Recall that $Z_l$ solves
\[
L Z_l = \partial_l \Lambda W,
\]
and $Z_{k,l}$ is obtained by scaling and translation. By standard elliptic regularity applied in the region $|x-\boldsymbol{y}_k|\le 2d$, together with the decay properties of the right-hand side, we obtain
\begin{equation*}
 \|Z_{k,l}\|_{L^2(|x-\boldsymbol{y}_k|\le 2d)} \lesssim 1\text{ , }\|Z_{k,l}\|_{\dot H^1(|x-\boldsymbol{y}_k|\le 2d)} \lesssim \left(\log\left(\frac{2d}{\lambda}\right)\right)^{1/2}   
\end{equation*}
and
\begin{equation*}
 \|Z_{k,l}\|_{\dot H^2(|x-\boldsymbol{y}_k|\le 2d)} \lesssim \frac{1}{\lambda}.   
\end{equation*}
\end{proof}
\begin{coro}\label{on the ring region}
   In the subregion $d\leq|x-\boldsymbol{y}_k|\leq 2d$, using the same approach we obtain
   \begin{equation*}
     \left\||\Delta Q_k|+|\nabla Q_k|+|Q_k|+|f'(W_k)Q_k|\right\|_{H^1}\lesssim \frac{\log t}{\lambda(t)},
   \end{equation*}
   \begin{equation*}
   \left\||\Delta S_k|+|\nabla S_k|+|S_k|+|f'(W_k)S_k|\right\|_{H^1}\lesssim   {\frac{\log t}{\lambda(t)\log(t/\lambda(t))}},
   \end{equation*}
   and 
    \begin{equation*}
     \left\||\Delta Z_{k,l}|+|\nabla Z_{k,l}|+|Z_{k,l}|+|f'(W_k)Z_{k,l}|\right\|_{H^1}\lesssim \left(\log\left(\frac{2d}{\lambda(t)}\right)\right)^{1/2} .
   \end{equation*} 
\end{coro}
\begin{coro}\label{energy norm of d dt profile}
 Using Lemma 5.2 together with the pointwise bounds on $\frac{dQ}{dt}$ and $\frac{dS}{dt}$, one obtains, in the region $|x-\boldsymbol{y}_k|\le 2d$,
    \begin{equation*}
        \left\|\frac{dQ_k}{dt}\right\|_{H^1}\lesssim \frac{\log t }{\lambda(t) t^{\frac{1}{3}}}\text{ , } \left\|\frac{dS_k}{dt}\right\|_{{H}^1}\lesssim \frac{\log t  }{\lambda(t)\log(t/\lambda(t)) t^{\frac{1}{3}}} \text{ and }\left\|\frac{dZ_{k,l}}{dt}\right\|_{{H}^1}\lesssim \frac{(\log(2d/\lambda(t)))^{1/2}}{t^{\frac{1}{3}}},
    \end{equation*}
and
 \begin{equation*}
        \left\|\frac{dQ_k}{dt}\right\|_{H^2}\lesssim \frac{\log (t/\lambda(t)) }{\lambda(t)t^{\frac{1}{3}} }\text{ , } \left\|\frac{dS_k}{dt}\right\|_{{H}^2}\lesssim \frac{1 }{\lambda(t) t^{\frac{1}{3}}} \text{ and }\left\|\frac{dZ_{k,l}}{dt}\right\|_{{H}^2}\lesssim  \frac{1 }{\lambda(t)t^{\frac{1}{3}}} .
    \end{equation*}
\end{coro}
We now turn to the construction of approximate solutions adapted to the modulation analysis. Using the profiles $Q_k$, $S_k$ and $Z_{k,l}$ together with the norm estimates established above, we define two approximate solutions. The first one, denoted by $\TPG$, is constructed as
\begin{equation*}
\begin{aligned}
    \TPG:=\sum_{k}\chi\left(\frac{\cdot-\boldsymbol{y}_k}{d}\right)\Big(\beta_k^2S_k+\frac{128\pi^2\lambda_kB_k(\boldsymbol{\lambda},\boldsymbol{y})}{\lelan \Lambda W\chi_t(\lambda \cdot),\Lambda W \rilan}Q_k+\sum_{l=1}^4(\lambda_k'+\beta_k)s_{k,l} Z_{k,l}\Big),   
\end{aligned}
\end{equation*}
using the refined parameters $(\lambda_k, \beta_k, \boldsymbol{y}_k, \boldsymbol{s}_k)$ and is adapted to the first order energy estimates. The second one, denoted by $\PG$, is defined as 
\begin{equation*}
\begin{aligned}
     {P}_\Gamma:=\sum_{k}\chi\left(\frac{\cdot-\boldsymbol{y}_k}{d}\right)\Big(b_k^2S_k+\frac{128\pi^2\lambda_kB_k(\boldsymbol{\lambda},\boldsymbol{y})}{\lelan \Lambda W\chi_t(\lambda \cdot),\Lambda W \rilan}Q_k+\sum_{l=1}^4(\lambda_k'+b_k)v_{k,l} Z_{k,l}\Big),   
\end{aligned}
\end{equation*}
using $(\lambda_k, b_k, y_k, v_k)$ and is adapted to the second order energy framework. Using Lemma \ref{Norm estimates for the rescaled profiles}, their size is quantified by the following norm estimates.
\begin{equation}\label{tpg h dot 1}
    \begin{aligned}
     \left\|\TPG\right\|_{\dot{H}^1}\lesssim& \sum_{k=1}^N\Bigg\{\beta_k^2\left\|S_k\right\|_{H^1}+\frac{\lambda_k^2}{t^{\frac{2}{3}}}\left\|Q_k\right\|_{H^1}+\sum_{l=1}^4\left|\lambda_k'+\beta_k\right||\boldsymbol{s}_k|\left\|Z_{k,l}\right\|_{H^1} \Bigg\}\\
     \lesssim&\sum_{k=1}^N\left({\frac{\beta_k^2\log t}{\lambda_k\log(t/\lambda_k)}}+\frac{\lambda_k\log t}{t^{\frac{2}{3}}}+\left|\lambda_k'+\beta_k\right||\boldsymbol{s}_k|(\log(2d/\lambda_k))^{1/2}\right)\lesssim (t^{-\frac{2}{3}}\log t)\lambda(t),
    \end{aligned}
\end{equation}
and
\begin{equation}\label{pg h dot 2}
 \begin{aligned}
     \left\|{P}_\Gamma\right\|_{{H}^2}\lesssim& \sum_{k=1}^N\Bigg\{b_k^2\left\|S_k\right\|_{H^2}+\frac{\lambda_k^2}{t^{\frac{2}{3}}}\left\|Q_k\right\|_{H^2}+\sum_{l=1}^4\left|\lambda_k'+b_k\right||\boldsymbol{v}_k|\left\|Z_{k,l}\right\|_{H^2} \Bigg\}\\
     \lesssim&\sum_{k=1}^N\left( \frac{b_k^2}{\lambda_k}+\frac{\lambda_k t^{\frac{2}{3}}}{t^{\frac{2}{3}}}+\left|\lambda_k'+b_k\right||\boldsymbol{v}_k|\frac{1}{\lambda_k}\right)\lesssim \lambda(t).
    \end{aligned}    
\end{equation}
Applying Corollary \ref{energy norm of d dt profile} and Lemma \ref{close the bootstrap for parameters system}, we have
\begin{equation}\label{partial t tilde p H 1}
\begin{aligned}
      \left\|\partial_t\TPG\right\|_{H^1}\lesssim& \sum_{k=1}^N\Bigg\{\beta_k|\beta_k'|\left\|S_k\right\|_{H^1}+\frac{|\lambda_k'|\lambda_k}{t^{\frac{2}{3}}}\left\|Q_k\right\|_{H^1}+\sum_{l=1}^4\Big[(|\lambda_k''|+|\beta_k'|)|\boldsymbol{s}_k|+\left|\lambda_k'+\beta_k\right||\boldsymbol{s}_k'|\Big]\left\|Z_{k,l}\right\|_{H^1}\\
      &+\beta_k^2\left\|\frac{dS_k}{dt}\right\|_{H^1}+\frac{\lambda_k^2}{t^{\frac{2}{3}}}\left\|\frac{dQ_k}{dt}\right\|_{H^1}+\sum_{l=1}^4\left|\lambda_k'+\beta_k\right||\boldsymbol{s}_k|\left\|\frac{dZ_{k,l}}{dt}\right\|_{H^1}\Bigg\}\\
      \lesssim& (t^{-\frac{5}{3}}\log t)\lambda+ (t^{-1}\log t)\lambda+\lambda^{\frac{3}{2}}+(t^{-1}\log t)\lambda\lesssim (t^{-1}\log t)\lambda(t),
\end{aligned}
\end{equation}
\begin{equation}\label{partial t P h dot 1}
    \begin{aligned}
        \left\|\partial_t{P}_\Gamma\right\|_{ H^1}\lesssim& \sum_{k=1}^N\Bigg\{b_k|b_k'|\left\|S_k\right\|_{{H}^1}+\frac{|\lambda_k'|\lambda_k}{t^{\frac{2}{3}}}\left\|Q_k\right\|_{{H}^1}+\sum_{l=1}^4\Big[(|\lambda_k''|+|b_k'|)|\boldsymbol{v}_k|+\left|\lambda_k'+b_k\right||\boldsymbol{v}_k'|\Big]\left\|Z_{k,l}\right\|_{{H}^1}\\
      &+b_k^2\left\|\frac{dS_k}{dt}\right\|_{{H}^1}+\frac{\lambda_k^2}{t^{\frac{2}{3}}}\left\|\frac{dQ_k}{dt}\right\|_{{H}^1}+\sum_{l=1}^4\left|\lambda_k'+b_k\right||\boldsymbol{v}_k|\left\|\frac{dZ_{k,l}}{dt}\right\|_{{H}^1}\Bigg\}\\
      \lesssim& \frac{t^{\frac{1}{2}}\lambda}{t}\log t+(t^{-1}\log t)\lambda+\lambda^{\frac{3}{2}}+(t^{-1}\log t)\lambda\lesssim (t^{\frac{1}{2}}\log t)\lambda(t),
    \end{aligned}
\end{equation}
and
\begin{equation}\label{partial t pg h dot 2}
     \begin{aligned}
        \left\|\partial_t{P}_\Gamma\right\|_{\dot H^2}\lesssim& \sum_{k=1}^N\Bigg\{b_k|b_k'|\left\|S_k\right\|_{\dot{H}^2}+\frac{|\lambda_k'|\lambda_k}{t^{\frac{2}{3}}}\left\|Q_k\right\|_{\dot{H}^2}+\sum_{l=1}^4\Big[(|\lambda_k''|+|b_k'|)|\boldsymbol{v}_k|+\left|\lambda_k'+b_k\right||\boldsymbol{v}_k'|\Big]\left\|Z_{k,l}\right\|_{\dot{H}^2}\\
      &+b_k^2\left\|\frac{dS_k}{dt}\right\|_{\dot{H}^2}+\frac{\lambda_k^2}{t^{\frac{2}{3}}}\left\|\frac{dQ_k}{dt}\right\|_{\dot{H}^2}+\sum_{l=1}^4\left|\lambda_k'+b_k\right||\boldsymbol{v}_k|\left\|\frac{dZ_{k,l}}{dt}\right\|_{\dot{H}^2}\Bigg\}\\
      \lesssim& \frac{t^{\frac{1}{2}}\lambda}{t^{\frac{1}{3}}}+t^{-\frac{1}{3}}\lambda+\lambda^{\frac{3}{2}}+\lambda\log t\lesssim t^{\frac{1}{6}}\lambda(t).
    \end{aligned} 
\end{equation}
Here we emphasize all the bounds of $\PG$, $\TPG$ are sufficiently small in the analysis later except the $\dot{H}^2$ norm of $\partial_t\PG$. The bad part comes from $b_k|b_k'|\left\|S_k\right\|_{\dot{H}^2}$ since we only have a rough estimate (\ref{rough est on b'}) for $b_k'$. More precisely, using Corollary \ref{on the ring region}, (\ref{partial t pg h dot 2}) can be written as 
\begin{equation}\label{refine dt pg h dot 2}
     \left\|\partial_t{P}_\Gamma\right\|_{\dot H^2}\lesssim \sum_{j=1}^N|b_j'b_j|\left\|S_k\right\|_{\dot{H}^2(|x-\boldsymbol{y}_k|\leq d)}+\lambda(t).
\end{equation}
This term will be treated more carefully in the second order energy estimate.
\subsection{Error estimates}
Recall the definition of $\WTG$ and $\WG$, the solution $\Vec{u}$ can be decomposed in two ways:
\begin{equation*}
\left(
    \begin{array}{c}
         u \\
         \partial_t u 
    \end{array}\right)=
\left(
    \begin{array}{c}
         \sum_{j=1}^N W_j\\
         \sum_{j=1}^N \frac{\beta_j}{\lambda_j}(\Lambda_j W_j)\chi_{t}+\frac{\boldsymbol{s}_j\cdot\nabla_jW_j}{\lambda_j}
    \end{array}\right)+
    \left(
    \begin{array}{c}
         g \\
         \dot{h} 
    \end{array}\right),
\end{equation*}
and
\begin{equation*}
  \left(
    \begin{array}{c}
         u \\
         \partial_t u 
    \end{array}\right)=
    \left(
    \begin{array}{c}
         \sum_{j=1}^N W_j \\
        \sum_{j=1}^N \frac{b_j}{\lambda_j}(\Lambda_jW_j)\chi_{t}+\frac{\boldsymbol{v}_j\cdot\nabla_jW_j}{\lambda_j} 
    \end{array}\right)+
    \left(
    \begin{array}{c}
         g\\
         \dot{g} 
    \end{array}\right). 
\end{equation*}
The corrected decomposition is used to exploit the sharper modulation system in
the first order estimate, whereas the original decomposition is kept for the
second order estimate. Accordingly, we set
\begin{equation*}
    \Tilde{\phi}:=\sum_{j=1}^NW_j+\TPG\text{ , }\Tilde{h}=g-\TPG
\end{equation*}
and
\begin{equation*}
    \phi:=\sum_{j=1}^NW_j+P_{\Gamma}\text{ , }h:=g-P_{\Gamma},
\end{equation*}
so that 
\begin{equation*}
    u=\sum_{j=1}^NW_j+g=\Tilde{\phi}+\Tilde{h}=\phi+h.
\end{equation*}
We first estimate the error associated with
\((\tilde \phi,\tilde h,\dot{\tilde h})\). For brevity, define
\begin{equation}\label{first order PSI}
   \begin{aligned}
         \Psi(\boldsymbol{\lambda},\boldsymbol{\beta},{\boldsymbol{y}},{\boldsymbol{s}}):=&\sum_{j=1}^N \frac{(\lambda_j'+\beta_j)\beta_j}{\lambda_j^2}(\underline{\Lambda}_j\Lambda_j W_j) \chi_t+\sum_{j=1}^N\frac{\beta_j((\boldsymbol{y}_j'+\boldsymbol{s}_j)\cdot \nabla_j)}{\lambda_j^2}(\Lambda_jW_j )\chi_t
           \\&-\sum_{j=1}^N\left(\beta_j'-\frac{\beta_j \lambda_j'}{2\log(t/\lambda_j)\lambda_j} -\frac{B_j(\boldsymbol{\lambda},\boldsymbol{y})}{\log(t/\lambda_j)}\right)\lambda_j^{-1}(\Lambda
           _j W_j)\chi_t,  
    \end{aligned}  
\end{equation}
and
\begin{equation*}
   \Tilde{\epsilon}(\tilde{h},\dot{h}):=\partial_t \dot{h}-\{ \Delta \Tilde{h}+f(\Tilde{\phi}+\Tilde{h})-f(\Tilde{\phi})+\Psi(\boldsymbol{\lambda},\boldsymbol{\beta},{\boldsymbol{y}},{\boldsymbol{s}})\}.
\end{equation*}
With these notations we have
\begin{lemma}\label{lemma of partial t dot g 2}
    Under the bootstrap assumption, the following estimate holds
      \begin{equation}\label{partial t g dot sim 2}
    \begin{aligned}
           \lVert \Tilde{\epsilon}(\tilde{h},\dot{h})\rVert_{L^2} 
           \lesssim {{(t^{-\frac{2}{3}}\log t)\lambda(t) }}.
    \end{aligned}
    \end{equation}
\end{lemma}
\begin{proof}
    By definition, $\partial_t \dot{h}$ satisfies  
  \begin{equation*}\label{partial h dot}
    \begin{aligned}
        \partial_t \dot{h}=&\Delta g+f(W_{\boldsymbol{\Gamma}}+g)-\sum_{j=1}^Nf(W_j)+\sum_{j=1}^N\frac{\lambda_j'\beta_j}{\lambda_j^2} (\underline{\Lambda}_j\Lambda_jW_j)\chi_{t}\\
        &-\sum_{j=1}^N\frac{\beta_j'}{\lambda_j}(\Lambda_jW_j\chi_{t})+\sum_
        {j=1}^N\frac{\beta_j|x-\boldsymbol{y}_j|}{\lambda_jt^2}(\Lambda_jW_j\chi'_{t})+\sum_{j=1}^N\frac{\beta_j}{\lambda_j^2}(\boldsymbol{y}_j'\cdot\nabla)(\Lambda_j W_j\chi_{t})\\
         & -\sum_{j=1}^N\sum_{l=1}^4\frac{s'_{j,l}(\partial_lW)(\frac{\cdot-\boldsymbol{y}_j}{\lambda_j})}{\lambda_j^2}-\sum_{j=1}^N\sum_{l=1}^4\frac{\lambda'_js_{j,l}(\underline{\Lambda} \partial_l W)\left(\frac{\cdot-\boldsymbol{y}_j}{\lambda_j}\right)}{\lambda_j^3}\\
   &+\sum_{j=1}^N\sum_{l,m=1}^4\frac{s_{j,l}y'_{j,m}(\partial_l\partial_mW)(\frac{\cdot-\boldsymbol{y}_j}{\lambda_j})}{\lambda_j^3}.
    \end{aligned}
\end{equation*}  
Hence, $\Tilde{\epsilon}(\tilde{h},\dot{h})$ can be written as
\begin{equation}\label{tilde epsilon equal}
\begin{aligned}
    \Tilde{\epsilon}(\tilde{h},\dot{h})=&\Delta \TPG+f(W_{\boldsymbol{\Gamma}}+\TPG)-\sum_{j=1}^Nf(W_j)-\sum_{j=1}^N\frac{\beta_j^2}{\lambda_j^2} (\underline{\Lambda}_j\Lambda_jW_j)\chi_{t}\\
        &-\sum_{j=1}^N\left(\frac{\beta_j \lambda_j'}{2\log(t/\lambda_j)\lambda_j} +\frac{B_j(\boldsymbol{\lambda},\boldsymbol{y})}{\log(t/\lambda_j)}\right)\lambda_j^{-1}\Lambda
           _j W_j\chi_t\\&+\sum_
        {j=1}^N\frac{\beta_j|x-\boldsymbol{y}_j|}{\lambda_jt^2}(\Lambda_jW_j\chi'_{t})-\sum_{j=1}^N\frac{\beta_j}{\lambda_j^2}(\boldsymbol{s}_j\cdot\nabla)(\Lambda_j W_j\chi_{t})\\
         & -\sum_{j=1}^N\sum_{l=1}^4\frac{s'_{j,l}(\partial_lW)(\frac{\cdot-\boldsymbol{y}_j}{\lambda_j})}{\lambda_j^2}-\sum_{j=1}^N\sum_{l=1}^4\frac{\lambda'_js_{j,l}(\underline{\Lambda} \partial_l W)\left(\frac{\cdot-\boldsymbol{y}_j}{\lambda_j}\right)}{\lambda_j^3}\\
   &+\sum_{j=1}^N\sum_{l,m=1}^4\frac{s_{j,l}y'_{j,m}(\partial_l\partial_mW)(\frac{\cdot-\boldsymbol{y}_j}{\lambda_j})}{\lambda_j^3}.
    \end{aligned}
\end{equation}
In order to prove (\ref{partial t g dot sim 2}), we again divide $\Tilde{B}:=\bigcup_{1\leq j\leq N}B(|x-\boldsymbol{y}_j|\leq d)$ and $\mathbb{R}^4 \setminus \tilde {B}$. We first consider the exterior region $\mathbb{R}^4 \setminus \tilde {B}$ and focus on the terms not involving $\TPG$.  In this region,  the multi-bubble term is bounded by
\begin{equation*}
    \left\|\sum_{j=1}^Nf(W_j)\right\|_{L^2(\mathbb{R}^4 \setminus \tilde {B})}\lesssim \sum_{j=1}^N\left(\int_d^{\infty}\left(\frac{\lambda_j^3}{r^6}\right)^2r^3dr\right)^{\frac{1}{2}}\lesssim \lambda^3.  
\end{equation*}
Then, the decay of $\underline{\Lambda}\Lambda W$ yields
\begin{equation*}
    \left\|\frac{\beta_j^2}{\lambda_j^2} (\underline{\Lambda}_j\Lambda_jW_j)\chi_{t}\right\|_{L^2(\mathbb{R}^4 \setminus \tilde {B})}\lesssim \beta_j^2\left(\int_d^{2t}\left(\frac{\lambda_j}{r^4}\right)^2r^3dr\right)^{\frac{1}{2}}\lesssim \lambda^3.
\end{equation*}
The following term provides the leading contribution in this region:
\begin{equation*}
    \begin{aligned}
   &\left\|\left(\frac{\beta_j \lambda_j'}{2\log(t/\lambda_j)\lambda_j} +\frac{B_j(\boldsymbol{\lambda},\boldsymbol{y})}{\log(t/\lambda_j)}\right)\lambda_j^{-1}\Lambda
           _j W_j\chi_t(\cdot-\boldsymbol{y}_j)\right\|_{L^2(\mathbb{R}^4 \setminus \tilde {B})}\\
           \lesssim& \left(\left|\frac{\beta_j \lambda_j'}{2\log(t/\lambda_j)\lambda_j}\right|+\left|\frac{B_j(\boldsymbol{\lambda},\boldsymbol{y})}{\log(t/\lambda_j)}\right|\right)\left(\int_d^{2t}\left(\frac{1}{r^2}\right)^2r^3dr\right)^{\frac{1}{2}}\lesssim (t^{-\frac{2}{3}}\log t)\lambda,  
    \end{aligned}
\end{equation*}
The remaining terms are treated in the same way. Here we emphasize two terms. First,  the cut-off derivative term is controlled by
\begin{equation*}
    \left\|\frac{\beta_j|x-\boldsymbol{y}_j|}{\lambda_jt^2}(\Lambda_jW_j\chi'_{t})\right\|_{L^2(\mathbb{R}^4 \setminus \tilde {B})}\lesssim \frac{\beta_j}{t}\left(\int_t^{2t}\left(\frac{1}{r^2}\right)^2r^3dr\right)^{\frac{1}{2}}\lesssim t^{-\frac{4}{3}}\lambda.
\end{equation*}
Moreover, by (\ref{s'+=to}), the term involving $\boldsymbol{s}'$ is bounded by
\begin{equation*}
   \left\|\frac{s'_{j,l}(\partial_lW)(\frac{\cdot-\boldsymbol{y}_j}{\lambda_j})}{\lambda_j^2}\right\|_{L^2(\mathbb{R}^4 \setminus \tilde {B})}\lesssim |\boldsymbol{s}'|\left(\int_d^{\infty}\left(\frac{\lambda_j}{r^3}\right)^2r^3dr\right)^{\frac{1}{2}}\ll\lambda^2. 
\end{equation*}
We next consider the terms involving $\TPG$. Since $\TPG=0$ in $ \bigcup_{1\leq j \leq N}B(|x-\boldsymbol{y}_j|\geq 2d)$, it suffices to consider on $d\leq|x-\boldsymbol{y}_j|\leq 2d$, for $1\leq j\leq N$. Using Corollary \ref{on the ring region}, these terms are controlled by
\begin{equation*}
    \begin{aligned}
        &\left\|\Delta\TPG+f(\WG+\TPG)\right\|_{L^2(d\leq |x-\boldsymbol{y}_k|\leq 2d)}\\
        \lesssim& \frac{\beta_k^2}{\lambda_k}+\frac{\lambda_k c\lambda}{\Xi(t)}\frac{\log t}{\lambda_k}+\frac{|\lambda_k'+\beta_k||s_{k,l}|}{\lambda_k}\lesssim (t^{-\frac{2}{3}}\log t)\lambda, 
    \end{aligned}
\end{equation*}
which completes the error estimates on $\mathbb{R}^4 \setminus \tilde {B}$. In the near k-bubble region $|x-\boldsymbol{y}_k|\leq d$, the definition of $\TPG$ yields 
\begin{equation*}
    \begin{aligned}
     \ML_k \TPG=&\beta_k^2\ML_k S_k+\frac{128\pi^2\lambda_kB_k(\boldsymbol{\lambda},\boldsymbol{y})}{\lelan \Lambda W\chi_t(\lambda \cdot),\Lambda W \rilan}\ML_kQ_k+\sum_{l=1}^4(\lambda_k'+\beta_k)s_{k,l}\ML_k Z_{k,l}\\
     =&-\frac{\beta_k^2}{\lambda_k^2}(\underline{\Lambda}_k\Lambda_kW_k)\chi_{t}+\frac{\beta_k^2}{2\log (t/\lambda_k)\lambda_k^2}\Lambda_k W_k\chi_t-\frac{B_k(\boldsymbol{\lambda},\boldsymbol{y})}{\log(t/\lambda_k)\lambda_k}\Lambda_k W_k\chi_t\\
     &-4{B_k(\boldsymbol{\lambda},\boldsymbol{y})}f'(W_k)-\frac{\beta_k}{\lambda_k^2}(\boldsymbol{s}_k\cdot\nabla)(\Lambda_k W_k)-\sum_{l=1}^4\frac{\lambda_k's_{k,l}(\underline{\Lambda} \partial_l W)\left(\frac{\cdot-\boldsymbol{y}_k}{\lambda_k}\right)}{\lambda_k^3}.
    \end{aligned}
\end{equation*}
Plugging this back into (\ref{tilde epsilon equal}) and using the definition of $Q_k,S_k,Z_{k.l}$, we obtain  the error $\Tilde{\epsilon}(\tilde{h},\dot{h})$ equals to 
\begin{equation*}
    \begin{aligned}
       &f(W_{\boldsymbol{\Gamma}}+\TPG)-3W_k^2\TPG-\sum_{j=1}^Nf(W_j)+{4B_k(\boldsymbol{\lambda},\boldsymbol{y})}f'(W_k)\\ -&\frac{\beta_k (\lambda_k'+\beta_k)}{2\log(t/\lambda_k)\lambda_k} \lambda_k^{-1}\Lambda
           _k W_k-\sum_{j=1}^N\sum_{l=1}^4\frac{s'_{j,l}(\partial_lW)(\frac{\cdot-\boldsymbol{y}_j}{\lambda_j})}{\lambda_j^2}+\sum_{j=1}^N\sum_{l,m=1}^4\frac{s_{j,l}y'_{j,m}(\partial_l\partial_mW)(\frac{\cdot-\boldsymbol{y}_j}{\lambda_j})}{\lambda_j^3},
    \end{aligned}
\end{equation*}
inside the ball $|x-\boldsymbol{y}_k|\leq d$. The $L^2$ norm of the second line can be absorbed into the right-hand side of (\ref{partial t g dot sim 2}) by (\ref{lambda'+b}), (\ref{b and beta}), (\ref{y'+v}), (\ref{s'+=to}). Here the leading error term comes from
\begin{equation*}
   \left\|\frac{\beta_k (\lambda_k'+\beta_k)}{2\log(t/\lambda_k)\lambda_k} \lambda_k^{-1}\Lambda
           _k W_k\right\|_{L^2(|x-\boldsymbol{y}_k|\leq d)}\lesssim \frac{|\lambda'+\beta|}{t}\left\|\frac{1}{\lambda_k}\Lambda
           _k W_k\right\|_{L^2(|x-\boldsymbol{y}_k|\leq d)}\lesssim (t^{-\frac{3}{2}}\sqrt{\log t})\lambda.
\end{equation*}
The first line can be written as
\begin{equation*}
\begin{aligned}
  &\Bigg(f(W_{\boldsymbol{\Gamma}}+\TPG)-3W_k^2\TPG-\sum_{j=1}^Nf(W_j)-\sum_{j\neq k}f'(W_k)W_j\Bigg)\\
  + &  \Bigg( \sum_{j\neq k}f'(W_k)W_j+{4B_k(\boldsymbol{\lambda},\boldsymbol{y})}f'(W_k)\Bigg):=\Rmnum{1}+\Rmnum{2}.
\end{aligned}
\end{equation*}
For $\Rmnum{1}$, the dominating terms are $\left\|W_k\TPG^2\right\|_{L^2}$, and $\left\|W_kW_j^2\right\|_{L^2}$. The quadratic term is bounded by Cauchy’s inequality and interpolation:
\begin{equation*}
  \left\|W_k\TPG^2\right\|_{L^2}\lesssim \left\|W_k\right\|_{L^4}\left\|\TPG\right\|_{L^8}^2\lesssim\left\|\TPG\right\|_{\dot{H}^1}^2+\left\|\TPG\right\|^2_{\dot{H}^2}\ll \lambda^{\frac{3}{2}},  
\end{equation*}
which can be absorbed. While the decay estimates yields 
\begin{equation*}
\left\|W_kW_j^2\right\|_{L^2}\lesssim \left(\int_{\lambda_k}^d\left(\frac{\lambda_k}{r^2}\frac{\lambda_j^2}{|z_k-z_j|^4}\right)^2r^3dr\right)^{\frac{1}{2}}\lesssim \lambda_k \lambda_j^2\log(d/\lambda_k)\ll  \lambda^2 . 
\end{equation*}
Consequently, $\left\|\Rmnum{1}\right\|_{L^2}$ is sufficiently small. Furthermore, for \Rmnum{2}, we notice that near k-bubble, the expansion
\begin{equation*}
    W_j(x)=\frac{8\lambda_j}{|\boldsymbol{y}_k-\boldsymbol{y}_j|^2}+O\left(\frac{\lambda_j|x-\boldsymbol{y}_k|}{|\boldsymbol{y}_k-\boldsymbol{y}_j|^3}\right)
\end{equation*}
holds. As a result we have
\begin{equation*}
    \begin{aligned}
        \left\|\Rmnum{2}\right\|_{L^2}\lesssim&\sum_{j\neq k}\left\|W_k^2\left(W_j-\frac{8\lambda_j}{|\boldsymbol{y}_k-\boldsymbol{y}_j|^2}\right)\right\|_{L^2(|x-\boldsymbol{y}_k|\leq d)}\\
        \lesssim&\lambda_j\left(\int_0^d\left(\frac{\lambda_k}{\lambda_k^2+r^2}\right)^4r^2r^3dr\right)^{\frac{1}{2}}\lesssim \lambda^{\frac{3}{2}},
    \end{aligned}
\end{equation*}
which is also sufficiently small. Collecting all these inequalities completes the proof of (\ref{partial t g dot sim 2}).
\end{proof}

We now pass to the error term associated with the decomposition $(\phi,h,\dot g)$. We introduce the operator $\ML_{\phi}$ which is defined as 
\begin{equation*}
    \ML_{\phi}h:=-\Delta h-3\phi^2h.
\end{equation*}
We denote
\begin{equation*}
   \begin{aligned}
         \Psi(\boldsymbol{\lambda},\boldsymbol{b},{\boldsymbol{y}},{\boldsymbol{v}}):=&\sum_{j=1}^N \frac{(\lambda_j'+b_j)b_j}{\lambda_j^2}\underline{\Lambda}_j\Lambda_j W_j \chi_t+\sum_{j=1}^N\frac{b_j((\boldsymbol{y}_j'+\boldsymbol{v}_j)\cdot \nabla_j)}{\lambda_j^2}\Lambda_jW_j \chi_t
           \\&-\sum_{j=1}^N\left(b_j'-\frac{b_j \lambda_j'}{2\log(t/\lambda_j)\lambda_j} -\frac{B_j(\boldsymbol{\lambda},\boldsymbol{y})}{\log(t/\lambda_j)}\right)\lambda_j^{-1}\Lambda
           _j W_j\chi_t,  
    \end{aligned}  
\end{equation*}
and
\begin{equation*}
  {\epsilon}(h,\dot{g}):=\partial_t \dot{g}-\{ \Delta h+f({\phi}+{h})-f({\phi})+ \Psi(\boldsymbol{\lambda},\boldsymbol{b},{\boldsymbol{y}},{\boldsymbol{v}})\}.
\end{equation*}
Analogously, we establish the following error estimate.
\begin{lemma}\label{lemma of partial t dot g}
    Under the bootstrap assumption we have
    \begin{equation}\label{partial t g dot sim}
    \begin{aligned}
          \left\|\ML_{\phi}( {\epsilon}(h,\dot{g}))\right\|_{\dot{H}^{-1}}\lesssim \lambda(t).
    \end{aligned}
    \end{equation}
\end{lemma}
\begin{proof} 
In this proof, a term is called negligible if it is much smaller than the right-hand side of (\ref{partial t g dot sim}). Inserting the expression of $\partial_t \dot{g}$ (\ref{partial g dot})  back into (\ref{partial t g dot sim}), it suffices for us to prove 
\begin{equation*}
    \begin{aligned}
        \Bigg\lVert &\ML_{\phi}\Big(\Delta \PG+f(W_{\boldsymbol{\Gamma}}+\PG
        )-\sum_{j=1}^Nf(W_j)-\sum_{j=1}^N \frac{b_j^2}{\lambda_j^2}(\underline{\Lambda}_j\Lambda_jW_j)\chi_{t}\\
        &-\sum_{j=1}^N\left(\frac{b_j \lambda_j'}{2|\log(t/\lambda_j)|\lambda_j} +\frac{B_j(\boldsymbol{\lambda},\boldsymbol{y})}{\left\|\Lambda W\chi_t (\lambda_j \cdot)\right\|_{L^2}^2}\right)\lambda_j^{-1}(\Lambda
           _j W_j)\chi_t-\sum_{j=1}^N\frac{b_j}{\lambda_j^2}(\boldsymbol{v}_j\cdot\nabla)((\Lambda_j W_j)\chi_{t})
           \\&+\sum_
        {j=1}^N\frac{b_j|x-\boldsymbol{y}_j|}{\lambda_jt^2}((\Lambda_jW_j)\chi'_{t})
          -\sum_{j=1}^N\sum_{l=1}^4\frac{v'_{j,l}(\partial_lW)(\frac{\cdot-\boldsymbol{y}_j}{\lambda_j})}{\lambda_j^2}+\sum_{j=1}^N\sum_{l=1}^4\frac{\lambda'_jv_{j,l}(\Lambda \partial_l W)(\frac{\cdot-\boldsymbol{y}_j}{\lambda_j})}{\lambda_j^3}\\
   &+\sum_{j=1}^N\sum_{l,m=1}^4\frac{v_{j,l}(t)y'_{j,m}(t)(\partial_l\partial_mW)(\frac{\cdot-\boldsymbol{y}_j}{\lambda_j})}{\lambda_j^3} \Big)\Bigg\rVert_{\dot{H}^{-1}} \lesssim \lambda.
    \end{aligned}
\end{equation*}
The main difference with the proof of Lemma \ref{lemma of partial t dot g 2} lies in the use of the second order structure. Instead of estimating $\tilde\varepsilon$ in $L^2$, we control $\ML_\phi \varepsilon$ in $\dot H^{-1}$. This is essential since many error terms involve $\Lambda W$ and $\nabla W$, which lie in the kernel of $\ML_k$. To see this, we start the analysis with the term $ \left\|\ML_{\phi}\Big(\frac{v'_{k,l}\partial_lW(\frac{\cdot-\boldsymbol{y}_k}{\lambda_k})}{\lambda_k^2}\Big)\right\|_{\dot{H}^{-1}}$. Since $\partial_l W\in \ker \ML$,
\begin{equation*}
\begin{aligned}
  \left\|\ML_{\phi}\Bigg(\frac{v'_{k,l}\partial_lW(\frac{\cdot-\boldsymbol{y}_k}{\lambda_k})}{\lambda_k^2}\Bigg)\right\|_{\dot{H}^{-1}}=&\left\|(\ML_{\phi}-\ML_{k})\Bigg(\frac{v'_{k,l}\partial_lW(\frac{\cdot-\boldsymbol{y}_k}{\lambda_k})}{\lambda_k^2}\Bigg)\right\|_{\dot{H}^{-1}}\\
  =&\left\|(3{\phi}^2-3W_{k}^2)\Bigg(\frac{v'_{k,l}\partial_lW(\frac{\cdot-\boldsymbol{y}_k}{\lambda_k})}{\lambda_k^2}\Bigg)\right\|_{\dot{H}^{-1}}
\end{aligned}
\end{equation*}
From the definition of $\phi$, 
\begin{equation*}
    3\phi^2-3W_k^2=3\left(\sum_{j\neq k}W_j+\PG\right)^2+6W_k\left(\sum_{j\neq k}W_j+\PG\right).
\end{equation*}
Then using Cauchy and Sobolev, it holds 
\begin{equation*}
    \left\|W_j W_k \frac{v'_{k,l}\partial_lW(\frac{\cdot-\boldsymbol{y}_k}{\lambda_k})}{\lambda_k^2}\right\|_{\dot{H}^{-1}}\lesssim |v'_{k,l}|\Bigg( \int_{\lambda_k}^{\infty}\left(\frac{\lambda_j}{|\boldsymbol{y}_k-\boldsymbol{y}_j|^2}\cdot\frac{\lambda_k}{r^2}\cdot \frac{\lambda_k}{r^3}\right)^{\frac{4}{3}}r^3dr \Bigg)^{3/4}\lesssim \lambda|\boldsymbol{v}'|;
\end{equation*}
\begin{equation*}
    \left\|\PG W_k \frac{v'_{k,l}\partial_lW(\frac{\cdot-\boldsymbol{y}_k}{\lambda_k})}{\lambda_k^2}\right\|_{\dot{H}^{-1}}\lesssim \left\| \PG\right\|_{\dot{H}^{2}} |v'_{k,l}|\left( \int_{\lambda_k}^{\infty}\left(\frac{\lambda_k}{r^2}\cdot \frac{\lambda_k}{r^3}\right)^{\frac{4}{3}}r^3dr \right)^{3/4}\lesssim |\boldsymbol{v}'|\left\| \PG\right\|_{\dot{H}^{2}};
\end{equation*}
\begin{equation*}
\begin{aligned}
    \left\|W_j^2 \frac{v'_{k,l}\partial_lW(\frac{\cdot-\boldsymbol{y}_k}{\lambda_k})}{\lambda_k^2}\right\|_{\dot{H}^{-1}}\lesssim&  |v'_{k,l}|\left( \int_{\lambda_k}^{d}\left(\frac{\lambda_j^2}{|\boldsymbol{y}_k-\boldsymbol{y}_j|^4}\cdot \frac{\lambda_k}{r^3}\right)^{\frac{4}{3}}r^3dr \right)^{3/4}
    +|v_{k,l}'|\lambda_j^2\lambda_k\left\||x|^{-7}\right\|_{L^{\frac{4}{3}}(|x|\geq d/2)}\\
    &+|v'_{k,l}|\left( \int_{\lambda_j}^{d}\left(\frac{\lambda_j^2}{|r|^4}\cdot \frac{\lambda_k}{|\boldsymbol{y}_k-\boldsymbol{y}_j|^3}\right)^{\frac{4}{3}}r^3dr \right)^{3/4}
    \lesssim |\boldsymbol{v}'|\lambda^2;
\end{aligned}
\end{equation*}
and 
\begin{equation*}
      \left\|\PG^2 \frac{v'_{k,l}\partial_lW(\frac{\cdot-\boldsymbol{y}_k}{\lambda_k})}{\lambda_k^2}\right\|_{\dot{H}^{-1}}\lesssim \left\| \PG\right\|_{\dot{H}^{2}}^2 |v'_{k,l}|\left( \int_{\lambda_k}^{d}\left( \frac{\lambda_k}{r^3}\right)^{\frac{4}{3}}r^3dr \right)^{3/4}\lesssim \lambda|\boldsymbol{v}'|\left\| \PG\right\|_{\dot{H}^{2}}^2\left(\log\left(\frac{d}{\lambda}\right)\right)^{\frac{3}{4}}.
\end{equation*}
Collecting the above bounds and using (\ref{rough est on v'}), (\ref{pg h dot 2}), we obtain 
\begin{equation*}
  \left\|\ML_{\phi}\Big(-\frac{v'_{k,l}\partial_lW(\frac{\cdot-\boldsymbol{y}_k}{\lambda_k})}{\lambda_k^2}\Big)\right\|_{\dot{H}^{-1}}\ll \lambda^{\frac{3}{2}},
\end{equation*}
which is sufficiently small. Furthermore, from the definition of $\ML_{\phi}$, for any $f\in \dot{H}^1$ it holds
\begin{equation*}
    \left\|\ML_{\phi}f\right\|_{\dot{H}^{-1}}\lesssim \left\|\Delta f\right\|_{\dot{H}^{-1}}+\left\|W_k^2f\right\|_{\dot{H}^{-1}}\lesssim \left\|f\right\|_{\dot{H}^1}.
\end{equation*}
Therefore, we derive estimates
\begin{equation*}
    \begin{aligned}
        \left\|\ML_{\phi}\left(\frac{b_k|x-\boldsymbol{y}_k|}{\lambda_kt^2}(\Lambda_kW_k\chi'_{t})\right)\right\|_{\dot{H}^{-1}}\lesssim \frac{b_k}{\lambda_k t}\left(\int_t^{2t}\left(\frac{\lambda_k}{r^3}\right)^2r^3dr\right)^{\frac{1}{2}}\lesssim \frac{b}{t^2},
    \end{aligned}
\end{equation*}
and
\begin{equation*}
    \left\|\ML_{\phi}\left( \frac{v_{k,l}y'_{k,m}(\partial_l\partial_mW)(\frac{\cdot-\boldsymbol{y}_j}{\lambda_k})}{\lambda_k^3}   \right)\right\|_{\dot{H}^{-1}}\lesssim  \left\|\frac{v_{k,l}y'_{k,m}(\partial_l\partial_mW)(\frac{\cdot-\boldsymbol{y}_k}{\lambda_k})}{\lambda_j^3}   \right\|_{\dot{H}^{1}}\lesssim \frac{|\boldsymbol{v}\boldsymbol{y}'|}{\lambda^2}.
\end{equation*}
Thus, both terms are negligible due to the bootstrap assumption. For the remaining term, we again work in the exterior region $\mathbb R^4\setminus \widetilde B$ and compute
\begin{equation*}
\begin{aligned}
      \left\|f(W_k)\right\|_{\dot{H}^1({|x-\boldsymbol{y}_k|}\geq d)}\lesssim& \lambda_k^{3}\left(\int_{d}^{\infty}r^{-14}r^3dr\right)^{1/2}\lesssim \lambda^3,\\
    \left\|\lambda_k^{-1}(\underline{\Lambda}_k\Lambda_kW_k)\chi_{t}\right\|_{\dot{H}^1(d\leq|x-\boldsymbol{y}_k|\leq 2t)}\lesssim& \lambda_k^2\left(\int_{d}^{2t}r^{-10}r^3dr\right)^{1/2}\lesssim \lambda^2, \\
    \left\|\frac{|x-\boldsymbol{y}_k|}{\lambda_kt^2}(\Lambda_kW_k)\chi_t'\right\|_{\dot{H}^1({|x-\boldsymbol{y}_k|}\geq d)}\lesssim& \frac{1}{t}\left(\int_t^{2t}\left(\frac{1}{r^3}\right)^2r^3dr\right)^{\frac{1}{2}}\lesssim \frac{1}{t},\\
      \left\|\lambda_k^{-1}\Lambda_k W_k\right\|_{\dot{H}^1(d\leq|x-\boldsymbol{y}_k|\leq 2t)}\lesssim& \left(\int_{d}^{2t}r^{-6}r^3dr\right)^{1/2}\lesssim 1/d,\\      \left\|\lambda_k^{-2}\partial_{l,k}\Lambda_k W_k\right\|_{\dot{H}^1(d\leq|x-\boldsymbol{y}_k|)}\lesssim&\left(\int_{d}^{2t}r^{-8}r^3dr\right)^{1/2}\lesssim 1/d,\\  
       \left\|\lambda_k^{-2}\Lambda_k \partial_{l,k} W_k\right\|_{\dot{H}^1(d\leq|x-\boldsymbol{y}_k|)}\lesssim&\left(\int_{d}^{2t}r^{-8}r^3dr\right)^{1/2}\lesssim 1/d.
\end{aligned}
\end{equation*}
Combining these with the bootstrap assumption yields all the terms not involving $\PG$ are negligible in the exterior part. Then we turn to the terms containing $\PG$. Similar to the proof of Lemma \ref{lemma of partial t dot g 2}, it suffices to consider on $d\leq|x-\boldsymbol{y}_j|\leq 2d$ and using corollary \ref{on the ring region} we obtain
\begin{equation*}
    \begin{aligned}
        &\left\|\Delta\PG+f(W_{\boldsymbol{\Gamma}}+\PG)\right\|_{\dot{H}^1(d\leq |x-\boldsymbol{y}_k|\leq 2d)}\\
        \lesssim& \frac{b_k^2}{\lambda_k}+\frac{\lambda_k B_k(\boldsymbol{\lambda},\boldsymbol{y})}{\lelan \Lambda W \chi_t(\lambda_k \cdot),\Lambda W \rilan}\frac{\log t}{\lambda_k}+\frac{|\lambda_k'+b_k||\boldsymbol{v}_k|}{\lambda_k}\lesssim (t^{-\frac{2}{3}}\log t)\lambda.
    \end{aligned}
\end{equation*}
Thus, all terms are negligible in $\mathbb R^4\setminus \widetilde B$, and it remains to consider the interior region. Since $P_\Gamma$ and $\TPG$ are constructed in the same way, the same cancellation occurs in the region $|x-\boldsymbol{y}_k|\le d$. Consequently, it suffices for us to estimate
\begin{equation*}
    \left\|\ML_{\phi}\left(f(\WG+{P}_{\Gamma})-3W_k^2\PG-\sum_{j=1}^Nf(W_j)+{4B_k(\boldsymbol{\lambda},\boldsymbol{y})}f'(W_k)-\frac{b_k (\lambda_k'+b_k)}{2\log(t/\lambda_k)\lambda_k^2}\Lambda
           _k W_k\right)\right\|_{\dot{H}^{-1}}.
\end{equation*}
Applying (\ref{bootstrap a for lambda}), (\ref{lambda'+b}), the $\Lambda_k W_k$ term is dominated by 
\begin{equation*}
 \frac{|b_k| |\lambda_k'+b_k|}{2\log(t/\lambda_k)\lambda_k^2} \left\|\Lambda_kW_k\right\|_{\dot{H}^1}\lesssim \frac{t^{\frac{1}{6}}\lambda}{t^{\frac{2}{3}}}\lesssim t^{-\frac{1}{2}}\lambda.
\end{equation*}
It remains to handle the multi-bubble interaction. As in Lemma \ref{lemma of partial t dot g 2}, we split this contribution into two terms, denoted by \Rmnum{1} and \Rmnum{2}. The only difference lies in the required norm. For \Rmnum{1}, the dominating terms are $\left\|W_k\PG^2\right\|_{\dot{H}^1}$ and $\left\|W_kW_j^2\right\|_{\dot{H}^1}$.Repeating the argument from Lemma \ref{lemma of partial t dot g 2}, we obtain
\begin{equation*}
    \left\|W_k\PG^2\right\|_{\dot{H}^1}\lesssim \left\|W_k\right\|_{\dot{H}^1}\left\|\PG\right\|_{\dot{H}^2}^2\ll \lambda^{\frac{3}{2}},
\end{equation*}
and 
\begin{equation*}
 \left\|W_kW_j^2\right\|_{\dot{H}^1}\lesssim  \left(\int_{\lambda_k}^d\left(\frac{\lambda_k}{r^3}\frac{\lambda_j^2}{|\boldsymbol{y}_k-\boldsymbol{y}_j|^4}\right)^2r^3dr\right)^{\frac{1}{2}}\lesssim  \lambda^2.
\end{equation*}
Thus $\left\|\Rmnum{1}\right\|_{\dot{H}^1}$ can be neglected. While for II, it holds 
\begin{equation*}
    \begin{aligned}
        \left\|\Rmnum{2}\right\|_{\dot{H}^1}\lesssim&\sum_{j\neq k}\left\|W_k^2\left(W_j-\frac{8\lambda_j}{|\boldsymbol{y}_k-\boldsymbol{y}_j|^2}\right)\right\|_{\dot{H}^1(|x-\boldsymbol{y}_k|\leq d|)}\\
        \lesssim&\lambda_j\left(\int_0^d\left(\frac{\lambda_k}{\lambda_k^2+r^2}\right)^4r^3dr\right)^{\frac{1}{2}}\lesssim \lambda,
    \end{aligned}
\end{equation*}
which gives the leading contribution. Collecting the above bounds completes the proof of the second order error estimate.

\end{proof}

\section{Energy estimates}
This section is devoted to the first and second order energy estimates for the remainders introduced in Section \ref{Construction of the refined approximate solution} . The second order estimate is the main ingredient of the argument, as it captures the cancellations associated with the linearized operator and allows us to handle the rough modulation terms.

We start with the first order energy estimate. Unlike in the five-dimensional setting, the first order estimate does not require an additional localized virial correction in the present argument. The second order control, together with the error bounds from Section~5, is sufficient to recover the required energy lower bound.
\subsection{First order energy estimate}
In this subsection, we establish the first order energy estimate for the decomposition associated with $(\tilde \phi,\tilde h,\dot h)$. By the definitions of $\tilde h$ and $\dot h$, together with the estimates on $\TPG$ obtained in Section \ref{Construction of the refined approximate solution}, we have
\begin{equation*}
    \left\|g-\tilde{h}\right\|_{\dot{H}^1}+\left\|\dot{g}-\dot{h}\right\|_{L^2}\lesssim \left\|\TPG\right\|_{\dot{H}^1}+\left\| \sum_{j=1}^N \frac{(\beta_j-b_j)}{\lambda_j}(\Lambda_j W_j)\chi_{t}+\frac{(\boldsymbol{s}_j-\boldsymbol{v}_j)\cdot\nabla_jW_j}{\lambda_j}\right\|_{L^2}.
\end{equation*}
From (\ref{tpg h dot 1}), the first term is bounded by $(t^{-\frac{2}{3}}\log t)\lambda(t)$. Moreover, by (\ref{b and beta}) and (\ref{v and s}), the second term is dominated by
\begin{equation*}
    \sum_{j=1}^N|\beta_j-b_j|\sqrt{\log(t/\lambda_j)}+\sum_{j=1}^N|\boldsymbol{s}_j-\boldsymbol{v}_j|\lesssim (t^{-\frac{1}{2}}\log t)\lambda(t) .
\end{equation*}
Consequently, the pair $(\tilde{h},\dot{h})$ satisfies
\begin{equation}\label{1 energy norm of th dot h}
    \left\|\tilde{h}\right\|_{\dot{H}^1}+\left\|\dot{h}\right\|_{L^2}\lesssim \left\|g\right\|_{\dot{H}^1}+\left\|\dot{g}\right\|_{L^2}+\left\|g-\tilde{h}\right\|_{\dot{H}^1}+\left\|\dot{g}-\dot{h}\right\|_{L^2}\lesssim t^{-\frac 16}\lambda(t).
\end{equation}
We define the first order energy functional by
\begin{equation*}
   \mathcal{I}_1(t):=\int_{\mathbb R^4}\left\{\frac{1}{2}\dot{h}^2+\frac{1}{2}|\nabla \tilde{h}|^2-(F(\tilde{\phi}+\tilde{h})-F(\tilde{\phi})-f(\tilde{\phi})\tilde{h})\right\}dx
\end{equation*}

\begin{lemma}\label{first order energy est lemma}
On $[T_{*},T]$ where the bootstrap assumption holds, one obtains
\begin{equation}\label{d dt t H}
   \mathcal{I}_1'(t)\gtrsim - t^{-\frac{3}{4}}\lambda(t)^2.
\end{equation}
\end{lemma}
\begin{proof}
    In this proof, the symbol "$\simeq$" means equality up to error terms of order $O(t^{-\frac{4}{5}}\lambda(t)^2)$, which will be called negligible. Taking time derivative of ${\mathcal{I}}_1$ we obtain
    \begin{equation*}
       \frac{d}{dt}\tilde{\mathcal{I}}=\lelan \partial_t \dot{h},\dot{h}  \rilan-\lelan \partial_t\tilde{h},\Delta \tilde{h}+f(\tilde{\phi}+\tilde{h})-f(\tilde{\phi}) \rilan-\lelan \partial_t \tilde{\phi}, f(\tilde{\phi}+\tilde{h})-f(\tilde{\phi})-f'(\tilde{\phi})\tilde{h}\rilan. 
    \end{equation*}
From (\ref{partial t g dot sim 2}) and (\ref{1 energy norm of th dot h}), it holds
\begin{equation*}
  \left|\lelan  \Tilde{\epsilon}(\tilde{h},\dot{h}),    \dot{h}\rilan\right|\lesssim \left\|\Tilde{\epsilon}(\tilde{h},\dot{h})\right\|_{L^2}\left\|\dot{h}\right\|_{L^2}\lesssim {(t^{-\frac{2}{3}}}\log t)\lambda \cdot t^{-\frac 16}\lambda\lesssim t^{-\frac{5}{6}}\lambda^2.
\end{equation*}
Thus,
\begin{equation*}
    \lelan \partial_t \dot{h},\dot{h}  \rilan\simeq \lelan \dot{h},  \Delta \Tilde{h}+f(\Tilde{\phi}+\Tilde{h})-f(\Tilde{\phi})\rilan+\lelan  \dot{h}, \Psi(\boldsymbol{\lambda},\boldsymbol{\beta},{\boldsymbol{y}},{\boldsymbol{s}}) \rilan.
\end{equation*}
We first estimate the term involving $\Psi(\boldsymbol{\lambda},\boldsymbol{\beta},{\boldsymbol{y}},{\boldsymbol{s}})$. Recall the definition of $\Psi(\boldsymbol{\lambda},\boldsymbol{\beta},{\boldsymbol{y}},{\boldsymbol{s}})$ in (\ref{first order PSI}). By the scaling property, $\left\|\Psi(\boldsymbol{\lambda},\boldsymbol{\beta},{\boldsymbol{y}},{\boldsymbol{s}})\right\|_{L^2}$  is controlled by
\begin{equation*}
    \begin{aligned}
     &\sum_{j=1}^N \frac{|\lambda_j'+\beta_j|\beta_j}{\lambda_j}\left\|\underline{\Lambda}\Lambda W\right\|_{L^2}+\sum_{j=1}^N\frac{|\boldsymbol{y}_j'+\boldsymbol{s}_j|\beta_j}{\lambda_j}\left\|\nabla \Lambda W\right\|_{L^2}\\
     +&\sum_{j=1}^N\left|\beta_j'-\frac{\beta_j \lambda_j'}{2\log(t/\lambda_j)\lambda_j} -\frac{B_j(\boldsymbol{\lambda})}{\log(t/\lambda_j)}\right|\left\|\Lambda W \chi\left(\frac{\lambda_k\cdot}{t}\right)\right\|_{L^2}.
    \end{aligned}
\end{equation*}
For the first term, from (\ref{lambda'+b}) and (\ref{b and beta}), it holds
\begin{equation*}
  \sum_{j=1}^N \frac{|\lambda_j'+\beta_j|\beta_j}{\lambda_j}\left\|\underline{\Lambda}\Lambda W\right\|_{L^2}\lesssim \sum_{j=1}^N t^{-\frac{1}{3}}(|\lambda_j'+b_j|+|\beta_j-b_j|)\lesssim (t^{-\frac{7}{6}}\log t)\lambda.  
\end{equation*}
Similarly, using (\ref{y'+v}) and (\ref{v and s}) the second term is bounded by
\begin{equation*}
    \sum_{j=1}^N\frac{|\boldsymbol{y}_j'+\boldsymbol{s}_j|\beta_j}{\lambda_j}\left\|\nabla \Lambda W\right\|_{L^2}\lesssim \sum_{j=1}^Nt^{-\frac{1}{3}}(|\boldsymbol{y}_j'+\boldsymbol{v}_j|+|\boldsymbol{v}_j-\boldsymbol{s}_j|)\lesssim t^{\frac{2}{3}}\lambda^2.
\end{equation*}
For the last term, a direct computation gives $\left\|\Lambda W \chi\left(\frac{\lambda_k\cdot}{t}\right)\right\|_{L^2}\lesssim \sqrt{\log(t/\lambda_k)}\lesssim t^{\frac{1}{3}} $. While (\ref{beta'+=to}) and (\ref{rough est of beta'}) imply
\begin{equation*}
\begin{aligned}
  \sum_{j=1}^N\left|\beta_j'-\frac{\beta_j \lambda_j'}{2\log(t/\lambda_j)\lambda_j} -\frac{B_j(\boldsymbol{\lambda})}{\log(t/\lambda_j)}\right|\lesssim& \sum_{j\neq k}\frac{8|\beta'_j(t)|}{\log(t/\lambda_j)}\log t+( t^{-\frac{7}{6}}\sqrt{\log t})\lambda  \\
  \lesssim& \frac{(t^{-\frac{2}{3}}\log t)\lambda}{t^{\frac{2}{3}}}+( t^{-\frac{7}{6}}\sqrt{\log t})\lambda\lesssim ( t^{-\frac{7}{6}}\sqrt{\log t})\lambda.
\end{aligned}  
\end{equation*}
Thus, we obtain 
\begin{equation*}
  \sum_{j=1}^N\left|\beta_j'-\frac{\beta_j \lambda_j'}{2\log(t/\lambda_j)\lambda_j} -\frac{B_j(\boldsymbol{\lambda})}{\log(t/\lambda_j)}\right|\left\|\Lambda W \chi\left(\frac{\lambda_k\cdot}{t}\right)\right\|_{L^2}\lesssim  ( t^{-\frac{5}{6}}\sqrt{\log t})\lambda.
\end{equation*}
Collecting the above estimates for $\|\Psi(\boldsymbol{\lambda},\boldsymbol{\beta},{\boldsymbol{y}},{\boldsymbol{s}})\|_{L^2}$ and applying Cauchy's inequality, we obtain
\begin{equation*}
\begin{aligned}
 \left|\lelan  \dot{h}, \Psi(\boldsymbol{\lambda},\boldsymbol{\beta},{\boldsymbol{y}},{\boldsymbol{s}}) \rilan\right| \lesssim \left\|\dot{h}\right\|_{L^2}\left\|\Psi(\boldsymbol{\lambda},\boldsymbol{\beta},{\boldsymbol{y}},{\boldsymbol{s}}) \right\|_{L^2}\lesssim  (t^{-1}\sqrt{\log t})\lambda^2.    
\end{aligned}
\end{equation*}
It follows that
\begin{equation*}
    \lelan \partial_t \dot{h},\dot{h}  \rilan\simeq \lelan \dot{h},  \Delta \Tilde{h}+f(\Tilde{\phi}+\Tilde{h})-f(\Tilde{\phi})\rilan.
\end{equation*}

 Next, we claim 
\begin{equation}\label{canclelation of partial t h and dot h}
    \lelan \partial_t\tilde{h},\Delta \tilde{h}+f(\tilde{\phi}+\tilde{h})-f(\tilde{\phi}) \rilan\simeq \lelan \dot{h},  \Delta \Tilde{h}+f(\Tilde{\phi}+\Tilde{h})-f(\Tilde{\phi})\rilan.
\end{equation}
From the definition of $\dot{h}$, we have
\begin{equation*}
     \partial_t \tilde{h}-\dot{h}=-\partial_t\TPG+\sum_{j=1}^N \lambda_j^{-1}(\lambda_j'\Lambda_jW_j+\beta_j(\Lambda_j W_j)\chi_{t})+\sum_{j=1}^N \lambda_j^{-1}(\boldsymbol{y}_j'+\boldsymbol{s}_j)\cdot\nabla_j W_j.
\end{equation*}
On the one hand, using the estimate (\ref{partial t tilde p H 1}) together with (\ref{lambda'+b}), (\ref{b and beta}) and (\ref{y'+v}), the difference $\left\|\partial_t \tilde{h}-\dot h\right\|_{\dot{H}^1}$ satisfies the rough bound
\begin{equation}\label{rouogh est of partial t h-dot h}
    \begin{aligned}
&\left\|\partial_t\TPG\right\|_{\dot{H}^1}+\sum_{k=1}^N\frac{|\lambda_k'+\beta_k|}{\lambda_k}\left\|\Lambda W\right\|_{\dot{H}^1}+\frac{|\beta_k|}{\lambda_k}\left\|\Lambda_k W_k(1-\chi(t))\right\|_{\dot{H}^1}+\frac{|\boldsymbol{y}_k'+\boldsymbol{s}_k|}{\lambda_k}\left\|\nabla W\right\|_{\dot{H}^1}\\
\lesssim&(t^{-1}\log t)\lambda+\frac{(t^{-\frac{5}{6}}\sqrt{\log t})\lambda}{\lambda}+\frac{t^{-\frac{1}{3}}\lambda}{t}+\frac{(t\sqrt{\log t})\lambda^2}{\lambda}\lesssim t^{-\frac{5}{6}}\sqrt{\log t}.
    \end{aligned}
\end{equation}
On the other hand,
\begin{equation*}
    \Delta \tilde{h}+f(\tilde{\phi}+\tilde{h})-f(\tilde{\phi})=\Delta  \tilde{h}+3\tilde{\phi}^2\tilde{h}+3\tilde{\phi}\tilde{h}^2+\tilde{h}^3.
\end{equation*}
By Cauchy's inequality and (\ref{bootstap a for energy}), the nonlinear part is dominated by 
\begin{equation*}
    \left\|3\tilde{\phi}\tilde{h}^2+\tilde{h}^3\right\|_{L^{\frac{4}{3}}}\lesssim \left\|\tilde{h}\right\|_{L^4}^2+\left\|\tilde{h}\right\|_{L^4}^3\lesssim\left\|\tilde{h}\right\|_{\dot{H}^1}^2+\left\|\tilde{h}\right\|_{\dot{H}^1}^3 \lesssim t^{-\frac{1}{3}}\lambda^2.
\end{equation*}
Combining this with the rough estimate (\ref{rouogh est of partial t h-dot h}),   it suffices to prove
\begin{equation*}
\left|\lelan  \partial_t \tilde{h}-\dot{h},\Delta\tilde{h}+3\tilde{\phi}^2\tilde{h} \rilan\right|\lesssim t^{-\frac{4}{5}}\lambda^2.
\end{equation*}
The smallness of the $\partial_t\TPG$ term has already been obtained in the proof of (\ref{rouogh est of partial t h-dot h}). Indeed, due to (\ref{tpg h dot 1}), all the terms involving $\TPG$ are controlled by 
\begin{equation*}
    \left|\lelan\partial_t \tilde{h}-\dot{h}, W_k\TPG h\rilan\right|\lesssim \left\|\partial_t \tilde{h}-\dot{h}\right\|_{\dot{H}^1}\left\|W_k\right\|_{\dot{H}^1}\left\|\TPG\right\|_{\dot{H}^1}\left\|\tilde{h}\right\|_{\dot{H}^1}\lesssim t^{-\frac{5}{3}}\lambda^2,
\end{equation*}
and are therefore negligible. For the remaining part, since $\ML_k\Lambda_kW_k=\ML_k \nabla_k W_k=0$, we deduce
\begin{equation*}
    \begin{aligned}
        &\left|\lelan \Lambda_k W_k, \Delta \tilde{h}+3\left(\sum_{j=1}^NW_j\right)^2\tilde{h} \rilan\right|\lesssim \sum_{j\neq k}\left|\lelan \Lambda_k W_k, W_jW_k \tilde{h} \rilan\right|\\\lesssim
        &\sum_{j\neq k}\left(\int_{\lambda_k}^{\infty}\left(\frac{\lambda_j\lambda_k^2}{r^4|\boldsymbol{y}_k-\boldsymbol{y}_j|^2}\right)^{\frac{4}{3}}r^3dr\right)^{\frac{3}{4}}\left\|\tilde{h}\right\|_{\dot{H}^1}\lesssim \sum_{j\neq k}\left\|\tilde{h}\right\|_{\dot{H}^1}\frac{\lambda_j\lambda_k}{|\boldsymbol{y}_k-\boldsymbol{y}_j|^2},
    \end{aligned}
\end{equation*}
and 
\begin{equation*}
\begin{aligned}
    &\left|\lelan \nabla_k W_k, \Delta \tilde{h}+3\left(\sum_{j=1}^NW_j\right)^2\tilde{h} \rilan\right|\lesssim \sum_{j\neq k}\left|\lelan \nabla_k W_k, W_jW_k \tilde{h} \rilan\right|\\
   \lesssim&\sum_{j\neq k}\left(\int_{\lambda_k}^{\infty}\left(\frac{\lambda_j\lambda_k^3}{r^5|\boldsymbol{y}_k-\boldsymbol{y}_j|^2}\right)^{\frac{4}{3}}r^3dr\right)^{\frac{3}{4}}\left\|\tilde{h}\right\|_{\dot{H}^1} \lesssim \sum_{j\neq k}\left\|\tilde{h}\right\|_{\dot{H}^1}\frac{\lambda_j\lambda_k}{|\boldsymbol{y}_k-\boldsymbol{y}_j|^2}.  
\end{aligned}
\end{equation*}
Therefore, the sums involving $\Lambda_j W_j$ and $\nabla_j W_j$ are also negligible. Finally, we consider the truncation term. The definition  of $\chi_t$ yields
\begin{equation*}
  {\left|\lelan \frac{\beta_k}{\lambda_k}\Lambda_kW_k\chi_t,\ML_k\tilde{h} \rilan \right|\lesssim \beta_k\int_{t}^{2t}\left(\left(\frac{1}{r^3}\frac{1}{t}+\frac{1}{r^2}\frac{1}{t^2}\right)^{\frac{4}{3}}r^3dr\right)^{\frac{3}{4}}\left\|\tilde{h}\right\|_{\dot{H}^1}\lesssim \frac{\beta_k}{t}\left\|\tilde{h}\right\|_{\dot{H}^1}.} 
\end{equation*}
By the estimate for $\beta$ in (\ref{b and beta}), this term is bounded by 
$t^{-\frac{3}{2}}\lambda^2$. Collecting the above estimates proves claim (\ref{canclelation of partial t h and dot h}).

It remains to estimate the nonlinear term $\lelan \partial_t \tilde{\phi}, f(\tilde{\phi}+\tilde{h})-f(\tilde{\phi})-f'(\tilde{\phi})\tilde{h}\rilan$. Using the identity $\partial_t\phi=\partial_tW_{\Gamma}+\TPG$ and (\ref{partial t tilde p H 1}), we obtain
\begin{equation*}
   \left\|\partial_t \tilde{\phi}\right\|_{\dot{H}^1}\lesssim \left\|\partial_t \TPG\right\|_{\dot{H}^1}+ \sum_{j=1}^N\frac{|\lambda_j'|}{\lambda_j}\left\|\Lambda_k W_k\right\|_{\dot{H}^1}+\sum_{j=1}^N\frac{|\boldsymbol{y}_j'|}{\lambda_j}\left\|\nabla_k W_k\right\|_{\dot{H}^1}\lesssim t^{-\frac{1}{3}}.
\end{equation*}
Furthermore, the nonlinear term is given by
\begin{equation*}
  f(\tilde{\phi}+\tilde{h})-f(\tilde{\phi})-f'(\tilde{\phi})\tilde{h}=3\tilde{\phi}\tilde{h}^2+\tilde{h}^3.  
\end{equation*}
By interpolation inequality and Sobolev embedding, we obtain 
\begin{equation*}
    \begin{aligned}
        &\left|\lelan \partial_t \tilde{\phi}, f(\tilde{\phi}+\tilde{h})-f(\tilde{\phi})-f'(\tilde{\phi})\tilde{h}\rilan\right|\lesssim\left\|\partial_t \tilde{\phi}\right\|_{\dot{H}^1}\left(\left\|\tilde{\phi}\right\|_{L^{\frac{8}{3}}}\left\|\tilde{h}\right\|_{\dot{H}^1}\left\|\tilde{h}\right\|_{\dot{H}^{\frac{3}{2}}}+\left\|\tilde{h}\right\|_{\dot{H}^1}^3\right)\\
        \lesssim& t^{-\frac{1}{3}}\left(\int_{\lambda_k}^{\infty}\left(\frac{\lambda_k}{r^2}\right)^{\frac{8}{3}}r^3dr\right)^{\frac{3}{8}}\left\|\tilde{h}\right\|_{\dot{H}^1}^{\frac{3}{2}}\left\|\tilde{h}\right\|_{\dot{H}^2}^{\frac{1}{2}}\lesssim t^{-\frac{1}{3}}\lambda^{\frac{5}{2}}.
    \end{aligned}
\end{equation*}
This completes the proof of (\ref{d dt t H}).
\end{proof}

\subsection{Construction of the truncated virial operators}
We now introduce the localized virial operators needed for the second order energy estimate. These operators are designed to extract coercive control from the commutator structure of the linearized operator.

In particular, the borderline behavior of the scaling direction and the cutoff terms in the refined approximate solution make the choice of the weight more
delicate.

\begin{lemma}\label{function q}For any $\epsilon>0$ and $R>0$, there exists a radially symmetric function $q=q_{\epsilon,R}\in \mathcal{C}^{5,1}$ with the following properties:

(1) $q(r)=\frac{1}{2}r^2$ for $r\leq R$.

(2) There exists $\Tilde{R}=\Tilde{R}(R,\epsilon)>R$ and $\Tilde{C}_1$, $\Tilde{C}_2$, $\Tilde{C}_3$, $\Tilde{C}_4\in \mathbb R$ (dependent on $\epsilon$ and $R$) such that for any $r\geq \Tilde{R}$, 
\begin{equation*}
    q(r)=\Tilde{C}_1+\Tilde{C}_2\log r+\Tilde{C}_3r^{-2}+\Tilde{C_4}r^{-1}.
\end{equation*}

(3) $\frac{q'(r)}{r}\geq -\epsilon$, $q''(r)-\frac{q'(r)}{r}\geq -\epsilon$.

(4) $\left|\left(\partial_{rr}+\frac{3\partial_r}{r}\right)^3 q(r)\right|\leq\frac{\epsilon}{r^4(1+|\log r|)^2}$.

(5) $\left|\left(\partial_{rr}+\frac{3\partial_r}{r}\right)^2 q(r)\right|\leq \frac{\epsilon}{r^2(1+|\log r|)}$.

(6) $\left|\left(\left(\partial_{rr}+\frac{3\partial_r}{r}\right) q(r)\right)'\right|\leq \frac{\epsilon}{(1+|\log r|)r}$.

(7) $| q'' (r)|\leq 4$, $|q'(r)|\leq 4r$ .

\end{lemma}
\begin{proof}
Let $\alpha_1>0$ be small constant to be chosen later. Set $X:=\log R$ and consider the auxiliary function
\begin{equation*}
    f(x):=\frac{1}{X}\left(\frac{1}{2X+1}-\frac{1}{x+1}\right)^{-1}\left(\log(1+x)-\log(1+2X)\right)-\frac{3X+2}{2X}.
\end{equation*}
For $X$ sufficiently large, one has $f(4X+1)<10$. Since $\lim_{x\to\infty} f(x)=+\infty$, there exists $\tilde{X} >4X+1$ such that $f(\Tilde{X})=\frac{1}{2\alpha_1}$. We note that $\log\left(\frac{1+\Tilde{X}}{1+2X}\right)$ is comparable to $\frac{1}{\alpha_1}$. In particular, $\Tilde{X}\gg X$. We let $\alpha_2:=\frac{\alpha_1}{X}\left(\frac{1}{2X+1}-\frac{1}{\Tilde{X}+1}\right)^{-1}$, then $\alpha_1$, $\alpha_2$ and $\Tilde{X}$ satisfy
\begin{equation*}
    \frac{\alpha_1}{X}=\alpha_2\left(\frac{1}{2X+1}-\frac{1}{\Tilde{X}+1}\right)
\end{equation*}
and 
\begin{equation*}
    \alpha_2(\log(1+\Tilde{X})-\log(1+2X))=\frac{1}{2X}\alpha_1(3X+2)+\frac{1}{2}.
\end{equation*}
Note that $\alpha_2=2(1+o_{X\to \infty}(1))\alpha_1$. We then define the function
\begin{equation*}
    F(x):=\begin{cases}
        -\alpha_1/X^2,& x\in(X,2X)\\
        \alpha_2/(1+x)^2, & x\in (2X,\Tilde{X})\\
        0,&\text{elsewhere}.
    \end{cases}
\end{equation*}
Then we have
\begin{equation}\label{int F(x)=0}
    \int_{-\infty}^{\infty}F(x)dx=-\frac{\alpha_1}{X}+\alpha_2\left(\frac{1}{2X+1}-\frac{1}{\Tilde{X}+1}\right)=0,
\end{equation}
and
\begin{equation*}\label{int xF(x)=1/2}
\begin{aligned}
     \int_{-\infty}^{\infty}(1+x)F(x)dx= &-\frac{\alpha_1}{2X^2}\left((1+2X)^2-(1+X)^2\right)\\
     &+\alpha_2\left(\log(1+\Tilde{X})-\log(1+2X)\right)=\frac{1}{2}.
\end{aligned}
\end{equation*}
Now we define 
\begin{equation*}\begin{aligned}
    c_1(x)&=-\int_{-\infty}^x(6z+3)e^{2z}F(z)dz,&c_2(x)=\int_{-\infty}^x6e^{2z}F(z)dz,\\
c_3(x)&=\frac{1}{2}-\int_{-\infty}^x\left(\frac{19}{12}+z\right)F(z)dz,&c_4(x)=\int_{-\infty}^xF(z)dz,\\
 c_5(x)&=\int_{-\infty}^x-\frac{3}{4}e^{4z}F(z)dz,&c_6(x)=\int_{-\infty}^x\frac{16}{3}e^{3z}F(z)dz.
    \end{aligned}
\end{equation*}
These functions are smooth away from the points $X$, $2X$ and $\tilde{X}$, and are globally Lipschitz continuous. For $x\leq X$, we have
\begin{equation*}
  (c_1(x),c_2(x),c_3(x),c_4(x),c_5(x),c_6(x))=(0,0,1/2,0,0,0).  
\end{equation*}
When $x\geq X$, integration by parts shows that
\begin{equation*}
\begin{aligned}
   &|c_1(x)|\lesssim \alpha_1 e^{2x}(1+x)^{-1}\text{ , }|c_2(x)|\lesssim \alpha_1 e^{2x}(1+x)^{-2}, \\
   &|c_5(x)|\lesssim \alpha_1 e^{4x}(1+x)^{-2}\text{ , }|c_6(x)|\lesssim \alpha_1 e^{3x}(1+x)^{-2}.
\end{aligned} 
\end{equation*}
Furthermore, from the definition of $F(x)$, the function $c_3(x)$ is monotone increasing on $x\in [X,2X]$ and reaching its maximum at $2X$. A direct computation gives $c_3(2X)=\frac{1}{2}+\frac{\alpha_1(18X+19)}{12X}$; and decreasing till $c_3(\title{X})=0$. Finally for $c_4(x)$, when $X\leq x\leq \Tilde{X}$, the same argument gives $|c_4(x)|\lesssim \frac{\alpha_1}{x}$. When $x\geq \Tilde{X}$, (\ref{int F(x)=0})  indicates $c_4(x)=0$.
Define
\begin{equation*}
    p(x):=c_1(x)+c_2(x)x+c_3(x)e^{2x}+c_4(x)xe^{2x}+c_5(x)e^{-2x}+c_6(x)e^{-x}.
\end{equation*}
Then $p$ is Lipschitz continuous and 
\begin{equation*}
    \begin{aligned}
        p'(x)=&c_1'(x)+c_2'(x)x+c_3(x)e^{2x}+c_4(x)'xe^{2x}+c_5'(x)e^{-2x}+c_6'(x)e^{-x}\\
       &+c_2(x)+2c_3(x)e^{2x}+c_4(x)(1+2x)e^{2x}-2c_5(x)e^{-2x}-c_6(x)e^{-x}.
    \end{aligned}
\end{equation*}
The terms containing the derivatives of $c_j$ cancel by construction: 
\begin{equation*}
\begin{aligned}
  &c_1'(x)+c_2'(x)x+c_3(x)e^{2x}+c_4(x)'xe^{2x}+c_5'(x)e^{-2x}+c_6'(x)e^{-x}\\
 =&e^{2x}\left(-3-\frac{19}{12}-\frac{3}{4}+\frac{16}{3}\right)+xe^{2x}(-6+6-1+1)=0.
\end{aligned}
\end{equation*}
Therefore
\begin{equation*}
    p'(x)=c_2(x)+2c_3(x)e^{2x}+c_4(x)(1+2x)e^{2x}-2c_5(x)e^{-2x}-c_6(x)e^{-x},
\end{equation*}
which is Lipschitz. Repeating this procedure, we obtain
\begin{equation*}
    \begin{aligned}
        p''(x)=&c_2'(x)+2c_3'(x)e^{2x}+c_4'(x)(1+2x)e^{2x}-2c_5'(x)e^{-2x}-c_6'(x)e^{-x}\\
        &+4c_3(x)e^{2x}+4c_4(x)e^{2x}+4c_4(x)xe^{2x}+4c_5(x)e^{-2x}+c_6(x)e^{-x}.
    \end{aligned}
\end{equation*}
Since 
\begin{equation*}
    \begin{aligned}
     &c_2'(x)+2c_3'(x)e^{2x}+c_4'(x)(1+2x)e^{2x}-2c_5'(x)e^{-2x}-c_6'(x)e^{-x}\\
     =&e^{2x}\left(6-\frac{19}{6}+1+\frac{3}{2}-\frac{16}{3}\right)+xe^{2x}\left(-2+2\right)=0,
    \end{aligned}
\end{equation*}
 we have
\begin{equation*}
    p''(x)=4c_3(x)e^{2x}+4c_4(x)e^{2x}+4c_4(x)xe^{2x}+4c_5(x)e^{-2x}+c_6(x)e^{-x},
\end{equation*}
again Lipschitz. Next, for the third order derivative, 
\begin{equation*}
\begin{aligned}
     p^{(3)}(x)=&4c_3'(x)e^{2x}+4c_4'(x)e^{2x}+4c_4'(x)xe^{2x}+4c_5'(x)e^{-2x}+c_6'(x)e^{-x}\\
     &+8c_3(x)e^{2x}+12c_4(x)e^{2x}+8c_4(x)xe^{2x}-8c_5(x)e^{-2x}-c_6(x)e^{-x}.
\end{aligned}
\end{equation*}
The same cancellation yields
\begin{equation*}
   \begin{aligned}
    &4c_3'(x)e^{2x}+4c_4'(x)e^{2x}+4c_4'(x)xe^{2x}+4c_5'(x)e^{-2x}+c_6'(x)e^{-x}\\
    =&e^{2x}\left(-\frac{19}{3}+4-3+\frac{16}{3}\right)+xe^{2x}\left(-4+4\right)=0.
   \end{aligned} 
\end{equation*}
It follows that $p^{(3)}$ is Lipschitz continuous and is given by
\begin{equation*}
    p^{(3)}(x)=8c_3(x)e^{2x}+12c_4(x)e^{2x}+8c_4(x)xe^{2x}-8c_5(x)e^{-2x}-c_6(x)e^{-x}.
\end{equation*}
Furthermore, the fourth order derivative is given by
\begin{equation*}
    \begin{aligned}
        p^{(4)}(x)=&8c_3'(x)e^{2x}+12c_4'(x)e^{2x}+8c_4'(x)xe^{2x}-8c_5'(x)e^{-2x}-c_6'(x)e^{-x}\\
        &+16c_3(x)e^{2x}+32c_4(x)e^{2x}+16c_4(x)e^{2x}+16c_5(x)e^{-2x}+c_6(x)e^{-x}.
    \end{aligned}
\end{equation*}
Similarly as before, we have
\begin{equation*}
    \begin{aligned}
     &8c_3'(x)e^{2x}+12c_4'(x)e^{2x}+8c_4'(x)xe^{2x}-8c_5'(x)e^{-2x}-c_6'(x)e^{-x} \\
     =&e^{2x}\left(-\frac{38}{3}+12+6-\frac{16}{3}\right)+xe^{2x}\left(-8+8\right)=0.
    \end{aligned}
\end{equation*}
As a result, $p^{(4)}$ is also Lipschitz and 
\begin{equation*}
    p^{(4)}(x)=16c_3(x)e^{2x}+32c_4(x)e^{2x}+16c_4(x)e^{2x}+16c_5(x)e^{-2x}+c_6(x)e^{-x}.
\end{equation*}
Finally, 
\begin{equation*}
    \begin{aligned}
       p^{(5)}(x)=&16c_3'(x)e^{2x}+32c_4'(x)e^{2x}+16c_4'(x)e^{2x}+16c_5'(x)e^{-2x}+c_6'(x)e^{-x} \\
       &+32c_3(x)e^{2x}+80c_4(x)e^{2x}+32c_4(x)xe^{2x}-32c_5(x)e^{-2x}-c_6(x)e^{-x}.
    \end{aligned}
\end{equation*}
The first line vanishes, since 
\begin{equation*}
    \begin{aligned}
    &16c_3'(x)e^{2x}+32c_4'(x)e^{2x}+16c_4'(x)e^{2x}+16c_5'(x)e^{-2x}+c_6'(x)e^{-x} \\
    =&e^{2x}\left(-\frac{76}{3}+32-12+\frac{16}{3}\right)+xe^{2x}\left(-16+16\right)=0.
    \end{aligned}
\end{equation*}
Consequently, 
\begin{equation*}
    p^{(5)}(x)=32c_3(x)e^{2x}+80c_4(x)e^{2x}+32c_4(x)xe^{2x}-32c_5(x)e^{-2x}-c_6(x)e^{-x}
\end{equation*}
and $p(x)\in C^{5,1}$. In order to prove property (4), we compute $p^{(6)}$:
\begin{equation*}
    \begin{aligned}
        p^{(6)}(x):=&32c_3'(x)e^{2x}+80c_4'(x)e^{2x}+32c_4'(x)xe^{2x}-32c_5'(x)e^{-2x}-c_6'(x)e^{-x}\\
        &+64c_3(x)e^{2x}+192c_4(x)e^{2x}+64c_4(x)xe^{2x}+64c_5(x)e^{-2x}+c_6(x)e^{-x}.
    \end{aligned}
\end{equation*}
We claim that the function $q(y):=p(\log y)$ satisfies all the desired properties. Denote $x:=\log y$, then for $y\leq R$, we have $x\leq X$, so $q(y)=p(x)=\frac{1}{2}e^{2x}=\frac{1}{2}y^2$, which proves (1). Also by defining $\Tilde{R}=e^{\Tilde{X}}$, property (2) directly holds. For property (3),   we have
\begin{equation*}
    \frac{q'(y)}{y}=e^{-2x}p'(x)\text{ , and }q''(y)=e^{-2x}\left(p''(x)-p'(x)\right).
\end{equation*}
From the expression of $p'(x)$ and $p''(x)$ we know
\begin{equation*}
  e^{-2x}p'(x)=2c_3(x)+e^{-2x}(c_2(x)+c_4(x)(1+2x)e^{2x}-2c_5(x)e^{-2x}-c_6(x)e^{-x}).
\end{equation*}
Then $\frac{q'}{r}\geq -\epsilon$ holds due to the fact that $c_3(x)\geq 0$ and the estimate of $|c_j(x)|$. Furthermore, $ q''(r)-\frac{q'(r)}{r}$ can be expressed by
\begin{equation*}
    q''(r)-\frac{q'(r)}{r}=e^{-2x}\left(p''(x)-2p'(x)\right)=-2c_2(x)e^{-2x}+2c_4(x)+8c_5(x)e^{-4x}+3c_6(x)e^{-3x}.
\end{equation*}
The smallness of the coefficients $c_j(x)$ gives the desired bound and proves (3). We next prove the crucial property (4). A direct change of variables gives 
\begin{equation*}
    \left(\partial_{yy}+\frac{3\partial_y}{y}\right)^3 q(y)=\frac{p^{(6)}(x)-6p^{(5)}(x)+4p^{(4)}(x)+24p^{(3)}(x)-32p''(x)}{y^6}.
\end{equation*}
Plugging the expression of $p^{(j)}$ into it yields
\begin{equation*}
    \begin{aligned}
         \left(\partial_{yy}+\frac{3\partial_y}{y}\right)^3 q(y)=&\frac{\left(32c_3'(x)e^{2x}+80c_4'(x)e^{2x}+32c_4'(x)xe^{2x}-32c_5'(x)e^{-2x}-c_6'(x)e^{-x}-45c_6(x)e^{-x}\right)}{y^6}\\
        =&\frac{48F(\log y)}{y^4}-\frac{45c_6(x)e^{-x}}{y^6}.
    \end{aligned}
\end{equation*}
The definition of $F$ indicates that $|F(\log y)|\leq\frac{3\alpha_1}{(1+\log y)^2}$. Together with the pointwise estimate of $c_6(x)$ we have $ \left(\partial_{yy}+\frac{3\partial_y}{y}\right)^3 q(y)\gtrsim -\frac{\alpha_1}{y^4(1+\log y)^2}$, which proves (4). For property (5), observe that
\begin{equation*}
\begin{aligned}
  \Delta^2 q(y)=&\frac{1}{y^4}\left(p^{(4)}(x)-4p^{(2)} (x)\right)
  =\frac{1}{y^4}\left(16c_4(x)e^{2x}-3c_6(x)e^{-x}\right).  
\end{aligned}  
\end{equation*}
Since $|c_4(x)|\lesssim \frac{\alpha_1}{x}$, $|c_6(x)|\lesssim \alpha_1\frac{e^{3x}}{(1+x)^2}$, the $| \Delta^2 q(y)|$ is bounded by $\frac{\alpha_1}{|y|^2(1+|\log y|)}$, and thus (5) holds. To prove (6), we compute $(\Delta q(y))'$ , which is equal to  
\begin{equation*}
   (\Delta q(y))'=\frac{1}{y^3}\left(p^{(3)}(x)-4p'(x)\right)=\frac{1}{e^{3x}}\left(-4c_2(x)+8c_4(x)e^{2x}+3c_6(x)e^{-x}\right).
\end{equation*}
The bound on $|c_j(x)|$ then implies (6). Finally, property (7) follows from 
\begin{equation*}
\begin{aligned}
   |q''(y)|=& \frac{|-c_2(x)+2c_3(x)e^{2x}+3c_4(x)e^{2x}+2c_4(x)xe^{2x}+6c_5(x)e^{-2x|}+2c_6(x)e^{-x}}{e^{2x}}\\\leq& |2c_3|+\alpha_1O(1)
   \leq 2\left(\frac{1}{2}+\frac{\alpha_1(18X+19)}{12X}\right)+\alpha_1O(1)\leq 4, 
\end{aligned}
\end{equation*}
and
\begin{equation*}
    \begin{aligned}
        \left|\frac{q'(y)}{y}\right|=&\frac{|c_2(x)+2c_3(x)e^{2x}+c_4(x)(1+2x)e^{2x}-2c_5(x)e^{-2x}-c_6(x)e^{-x}|}{e^{2x}}\\
        \leq& \left( |2c_3|+\alpha_1O(1)\right)\leq 4,
    \end{aligned}
\end{equation*}
which completes the proof of the Lemma.
\end{proof}

With this choice of $q$, we define the localized operators
\begin{equation*}
    \begin{aligned}
        &\left[A_k h\right](x):=\frac{1}{4}\frac{1}{\lambda_k}\Delta q \left(\frac{x-\boldsymbol{y}_k}{\lambda_k}\right)h(x)+\nabla q\left(\frac{x-\boldsymbol{y}_k}{\lambda_k}\right)\cdot \nabla h(x),\\
        &\left[\underline{A}_k h\right](x):=\frac{1}{2}\frac{1}{\lambda_k}\Delta q \left(\frac{x-\boldsymbol{y}_k}{\lambda_k}\right)h(x)+\nabla q\left(\frac{x-\boldsymbol{y}_k}{\lambda_k}\right)\cdot \nabla h(x).
    \end{aligned}
\end{equation*}
Note the similarity between $A_k$ and $\frac{1}{\lambda_k}\Lambda$ and between $\underline{A}_k$ and $\frac{1}{\lambda_k}\underline{\Lambda}$. For later use, we also introduce the unscaled operators
\begin{equation*}
    \begin{aligned}
        \left[Ah\right](x):=&\frac{1}{4}\Delta q(x)h(x)+\nabla q(x)\cdot \nabla h(x)\\ \left[\underline{A}h\right](x):=&\frac{1}{2}\Delta q(x)h(x)+\nabla q(x)\cdot \nabla h(x).
    \end{aligned}
\end{equation*}
\begin{lemma}\label{property of Ak}
   For any $k=1,...,N$, the operators $A_k$ and $\underline{A}_k$ satisfy the following properties.

 (\rmnum{1}) The families
\begin{equation*}
    \begin{aligned}
        \left\{A_k:\lambda_k>0,y_k\in \mathbb R^4\right\}\text{ , }&\left\{\underline{A}_k:\lambda_k>0,y_k\in \mathbb R^4\right\}\\
        \left\{\lambda_k\partial_{\lambda_k}A_k:\lambda_k>0,y_k\in \mathbb R^4\right\}\text{ , }&\left\{\lambda_k\partial_{\lambda_k}\underline{A}_k:\lambda_k>0,y_k\in \mathbb R^4\right\}\\
        \left\{\lambda_k\partial_{y_k}A_k:\lambda_k>0,y_k\in \mathbb R^4\right\}\text{ and }&\left\{\lambda_k\partial_{y_k}\underline{A}_k:\lambda_k>0,y_k\in \mathbb R^4\right\}
    \end{aligned}
\end{equation*}
are bounded in $\ML(\dot{H}^1,L^2)$ and $\ML(L^2,L^{\frac{4}{3}})$,  with norms depending on $q$. Moreover, for any $h\in \dot{H}^1\cap \dot{H}^2$,
\begin{equation}\label{Lk 2 bounded for Ak}
    \left\|\underline{A}_kh\right\|_{\dot{H}^1}+\left\|\lambda_k\partial_{\lambda_k}\underline{A}_kh\right\|_{\dot{H}^1}+\left\|\lambda_k\partial_{y_k}\underline{A}_kh\right\|_{\dot{H}^1}\lesssim \Tilde{R}\left\|\ML_k h\right\|_{L^2}+\frac{1}{\lambda_k}\left\|h\right\|_{\dot{H}^1}.
\end{equation}

(\rmnum{2}) For any $g,h\in \dot{H}^1\cap \dot{H}^2$,
\begin{equation}\label{Ak key property 1}
  \lelan A_kh,f(h+g)-f(h)-f'(h)g \rilan=-\lelan A_k g,f(h+g)-f(h) \rilan. 
\end{equation}

(\rmnum{3}) For any $h\in \dot{H}^1\cap \dot{H}^2$, $g\in\dot{H}^1\cap L^2$, it holds
\begin{equation}\label{Ah delta g}
    \left|\lelan \underline{A}_k h,\Delta g \rilan\right|\lesssim \Tilde{R}\left\|\ML_kh\right\|_{L^2}\left\|g\right\|_{\dot{H}^1}.
\end{equation}

(\rmnum{4}) For any $\eta>0$, choosing $\epsilon>0$ small enough in Lemma \ref{function q}, it holds for all $g\in \dot{H}^1\cap \dot{H}^2$,
\begin{equation}\label{Ak key property 3}
    \lelan \underline{A}_kg,\Delta g \rilan\leq \frac{\eta}{\lambda_k}\left\|g\right\|_{\dot{H}^1}^2-\frac{1}{\lambda_k}\int_{|x-\boldsymbol{y}_k|\leq R}|\nabla g(x)|^2dx,
\end{equation}
\begin{equation}\label{Ak key property 2}
      \lelan [\Delta,\underline{A}_k]g,\Delta g \rilan\geq  \frac{1}{\lambda_k}\int_{|x-\boldsymbol{y}_k|\leq \lambda_k R}(\Delta g(x))^2dx-\frac{\eta}{\lambda_k}\left\|\ML_k g\right\|_{L^2}^2.
\end{equation}

(\rmnum{5}) For any $\eta>0$, choosing $\epsilon>0$ and $R$ is from Lemma \ref{function q}, it holds
\begin{equation}\label{A and Lambda W nabla W}
    \left\|\underline{A}\Lambda W\right\|_{L^2(|x|\geq R)}\leq \left(\eta+\frac{1}{R}\right)\text{ , }\left\|\underline{A}\nabla W\right\|_{L^2(|x|\geq R)}\leq \frac{1}{R}.
\end{equation}
\end{lemma}
\begin{Rmk}
To keep the exposition focused, we defer the proof of Lemma \ref{property of Ak} to Appendix \ref{Proof of Lemma property of Ak}. We emphasize that (\ref{Ak key property 2}) is the key estimate. The rather involved construction of the weight in Lemma \ref{function q} is designed precisely to obtain this property.
\end{Rmk}

\subsection{Second order energy estimates}
In this subsection, we establish the second order energy estimate for the decomposition associated with $( \phi, h,\dot g)$. By the definition of $h$ and estimate (\ref{pg h dot 2}) we obtain
\begin{equation*}
    \left\|h\right\|_{\dot{H}^2}\lesssim \left\|g\right\|_{\dot{H}^2}+\left\|\PG\right\|_{\dot{H}^2}\lesssim t^{\frac{1}{2}}\lambda.
\end{equation*}
We define the second order energy by
\begin{equation*}
    \MI_2(t):=\frac{1}{2}\int_{\mathbb R^4}\Big(|\nabla \dot{g}|^2-3\phi^2\dot{g}^2\Big)dx+\frac{1}{2}\int_{\mathbb R^4}\Big(\Delta h+3\phi^2h\Big)^2dx.
\end{equation*}
We also set
\begin{equation*}
    \mathcal{J}_k(t)=\mathcal{J}_{k,1}(t)+\mathcal{J}_{k,2}(t),
\end{equation*}
where 
\begin{equation*}
    \mathcal{J}_{k,1}(t):=-b_k\lelan \underline{A}_kh, \ML_{\phi}\dot{g} \rilan,\qquad\mathcal{J}_{k,2}(t):=2b_k\lelan \underline{A}_k\dot{g}, \mathcal{L}_{\phi}h \rilan.
\end{equation*}
Denote $\mathcal{H}_2(t):=\MI_2(t)+\sum_{k=1}^NJ_k(t)$. We prove the following Lemma.
\begin{lemma}\label{energy est 2nd}
    For any $\delta>0$, choosing $\epsilon>0$ small enough, it holds
    \begin{equation}\label{H 2}
  \mathcal{H}'_2(t)\gtrsim -\delta t^{\frac{2}{3}}\lambda(t)^2.      
    \end{equation}
\end{lemma}
\begin{proof} In this proof, the symbol $\simeq$ denotes equality up to terms of order 
$O(t^{\frac{1}{2}}\lambda^2)$, which will be called negligible. We call such error terms "negligible". All constants $C>0$ below are independent of the choices of $\epsilon$ and $R$ in Lemma \ref{function q}.

\textbf{Step1 Analysis of $\MI_2'.$} We first compute $\MI_2'$.
\begin{equation*}
    \begin{aligned}
        \frac{d}{dt}\MI_2=&\lelan \ML_{\phi}\dot{g}, \partial_t\dot{g}\rilan+\frac{1}{2}\lelan [\partial_t,\ML_{\phi}]\dot{g}, \dot{g} \rilan
        +\lelan \ML_{\phi}^2h, \partial_t h \rilan +\lelan [\partial_t,\ML_{\phi}]h, \ML_{\phi}h \rilan.
    \end{aligned}
\end{equation*}
Since $[\partial_t,\ML_{\phi}]=-6\phi \partial_t\phi$ is a self-adjoint operator, term $\lelan [\partial_t,\ML_{\phi}]h, \ML_{\phi}h \rilan$ can be rewritten as
\begin{equation*}
\begin{aligned}
  &\lelan [\partial_t,\ML_{\phi}]h, \ML_{\phi}h \rilan \\
  =&\frac{1}{2}\lelan [\partial_t, \ML_{\phi}]\ML_{\phi}h,h \rilan+\frac{1}{2}\lelan \ML_{\phi}[\partial_t,\ML_{\phi}]h,h \rilan\\
  =&\frac{1}{2}\lelan \partial_t(\ML_{\phi}(\ML_{\phi}h))-\ML_{\phi}(\partial_t(\ML_{\phi}h))  ,h\rilan+\frac{1}{2}\lelan  \ML_{\phi}(\partial_t(\ML_{\phi}h))-\ML_{\phi}(\ML_{\phi}(\partial_t h)),h\rilan\\
  =&\frac{1}{2}\lelan [\partial_t,\ML_{\phi}^2]h,h \rilan.
\end{aligned}  
\end{equation*}
This identity allows us to rewrite the derivative in a symmetric form. Therefore, 
\begin{equation}\label{d dt I 2 equal to}
    \begin{aligned}
        \frac{d}{dt}\MI_2=&\lelan \ML_{\phi}\dot{g}, \partial_t\dot{g}\rilan
        +\lelan \ML_{\phi}^2h, \partial_t h \rilan +\frac{1}{2}\lelan [\partial_t,\ML_{\phi}]\dot{g}, \dot{g} \rilan+\frac{1}{2}\lelan [\partial_t,\ML_{\phi}^2]h, h \rilan.
    \end{aligned}
\end{equation}
The last two terms in (\ref{d dt I 2 equal to}), which contain the commutators, will be canceled after the corrections $\mathcal J_k$ are taken into account. This cancellation will be discussed in the later steps. For the remaining part, using (\ref{partial t g dot sim}),  we obtain
\begin{equation}\label{L phi dot g and L phi 2 h}
  \begin{aligned}
     &\lelan \ML_{\phi}\dot{g}, \partial_t\dot{g}\rilan+ \lelan \ML_{\phi}^2h, \partial_t h \rilan\\
     =&\lelan \ML_{\phi}\dot{g}, -\ML_{\phi}h \rilan+\lelan \ML_{\phi}\dot{g},   \Psi(\boldsymbol{\lambda},\boldsymbol{b},{\boldsymbol{y}},{\boldsymbol{v}})\rilan+\lelan \ML_{\phi}\dot{g}, {\epsilon}(h,\dot{g}) \rilan+\lelan \ML_{\phi}^2h, \partial_t h \rilan\\
     =&\lelan \ML_{\phi}^2h, \partial_t (h-g) \rilan+\lelan \ML_{\phi}^2h, \partial_t g-\dot{g} \rilan+\lelan \ML_{\phi}\dot{g},   \Psi(\boldsymbol{\lambda},\boldsymbol{b},{\boldsymbol{y}},{\boldsymbol{v}})\rilan+\lelan \ML_{\phi}\dot{g}, {\epsilon}(h,\dot{g}) \rilan.
  \end{aligned}  
\end{equation}
We estimate the four terms on the right hand side of (\ref{L phi dot g and L phi 2 h}) separately. Since $h-g=-\PG$, and using the analysis from the estimates (\ref{partial t pg h dot 2}) and (\ref{refine dt pg h dot 2}), the first term is bounded by
\begin{equation*}
  \begin{aligned}
    \left| \lelan \ML_{\phi}^2h, \partial_t (h-g) \rilan\right|=&\left|\lelan \ML_{\phi}h,\ML_{\phi}(\partial_t \PG) \rilan\right|\\\lesssim& \sum_{k=1}^N{|b_k'b_k|}\left|\int_{|x-\boldsymbol{y}_k|\leq d}(\ML_{\phi}h)(\ML_k S_k)dx\right|+t^{\frac{1}{2}}\lambda^2.
  \end{aligned}  
\end{equation*}
Since on the region $|x-\boldsymbol{y}_k|\leq d$, it holds
\begin{equation*}
    \ML_k S_k= -\frac{1}{\lambda_k^3}(\underline{\Lambda}\Lambda W)\left(\frac{x-\boldsymbol{y}_k}{\lambda_k}\right)+\frac{\lelan (\underline{\Lambda}\Lambda W)\chi_t(\lambda r),\Lambda W \rilan}{\lelan \Lambda W \chi_t(\lambda \cdot),\Lambda W \rilan\lambda_k^3}(\Lambda W)\left(\frac{x-\boldsymbol{y}_k}{\lambda_k}\right).
\end{equation*}
The $\Lambda W$ term, by Cauchy and (\ref{rough est on b'}), is controlled by
\begin{equation*}
    \frac{|b_k'b_k|}{\lambda_k}\left\|\ML_{\phi}h\right\|_{L^2}\left(\int_{\lambda_k}^d\left(\frac{1}{r^2t^{\frac{2}{3}}}\right)^2r^3dr\right)^{\frac{1}{2}}\lesssim \frac{t^{\frac{1}{2}}\lambda}{t^{\frac{1}{3}}}t^{\frac{1}{2}}\lambda(t^{-\frac{1}{3}})=t^{\frac{1}{3}}\lambda^2,
\end{equation*}
which can be neglected. The leading contribution comes from $\frac{1}{\lambda_k^3}\underline{\Lambda}\Lambda W$ term. Using (\ref{slight refine of b'}), we obtain 
\begin{equation*}
\begin{aligned}
   &\sum_{k=1}^N\frac{|b_k'b_k|}{\lambda_k}\left|\int_{|x-\boldsymbol{y}_k|\leq d}(\ML_{\phi} h)\frac{1}{\lambda_k^2}\left(\underline{\Lambda}\Lambda W\right)\left(\frac{x-\boldsymbol{y}_k}{\lambda_k}\right)dx\right|\\
    \lesssim& \frac{\left\|\ML_k g\right\|_{L^2(|x-\boldsymbol{y}_k|\leq 2\lambda_kM)}}{t^{\frac{1}{3}}\log M} \left|\int_{|x-\boldsymbol{y}_k|\leq \lambda_k R}(\ML_{\phi} h)\frac{1}{\lambda_k^2}\left(\underline{\Lambda}\Lambda W\right)\left(\frac{x-\boldsymbol{y}_k}{\lambda_k}\right)dx\right|\\
    &+\frac{\left\|\ML_k g\right\|_{L^2(|x-\boldsymbol{y}_k|\leq 2\lambda_kM)}}{t^{\frac{1}{3}}\log M}\left|\int_{\lambda_k R\leq|x-\boldsymbol{y}_k|\leq d}(\ML_{\phi} h)\frac{1}{\lambda_k^2}\left(\underline{\Lambda}\Lambda W\right)\left(\frac{x-\boldsymbol{y}_k}{\lambda_k}\right)dx\right|,  
\end{aligned}
\end{equation*}
where $R$ is given by Lemma \ref{function q}. Without loss of generality, we assume $R\geq 2M$ so that $\left\|\ML_k g\right\|_{L^2(|x-\boldsymbol{y}_k|\leq 2\lambda_kM)}\leq \left\|\ML_k g\right\|_{L^2(|x-\boldsymbol{y}_k|\leq \lambda_kR)}$.  Then by Cauchy inequality and the decay of $\underline{\Lambda}\Lambda W$, we derive
\begin{equation*}
\begin{aligned}
   &\frac{|b_k'b_k|}{\lambda_k}\left|\int_{|x-\boldsymbol{y}_k|\leq d}(\ML_{\phi} h)\frac{1}{\lambda_k^2}\left(\underline{\Lambda}\Lambda W\right)\left(\frac{x-\boldsymbol{y}_k}{\lambda_k}\right)\right|
    \lesssim \frac{\left\|\ML_\phi h\right\|^2_{L^2(|x-\boldsymbol{y}_k|\leq \lambda_kR)}}{t^{\frac{1}{3}}\log M} +\frac{\left\|\ML_\phi h\right\|^2_{L^2(|x-\boldsymbol{y}_k|\leq d)}}{t^{\frac{1}{3}}R^2\log M}. 
\end{aligned}
\end{equation*}
Consequently,
\begin{equation}\label{L phi 2 h partial t(h-g)}
 \lelan \ML_{\phi}^2h, \partial_t (h-g) \rilan\gtrsim -  \frac{\left\|\ML_\phi h\right\|^2_{L^2(|x-\boldsymbol{y}_k|\leq \lambda_kR)}}{t^{\frac{1}{3}}\log M} -\frac{\left\|\ML_\phi h\right\|^2_{L^2(|x-\boldsymbol{y}_k|\leq d)}}{t^{\frac{1}{3}}R^2\log M}-O(t^{\frac{1}{2}}\lambda^2). 
\end{equation}

 Next, using (\ref{g-dot g}) and $\left\{\Lambda W,\partial_jW_{j=1,2,3,4}\right\}\in \ker \ML$, we estimate the term $\lelan \ML_{\phi}^2h, \partial_t g-\dot{g} \rilan$ as 
\begin{equation}\label{ML phi on trunction}
\begin{aligned}
 |\lelan \ML_{\phi}^2h, \partial_t g-\dot{g} \rilan|\lesssim&   \sum_k\left\|\ML_{\phi}g\right\|_{L^2}\left\|\ML_k(\lambda_k^{-1}b_k(\Lambda_kW_k)\chi_{t})\right\|_{L^2}\\
 &+\sum_k\left\|\ML_{\phi}g\right\|_{L^2}\left\|(\ML_{\phi}-\ML_k)(\lambda_k^{-1}(\lambda_k'\Lambda_kW_k+b_k(\Lambda_k W_k)\chi_{t})\right\|_{L^2}\\
  &+\sum_k\left\|\ML_{\phi}g\right\|_{L^2}\left\|(\ML_{\phi}-\ML_k)(\lambda_k^{-1}(\boldsymbol{y}_k'+\boldsymbol{v}_k)\cdot\nabla_k W_k)\right\|_{L^2}.
\end{aligned}
\end{equation}
For the first term on the right-hand side of (\ref{ML phi on trunction}), the definition of $\chi_t$ and Cauchy's inequality give
\begin{equation*}
   \begin{aligned}
&\sum_k\left\|\ML_{\phi}g\right\|_{L^2}\left\|\ML_k(\lambda_k^{-1}b_k(\Lambda_kW_k)\chi_{t})\right\|_{L^2}\\ 
 \lesssim&\sum_k\frac{b_k}{\lambda_k}t^{\frac{1}{2}}\lambda\left(\int_{t\leq |x-\boldsymbol{y}_k|\leq 2t}(\nabla(\Lambda_kW_k)\cdot\nabla(\chi_{t})+\Lambda_kW_k\Delta(\chi_{t}))^2dx\right)^{\frac{1}{2}}\\
 \lesssim&t^{\frac16}\lambda\left(\int_{t\leq r \leq 2t}\Big(\frac{\lambda_k}{r^3}\frac{1}{t}\Big)^2+\Big(\frac{\lambda_k}{r^2}\frac{1}{t^2}\Big)^2r^3dr\right)^{\frac{1}{2}}     
 \lesssim t^{-\frac{11}{6}}\lambda^2.
   \end{aligned} 
\end{equation*}
For the second term of (\ref{ML phi on trunction}), we split the domain into the region near multi-bubbles, $\Tilde{B}:=\bigcup_{1\leq j\leq N}B(|x-\boldsymbol{y}_j|\leq d)$, and its complement $\mathbb{R}^4 \setminus \tilde {B}$. First, in $B(|x-\boldsymbol{y}_k|\leq d
)$ from (\ref{lambda'+b}) and (\ref{pg h dot 2}) we have 
\begin{equation*}
    \begin{aligned}
    &\left\|(\ML_{\phi}-\ML_k)(\lambda_k^{-1}(\lambda_k'\Lambda_kW_k+b_k(\Lambda_k W_k)\chi_{t})\right\|_{L^2(\mathcal{B}(|x-\boldsymbol{y}_k|\leq d)}\\
    \lesssim &\frac{|\lambda_k'+b_k|}{\lambda_k}\left(\sum_{j\neq k}\left\|(W_jW_k)\Lambda_kW_k\right\|_{L^2(\mathcal{B}(|x-\boldsymbol{y}_k|\leq d)}+\left\|(\PG W_k)\Lambda_kW_k\right\|_{L^2(\mathcal{B}(|x-\boldsymbol{y}_k|\leq d)}\right)\\
    \lesssim &t^{\frac{1}{2}}\lambda\left(\frac{\lambda_j\lambda_k^2}{|z_k-z_j|^2}\left(\int_{\lambda_k}^d(\frac{1}{r^4})^2r^3dr\right)^{1/2}+\sqrt{\log(d/\lambda_k)}\left\|\PG\right\|_{{H}^2}\right)
    \lesssim \lambda^{\frac{3}{2}}.
    \end{aligned}
\end{equation*}
The same estimate holds on $\bigcup_{j\neq k}{B}(|x-\boldsymbol{y}_j|\leq d)$,
\begin{equation*}
  \begin{aligned}
    &\left\|(\ML_{\phi}-\ML_k)(\lambda_k^{-1}(\lambda_k'\Lambda_kW_k+b_k(\Lambda_k W_k)\chi_{t})\right\|_{L^2(\mathcal{B}(|x-\boldsymbol{y}_j|\leq d)}\\
    \lesssim &\frac{|\lambda_k'+b_k|}{\lambda_k}\left(\left\|(W_j)^2\Lambda_kW_k\right\|_{L^2(\mathcal{B}(|x-\boldsymbol{y}_j|\leq d)}+\left\|(\PG W_j)\Lambda_kW_k\right\|_{L^2(\mathcal{B}(|x-\boldsymbol{y}_j|\leq d)}\right)
    \lesssim  \lambda^{\frac{3}{2}}.
    \end{aligned}   
\end{equation*}
Hence the contribution from $\tilde{B}$ is bounded by
\begin{equation*}
 \sum_k\left\|\ML_{\phi}g\right\|_{L^2}\left\|(\ML_{\phi}-\ML_k)(\lambda_k^{-1}(\lambda_k'\Lambda_kW_k+b_k(\Lambda_k W_k)\chi_{t})\right\|_{L^2({\tilde{B})}}\lesssim t^{\frac{1}{2}}\lambda^{\frac{5}{2}}.   
\end{equation*}
 In the exterior region $\mathbb{R}^4 \setminus \tilde {B}$ , by (\ref{pg h dot 2}) we obtain
\begin{equation*}
  \begin{aligned}
    &\left\|(\ML_{\phi}-\ML_k)(\lambda_k^{-1}(\lambda_k'\Lambda_kW_k+b_k(\Lambda_k W_k)\chi_{t})\right\|_{L^2(\mathbb{R}^4 \setminus \tilde {B})}\\
    \lesssim &\frac{|\lambda_k'|+|b_k|}{\lambda_k}\left(\left(\int_d^{\infty}\left(\frac{\lambda_k^3}{r^6}\right)^2r^3dr\right)^{1/2}+\left\|\PG\right\|_{\dot{H}^1}\left(\int_d^{\infty}\left(\frac{\lambda_k^2}{r^4}\right)^4\right)^{1/4}\right)
    \lesssim \lambda^2.
    \end{aligned}   
\end{equation*}
Another application of Cauchy's inequality shows that the exterior contribution is also of lower order.  The same argument can be applied to the last line of (\ref{ML phi on trunction}). Plugging these estimates back into (\ref{ML phi on trunction}) we conclude that $ \left|\lelan \ML_{\phi}^2h, \partial_t g-\dot{g} \rilan\right|$ is negligible. 

Finally, from Lemma \ref{lemma of partial t dot g}, the term involving $\epsilon(h,\dot{g})$ is dominated by
\begin{equation*}
\begin{aligned}
    |\lelan \ML_{\phi}\dot{g}, {\epsilon}(h,\dot{g}) \rilan  |\lesssim \left\|\dot{g}\right\|_{\dot{H}^1}\left\|\ML_{\phi} ({\epsilon}(h,\dot{g}))\right\|_{\dot{H}^{-1}}
    \lesssim& t^{\frac{1}{2}}\lambda^2,
\end{aligned}
\end{equation*}
which is lower order. It remains for us to estimate $\lelan \ML_{\phi}\dot{g},   \Psi(\boldsymbol{\lambda},\boldsymbol{b},{\boldsymbol{y}},{\boldsymbol{v}})\rilan$. We separate the main part of $\Psi$ by setting
\begin{equation*}
      \Psi_0(\boldsymbol{\lambda},\boldsymbol{b},{\boldsymbol{y}},{\boldsymbol{v}}):=\sum_{j=1}^N \frac{(\lambda_j'+b_j)b_j}{\lambda_j^2}(\underline{\Lambda}_j\Lambda_j W_j) \chi_t(\cdot-\boldsymbol{y}_k)+\sum_{j=1}^N\frac{b_j((\boldsymbol{y}_j'+\boldsymbol{v}_j)\cdot \nabla_j)}{\lambda_j^2}(\Lambda_jW_j) \chi_t(\cdot-\boldsymbol{y}_j),  
\end{equation*}
and one directly has
\begin{equation*}
 \Psi(\boldsymbol{\lambda},\boldsymbol{b},{\boldsymbol{y}},{\boldsymbol{v}})=     \Psi_0(\boldsymbol{\lambda},\boldsymbol{b},{\boldsymbol{y}},{\boldsymbol{v}})-\sum_{j=1}^N\left(b_j'-\frac{b_j \lambda_j'}{2\log(1/\lambda_j)\lambda_j} -\frac{B_j(\Vec{\lambda})}{\left\|\Lambda W\chi_t (\lambda_j \cdot)\right\|_{L^2}^2}\right)\lambda_j^{-1}(\Lambda
           _j W_j)\chi_t(\cdot-\boldsymbol{y}_j).
\end{equation*}
We claim that the remaining part of $\Psi(\boldsymbol{\lambda},\boldsymbol{b},{\boldsymbol{y}},{\boldsymbol{v}})$ is negligible, namely
\begin{equation}\label{PSi and Psi 0}
  \lelan \ML_{\phi}\dot{g},   \Psi(\boldsymbol{\lambda},\boldsymbol{b},{\boldsymbol{y}},{\boldsymbol{v}})\rilan\simeq \lelan \ML_{\phi}\dot{g},   \Psi_0(\boldsymbol{\lambda},\boldsymbol{b},{\boldsymbol{y}},{\boldsymbol{v}})\rilan.
\end{equation}
In order to prove the claim, using an argument similar to the proof of (\ref{ML phi on trunction}) we obtain
\begin{equation*}
\begin{aligned}
    \left\|\ML_{\phi}((\Lambda_kW_k)\chi_{t})\right\|_{\dot{H}^{-1}}\lesssim& \left\|\ML_k((\Lambda_kW_k)\chi_{t})\right\|_{L^{\frac{4}{3}}}+\left\|(\ML_{\phi}-\ML_k)((\Lambda_kW_k)\chi_{t})\right\|_{L^{\frac{4}{3}}}\\
    \lesssim& \left(\int_{t\leq |x-\boldsymbol{y}_k|\leq 2t}\left((\Lambda_k W_k)\Delta(\chi_{t})+\nabla(\Lambda_k W_k)\cdot \nabla(\chi_{t})\right)^{\frac{4}{3}}\right)^{\frac{3}{4}}\\
    &+\sum_{j\neq k}\frac{\lambda_j}{|\boldsymbol{y}_k-\boldsymbol{y}_j|^2}\left(\int_{\lambda_k}^d(\frac{\lambda_k^2}{r^4})^{\frac{4}{3}}r^3dr\right)^{\frac{3}{4}}+\left\|\PG\right\|_{\dot{H}^1}\\
    \lesssim& \frac{\lambda_k}{t}+\sum_{j\neq k}\lambda_j \lambda_k +(t^{-\frac{2}{3}}\log t)\lambda\lesssim (t^{-\frac{2}{3}}\log t)\lambda.
\end{aligned}
\end{equation*}
Combining this with (\ref{lambda'+b}), (\ref{rough est on b'}) and bootstrap assumptions for $\lambda_k$, $b_k$,  we have
\begin{equation*}
\begin{aligned}
   &\left| \lelan \ML_\phi \dot{g},\sum_{j=1}^N(b_j'-\frac{b_j \lambda_j'}{2\log(1/\lambda_j)\lambda_j} -\frac{B_j(\Vec{\lambda})}{\left\|\Lambda W\chi_t (\lambda_j \cdot)\right\|_{L^2}^2})\lambda_j^{-1}\Lambda
           _j W_j\chi_t \rilan\right|\\
           \lesssim& \sum_{j=1}^N\left(\frac{|b_j'|}{\lambda_j}+\frac{|b_j||\lambda_j'|}{\log(1/\lambda_j)\lambda_j^2}+\frac{\lambda_j}{\log(1/\lambda_j)\lambda_j}\right)\left\|\dot{g}\right\|_{\dot{H}^1} \left\|\ML_{\phi}((\Lambda_kW_k)\chi_{t})\right\|_{\dot{H}^{-1}}\\\lesssim& \left(t^{\frac{1}{2}}+t^{-\frac{4}{3}}+t^{-\frac{2}{3}}\right)t^{\frac{1}{2}}\lambda (t^{-\frac{2}{3}}\log t)\lambda \lesssim (t^{\frac{1}{3}}\log t)\lambda^2,
\end{aligned}
\end{equation*}
which proves (\ref{PSi and Psi 0}). Collecting all the estimates in Step 1, $\mathcal{I}_2'$ satisfies
\begin{equation}\label{final est for I2'}
\begin{aligned}
   \frac{d}{dt}\MI_2\geq& \frac{1}{2}\lelan [\partial_t,\ML_{\phi}]\dot{g}, \dot{g} \rilan
       +\frac{1}{2}\lelan [\partial_t,\ML_{\phi}^2]h, h \rilan+\lelan \ML_{\phi}\dot{g},   \Psi_0(\boldsymbol{\lambda},\boldsymbol{b},{\boldsymbol{y}},{\boldsymbol{v}})\rilan\\
       &-C\frac{\left\|\ML_\phi h\right\|^2_{L^2(|x-\boldsymbol{y}_k|\leq \lambda_kR)}}{t^{\frac{1}{3}}\log M} -C\frac{t^{\frac{2}{3}}\lambda^2}{R^2\log M}.   
\end{aligned} 
\end{equation}

\textbf{Step 2. Estimate of $\mathcal{J}_{k,1}'.$} We now analyze the correction terms $\mathcal J_k$. Their time derivatives are designed to cancel the commutator terms on the right-hand side of (\ref{final est for I2'}). We first consider $\mathcal J_{k,1}$. A direct computation gives
\begin{equation*}\label{J_k'1}
\begin{aligned}
   \mathcal{J}_{k,1}'=&-b_k'\lelan \underline{A}_kh, \ML_{\phi}\dot{g} \rilan -b_k\lambda_k' \lelan (\partial_{\lambda_k}\underline{A}_k)h, \ML_{\phi}\dot{g} \rilan-b_k\lelan (\boldsymbol{y}_k'\cdot \partial_{y_k}\underline{A}_k)h, \ML_{\phi}\dot{g}\rilan \\
   &-b_k\lelan \underline{A}_k\partial_th, \ML_{\phi}\dot{g} \rilan-b_k \lelan \underline{A}_kh, [\partial_t,\ML_{\phi}]\dot{g}\rilan-b_k\lelan \underline{A}_kh, \ML_{\phi}\partial_t\dot{g} \rilan\\
   :=&\Rmnum{1}+\Rmnum{2}+\Rmnum{3}+\Rmnum{4}+\Rmnum{5}+\Rmnum{6}.
\end{aligned}
\end{equation*}
For $\Rmnum{1}$, the definition of $\ML_{\phi}$ gives
\begin{equation*}
\begin{aligned}
    |\Rmnum{1}|\lesssim& |b_k'|\left(\left|\lelan \underline{A}_kh,\Delta \dot{g}\rilan\right|+\left|\lelan \underline{A}_kh,W_k^2 \dot{g}\rilan\right|+\left|\lelan \underline{A}_kh,\PG^2 \dot g \rilan\right| \right).
\end{aligned}
\end{equation*}
From (\ref{Ah delta g}), we deduce
\begin{equation*}
   |b_k'|\left|\lelan \underline{A}_kh,\Delta \dot{g}\rilan\right|
    \lesssim|b_k'|\Tilde{R}\left\|\ML_kh\right\|_{L^2}\left\|g\right\|_{\dot{H}^1}\lesssim \lambda^{\frac{5}{2}}.
\end{equation*}
Furthermore, the fact that $\lelan \underline{A}_kf,g  \rilan=-\lelan f,\underline{A}_kg  \rilan$ implies
\begin{equation*}
    |b_k'|\left|\lelan \underline{A}_kh,W_k^2 \dot{g}\rilan\right|\lesssim |b_k'|\left(\left|\lelan h,\underline{A}_k(W_k^2)\dot{g} \rilan \right|+\left|\lelan h,W_k^2\underline{A}_k\dot{g} \rilan \right|  \right).
\end{equation*}
Due to the failure of Sobolev embedding for $\dot{H}^2$ in $\mathbb R^4$, here we use $\left\|h\right\|_{\dot{H}^{\frac{3}{2}}}$ instead. From (\ref{h dot 3 2 for g}) and (\ref{pg h dot 2}), we have 
\begin{equation*}
   \left\|h\right\|_{\dot{H}^{\frac{3}{2}}}\lesssim t^{\frac16}\lambda. 
\end{equation*}
 Meanwhile, (i) of Lemma \ref{property of Ak} and interpolation inequality together gives
\begin{equation*}
    \left\|\underline{A}_k W_k^2\right\|_{L^{\frac{8}{5}}}\lesssim \left\|\underline{A}_k W_k^2\right\|_{L^{\frac{4}{3}}}^{\frac{1}{2}}\left\|\underline{A}_k W_k^2\right\|_{L^2}^{\frac{1}{2}}\lesssim \left\|W_k^2\right\|_{L^{2}}^{\frac{1}{2}}\left\|W_k^2\right\|_{\dot{H}^1}^{\frac{1}{2}}\lesssim\sqrt{\lambda_k}.
\end{equation*}
Using Cauchy inequality and (\ref{rough est on b'}), we then obtain
\begin{equation*}
\begin{aligned}
     &|b_k'|\left(\left|\lelan h,\underline{A}_k(W_k^2)\dot{g} \rilan \right|+\left|\lelan h,W_k^2\underline{A}_k\dot{g} \rilan \right|  \right)\\ \lesssim& |b_k'|\left(\left\|h\right\|_{\dot{H}^{\frac{3}{2}}}\left\|\underline{A}_k W_k^2\right\|_{L^{\frac{8}{5}}}\left\|\dot{g}\right\|_{\dot{H}^1}+\left\|h\right\|_{\dot{H}^{\frac{3}{2}}}\left\| W_k^2\right\|_{L^{\frac{8}{3}}}\left\|\underline{A}_k\dot{g}\right\|_{L^2}\right)\\
     \lesssim& t^{\frac{1}{2}} \lambda \left\|h\right\|_{\dot{H}^{\frac{3}{2}}}\left\|\dot{g}\right\|_{\dot{H}^1}\frac{1}{\sqrt{\lambda_k}}\lesssim t^{\frac{2}{3}}\lambda^{\frac{5}{2}}. 
\end{aligned}
\end{equation*}
The term involving $\PG$, by (\rmnum{1}) of Lemma \ref{property of Ak} and (\ref{pg h dot 2}), is dominated by
\begin{equation*}
    |b_k'|\left|\lelan \underline{A}_kh,\PG^2 \dot g \rilan\right|\lesssim |b_k'|\left\|\PG\right\|_{\dot{H}^1}^2\left\|\dot{g}\right\|_{\dot{H}^1}\left(\left\|h\right\|_{\dot{H}^2}+\frac{1}{\lambda_k}\left\|h\right\|_{\dot{H}^1}\right)\lesssim \lambda^{\frac{5}{2}}.
\end{equation*}
Combining these estimates together, we conclude that $\Rmnum{1}$ is negligible. Next for $\Rmnum{2}$ and $\Rmnum{3}$, (\rmnum{1}) of Lemma \ref{property of Ak} yields
\begin{equation*}
\begin{aligned}
       &|\Rmnum{2}|+|\Rmnum{3}|
       \lesssim \left(\frac{|b_k \lambda_k'|}{\lambda_k}+\frac{|b_k||\boldsymbol{y}_k'|}{\lambda_k}\right)(\Tilde{R}\left\|\MLWG h\right\|_{L^2}+\frac{1}{\lambda_k}\left\|h\right\|_{\dot{H}^1})\left\|\dot{g}\right\|_{\dot{H}^1}
       \lesssim t^{-\frac{1}{3}}\lambda^2.
\end{aligned}
\end{equation*}
Therefore, both terms are negligible. We next consider term $\Rmnum{4}$. Recall that 
\begin{equation*}
    \partial_t h=\partial_t g-\partial_t \PG=\dot{g}+\sum_{j=1}^N \lambda_j^{-1}(\lambda_j'\Lambda_jW_j+b_j\Lambda_j W_j\chi_{t})+\sum_{j=1}^N \lambda_j^{-1}(\boldsymbol{y}_j'+\boldsymbol{v}_j)\cdot\nabla_j W_j-\partial_t \PG.
\end{equation*}
Hence, $\Rmnum{4}$ can be decomposed as 
\begin{equation}\label{Jk' 2.1}
    \begin{aligned}
    - b_k\lelan \underline{A}_k\partial_th, \ML_{\phi}\dot{g} \rilan=&  -b_k\lelan \underline{A}_k\dot{g}, \ML_{\phi}\dot{g} \rilan -b_k\lelan \underline{A}_k\left(\sum_{j=1}^N \lambda_j^{-1}(\lambda_j'\Lambda_jW_j+b_j\Lambda_j W_j\chi_{t})\right), \ML_{\phi}\dot{g} \rilan\\
     &-b_k\lelan \underline{A}_k\left(\sum_{j=1}^N \lambda_j^{-1}(\boldsymbol{y}_j'+\boldsymbol{v}_j)\cdot\nabla_j W_j\right), \ML_{\phi}\dot{g} \rilan+b_k\lelan \underline{A}_k\partial_t \PG, \ML_{\phi}\dot{g} \rilan.
    \end{aligned}
\end{equation}
The first term on the right-hand side of (\ref{Jk' 2.1}) will be treated in Step 3, where it is canceled by terms coming from $\mathcal J'_{k,2}$. Furthermore, the term involving $\underline{A}_k \partial_t \PG$ is bounded by
\begin{equation}\label{Jk 2 4 4}
\begin{aligned}
    |b_k\lelan \underline{A}_k\partial_t \PG, \ML_{\phi}\dot{g} \rilan|\lesssim& b_k\left\|\underline{A}_k\partial_t \PG\right\|_{\dot{H}^1}\left\|\ML_{\phi}\dot{g}\right\|_{\dot{H}^{-1}}\\\lesssim& b_k\left(\tilde{R}\left\|\partial_t\PG\right\|_{\dot{H}^2} +\frac{\left\|\partial_t \PG\right\|_{\dot{H}^1}}{\lambda_k}\right)\left\|\dot{g}\right\|_{\dot{H}^1}\lesssim t^{-\frac{1}{3}}\lambda^2, 
\end{aligned}
\end{equation}
where we used (\ref{partial t P h dot 1}), (\ref{partial t pg h dot 2}) and (\ref{Lk 2 bounded for Ak}).
For the remaining part of $\Rmnum{4}$, in view of the similarity between $\underline{A}_k$ and $\frac{1}{\lambda_k}\underline{\Lambda}$, these terms can be compared with
\begin{equation*}
  - \lelan\frac{(\lambda_k'+b_k)b_k}{\lambda_k^2}\underline{\Lambda}_k\Lambda_k W_k \chi_t, \ML_{\phi}\dot{g}\rilan  -\lelan \frac{b_k((\boldsymbol{y}_k'+\boldsymbol{v}_k)\cdot \nabla_k)}{\lambda_k^2}\Lambda_kW_k, \ML_{\phi}\dot{g}\rilan,
\end{equation*}
where $\underline{\Lambda}\partial_l W=\partial_l \Lambda W$ is used. In order to prove this, we first show that 
\begin{equation*}
    \left\|\underline{A}_k\Lambda_k W_k-\frac{1}{\lambda_k}\underline{\Lambda}_k \Lambda_k W_k\right\|_{\dot{H}^1}+\left\|\underline{A}_k\nabla_k W_k-\frac{1}{\lambda_k}\underline{\Lambda}_k \nabla_kW_k \right\|_{\dot{H}^1}\leq \frac{\eta}{\lambda_k}.
\end{equation*}
Indeed, by the definition of $\underline{A}_k$, for $|x-\boldsymbol{y}_k|\leq \lambda_k R$, we have $\underline{A}_k \Lambda_k W_k=\frac{1}{\lambda_k}\underline{\Lambda}_k \Lambda_k W_k$. For $|x-\boldsymbol{y}_k|\geq \lambda_k R$, the decay of $W$ and function $q$ leads $ \left\|\underline{A}_k\Lambda_k W_k-\frac{1}{\lambda_k}\underline{\Lambda}_k \Lambda_k W_k\right\|_{\dot{H}^1}\leq \frac{\eta}{\lambda_k}.$ The estimate on $\nabla_k W_k$ is proved analogously. For $j\neq k$, one verifies that 
$\left\|\underline{A}_k\Lambda_j W_j\right\|_{\dot{H^1}}+\left\|\underline{A}_k\nabla_j W_j\right\|_{\dot{H^1}}\ll t^{-1}$. Combining these estimates with (\ref{lambda'+b}) and (\ref{y'+v}), we arrive at the estimate
\begin{equation}\label{Jk2 4 2}
\begin{aligned}
    &-b_k\lelan \underline{A}_k\left(\sum_{j=1}^N \lambda_j^{-1}(\lambda_j'\Lambda_jW_j+b_j\Lambda_j W_j\chi_{t})\right), \ML_{\phi}\dot{g} \rilan+  \lelan\frac{(\lambda_k'+b_k)b_k}{\lambda_k^2}\underline{\Lambda}_k\Lambda_k W_k \chi_t, \ML_{\phi}\dot{g}\rilan\\
    &\gtrsim - \frac{|(b_k+\lambda_k')b_k|}{\lambda_k^2}(t^{-1}+R^{-1}+\eta)\left\|\dot{g}\right\|_{\dot{H}^1}\gtrsim -{t^{\frac{2}{3}}}\lambda^2(t^{-1}+R^{-1}+\eta),
\end{aligned}
\end{equation}
and 
\begin{equation}\label{Jk 2 4 3}
    \begin{aligned}
      &-b_k\lelan \underline{A}_k\left(\sum_{j=1}^N \lambda_j^{-1}(\boldsymbol{y}_j'+\boldsymbol{v}_j)\cdot\nabla_j W_j\right), \ML_{\phi}\dot{g} \rilan+\lelan \frac{b_k((\boldsymbol{y}_k'+\boldsymbol{v}_k)\cdot \nabla_k)}{\lambda_k^2}\Lambda_kW_k \chi_t, \ML_{\phi}\dot{g}\rilan\\
      &\gtrsim -\frac{|(\boldsymbol{y}_k'+\boldsymbol{v}_k)b_k|}{\lambda_k}(t^{-1}+R^{-1})\left\|\dot{g}\right\|_{\dot{H}^1}\gtrsim -{t^{\frac{2}{3}}}\lambda^2(t^{-1}+R^{-1}+\eta).
    \end{aligned}
\end{equation}
Substituting (\ref{Jk 2 4 4})-(\ref{Jk 2 4 3}) into (\ref{Jk' 2.1}) and recalling the definition of $ \Psi_0(\boldsymbol{\lambda},\boldsymbol{b},{\boldsymbol{y}},{\boldsymbol{v}})$, we obtain 
\begin{equation*}
  \left|\Rmnum{4}+\lelan \ML_{\phi}\dot{g},   \Psi_0(\boldsymbol{\lambda},\boldsymbol{b},{\boldsymbol{y}},{\boldsymbol{v}})\rilan+b_k\lelan \underline{A}_k\dot{g}, \ML_{\phi}\dot{g} \rilan \right|\leq C\eta{t^{\frac{2}{3}}}\lambda^2,
\end{equation*}
This gives the desired estimate for $\Rmnum{4}$. Next, since $ [\partial_t, \ML_{\phi}]=-6\phi_t\phi$,  Lemma \ref{property of Ak} (\rmnum{1}) gives
\begin{equation*}
\begin{aligned}
     |b_k \lelan \underline{A}_kh, [\partial_t,\ML_{\phi}]\dot{g}\rilan| \leq&\left|b_k\lelan h, \underline{A}_k(6\phi_t\phi)\dot{g} \rilan\right|+\left|b_k\lelan h, 6\phi_t\phi\underline{A}_k(\dot{g}) \rilan\right|\\
     \lesssim& |b_k|\left\|h\right\|_{\dot{H}^2}\left\|\phi_t \phi\right\|_{L^2}\left\|\dot{g}\right\|_{\dot{H}^1}\lesssim \lambda^{\frac{5}{2}}.
\end{aligned}  
\end{equation*}
Finally , we consider the term $\Rmnum{6}$ and claim that
\begin{equation}\label{last of Jk'1}
    -\sum_{k=1}^Nb_k\lelan \underline{A}_kh, \ML_{\phi}\partial_t\dot{g} \rilan+\frac{1}{2}\lelan [\partial_t,\ML_{\phi}^2]h, h \rilan\geq \sum_{k=1}^N\frac{2b_k}{\lambda_k}\int_{|x-\boldsymbol{y}_k|\leq R\lambda_k}  (\ML_{\phi}h)^2 dx -C\eta{t^{\frac{2}{3}}}\lambda^2.
\end{equation}
To prove this claim, we first use the argument leading to  (\ref{PSi and Psi 0}) to obtain
\begin{equation*}
\begin{aligned}
    &|-b_k\lelan \underline{A}_kh, \ML_{\phi}\partial_t\dot{g} \rilan- b_k\lelan \underline{A}_kh, \ML_{\phi}^2h \rilan|\\
    \lesssim&|b_k|\left\|\underline{A}_k h\right\|_{\dot{H}^1}\left\|\ML_{\phi}({\epsilon}(h,\dot{g}))\right\|_{\dot{H}^{-1}}+|b_k|\left\|\underline{A}_k h\right\|_{\dot{H}^1}\left\|\ML_{\phi}\left(\Psi_0(\boldsymbol{\lambda},\boldsymbol{b},{\boldsymbol{y}},{\boldsymbol{v}})\right)\right\|_{\dot{H}^{-1}}.
\end{aligned}
\end{equation*}
Using (\ref{lambda'+b}), (\ref{y'+v}), (\ref{partial t g dot sim}) and (\ref{Lk 2 bounded for Ak}), these terms are bounded by
\begin{equation*}
    \begin{aligned}
     &|b_k|\left(\Tilde{R}\left\|\ML_k h\right\|_{L^2}+\frac{1}{\lambda_k}\left\|h\right\|_{\dot{H}^1}\right)\left(\left\|\ML_{\phi}({\epsilon}(h,\dot{g}))\right\|_{\dot{H}^{-1}}+\sum_{j=1}^N\frac{|\lambda_j'+b_j|b_j}{\lambda_j^2}+\sum_{j=1}^N\frac{|\boldsymbol{y}_j'+\boldsymbol{v}_j|b_j}{\lambda_j^2}\right)\\
     \lesssim&t^{-\frac{1}{3}}\lambda t^{-\frac{1}{6}}\left(\lambda+t^{\frac16}\lambda\right)\lesssim t^{-\frac{1}{3}}\lambda^2,
    \end{aligned}
\end{equation*}
and are thus negligible. Consequently, it remains to prove
\begin{equation}\label{to cancel [,Lphi2]}
    \sum_{k=1}^Nb_k\lelan \underline{A}_kh, \ML_{\phi}^2h \rilan+\frac{1}{2}\lelan [\partial_t,\ML_{\phi}^2]h, h \rilan\geq \sum_{k=1}^N\frac{2b_k}{\lambda_k}\int_{|x-\boldsymbol{y}_k|\leq R\lambda_k}  (\ML_{\phi}h)^2 dx -C\eta{t^{\frac{2}{3}}}\lambda^2.
\end{equation}
A direct computation gives
\begin{equation*}
    \begin{aligned}
       b_k\lelan \underline{A}_kh, \ML_{\phi}^2h \rilan=&b_k\lelan 
       [\ML_{\phi},\underline{A}_k]h,\ML_{\phi}h\rilan\\
       =&6b_k\lelan \left(\nabla q
       \left(\frac{x-\boldsymbol{y}_k}{\lambda_k}\right)\cdot\nabla\phi\right)\phi h , \ML_{\phi}h\rilan-b_k\lelan [\Delta,\underline{A}_k]h, \ML_{\phi}h \rilan.
    \end{aligned}
\end{equation*}
For the first term, since $\nabla q(x)=x $ for $|x|\leq R$, we have
\begin{equation}\label{6 nabla q nalba phi}
  \begin{aligned}
   6b_k\lelan \left(\nabla q
       \left(\frac{x-\boldsymbol{y}_k}{\lambda_k}\right)\cdot\nabla\phi\right)\phi h , \ML_{\phi}h\rilan\simeq&\frac{6b_k}{\lambda_k}\int_{|x-\boldsymbol{y}_k|\leq R\lambda_k}\left(\left(\frac{x-\boldsymbol{y}_k}{\lambda_k}\right)\cdot \nabla_kW_k\right)\phi h \ML_{\phi}hdx\\
       +&\frac{6b_k}{\lambda_k}\int_{|x-\boldsymbol{y}_k|\geq R\lambda_k}\left(\nabla q\left(\frac{x-\boldsymbol{y}_k}{\lambda_k}\right)\cdot \nabla_kW_k\right)\phi h \ML_{\phi}hdx,
  \end{aligned}  
\end{equation}
where the terms involving $\nabla \PG$ and $\nabla_j W_j$ for $j\neq k$  are negligible and have been omitted. Moreover, using the coercivity estimate (\ref{second order coercivity for L scaling verison}) and  $ |q'(x)|\leq 4|x|$, the contribution from
the exterior region is bounded by
\begin{equation}\label{6 exterior}
\begin{aligned}
   &\left|\frac{6b_k}{\lambda_k}\int_{|x-\boldsymbol{y}_k|\geq R\lambda_k}\left(\nabla q\left(\frac{x-\boldsymbol{y}_k}{\lambda_k}\right)\cdot \nabla_kW_k\right)\phi h \ML_{\phi}hdx\right|\\
   \leq& Ct^{-\frac{1}{3}}\left\|\ML_{\phi}h\right\|_{L^2}\left(\int_{|x-\boldsymbol{y}_k|\geq R\lambda_k}\left(\frac{\lambda_k^2h}{|x-\boldsymbol{y}_k|^4}\right)^2\right)^{\frac{1}{2}}\leq \frac{C}{R}t^{-\frac{1}{3}}\left\|\ML_{\phi}h\right\|_{L^2}^2\leq \frac{C}{R}t^{\frac{2}{3}}\lambda^2.
\end{aligned}
\end{equation}
Next, the argument used in the proof of (\ref{Ak key property 2}) gives 
\begin{equation*}\label{delta Ak Lphi}
\begin{aligned}
     &-b_k\lelan [\Delta,\underline{A}_k]h, \ML_{\phi}h \rilan \\\geq&\frac{2b_k}{\lambda_k}\int_{|x-\boldsymbol{y}_k|\leq R\lambda_k}  \Delta h(x)(\Delta h(x)+3\phi^2h)dx-  \frac{C\eta b_k}{\lambda_k}\left\|\ML_k h\right\|_{L^2}^2\\
     =&\frac{2b_k}{\lambda_k}\int_{|x-\boldsymbol{y}_k|\leq R\lambda_k}  (\ML_{\phi}h)^2 dx+\frac{b_k}{\lambda_k}\int_{|x-\boldsymbol{y}_k|\leq R\lambda_k}(6\phi\ML_{\phi}h)dx-  \frac{C\eta b_k}{\lambda_k}\left\|\ML_k h\right\|_{L^2}^2.
\end{aligned} 
\end{equation*}
On the other hand, the term containing $[\partial_t,\ML_{\phi}^2]$ in (\ref{to cancel [,Lphi2]}) can be expanded as 
\begin{equation*}
    \begin{aligned}
     &\frac{1}{2}\lelan [\partial_t,\ML_{\phi}^2]h, h \rilan= \lelan [\partial_t,\ML_{\phi}]\ML_{\phi}h, h \rilan\\
     =&\sum_{j=1}^N6\frac{\lambda_j'}{\lambda_j}\lelan \Lambda_j W_j\phi \ML_{\phi}h, h  \rilan+\sum_{j=1}^N6\lelan \frac{\boldsymbol{y}_j'}{\lambda_j}\cdot \nabla_j W_j  \phi \ML_{\phi}h, h \rilan+6\lelan \partial_t \PG \phi \ML_{\phi}h,h \rilan,
    \end{aligned}
\end{equation*}
By (\ref{y'+v}) and (\ref{partial t pg h dot 2}), the last two terms are negligible. Therefore, applying the same argument as in (\ref{6 exterior}), we obtain
\begin{equation}\label{6 [t, L2]}
   \frac{1}{2}\lelan [\partial_t,\ML_{\phi}^2]h, h \rilan\geq  \sum_{j=1}^N\int_{|x-\boldsymbol{y}_j|\leq R\lambda_j}\frac{6\lambda_j'}{\lambda_j}\Lambda_j W_j\phi (\ML_{\phi}h)hdx-\frac{C}{R}t^{\frac{2}{3}}\lambda^2.
\end{equation}
Collecting inequalities (\ref{6 nabla q nalba phi})-(\ref{6 [t, L2]}) leads to 
\begin{equation*}
    \begin{aligned}
     \sum_{k=1}^Nb_k\lelan \underline{A}_kh, \ML_{\phi}^2h \rilan+\frac{1}{2}\lelan [\partial_t,\ML_{\phi}^2]h, h \rilan
     \geq&\sum_{j=1}^N\int_{|x-\boldsymbol{y}_j|\leq R\lambda_j}\frac{6(\lambda_j'+b_j)}{\lambda_j}\Lambda_j W_j\phi (\ML_{\phi}h)hdx\\&+ \sum_{k=1}^N\frac{2b_k}{\lambda_k}\int_{|x-\boldsymbol{y}_k|\leq R\lambda_k}  (\ML_{\phi}h)^2 dx -C\eta{t^{\frac{2}{3}}}\lambda^2,    
    \end{aligned}
\end{equation*}
where we used $\Lambda f=f+x\cdot \nabla f$. Using Cauchy's inequality and (\ref{lambda'+b}), the first term on the right hand side is controlled by
\begin{equation*}
    \frac{|\lambda_j'+b_j|}{\lambda_j}\left\|\ML_{\phi}h\right\|_{L^2}\left\|h\right\|_{\dot{H}^2}\lesssim \lambda^{\frac{5}{2}},
\end{equation*}
and thus negligible. This proves (\ref{last of Jk'1}). 

Combining (\ref{Jk' 2.1})-(\ref{last of Jk'1}), we obtain
\begin{equation}\label{J k1' est and cancel}
    \begin{aligned}
        &\sum_{k=1}^N\mathcal{J}_{k,1}'+\frac{1}{2}\lelan [\partial_t,\ML_{\phi}^2]h, h \rilan+\lelan \ML_{\phi}\dot{g},   \Psi_0(\boldsymbol{\lambda},\boldsymbol{b},{\boldsymbol{y}},{\boldsymbol{v}})\rilan\\
        \geq& \sum_{k=1}^N\frac{2b_k}{\lambda_k}\left\|\ML_\phi h\right\|^2_{L^2(|x-\boldsymbol{y}_k|\leq \lambda_kR)}- \sum_{k=1}^Nb_k\lelan \underline{A}_k\dot{g}, \ML_{\phi}\dot{g} \rilan-C\eta{t^{\frac{2}{3}}}\lambda^2,
    \end{aligned}
\end{equation}
which completes Step 2.

\textbf{Step 3. Estimate of $\mathcal{J}_{k,2}'.$} In this step, we prove that $\mathcal{J}_{k,2}'+\lelan [\partial_t,\ML_{\phi}]\dot{g}, \dot{g} \rilan$ admits the required bound. Taking time derivative of $\mathcal{J}_{k,2}$ gives
\begin{equation}\label{J_k'2}
\begin{aligned}
   \mathcal{J}_{k,2}'=&2b_k'\lelan \underline{A}_k\dot{g}, \ML_{\phi}h \rilan +2b_k\lambda_k' \lelan (\partial_{\lambda_k}\underline{A}_k)\dot{g}, \ML_{\phi}h \rilan+2b_k\lelan (\boldsymbol{y}_k'\cdot \partial_{y_k}\underline{A}_k)\dot{g}, \ML_{\phi}h\rilan \\
   &+2b_k\lelan \underline{A}_k\partial_t\dot{g}, \ML_{\phi}h \rilan+2b_k \lelan \underline{A}_k\dot{g}, [\partial_t,\ML_{\phi}]h\rilan+2b_k\lelan \underline{A}_k\dot{g}, \ML_{\phi}\partial_th \rilan
\end{aligned}  
\end{equation}
For the first line,  Lemma \ref{property of Ak} (\rmnum 1) gives
\begin{equation*}
\begin{aligned}
   &\left(|b_k'|+\frac{b_k|\lambda_k'|}{\lambda_k}+\frac{b_k|\boldsymbol{y}_k'|}{\lambda_k}\right)\left(\left\|\underline{A}_k\dot{g}\right\|_{L^2}+\left\|\lambda_k\partial_{\lambda_k}\underline{A}_k\dot{g}\right\|_{L^2}+\left\|\lambda_k \partial_{y_k}\underline{A}_k\dot{g}\right\|_{L^2}\right)\left\|\ML_{\phi}h\right\|_{L^2}\\
   \lesssim&t^{\frac{1}{2}}\lambda\left\|\dot{g}\right\|_{\dot{H}^1}\left\|\ML_{\phi}h\right\|_{L^2}\lesssim \lambda^{\frac{5}{2}},   
\end{aligned}
\end{equation*}
and is thus negligible. 

Next, recall the definition of $\Psi(\boldsymbol{\lambda},\boldsymbol{b},{\boldsymbol{y}},{\boldsymbol{v}})$ and $\epsilon(h,\dot{g})$, the fourth term decomposes as
\begin{equation*}
    \begin{aligned}
       &2b_k\lelan \underline{A}_k\partial_t\dot{g}, \ML_{\phi}h \rilan
       \simeq -2b_k\lelan \underline{A}_k\ML_{\phi}h, \ML_{\phi}h \rilan+2b_k\lelan \underline{A}_k \Psi(\boldsymbol{\lambda},\boldsymbol{b},{\boldsymbol{y}},{\boldsymbol{v}}), \ML_{\phi}h  \rilan+2b_k\lelan \underline{A}_k{\epsilon}(h,\dot{g}), \ML_{\phi}h \rilan. 
    \end{aligned}
\end{equation*}
The first term vanishes by the skew-adjointness identity $\lelan \underline{A}_kf, f \rilan=0$. For  the term involving $\Psi(\boldsymbol{\lambda},\boldsymbol{b},{\boldsymbol{y}},{\boldsymbol{v}})$, denote
\begin{equation*}
    \Psi_1(\boldsymbol{\lambda},\boldsymbol{b},{\boldsymbol{y}},{\boldsymbol{v}}):= \Psi(\boldsymbol{\lambda},\boldsymbol{b},{\boldsymbol{y}},{\boldsymbol{v}})+\sum_{j=1}^N\frac{b_j'}{\lambda_j}\Lambda_jW_j\chi_t.
\end{equation*}
By (\ref{lambda'+b}), (\ref{y'+v}) and (\rmnum{1}) of Lemma \ref{property of Ak}, we have
\begin{equation*}
\begin{aligned}
  &\left\|\underline{A}_k \Psi_1(\boldsymbol{\lambda},\boldsymbol{b},{\boldsymbol{y}},{\boldsymbol{v}})\right\|_{L^2}\lesssim \left\| \Psi_1(\boldsymbol{\lambda},\boldsymbol{b},{\boldsymbol{y}},{\boldsymbol{v}})\right\|_{\dot{H}^1}\\\lesssim& \left(\frac{|\lambda_j'+b_j|b_j}{\lambda_j^2}+\frac{|\boldsymbol{y}_j'+\boldsymbol{v}_j|b_j}{\lambda_j^2}+\frac{|b_j\lambda_j'|}{2\log(t/\lambda_j)\lambda_j^2}+\frac{|B_j(\boldsymbol{\lambda})|}{\log(t/\lambda_j)\lambda_j}\right)\lesssim \frac{1}{t^{\frac{2}{3}}}.
\end{aligned} 
\end{equation*}
Therefore, Cauchy inequality yields
\begin{equation*}
    \left|2b_k\lelan\underline{A}_k \Psi_1(\boldsymbol{\lambda},\boldsymbol{b},{\boldsymbol{y}},{\boldsymbol{v}}),\ML_{\phi}h  \rilan\right|\lesssim |b_k|\left\|\underline{A}_k \Psi_1(\boldsymbol{\lambda},\boldsymbol{b},{\boldsymbol{y}},{\boldsymbol{v}})\right\|_{L^2}\left\|\ML_{\phi}h\right\|_{L^2}\lesssim t^{\frac{1}{2}}\lambda^2.
\end{equation*}
Furthermore, using $\underline{A}_k=\frac{1}{\lambda_k}\underline{\Lambda}$ when $|x-\boldsymbol{y}_k|\leq \lambda_k R$, we deduce
\begin{equation*}
    \begin{aligned}
        -2\sum_{j=1}^N\frac{b_kb_j'}{\lambda_j}\lelan \underline{A}_k\Lambda_jW_j,\ML_{\phi}h \rilan\simeq&-\frac{2b_kb_k'}{\lambda_k^2}\int_{|x-\boldsymbol{y}_k|\leq \lambda_kR}(\underline{\Lambda}\Lambda W)\ML_{\phi}hdx\\
        &-\frac{2b_kb_k'}{\lambda_k^2}\int_{|x-\boldsymbol{y}_k|\geq \lambda_kR}(\underline{A}_k\Lambda_kW_k)\ML_{\phi}hdx
    \end{aligned}
\end{equation*}
For the interior region, the argument used in (\ref{L phi 2 h partial t(h-g)}) gives the lower bound
\begin{equation*}
 -\frac{2b_kb_k'}{\lambda_k^2}\int_{|x-\boldsymbol{y}_k|\leq \lambda_kR}(\underline{\Lambda}\Lambda W)\ML_{\phi}hdx\geq  - C \frac{\left\|\ML_\phi h\right\|^2_{L^2(|x-\boldsymbol{y}_k|\leq \lambda_kR)}}{t^{\frac{1}{3}}\log M}. 
\end{equation*}
For the exterior region, the scaling invariance and (\ref{A and Lambda W nabla W}) provide 
\begin{equation*}
  -\frac{2b_kb_k'}{\lambda_k^2}\int_{|x-\boldsymbol{y}_k|\geq \lambda_kR}(\underline{A}_k\Lambda_kW_k)\ML_{\phi}hdx\geq -C\left(\eta+\frac{1}{R}\right)  \frac{\left\|\ML_\phi h\right\|^2_{L^2}}{t^{\frac{1}{3}}}.
\end{equation*}
Consequently, we conclude
\begin{equation*}
   2b_k\lelan \underline{A}_k \Psi(\boldsymbol{\lambda},\boldsymbol{b},{\boldsymbol{y}},{\boldsymbol{v}}), \ML_{\phi}h  \rilan \geq - C \frac{\left\|\ML_\phi h\right\|^2_{L^2(|x-\boldsymbol{y}_k|\leq \lambda_kR)}}{t^{\frac{1}{3}}\log M} -C\left(\eta+\frac{1}{R}\right)  \frac{\left\|\ML_\phi h\right\|^2_{L^2}}{t^{\frac{1}{3}}}.
\end{equation*}
Now we turn to the term involving $\epsilon(h,\dot{g})$. Recall the definition of $\epsilon(h,\dot{g})$ and the equation of $\PG$, it suffices to consider the term 
\begin{equation*}
    \sum_{j=1}^N\sum_{l=1}^42b_k\lelan \underline{A}_k \left(-\frac{v'_{j,l}(\partial_lW)(\frac{\cdot-\boldsymbol{y}_j}{\lambda_j})}{\lambda_j^2}
   +\frac{v_{j,l}(t)y'_{j,m}(t)(\partial_l\partial_mW)(\frac{\cdot-\boldsymbol{y}_j}{\lambda_j})}{\lambda_j^3} \right),\ML_{\phi}h\rilan.
\end{equation*}
From (\ref{y'+v}) and (\rmnum{1}) of Lemma \ref{property of Ak}, it remains to estimate
\begin{equation*}
    -\frac{2b_k}{\lambda_k^2}\lelan \underline{A}_k\left(v'_{j,l}(\partial_lW)\left(\frac{\cdot-\boldsymbol{y}_j}{\lambda_j}\right)\right),\ML_{\phi}h \rilan ,
\end{equation*}
Using the same argument as above and (\ref{slight refine of v'}), this term is bounded below by 
\begin{equation*}
   - C \frac{\left\|\ML_\phi h\right\|^2_{L^2(|x-\boldsymbol{y}_k|\leq \lambda_kR)}}{t^{\frac{1}{3}}M} -C\left(\eta+\frac{1}{R}\right)  \frac{\left\|\ML_\phi h\right\|^2_{L^2}}{t^{\frac{1}{3}}}.  
\end{equation*}
Collecting all these estimates, we arrive at the estimate for the fourth term:
\begin{equation*}\label{jk2' 4th term}
    2b_k\lelan \underline{A}_k\partial_t\dot{g}, \ML_{\phi}h \rilan\geq  - C \frac{\left\|\ML_\phi h\right\|^2_{L^2(|x-\boldsymbol{y}_k|\leq \lambda_kR)}}{t^{\frac{1}{3}}\log M} -C\left(\eta+\frac{1}{R}\right)  \frac{\left\|\ML_\phi h\right\|^2_{L^2}}{t^{\frac{1}{3}}}.
\end{equation*}

The fifth term on the right hand side of (\ref{J_k'2}) is dominated by
\begin{equation*}
    \left|2b_k \lelan \underline{A}_k\dot{g}, [\partial_t,\ML_{\phi}]h\rilan\right|= |12b_k\lelan \underline{A}_k\dot{g},(\partial_t\phi\phi) h\rilan|\lesssim |b_k|\left\|\dot{g}\right\|_{\dot{H}^1}\left\|h\right\|_{\dot{H}^2}\lesssim \lambda^{\frac{5}{2}},
\end{equation*}
and is thus negligible. Finally, inserting (\ref{g-dot g}) into the last term of (\ref{J_k'2}) we get
\begin{equation}\label{last of jk2'}
    \begin{aligned}
        2b_k\lelan \underline{A}_k\dot{g}, \ML_{\phi}\partial_th \rilan=&2b_k\lelan \underline{A}_k\dot{g}, \ML_{\phi}\dot{g} \rilan-2b_k\lelan \underline{A}_k\dot{g}, \ML_{\phi}\partial_t\PG \rilan\\
        +&2b_k\lelan \underline{A}_k\dot{g}, \ML_{\phi} \left(\sum_{j=1}^N \lambda_j^{-1}(\lambda_j'\Lambda_jW_j+b_j\Lambda_j W_j\chi_{t})\right)\rilan\\
        +&2b_k\lelan \underline{A}_k\dot{g},\ML_{\phi}\left(\sum_{j=1}^N \lambda_j^{-1}(\boldsymbol{y}_j'+\boldsymbol{v}_j)\cdot\nabla_j W_j\right) \rilan.
    \end{aligned}
\end{equation}
The last two terms are negligible due to (\ref{lambda'+b}), (\ref{y'+v}), and the $\Lambda W,\nabla W\in \ker\ML$. The argument is similar to (\ref{ML phi on trunction}) and we omit it. The term involving $\partial_t\PG$ is bounded by
\begin{equation*}
    \left|2b_k\lelan \underline{A}_k\dot{g}, \ML_{\phi}\partial_tP \rilan\right|\lesssim b_k\left\|\underline{A}_k\dot{g}\right\|_{L^2}\left\|\partial_t \PG\right\|_{{H}^2}\lesssim \lambda^{\frac{5}{2}}.
\end{equation*}
It remains to treat the first term on the right-hand side of (\ref{last of jk2'}). 
This term should be combined with the contribution obtained in (\ref{J k1' est and cancel}). 
Indeed, (\ref{J k1' est and cancel}) contains the opposite term $- b_k\langle \underline{A}_k\dot g,\ML_\phi\dot g\rangle$, while (\ref{last of jk2'}) contains two copies of it. Hence, after summing the estimates for $\mathcal J'_{k,1}$ and $\mathcal J'_{k,2}$ in the derivative of $\mathcal H_2$, only one copy remains. It is therefore enough to estimate
\begin{equation}\label{bk Ak dot g L phi dot g}
      b_k\lelan \underline{A}_k\dot{g}, \ML_{\phi}\dot{g} \rilan=-b_k\lelan \underline{A}_k\dot{g},\Delta \dot{g} \rilan-\frac{b_k}{4\lambda_k}\lelan \Delta q\left(\frac{x-\boldsymbol{y}_k}{\lambda_k}\right)\dot{g}, 3\phi^2 \dot{g}  \rilan-b_k\lelan A_k \dot{g},3\phi^2 \dot{g} \rilan.
\end{equation}
From (\ref{Ak key property 3}), the first two terms on the right hand side are bounded  below  by
\begin{equation*}
    \begin{aligned}
      & -\eta \frac{b_k}{\lambda_k}\left\|\dot{g}\right\|_{\dot{H}^1}^2+\frac{b_k}{\lambda_k}\left\{\int_{|x-\boldsymbol{y}_k|<R\lambda_k}|\nabla \dot{g}(x)|^2dx-\frac{1}{4}\lelan \Delta q(\frac{x-\boldsymbol{y}_k}{\lambda_k})\dot{g}, 3\phi^2 \dot{g}  \rilan \right\}.
    \end{aligned}
\end{equation*}
Furthermore, we claim 
\begin{equation}\label{Delta q 3 phi 2 claim}
    \left|\frac{1}{4}\lelan \Delta q(\frac{x-\boldsymbol{y}_k}{\lambda_k})\dot{g}, 3\phi^2 \dot{g}  \rilan-\int_{|x-\boldsymbol{y}_k|\leq \lambda_k R}3W_k^2\dot{g}^2(x)dx\right|\lesssim \frac{1}{R^2}t\lambda^2.
\end{equation}
To prove the claim, first from the smallness of $\left\|\PG\right\|_{{H}^2}$, we obtain
\begin{equation*}
    \left|\lelan \Delta q(\frac{x-\boldsymbol{y}_k}{\lambda_k})\dot{g}, \PG^2 \dot{g}  \rilan\right|\lesssim \left\|\dot{g}\right\|_{L^2}^2\left\|\PG\right\|_{H^2}\lesssim \lambda^{\frac{5}{2}}.
\end{equation*}
Hence, we can replace the $3\phi^2$ in (\ref{Delta q 3 phi 2 claim}) with $3(\sum_{j=1}^NW_j)^2$. Then recall that when $|x-\boldsymbol{y}_k|\geq \tilde{R}\lambda_k$, $\Delta q(\frac{x-\boldsymbol{y}_k}{\lambda_k})=O({\lambda_k^2}/{|x-\boldsymbol{y}_k|^2})$. It holds
\begin{equation*}
    \begin{aligned}
    \left| \int_{|x-\boldsymbol{y}_k|\geq \tilde{R}\lambda_k}\Delta q(\frac{x-\boldsymbol{y}_k}{\lambda_k})\dot{g}^2W_k^2dx\right|\lesssim \int_{|x-\boldsymbol{y}_k|\geq \tilde{R}\lambda_k}\frac{\lambda_k^4}{|x-\boldsymbol{y}_k|^4} \frac{\dot{g}^2}{|x-\boldsymbol{y}_k|^2}dx\lesssim \frac{1}{\tilde{R}^4}\left\|\dot{g}\right\|_{\dot{H}^1},  
    \end{aligned}
\end{equation*}
where Hardy's inequality is used. Moreover, for $j\neq k$, the main contribution lies in the region $|x-\boldsymbol{y}_j|\leq d$ and is bounded by
\begin{equation*}
    \left| \int_{|x-\boldsymbol{y}_j|\leq d\lambda_k}\Delta q(\frac{x-\boldsymbol{y}_k}{\lambda_k})\dot{g}^2W_j^2dx\right|\lesssim \int \frac{\lambda_k^2}{|\boldsymbol{y}_k-\boldsymbol{y}_j|^2\lambda_j^2}\dot{g}^2dx\lesssim \left\|\dot{g}\right\|_{L^2}.
\end{equation*}
In the interior region $|x-\boldsymbol{y}_k|\leq \tilde{R}$, since $|\Delta q|\leq C$, the multi-bubble interaction is dominated by
\begin{equation*}
    \left|\int_{|x-\boldsymbol{y}_k|\leq \tilde{R}\lambda_k}\Delta q(\frac{x-\boldsymbol{y}_k}{\lambda_k})\dot{g}^2W_jW_kdx \right|\lesssim \int\frac{\lambda_j}{|\boldsymbol{y}_k-\boldsymbol{y}_j|^2\lambda_k}\dot{g}^2dx\lesssim \left\|\dot{g}\right\|_{L^2}^2.
\end{equation*}
Also notice $\Delta q(x)=4$ when $|x|\leq R$, Cauchy inequality then yields
\begin{equation*}
    \begin{aligned}
        &\left|\frac{1}{4}\lelan \Delta q(\frac{x-\boldsymbol{y}_k}{\lambda_k})\dot{g}, 3W_k^2 \dot{g}  \rilan-\int_{|x-\boldsymbol{y}_k|\leq \lambda_k R}3W_k^2\dot{g}^2(x)dx\right|\\\lesssim& \left|\int_{R\lambda_k\leq|x-\boldsymbol{y}_k|\leq \tilde{R}\lambda_k}\Delta q(\frac{x-\boldsymbol{y}_k}{\lambda_k})\dot{g}^2W_k^2dx\right|
        \lesssim \left\|\dot{g}\right\|_{\dot{H}^1}^2\left\|W_k\right\|_{L^4(|x-\boldsymbol{y}_k|\geq R\lambda_k)}^2\leq \frac{1}{R^2}\left\|\dot{g}\right\|_{\dot{H}^1}^2.
    \end{aligned}
\end{equation*}
Collecting these estimates, the claim (\ref{Delta q 3 phi 2 claim}) is proved. Plugging the claim back into (\ref{bk Ak dot g L phi dot g}) and using Lemma \ref{Lemma 9 of 5D multi}, we derive
\begin{equation*}
    \begin{aligned}
     &-b_k\lelan \underline{A}_k\dot{g},\Delta \dot{g} \rilan-\frac{b_k}{4\lambda_k}\lelan \Delta q(\frac{x-\boldsymbol{y}_k}{\lambda_k})\dot{g}, 3\phi^2 \dot{g}  \rilan\\\geq&     \frac{b_k}{\lambda_k}\int_{|x-\boldsymbol{y}_k|
     \leq \lambda_kR}\left(|\nabla \dot{g}|^2-3W_k^2\dot{g}\right)dx-C\left(\eta+\frac{1}{R}\right)t^{\frac{2}{3}}\lambda^2.
    \end{aligned}
\end{equation*}
Applying Lemma \ref{Lemma 9 of 5D multi} and the bootstrap assumption of stable/unstable direction (\ref{bootstrap a for unstable})-(\ref{bootstrap a for stable}), the first term on the right hand side can be refined as 
\begin{equation}\label{Jk1'.1 first line}
   \begin{aligned}
  &   \frac{b_k}{\lambda_k}\int_{|x-\boldsymbol{y}_k|
     \leq \lambda_kR}\left(|\nabla \dot{g}|^2-3W_k^2\dot{g}\right)dx
    \geq\frac{b_k}{\lambda_k}\left(-\eta \left\|\dot{g}\right\|_{\dot{H}^1}^2+\lelan \frac{1}{\lambda_k^2}Y_k,\dot{g} \rilan^2\right)\\
    \geq& -\eta t^{\frac{2}{3}}\lambda^2-\frac{b_k}{\lambda_k}\left(\frac{a_k^{\pm}}{\lambda_k}\right)^2\geq -\eta t^{\frac{2}{3}}\lambda^2-\frac{t^{\frac{1}{2}}\lambda^4}{t^{\frac{1}{3}}\lambda^2}\geq  -\eta t^{\frac{2}{3}}\lambda^2.
   \end{aligned}
\end{equation}
Returning to (\ref{bk Ak dot g L phi dot g}), it remains to handle the term $-b_k\lelan A_k \dot{g},3\phi^2 \dot{g} \rilan$. From (\ref{Ak key property 1}), we obtain
\begin{equation*}
    -b_k\lelan A_k \dot{g},3\phi^2 \dot{g} \rilan=b_k\lelan {A}_k \dot{g},3\phi\dot{g}^2 +\dot{g}^3\rilan+b_k\lelan A_k\phi,3\phi\dot{g}^2+\dot{g}^3 \rilan.
\end{equation*}
The definition of $A_k$ and Lemma \ref{function q} (7) imply
\begin{equation*}
    \begin{aligned}
      \left|b_k\lelan {A}_k \dot{g},3\phi\dot{g}^2 +\dot{g}^3\rilan+b_k\lelan A_k\phi,\dot{g}^3 \rilan\right|\lesssim |b_k|\frac{1}{\lambda_k}\left(\left\|\dot{g}\right\|_{\dot{H}^1}^3+\left\|\dot{g}\right\|_{\dot{H}^1}^4\right)\lesssim \lambda^{\frac{5}{2}}.  
    \end{aligned}
\end{equation*}
For the remaining part, we claim
\begin{equation}\label{[partial t, L phi] and the Jk1'}
    \frac{1}{2}\lelan [\partial_t,\ML_{\phi}]\dot{g}, \dot{g} \rilan+\sum_{k=1}^Nb_k\lelan A_k \phi, 3\phi\dot{g}^2 \rilan \geq -\frac{C}{R}{t^{\frac{2}{3}}\lambda^2}.
\end{equation}
To prove this, first a direct computation provides
\begin{equation*}
    \frac{1}{2}\lelan [\partial_t,\ML_{\phi}]\dot{g}, \dot{g} \rilan=\sum_{k=1}^N \frac{\lambda_k'}{\lambda_k}\lelan \Lambda_k W_k, 3\phi\dot{g}^2 \rilan +\sum_{k=1}^N\frac{1}{\lambda_k}\lelan \boldsymbol{y}_k'\cdot \nabla_kW_k, 3\phi\dot{g}^2 \rilan-\lelan \partial_t \PG, 3\phi\dot{g}^2 \rilan.
\end{equation*}
The bootstrap assumption for $\boldsymbol{v}$ together with (\ref{y'+v}) and (\ref{partial t P h dot 1}) yields that
\begin{equation*}
   \left|\sum_{k=1}^N\frac{1}{\lambda_k}\lelan \boldsymbol{y}_k'\cdot \nabla_kW_k, 3\phi\dot{g}^2 \rilan-\lelan \partial_t \PG, 3\phi\dot{g}^2 \rilan\right| \lesssim \frac{|\boldsymbol{y}_k'|}{\lambda_k} \left\|\dot{g}\right\|_{\dot{H}^1}^2+\left\|\partial_t\PG\right\|_{L^4}\left\|\dot{g}\right\|_{\dot{H}^1}^2\lesssim \lambda^{\frac{5}{2}}.
\end{equation*}
Therefore, the problem is reduced to proving
\begin{equation*}
    \sum_{k=1}^N\lelan b_k A_k\phi+\frac{\lambda_k'}{\lambda_k}\Lambda_kW_k , 3\phi\dot{g}^2\rilan\gtrsim -\frac{1}{R}t^{\frac{2}{3}}\lambda^2.
\end{equation*}
Since $A_k=\frac{1}{\lambda_k}\Lambda$ when $|x-\boldsymbol{y}_k|\leq R\lambda_k$, the left-hand side is thus bounded by
\begin{equation*}
    \left| \int_{|x-\boldsymbol{y}_k|\leq R\lambda_k}\frac{\lambda_k'+b_k}{\lambda_k}\Lambda_kW_k, 3\phi \dot{g}^2dx\right|\lesssim \frac{|\lambda_k'+b_k|}{\lambda_k}\left\|\dot{g}\right\|_{\dot{H}^1}\lesssim \lambda^{\frac{5}{2}}
\end{equation*}
in the interior region. On the exterior region  $|x-\boldsymbol{y}_k|\geq R\lambda_k$, from (7) of Lemma \ref{function q} and Cauchy, it holds
\begin{equation*}
    \begin{aligned}
        &\left|\int_{|x-\boldsymbol{y}_k|\geq R\lambda_k}\left(b_kA\phi+\frac{\lambda_k'}{\lambda_k}\Lambda_kW_k\right)3\phi\dot{g}^2dx\right|\\
        \lesssim& \frac{|\lambda_k'|+|b_k|}{\lambda_k}\left(\int_{|x-\boldsymbol{y}_k|\geq R\lambda_k}\left(\frac{\lambda_k}{|x|^2}\right)^4dx\right)^{\frac{1}{4}}\left\|\phi\right\|_{\dot{H}^1}\left\|\dot{g}\right\|_{\dot{H}^1}^2\lesssim \frac{1}{R}t^{\frac{2}{3}}\lambda^2.
    \end{aligned}
\end{equation*}
Consequently,  (\ref{[partial t, L phi] and the Jk1'}) is proved. Now inserting (\ref{Jk1'.1 first line}) and (\ref{[partial t, L phi] and the Jk1'}) back into (\ref{bk Ak dot g L phi dot g})
\begin{equation*}\label{-bk Ak Lphi with 1/2 [,Lphi]}
    \sum_{k=1}^Nb_k\lelan \underline{A}_k\dot{g}, \ML_{\phi}\dot{g} \rilan+\frac{1}{2}\lelan [\partial_t,\ML_{\phi}]\dot{g}, \dot{g} \rilan\geq-C\left(\eta+ \frac{1}{R}\right)t^{\frac{2}{3}}\lambda^2.
\end{equation*}

Combining the estimates above, we obtain
\begin{equation}\label{Jk 2' final estmate}
\begin{aligned}
    &\sum_{k=1}^N\mathcal{J}_{k,2}'-\sum_{k=1}^Nb_k\lelan \underline{A}_k\dot{g}, \ML_{\phi}\dot{g} \rilan+\frac{1}{2}\lelan [\partial_t,\ML_{\phi}]\dot{g}, \dot{g}\ \rilan \\
    \geq& - C \frac{\left\|\ML_\phi h\right\|^2_{L^2(|x-\boldsymbol{y}_k|\leq \lambda_kR)}}{t^{\frac{1}{3}}\log M} -C\left(\eta+ \frac{1}{R}\right)t^{\frac{2}{3}}\lambda^2.
\end{aligned}
\end{equation}

\textbf{Step 4. Analysis of $\mathcal{H}_2'$.} Combining (\ref{final est for I2'}), (\ref{J k1' est and cancel}) and (\ref{Jk 2' final estmate}) together yields
\begin{equation*}
    \begin{aligned}
        \mathcal{I}_2'+\sum_{k=1}^N\mathcal{J}_{k,1}'+\sum_{k=1}^N\mathcal{J}_{k,2}'\geq& \sum_{k=1}^N\frac{2b_k}{\lambda_k}\left\|\ML_\phi h\right\|^2_{L^2(|x-\boldsymbol{y}_k|\leq \lambda_kR)} -C\frac{\left\|\ML_\phi h\right\|^2_{L^2(|x-\boldsymbol{y}_k|\leq \lambda_kR)}}{t^{\frac{1}{3}}\log M}\\
        &-C \frac{\left\|\ML_\phi h\right\|^2_{L^2(|x-\boldsymbol{y}_k|\leq \lambda_kR)}}{t^{\frac{1}{3}}M}-C\left(\eta+ \frac{1}{R}\right){t^{\frac{2}{3}}\lambda^2}.
    \end{aligned}
\end{equation*}
From (\ref{bootstrap a for b}) and (\ref{lambda'+b}), $\frac{2b_k}{\lambda_k}\geq \frac{1}{3}t^{-\frac{1}{3}}$. Then by choosing $M$, which comes from orthogonal condition (\ref{orthogonal condition 1}) and (\ref{orthogonal condition 2}), be a large constant such that
\begin{equation*}
    \frac{C}{M}+\frac{C}{\log M}\leq \frac{1}{10},
\end{equation*}
it holds that
\begin{equation*}
    \mathcal{H}'_2(t)\geq -C\eta{t^{\frac{2}{3}}\lambda(t)^2}.
\end{equation*}
which completes the proof of (\ref{H 2}) provided $R$ is chosen sufficiently large and $\eta$ sufficiently small.
\end{proof}
We are now ready to close the energy part of the bootstrap. The argument is the
standard backward-in-time integration of the modified energies.

\begin{lemma}\label{final boots for energy}
For all \(t\in[T_*,T]\), it holds
\[
\|\vec g(t)\|_{\mathcal E}\le \frac12 t^{-\frac16}\lambda(t),
\qquad
\|\vec g(t)\|_{\mathcal E_2}\le \frac12 t^{\frac12}\lambda(t).
\]
\end{lemma}

\begin{proof}
We first record an elementary consequence of the scale and velocity bootstrap.
By the definition of \(\alpha\),
\[
\alpha'=-\frac{\lambda'}{\lambda}.
\]
Using the modulation estimates for \(\lambda'+b\) and \(\beta-b\), together
with \(m=\beta/\lambda\), we have
\[
\alpha'=m+O(t^{-\frac56}\sqrt{\log t}).
\]
The bootstrap estimate on \(m-\mu\), together with \(\mu(t)\sim t^{-1/3}\),
therefore gives
\[
\alpha'(t)\ge c_0 t^{-\frac13},
\]
where $c_0>0$, for \(T_0\) sufficiently large. Hence, for \(t\le s\le T\),
\[
\lambda(s)^2
=
\lambda(t)^2\exp\left(-2\int_t^s\alpha'(\tau)\,d\tau\right)
\le
\lambda(t)^2 e^{-c(s^{2/3}-t^{2/3})}.
\]
Consequently, for any fixed \(a\),
\begin{equation}\label{6.14}
  \int_t^T s^a\lambda(s)^2\,ds
\lesssim
t^{a+\frac13}\lambda(t)^2.
\end{equation}

We next recall the comparison between the modified energies and the remainder
norms. By the coercivity estimates, the orthogonality conditions, and the
estimates on the correction profiles, the lower order terms in the modified
energies are perturbative. Thus, for \(T_0\) sufficiently large,
\begin{equation}\label{6.15}
  \|\vec g(t)\|_{\mathcal E}^2
\lesssim
\mathcal I_1(t)+(\widetilde {a}^+(t))^2+(\widetilde {a}^-(t))^2+o\bigl(t^{-\frac13}\lambda(t)^2\bigr),  
\end{equation}
and
\begin{equation}\label{6.16}
\|\vec g(t)\|_{\mathcal E_2}^2
\lesssim
\mathcal H_2(t)+\frac{( a^+(t))^2}{\lambda(t)^2}+\frac{( a^-(t))^2}{\lambda(t)^2}
+o\bigl(t\lambda(t)^2\bigr).   
\end{equation}
Moreover, by the bootstrap assumption on \(a^\pm\) and the estimate comparing
\(a^\pm\) with \(\widetilde a^\pm\), the unstable contributions satisfy
\begin{equation}\label{6.17}
  (\widetilde a^\pm(t))^2=o\bigl(t^{-\frac13}\lambda(t)^2\bigr),
\qquad
\frac{(a^+(t))^2}{\lambda(t)^2}
=o\bigl(t\lambda(t)^2\bigr).  
\end{equation}

We now integrate the first order energy estimate. Lemma \ref{first order energy est lemma}  gives
\[
\mathcal I_1'(t)\gtrsim -t^{-\frac34}\lambda(t)^2.
\]
Therefore, using \((\ref{6.14})\),
\[
\mathcal I_1(t)
\le
\mathcal I_1(T)+C\int_t^T s^{-\frac34}\lambda(s)^2\,ds
\le
\mathcal I_1(T)+C t^{-\frac5{12}}\lambda(t)^2.
\]
By the terminal data estimate in Lemma \ref{choice of u(T)}, \(\mathcal I_1(T)\) is lower order
with respect to \(t^{-1/3}\lambda(t)^2\). Since
\[
t^{-\frac5{12}}\lambda(t)^2=o\bigl(t^{-\frac13}\lambda(t)^2\bigr),
\]
we get, after increasing \(T_0\),
\[
\mathcal I_1(t)\le o\bigl(t^{-\frac13}\lambda(t)^2\bigr).
\]
Combining this with \((\ref{6.15})\) and \((\ref{6.17})\), we obtain
\[
\|\vec g(t)\|_{\mathcal E}^2
\le
\frac14 t^{-\frac13}\lambda(t)^2.
\]

For the second order energy, Lemma 6.4 gives, for every \(\eta>0\), after
choosing the auxiliary parameter \(\epsilon>0\) sufficiently small,
\[
\mathcal H_2'(t)\geq -C\eta\, t^{\frac23}\lambda(t)^2.
\]
Thus, by (\ref{6.14}),
\[
\mathcal H_2(t)
\le
\mathcal H_2(T)
+C\eta\int_t^T s^{\frac23}\lambda(s)^2\,ds
\le
\mathcal H_2(T)+C\eta\, t\lambda(t)^2.
\]
Again the terminal contribution \(\mathcal H_2(T)\) is lower order compared
with \(t\lambda(t)^2\). Choosing \(\eta>0\) sufficiently small and then
increasing \(T_0\), we obtain
\[
\mathcal H_2(t)\le o\bigl(t\lambda(t)^2\bigr).
\]
Using \((\ref{6.16})\)--\((\ref{6.17})\), this yields
\[
\|\vec g(t)\|_{\mathcal E_2}^2
\le
\frac14 t\lambda(t)^2.
\]
Taking square roots proves the lemma.
\end{proof}

\section{Proof of Theorem \ref{multi-bubble solution}}
\subsection{Proof of Proposition \ref{main step}.}
To prove our main result, we need to first close all the bootstrap estimates in Proposition \ref{main step}. The estimates (\ref{bootstap a for energy}) and (\ref{bootstap a for 2nd energy}) were proved in the last section. We next consider the scaling velocity parameter, the translation parameters and the stable direction. 
\begin{lemma}\label{close of the scaling and trans para}For all  $t\in [T_*,T]$ it holds
\begin{equation}\label{bootstrap a for b close time}
    |m(t)-\mu(t)|\leq \frac{C}{2}t^{-\frac{2}{3}},
\end{equation}
\begin{equation}\label{bootstrap a for y close time}
   |\boldsymbol{y}_k(t)-\boldsymbol{z}_k|\leq \frac{1}{2}t^{\frac{5}{2}}\lambda(t)^2, 
\end{equation}
\begin{equation}\label{bootstrap a for v close time}
  |\boldsymbol{v}_k(t)|\leq \frac{1}{2}t^{\frac{7}{6}}\lambda(t)^2,
  \end{equation}
and
\begin{equation}\label{close of a+}
   (a^+(t))^2\leq \frac{1}{2} t^{\frac{1}{2}}\lambda(t)^4,
\end{equation} 
\end{lemma}
\begin{proof}
We first improve the bootstrap estimate (\ref{bootstrap a for b}). Recall that
\[
m(t)=\frac{\beta(t)}{\lambda(t)},\qquad \mu(t)=\left(\frac{c}{\Xi(t)}\right)^{1/2}.
\]
We claim that, on \([T_*,T]\),
\begin{equation}\label{m and mu refine}
   \left|m'-m^2+\mu^2\right|
\lesssim t^{-\frac76}\sqrt{\log t}. 
\end{equation}
Indeed, since
\[
\alpha'=-\frac{\lambda'}{\lambda},
\]
the estimates (\ref{lambda'+b}) and (\ref{b and beta}) give
\begin{equation*}
   \alpha'
=
\frac{b}{\lambda}+O(t^{\frac12}\lambda)
=
\frac{\beta}{\lambda}+O(t^{-\frac56}\sqrt{\log t})
=
m+O(t^{-\frac56}\sqrt{\log t}). 
\end{equation*}
Hence
\[
\frac{\beta'}{\lambda}
=
m'-m\alpha'
=
m'-m^2+O(t^{-\frac76}\sqrt{\log t}),
\]
where we used \(m=O(t^{-1/3})\), which follows from the bootstrap bound
\(|m-\mu|\le Ct^{-2/3}\) and \(\mu\sim t^{-1/3}\).

Dividing (\ref{beta'+=to}) by \(\lambda\), and using the last identity, we obtain
\[
\Xi(t)(m'-m^2)+\frac12 m^2+c
=
O(t^{-\frac12}\sqrt{\log t}).
\]
Since \(\Xi(t)\sim t^{2/3}\), this gives
\[
m'-m^2+\frac{c}{\Xi(t)}
=
O(t^{-\frac76}\sqrt{\log t})
+
O\left(\frac{m^2}{\Xi(t)}\right).
\]
The last term is \(O(t^{-4/3})\), hence is absorbed by the right-hand side.
Since \(c/\Xi=\mu^2\), the claim (\ref{m and mu refine}) follows. Set
\[
w(t):=m(t)-\mu(t).
\]
By the definition of \(\mu\), and using
\[
\Xi(t)=\alpha(t)+O(\log t),\qquad
\alpha(t)=\xi t^{2/3}+O(t^{1/3}),
\]
we have
\[
\mu(t)\sim t^{-1/3},\qquad |\mu'(t)|\lesssim t^{-4/3}.
\]
Combining this with (\ref{m and mu refine}), we obtain
\begin{equation*}\label{w' equal}
  w'
=
m'-\mu'
=
m^2-\mu^2+O(t^{-\frac76}\sqrt{\log t})-\mu'
=
(m+\mu)w+O(t^{-\frac76}\sqrt{\log t}).  
\end{equation*}
Moreover, by the bootstrap assumption (\ref{bootstrap a for b}),
\[
m+\mu\ge c_0 t^{-1/3}
\]
where $c_0>0$ is a constant for \(T_0\) sufficiently large. By (\ref{m T- mu T}), we have
\[
w(T)=m(T)-\mu(T)=O(T^{-1}\log T).
\]
Solving the differential inequality backward on \([t,T]\), we get
\[
|w(t)|
\lesssim
e^{-c_0(T^{2/3}-t^{2/3})}|w(T)|
+
\int_t^T
e^{-c_0(s^{2/3}-t^{2/3})}
s^{-\frac76}\sqrt{\log s}\,ds .
\]
The integral is bounded by
\[
\int_t^T
e^{-c_0(s^{2/3}-t^{2/3})}
s^{-\frac76}\sqrt{\log s}\,ds
\lesssim
t^{-\frac56}\sqrt{\log t}.
\]
Therefore
\begin{equation*}
    |m(t)-\mu(t)|
\lesssim
t^{-\frac56}\sqrt{\log t}
+
e^{-c(T^{2/3}-t^{2/3})}O(T^{-1}\log T)\lesssim t^{-\frac{3}{4}}.
\end{equation*}
Thus, for \(T_0\) sufficiently large,
\[
|m(t)-\mu(t)|\le \frac C2 t^{-2/3},
\]
which proves (\ref{bootstrap a for b close time}) and  closes the bootstrap estimate (\ref{bootstrap a for b}). 

It remains to close the translation estimates. From (\ref{m and mu refine}), (\ref{lambda'+b}), and
(\ref{b and beta}), we have
\begin{equation}\label{alpha' and m}
   \alpha'(t)
=
m(t)+O(t^{-\frac56}\sqrt{\log t})
=
\mu(t)+O(t^{-\frac34})
\gtrsim t^{-\frac13}. 
\end{equation}
Hence, for \(t\le \tau\le T\),
\[
\lambda(\tau)^2
\lesssim
\lambda(t)^2 e^{-c(\tau^{2/3}-t^{2/3})}.
\]
Consequently, for any fixed \(a\),
\begin{equation}\label{integration kernel}
 \int_t^T \tau^a\lambda(\tau)^2\,d\tau
\lesssim
t^{a+\frac13}\lambda(t)^2,
\qquad
T^a\lambda(T)^2\lesssim t^a\lambda(t)^2.   
\end{equation}
By Lemma \ref{close the bootstrap for parameters system} , we have
\[
|\boldsymbol v_k-\boldsymbol s_k|\lesssim t\lambda(t)^2,
\qquad
|-\boldsymbol s_k'+\boldsymbol D_k(\lambda,\boldsymbol y)|
\lesssim t^{\frac23}\lambda(t)^2.
\]
Since the centers stay close to the fixed separated configuration, the
definition of \(\boldsymbol D_k\) gives
\[
|\boldsymbol D_k(\lambda,\boldsymbol y)|\lesssim \lambda(t)^2.
\]
Thus
\[
|\boldsymbol s_k'(t)|\lesssim t^{\frac23}\lambda(t)^2.
\]
Moreover, by \( \boldsymbol v_k(T)=0\) and the estimate on
\(\boldsymbol v_k-\boldsymbol s_k\),
\[
|\boldsymbol s_k(T)|\lesssim T\lambda(T)^2.
\]
Integrating backward and using (\ref{integration kernel}), we obtain
\[
|\boldsymbol s_k(t)|
\le
|\boldsymbol s_k(T)|
+
\int_t^T |\boldsymbol s_k'(\tau)|\,d\tau
\lesssim
t\lambda(t)^2.
\]
Therefore
\[
|\boldsymbol v_k(t)|
\le
|\boldsymbol v_k(t)-\boldsymbol s_k(t)|
+
|\boldsymbol s_k(t)|
\lesssim
t\lambda(t)^2.
\]
Since \(t\lambda(t)^2=o(t^{7/6}\lambda(t)^2)\), after increasing \(T_0\) we get
\[
|\boldsymbol v_k(t)|
\le
\frac12 t^{\frac76}\lambda(t)^2,
\]
which proves (\ref{bootstrap a for v close time}) and closes the bootstrap estimate for \(\boldsymbol v_k\). Finally, using \(\boldsymbol y_k(T)=\boldsymbol{z}_k\) and (\ref{y'+v}) we have
\[
\begin{aligned}
|\boldsymbol y_k(t)-\boldsymbol{z}_k|
&\le
\int_t^T |\boldsymbol y_k'(\tau)+\boldsymbol v_k(\tau)|\,d\tau
+
\int_t^T |\boldsymbol v_k(\tau)|\,d\tau  \\
&\lesssim
\int_t^T \tau^{\frac12}\lambda(\tau)^2\,d\tau
+
\int_t^T \tau\lambda(\tau)^2\,d\tau  \\
&\lesssim
t^{\frac56}\lambda(t)^2+t^{\frac43}\lambda(t)^2
\lesssim
t^{\frac43}\lambda(t)^2.
\end{aligned}
\]
Since \(t^{4/3}\lambda(t)^2=o(t^{5/2}\lambda(t)^2)\), increasing \(T_0\) if
necessary gives
\[
|\boldsymbol y_k(t)-\boldsymbol z_k|
\le
\frac12 t^{\frac52}\lambda(t)^2.
\]
This proves (\ref{bootstrap a for y close time}) closes the translation bootstrap estimates.

  In order to prove (\ref{close of a+}),  notice that (\ref{bootstrap a for unstable}) and (\ref{a' also satisfies}) together give
    \begin{equation*}
        \left|\frac{d}{dt}({a}^+)^2-2\nu\sum_k\frac{({a}^+)^2}{\lambda}\right|  \lesssim 2\left|{a}^{+}\sum_k\left(\frac{d}{dt}\tilde{a}^{+}-\nu\frac{\tilde{a}^+}{\lambda}\right)\right| \lesssim t^{\frac{1}{4}}\lambda^3.
    \end{equation*}
Hence, there exist constants $C_1, C_2>0$ (independent of $t$) such that
\begin{equation}\label{d dt a+}
    \frac{d}{dt}({a}^+(t))^2\geq\frac{C_1}{\lambda}({a}^+(t))^2-C_2t^{\frac{1}{4}}\lambda(t)^3.
\end{equation}
    Since $  a_k^+(T)=0$ by (\ref{initial data of stable unstable}) , it is clear that (\ref{close of a+}) holds for $t$ close to $T$. Suppose that (\ref{close of a+}) breaks down for the first time at some $T_1\in (T_*,T)$, which gives 
\begin{equation*}
    \left(a^+(T_1)\right)^2>\frac{1}{2}T_1^{\frac{1}{2}}\lambda(T_1)^4.
\end{equation*}
Then directly we would have on the one hand
\begin{equation*}
    \frac{d}{dt}({a}^+(T_1))^2\leq 0;
\end{equation*}
 on the other hand (\ref{d dt a+}) indicates
 \begin{equation*}
  \frac{d}{dt}({a}^+(T_1))^2\geq \frac{C_1}{2\lambda(T_1)}T_1^{\frac{1}{2}}\lambda(T_1)^4-C_2 T_1^{\frac{1}{4}}\lambda(T_1)^3>0 
 \end{equation*}
 for sufficiently large $T_0$. This contradiction completes the proof for $a^+$. 

\end{proof}

Finally, to prove Proposition \ref{main step}, it remains to handle estimates (\ref{bootstrap a for lambda}) and (\ref{bootstrap a for stable}). For the sake of contradiction, suppose that for any $(\omega_0,\omega_1)\in [-1,1]^2$, it holds $T_*=T_*(\omega_0,\omega_1)\in (T_0,T]$. By Lemma \ref{final boots for energy}, and Lemma \ref{close of the scaling and trans para}, on $[T_*,T]$, equality is reached in none of the estimates (\ref{bootstap a for energy})-(\ref{bootstrap a for unstable}). Therefore, from (i) of Lemma \ref{close the bootstrap for parameters system}, equality has to be reached at $t=T_*$ in estimates (\ref{bootstrap a for lambda}) or (\ref{bootstrap a for stable}). Let 
\begin{equation*}
A_0(t):=t^{-\frac{1}{3}}\left(\alpha(t)-\xi t^{\frac{2}{3}} \right)     ,\qquad
A_1(t):=t^{-\frac{1}{4}}\lambda(t)^{-2} a^-(t).  
\end{equation*}
 Then from (\ref{initial data of stable unstable}), $A_k(T)=\omega_k$ for $k=0,1$. The contradiction assumption gives that for any $(\omega_0,\omega_1)\in [-1,1]^2$, for all $t\in [T_*,T]$ it holds 
\begin{equation*}
    (\omega_0,\omega_1)\in [-1,1]^2,\quad(A_0(T_*),A_1(T_*))\in \mathbb \partial([-1,1]^2).
\end{equation*}
Consider the map $\mathcal{G}:[-1,1]^2\to \partial([-1,1]^2)$ defined by 
\begin{equation*}
    \mathcal{G}(\omega_0,\omega_1):=(A_0(T_*),A_1(T_*)).
\end{equation*}
To prove $\mathcal{G}$ is a continuous map, it suffice to verify that $(\omega_0,\omega_1)\to T_*$ is continuous. Applying the implicit function theorem, this can be reduced to the following transversality condition: for any $T_1\in [T_*,T]$ if $|A_0(T_1)|=1$, then 
\begin{equation}\label{A0A0' cond}
    A_0'(T_1)A_0(T_1)<0;
\end{equation}
if $|A_1|=1$, then
\begin{equation}\label{A1A1' cond}
     A_1'(T_1)A_1(T_1)<0.
\end{equation}
To prove (\ref{A0A0' cond}), notice that
\begin{equation*}
\begin{aligned}
    A_0'(T_1)A_0(T_1)=&-\frac{(\alpha(T_1)-\xi T_1^{\frac{2}{3}})}{3T_1^{\frac{2}{3}}}A_0(T_1)+T_1^{-\frac{1}{3}}\left(\alpha'(T_1)-\frac{2\xi}{3}T_1^{-\frac{1}{3}}\right)A_0(T_1) \\
    =&-\frac{A_0(T_1)^2}{3T_1}+T_1^{-\frac{1}{3}}\left(\mu(T_1)-\frac{2\xi}{3}T_1^{-\frac{1}{3}}\right)A_0(T_1)+T_1^{-\frac{1}{3}}\left(\alpha'(T_1)-\mu(T_1)\right)A_0(T_1).
\end{aligned}  
\end{equation*}
By (\ref{rough of mu}) and (\ref{alpha' and m}), we obtain
\begin{equation*}
    A_0'(T_1)A_0(T_1)=-\frac{2A_0(T_1)^2}{3T_1}+\frac{CA_0(T_1)}{T_1^{\frac{13}{12}}}.
\end{equation*}
Therefore, when $|A_0(T_1)|=1$, (\ref{A0A0' cond}) holds.

Next, for (\ref{A1A1' cond}), using (\ref{a' also satisfies}) we have
\begin{equation*}
\begin{aligned}
     (a^-)'(T_1)a^-(T_1)=-\frac{\nu}{\lambda(T_1)}(a^-(T_1))^2+O(\lambda(T_1)|a^-(T_1)|),
\end{aligned}
\end{equation*}
and thus
\begin{equation*}
\begin{aligned}
     A'_1(T_1)A_1(T_1)=&T_1^{-\frac{1}{2}}\lambda^{-4}\left(\frac{(T_1^{-\frac{1}{4}}\lambda^{-2})'}{T_1^{-\frac{1}{4}}\lambda^{-2}}a^-(T_1)+(a^-)'(T_1)\right)a^-(T_1)\\
     =&O(T_1^{-\frac{1}{3}})A_1(T_1)^2-\frac{\nu}{\lambda(T_1)}A_1(T_1)^2+O(T_1^{-\frac{1}{4}}\lambda^{-1}|A_1(T_1)|).
\end{aligned}
\end{equation*}
On the face \(|A_1(T_1)|=1\), the last term is
\(O(T_1^{-1/4}\lambda(T_1)^{-1})\), and is absorbed by
\(-\nu\lambda(T_1)^{-1}\) for \(T_0\) large. This completes the proof of (\ref{A1A1' cond}).

As a result, the transversality condition is proved and $\mathcal{G}$ is continuous on $[-1,1]^2$ and its restriction to $\partial([-1,1]^2)$ is the identity, which contradicts with the no-retraction theorem.

Combining this with Lemma \ref{final boots for energy} and Lemma \ref{close of the scaling and trans para}, Proposition \ref{main step} holds. With all the preparations before, we can now turn to the proof of Theorem \ref{multi-bubble solution}.
\subsection{Proof of Theorem \ref{multi-bubble solution} from Proposition \ref{main step}.}
This process is very similar to \cite{JM}. We need to introduce two propositions first.
\begin{prop}\label{APPENDIX 1}There exists a constant $\eta>0$ such that the following holds. Let $\Vec{u}:[t_0,T_{\max})\to \dot{H}^1\times L^2$ be a maximal solution of (\ref{NLW 4}) with $T_{\max}<\infty$. Then for any compact set $\mathcal{K}\subset \dot{H}^1\times L^2$ there exists $\tau<T_{\max}$ such that $dist(\Vec{u},\mathcal{K})>\eta$ for all $t\in [\tau,T_{\max})$.    
\end{prop}

\begin{prop}\label{APPENDIX 2}There exists a constant $\eta>0$ such that the following holds. Let $\mathcal{K}\subset \dot{H}^1\times L^2$ be a compact set and let $\Vec{u}_n:[T_1,T_2]\to\dot{H}^1\times L^2$ be a sequence of solutions to (\ref{NLW 4}) such that 
\begin{equation*}
    dist(\Vec{u}_n,\mathcal{K})\leq \eta\text{ , for all }n\in \mathbb N \text{ and }t\in[T_1,T_2].
\end{equation*}
Suppose that $\vec u_n(T_1)\rightharpoonup \Vec{u}_0$ weakly in $\dot{H}^1\times L^2$. Then the solution $\Vec{u}(t)$ of (\ref{NLW 4}) with the initial condition $\Vec{u}(T_1)=\Vec{u}_0$ is defined for $t\in[T_1,T_2]$ and 
\begin{equation*}
    \Vec{u}_n(t)\rightharpoonup \Vec{u}(t)\text{ , weakly in }\dot{H}^1\times L^2\text{ for all }t\in [T_1,T_2].
\end{equation*}
\end{prop}
These propositions have been proved in \cite{Jtb}. With these tools, we will prove our main result using compactness strategy, which was developed in \cite{CMM,Jtype2d5,Jtb,Mnsgkdv,Mkps,RS}. 
\begin{proof}[Proof of Theorem \ref{multi-bubble solution}.] Let \(T_n\to+\infty\). For each \(n\), let \(\vec u_n\) be the solution
given by Proposition \ref{main step} with terminal time \(T=T_n\).
Then \(\vec u_n\) is defined on \([T_0,T_n]\), and we write
\[
\vec u_n(t)=\vec W_{\boldsymbol\Gamma_n(t)}+\vec g_n(t),
\]
where
\[
\boldsymbol\Gamma_n(t)
=
(\lambda_{1,n},b_{1,n},\boldsymbol y_{1,n},\boldsymbol v_{1,n},\ldots,
 \lambda_{N,n},b_{N,n},\boldsymbol y_{N,n},\boldsymbol v_{N,n}).
\]
Since the solution is \(G\)-invariant and the modulation parameters are chosen
in the equivariant modulation class, we have, for each fixed \(n\),
\[
\lambda_{1,n}(t)=\cdots=\lambda_{N,n}(t)=:\lambda_n(t),
\qquad
b_{1,n}(t)=\cdots=b_{N,n}(t)=:b_n(t).
\]
We emphasize that no relation between \(\lambda_n\) and \(\lambda_m\) is assumed
when \(n\neq m\).
The estimates of Proposition \ref{main step} are uniform in \(n\).
In particular, for all \(t\in [T_0,T_n]\),
\[
\|\vec g_n(t)\|_{\ME} \lesssim t^{-1/6}\lambda_n(t),
\]
and
\[
\left|
\log\frac1{\lambda_n(t)}-\xi t^{2/3}
\right|
\lesssim t^{1/3},
\qquad
|y_{k,n}(t)-z_k|
\lesssim t^{5/2}\lambda_n(t)^2 .
\]
Moreover, by the bounds on \(b_n\) and \(v_n\), the velocity part of
\(\vec W_{\Gamma_n(t)}\) satisfies
\[
\left\|
\sum_{k=1}^N
\left(
\frac{b_{k,n}(t)}{\lambda_{k,n}(t)}
(\Lambda_k W_{k,n})\chi_t
+
\frac{v_{k,n}(t)}{\lambda_{k,n}(t)}
\cdot \nabla_k W_{k,n}
\right)
\right\|_{L^2}
\to 0
\]
as \(t\to+\infty\), uniformly in \(n\) whenever \(t\in[T_0,T_n]\).
Therefore,
\[
\left\|
u_n(t)-
\sum_{k=1}^N
\frac1{\lambda_n(t)}
W\left(
\frac{\cdot-y_{k,n}(t)}{\lambda_n(t)}
\right)
\right\|_{\dot H^1}
+
\|\partial_t u_n(t)\|_{L^2}
\to 0
\]
as \(t\to+\infty\), uniformly in the same sense.

We now apply the compactness argument.Fix \(T>T_0\). For \(n\) sufficiently large, \(T_n\geq T\). On \([T_0,T]\),
the bootstrap estimates imply that the parameters of \(\Gamma_n(t)\) stay in a
compact subset of the admissible modulation region. More precisely, let
\(\mathcal K_T\) be the set of all profiles \(\vec W_\Gamma(t)\), with
\(t\in[T_0,T]\), whose parameters satisfy the equivariance conditions and the
bootstrap bounds
\[
    \left|\log\frac1{\lambda(t)}-\xi t^{2/3}\right|
    \leq Ct^{1/3},
\]
together with the corresponding bounds on \(b,y,v\). Since \(t\in[T_0,T]\),
the scale \(\lambda(t)\) is bounded away from \(0\) and \(+\infty\), and the
parameter set is compact. Hence \(\mathcal K_T\) is compact in
\(\dot H^1\times L^2\). Taking \(T_0\) sufficiently large, the
estimate
\[
\|\vec g_n(t)\|_{\ME}\lesssim t^{-1/6}\lambda_n(t)
\]
implies
\[
\operatorname{dist}(\vec u_n(t),\mathcal K_T)\le \eta
\]
for all \(t\in[T_0,T]\) and all sufficiently large \(n\), where \(\eta>0\) is the
constant in Proposition \ref{APPENDIX 2}.

Since \((\vec u_n(T_0))_n\) is bounded in \(\dot H^1\times L^2\), after extracting
a subsequence we may assume that
\[
\vec u_n(T_0)\rightharpoonup \vec u_0
\quad\text{weakly in } \dot H^1\times L^2 .
\]
By Proposition \ref{APPENDIX 2}, the solution \(\vec u\) of \((1.1)\) with
initial data
\[
\vec u(T_0)=\vec u_0
\]
is defined on \([T_0,T]\), and for every \(t\in[T_0,T]\),
\[
\vec u_n(t)\rightharpoonup \vec u(t)
\quad\text{weakly in } \dot H^1\times L^2 .
\]
Since \(T>T_0\) is arbitrary, \(\vec u\) is defined on \([T_0,+\infty)\).

It remains to identify the asymptotic behavior of \(\vec u\). Applying the
modulation decomposition to \(\vec u(t)\) for \(t\) sufficiently large, and passing
to the limit in the uniform estimates above, we obtain parameters
\(\lambda(t)>0\) and \(\boldsymbol y_k(t)\) such that
\[
\left|
\log\frac1{\lambda(t)}-\xi t^{2/3}
\right|
\lesssim t^{1/3},
\qquad
|\boldsymbol y_k(t)-\boldsymbol z_k|\lesssim t^{5/2}\lambda(t)^2,
\]
and
\[
\left\|
u(t)-
\sum_{k=1}^N
\frac1{\lambda(t)}
W\left(
\frac{\cdot-\boldsymbol y_k(t)}{\lambda(t)}
\right)
\right\|_{\dot H^1}
+
\|\partial_t u(t)\|_{L^2}
\to 0 .
\]
Finally,
\[
\left\|
\frac1{\lambda(t)}
W\left(
\frac{\cdot-\boldsymbol y_k(t)}{\lambda(t)}
\right)
-
\frac1{\lambda(t)}
W\left(
\frac{\cdot-\boldsymbol z_k}{\lambda(t)}
\right)
\right\|_{\dot H^1}
\lesssim
\frac{|\boldsymbol y_k(t)-\boldsymbol z_k|}{\lambda(t)}
\lesssim t^{5/2}\lambda(t)
\to0.
\]
Therefore
\[
\left\|
u(t)-
\sum_{k=1}^N
\frac1{\lambda(t)}
W\left(
\frac{\cdot-\boldsymbol z_k}{\lambda(t)}
\right)
\right\|_{\dot H^1}
+
\|\partial_t u(t)\|_{L^2}
\to0,
\]
and
\[
\left|
\log\frac1{\lambda(t)}-\xi t^{2/3}
\right|
\lesssim t^{1/3}.
\]
After translating time, the solution is defined on \([0,+\infty)\). This proves
Theorem \ref{multi-bubble solution}.

\end{proof}
\appendix
\section{Hardy inequalities in dimension four}
In this appendix, we give a proof of the four-dimensional Hardy inequality. We include it because it leads to a key difference in the four-dimensional coercivity estimate compared with higher dimensions. 
\begin{lemma}[Hardy inequalities]\label{Hardy 4D} Let $D=4$. Then for all $R>2$ and $v\in H^2(\mathbb R^4)$, we have
\begin{equation}\label{Hardy 1}
    \int_{\mathbb R^4}\frac{|\nabla v|^2}{|x|^2}dx\lesssim \int_{\mathbb R^4}(\Delta v)^2dx,
\end{equation}
\begin{equation}\label{Hardy 2}
  \int_{|x|\leq R}\frac{|v|^2}{|x|^4(1+|\log|x| |)^2} dx \lesssim \int_{|x|\leq R}\frac{|\nabla v|^2}{|x|^2} dx+\int_{|x|\leq 2}|v|^2dx.
\end{equation}
\end{lemma}
\begin{proof}
Using Hardy's inequality and the Riesz transform, (\ref{Hardy 1}) holds from 
\begin{equation*}
   \int_{\mathbb R^4}\frac{|\nabla v|^2}{|x|^2}dx\lesssim \int_{\mathbb R^4}|\nabla (\nabla v)|^2dx\lesssim   \int_{\mathbb R^4}(\Delta v)^2dx.
\end{equation*}
Next to prove (\ref{Hardy 2}), let $a\in [1,2]$ such that 
\begin{equation*}
    \int_{|x|=a}|v(x)|^2dS\leq 8\int_{1\leq |x|\leq 2}|v|^2.
\end{equation*}
Let $f(x)=-(x/|
x|^4(1+\log |x|))$, so that $\nabla \cdot f=\frac{1}{|x|^4(1+\log |x|)^2}$, and integrate by parts to get
\begin{equation}\label{Hardy 2.1}
    \begin{aligned}
        &\int_{a\leq |x|\leq R}\frac{|v|^2}{|x|^4(1+\log |x|)^2}dx
        =\int_{a\leq |x|\leq R}|v|^2\nabla \cdot fdx\\
        =&\int_{|x|=R}-\frac{|v|^2}{|x|^3(1+\log |x|)} dS+\int_{|x|=a}\frac{|v|^2}{|x|^3(1+\log |x|)} dS
       +2\int_{a\leq |x|\leq R}\frac{2v \nabla v\cdot x}{|x|^4(1+\log |x|)^2}\\
       \lesssim &\int_{|x|=a}|v(x)|^2dS+\left(\int_{a\leq |x|\leq R}\frac{|v|^2}{|x|^4(1+|\log |x||)^2}\right)^{\frac{1}{2}}\left(\int_{a\leq |x|\leq R}\frac{|\nabla v|^2}{|x|^2}\right)^{\frac{1}{2}}.
    \end{aligned}
\end{equation}
Similarly, using $\Tilde{f}(x)=(x/|
x|^4(1-\log |x|))$, we have
\begin{equation}\label{Hardy 2.2}
    \begin{aligned}
        &\int_{\epsilon\leq |x|\leq a}\frac{|v|^2}{|x|^4(1-\log |x|)^2}dx
        =\int_{\epsilon\leq |x|\leq a}|v|^2\nabla \cdot \Tilde{f}dx\\
        =&\int_{|x|=a}\frac{|v|^2}{|x|^3(1-\log |x|)} dS-\int_{|x|=\epsilon}\frac{|v|^2}{|x|^3(1-\log |x|)} dS
       -2\int_{\epsilon\leq |x|\leq a}\frac{2v \nabla v\cdot x}{|x|^4(1-\log |x|)^2}\\
       \lesssim &\int_{|x|=a}|v(x)|^2dS+\left(\int_{\epsilon\leq |x|\leq a}\frac{|v|^2}{|x|^4(1+|\log |x||)^2}\right)^{\frac{1}{2}}\left(\int_{\epsilon\leq |x|\leq a}\frac{|\nabla v|^2}{|x|^2}\right)^{\frac{1}{2}}.
    \end{aligned}
\end{equation}
Combining (\ref{Hardy 2.1}) and (\ref{Hardy 2.2}) together we have (\ref{Hardy 2}) holds.
\end{proof}

\section{Proof of the pointwise estimates for the refined profile}\label{pointwise est of refined profile}
This appendix is devoted to the proof of the pointwise estimates for the profiles \(Q\) and \(S\) stated in Section~5. These estimates follow from standard ODE arguments and are included here for completeness.
\begin{proof}
We recall that $Q$ satisfies
\begin{equation*}
    \ML Q=-\Lambda W \chi_t(\lambda \cdot)-\frac{\lelan \Lambda W \chi_t(\lambda \cdot),\Lambda W \rilan}{32\pi^2}f'(W),
\end{equation*} 
As in \cite{HR}, let 
    \begin{equation*}
        \Gamma (y)=-\Lambda W(y)\int_1^y\frac{ds}{s^{3}(\Lambda W)^2(s)},
    \end{equation*}
then $\Gamma$ belongs to the kernel of $\ML$ and satisfies
\begin{equation*}
    \Gamma '\Lambda W-\Gamma(\Lambda W)'=\frac{-1}{y^3}.
\end{equation*}
A direct computation yields the pointwise estimates for $\Gamma$:
\begin{equation*}
    \frac{d^k \Gamma}{dy^k}(y)=\begin{cases}
        &O(y^{-2-k})\text{ as }y\to 0,\\
        &O(y^{-k})\text{ as }y\to \infty.
    \end{cases}
\end{equation*}
Then a smooth solution to $\ML w=F$ can be given by 
\begin{equation}\label{solution formula}
w(y)=\Gamma (y)\int_0^yF(s)\Lambda W(s)s^3ds-\Lambda W(y)\int_0^yF(s)\Gamma(s)s^3ds.  
\end{equation}
Differentiating \(w\), we obtain
\begin{equation*}
    \frac{d}{dy}w(y)=\Gamma' (y)\int_0^yF(s)\Lambda W(s)s^3ds-\Lambda W'(y)\int_0^yF(s)\Gamma(s)s^3ds, 
\end{equation*}
where we denote 
\begin{equation*}
    F=-\Lambda W \chi_t(\lambda \cdot)-\frac{\lelan \Lambda W \chi_t(\lambda \cdot),\Lambda W \rilan}{24}f'(W).
\end{equation*}
Here $\Lambda W'(y)=\frac{y(y^2-24)}{32(1+\frac{y^2}{8})^3}$. Since $\Lambda W$ is orthogonal to the right hand side of the equation, (\ref{solution formula}) can be rewritten as
\begin{equation}\label{solution formula for QTS}
  w(y)=-\left[\Gamma (y)\int_y^{\infty}F(s)\Lambda W(s)s^3ds+\Lambda W(y)\int_0^yF(s)\Gamma(s)s^3ds\right].  
\end{equation}
Notice that $\lelan \Lambda W \chi_t(\lambda \cdot),\Lambda W \rilan\sim \log(2t/\lambda)$.
Hence for $y\leq 1$, using (\ref{solution formula}), 
\begin{equation*}
    \begin{aligned}
     |Q(y)|\leq &\left|\Gamma (y)\int_0^yF(s)\Lambda W(s)s^3ds\right|+\left|\Lambda W(y)\int_0^yF(s)\Gamma(s)s^3ds\right|\\
     \lesssim&\left|y^{-2}\int_0^y(1+{\log(2t/\lambda)})s^3ds\right|+\left|\int_0^y(1+{\log(2t/\lambda)})sds\right|
     \lesssim \log(2t/\lambda)y^2.
    \end{aligned}
\end{equation*}
Analogously, on this region the derivative is controlled by
\begin{equation*}
\begin{aligned}
    \left|\frac{d Q}{dy}\right|\leq &\left|\Gamma' (y)\int_0^yF(s)\Lambda W(s)s^3ds\right|+\left|\Lambda W'(y)\int_0^yF(s)\Gamma(s)s^3ds\right|\\
     \lesssim&\left|y^{-3}\int_0^y(1+{\log(2t/\lambda)})s^3ds\right|+\left|\int_0^y(1+{\log(2t/\lambda)})sds\right|
     \lesssim \log(2t/\lambda)y.
\end{aligned}    
\end{equation*}
On the region $1\leq y\leq 2t/\lambda$, applying (\ref{solution formula for QTS}) yields
\begin{equation}\label{pointwise est of Q}
    \begin{aligned}
    |Q(y)|=&\left|\Gamma (y)\int_y^{\infty}F(s)\Lambda W(s)s^3ds\right|+\left|\Lambda W(y)\int_0^yF(s)\Gamma(s)s^3ds\right|\\
    \lesssim& \int_y^{\infty}\left(\frac{s^3\chi_{s\leq \frac{2t}{\lambda}}}{(1+s^2)^2}+\frac{s^3\log(2t/\lambda)}{(1+s^2)^3}\right)ds+\frac{1}{1+y^2}\int_1^y\left(\frac{s^3}{1+s^2}+\frac{s^3\log(2t/\lambda)}{(1+s^2)^2}\right)ds\\
    &+\frac{1}{1+y^2}\int_0^1\left(\frac{s}{1+s^2}+\frac{s\log(2t/\lambda)}{(1+s^2)^2}\right)ds\\
    \lesssim&\log(\frac{2t}{\lambda y})+\frac{y^2+(1+\log y) \log(2t/\lambda )}{1+y^2}.
    \end{aligned}
\end{equation}
Similarly, on the exterior region $\frac{dQ}{dy}$ is dominated by
\begin{equation*}
  \begin{aligned}
    \left|\frac{dQ}{dy}(y)\right|=&\left|\Gamma' (y)\int_y^{\infty}F(s)\Lambda W(s)s^3ds\right|+\left|\Lambda W'(y)\int_0^yF(s)\Gamma(s)s^3ds\right|\\
    \lesssim& \frac{1}{y}\int_y^{\infty}\left(\frac{s^3\chi_{s\leq \frac{2t}{\lambda}}}{(1+s^2)^2}+\frac{s^3\log(2t/\lambda)}{(1+s^2)^3}\right)ds+\frac{1}{1+y^3}\int_1^y\left(\frac{s^3}{1+s^2}+\frac{s^3\log(2t/\lambda)}{(1+s^2)^2}\right)ds\\
    &+\frac{1}{1+y^3}\int_0^1\left(\frac{s}{1+s^2}+\frac{s\log(2t/\lambda)}{(1+s^2)^2}\right)ds\\
    \lesssim&\frac{\log(\frac{2t}{\lambda y})}{y}+\frac{y^2+(1+\log y) \log(2t/\lambda )}{1+y^3}.
    \end{aligned}  
\end{equation*}
Furthermore, from the equation of $Q$, one can see that $Q$ is also a function of $t$. When $y\leq 1$, the $\frac{dF}{dt}(y)$ satisfies
\begin{equation*}
    \frac{dF}{dt}(y)=-\frac{d\lelan \Lambda W \chi_t(\lambda \cdot),\Lambda W \rilan }{dt}\frac{f'(W)(y)}{24}\sim \frac{|\lambda'|}{\lambda}+\frac{1}{t}.
\end{equation*}
As a result we obtain the pointwise estimate that, when $y\leq 1$,
\begin{equation*}
    \begin{aligned}
     \left|\frac{dQ}{dt}(y)\right|\leq &\left|\Gamma (y)\int_0^y\frac{dF}{dt}(s)\Lambda W(s)s^3ds\right|+\left|\Lambda W(y)\int_0^y\frac{dF}{dt}(s)\Gamma(s)s^3ds\right|\\
     \lesssim&\left|y^{-2}\int_0^y\left(\frac{|\lambda'|}{\lambda}+\frac{1}{t}\right)s^3ds\right|+\left|\int_0^y\left(\frac{|\lambda'|}{\lambda}+\frac{1}{t}\right)sds\right|
     \lesssim \left(\frac{|\lambda'|}{\lambda}+\frac{1}{t}\right)y^2,
    \end{aligned}
\end{equation*}
and
\begin{equation*}
\begin{aligned}
    \left|\frac{d^2 Q}{dydt}\right|\leq &\left|\Gamma' (y)\int_0^y\frac{dF}{dt}(s)\Lambda W(s)s^3ds\right|+\left|\Lambda W'(y)\int_0^y\frac{dF}{dt}(s)\Gamma(s)s^3ds\right|\\
     \lesssim&\left|y^{-3}\int_0^y\left(\frac{|\lambda'|}{\lambda}+\frac{1}{t}\right)s^3ds\right|+\left|\int_0^y\left(\frac{|\lambda'|}{\lambda}+\frac{1}{t}\right)sds\right|
     \lesssim \left(\frac{|\lambda'|}{\lambda}+\frac{1}{t}\right)y.
\end{aligned}        
\end{equation*}
On the outer region $1\leq y\leq 2t/\lambda$ , it holds 
\begin{equation*}
\begin{aligned}
    \frac{dF}{dt}(y)=&-\Lambda W\frac{d(\chi_t(\lambda \cdot))}{dt} -\frac{d\lelan \Lambda W \chi_t(\lambda \cdot),\Lambda W \rilan }{dt}\frac{f'(W)(y)}{24}\\
    \sim&  \left(\Lambda W \chi'\left(\frac{\lambda \cdot}{t}\right) +f'(W) \right)\left(\frac{|\lambda'|}{\lambda}+\frac{1}{t}\right).
\end{aligned}
\end{equation*}
Therefore, when $1\leq y \leq 2t/\lambda$,   $\frac{dQ}{dt}$ is bounded pointwise by
\begin{equation*}
    \begin{aligned}
     \left|\frac{dQ}{dt}(y)\right|\leq &\left|\Gamma (y)\int_y^{\infty}\frac{dF}{dt}(s)\Lambda W(s)s^3ds\right|+\left|\Lambda W(y)\int_0^y\frac{dF}{dt}(s)\Gamma(s)s^3ds\right|\\
    \lesssim& \Bigg[\int_y^{\infty}\left(\frac{s^3\chi'(\frac{\lambda s}{t})}{(1+s^2)^2}+\frac{s^3}{(1+s^2)^3}\right)ds+\frac{1}{1+y^2}\int_1^y\left(\frac{s^3\chi'(\frac{\lambda s}{t})}{1+s^2}+\frac{s^3}{(1+s^2)^2}\right)ds\\
    &+\frac{1}{1+y^2}\int_0^1\frac{s}{(1+s^2)^2}ds\Bigg]\left(\frac{|\lambda'|}{\lambda}+\frac{1}{t}\right) 
    \lesssim\left(\frac{|\lambda'|}{\lambda}+\frac{1}{t}\right),
    \end{aligned}   
\end{equation*}
and $\frac{d^2Q}{dy dt}$ satisfies
\begin{equation*}
    \begin{aligned}
    \left|\frac{d^2Q}{dydt}(y)\right|=&\left|\Gamma' (y)\int_y^{\infty}\frac{dF}{dt}(s)\Lambda W(s)s^3ds\right|+\left|\Lambda W'(y)\int_0^y\frac{dF}{dt}(s)\Gamma(s)s^3ds\right|\\
    \lesssim&\Bigg[ \frac{1}{y}\int_y^{\infty}\left(\frac{s^3\chi'(\frac{\lambda s}{t})}{(1+s^2)^2}+\frac{s^3}{(1+s^2)^3}\right)ds+\frac{1}{1+y^3}\int_1^y\left(\frac{s^3\chi'(\frac{\lambda s}{t})}{1+s^2}+\frac{s^3}{(1+s^2)^2}\right)ds\\
    &+\frac{1}{1+y^3}\int_0^1\frac{s}{(1+s^2)^2}ds\Bigg]\left(\frac{|\lambda'|}{\lambda}+\frac{1}{t}\right) 
    \lesssim\frac{1}{y}\left(\frac{|\lambda'|}{\lambda}+\frac{1}{t}\right) .
    \end{aligned}   
\end{equation*}
Next, we turn to $S$. Since $S$ satisfies
\begin{equation*}
    \ML S=-(\underline{\Lambda}\Lambda W)\chi_t(\lambda \cdot)+\frac{\lelan (\underline{\Lambda}\Lambda W)\chi_t(\lambda \cdot),\Lambda W \rilan}{\lelan \Lambda W \chi_t(\lambda \cdot),\Lambda W \rilan}\Lambda W \chi_t(\lambda \cdot),
\end{equation*}
and observe that 
\begin{equation*}
 \frac{\lelan (\underline{\Lambda}\Lambda W)\chi_t(\lambda \cdot),\Lambda W \rilan}{\lelan \Lambda W \chi_t(\lambda \cdot),\Lambda W \rilan}\sim \frac{1}{\log(2t/\lambda)}.
\end{equation*}
With the same strategy as $Q$, when $y\leq 1$, using (\ref{solution formula}) we obtain
\begin{equation*}
    |S(y)|\lesssim y^{-2}\int_0^y \left(1+\frac{1}{2\log(2t/\lambda)}\right)s^3ds+\int_0^y \left(1+\frac{1}{2\log(2t/\lambda)}\right)sds\lesssim y^2,
\end{equation*}
and
\begin{equation*}
   \left|\frac{dS}{dy}(y)\right|\lesssim  y^{-3}\int_0^y \left(1+\frac{1}{2\log(2t/\lambda)}\right)s^3ds+\int_0^y \left(1+\frac{1}{2\log(2t/\lambda)}\right)sds\lesssim y.
\end{equation*}
Also, since on this region,  
\begin{equation*}
    \frac{d}{dt }\left(-(\underline{\Lambda}\Lambda W)\chi_t(\lambda \cdot)+\frac{\lelan (\underline{\Lambda}\Lambda W)\chi_t(\lambda r),\Lambda W \rilan}{\lelan \Lambda W \chi_t(\lambda \cdot),\Lambda W \rilan}\Lambda W \chi_t(\lambda \cdot)\right)\sim \frac{1}{\log(2t/\lambda)^2}\left(\frac{|\lambda'|}{\lambda}+\frac{1}{t}\right),
\end{equation*}
the $\frac{dS}{st}$ and $\frac{d^2S}{dydt}$ are thus dominated by
\begin{equation*}
\begin{aligned}
      \left|\frac{dS}{dt}(y)\right|\lesssim& y^{-2}\int_0^y \left(\frac{1}{\log(2t/\lambda)^2}\left(\frac{|\lambda'|}{\lambda}+\frac{1}{t}\right)\right)s^3ds+\int_0^y \left(\frac{1}{\log(2t/\lambda)^2}\left(\frac{|\lambda'|}{\lambda}+\frac{1}{t}\right)\right)sds\\
      \lesssim& \frac{y^2}{\log(2t/\lambda)^2}\left(\frac{|\lambda'|}{\lambda}+\frac{1}{t}\right),
\end{aligned}
\end{equation*}
and
\begin{equation*}
\begin{aligned}
   \left|\frac{d^2S}{dydt}(y)\right|\lesssim& y^{-3}\int_0^y \left(\frac{1}{\log(2t/\lambda)^2}\left(\frac{|\lambda'|}{\lambda}+\frac{1}{t}\right)\right)s^3ds+\int_0^y \left(\frac{1}{\log(2t/\lambda)^2}\left(\frac{|\lambda'|}{\lambda}+\frac{1}{t}\right)\right)sds\\
      \lesssim& \frac{y}{\log(2t/\lambda)^2}\left(\frac{|\lambda'|}{\lambda}+\frac{1}{t}\right).
\end{aligned}
\end{equation*}
In the case $1\leq y \leq 2t/\lambda$,  (\ref{solution formula for QTS}) yields
\begin{equation*}
    \begin{aligned}
        |S(y)|=&\left|\Gamma (y)\int_y^{\infty}F(s)\Lambda W(s)s^3ds-\Lambda W(y)\int_0^yF(s)\Gamma(s)s^3ds\right|\\
        \lesssim&\int_y^{\infty}\left(\frac{s^3\chi_{s\leq \frac{2t}{\lambda}}}{(1+s^2)(1+s^4)}+\frac{s^3\chi_{s\leq \frac{2t}{\lambda}}}{(1+s^2)^2\log(2t/\lambda)}\right)ds\\
        &+\frac{1}{1+y^2}\int_1^y\left(\frac{s^3}{1+s^4}+\frac{s^3}{(1+s^2)\log(2t/\lambda)}\right)ds\\
    &+\frac{1}{1+y^2}\int_0^1\left(\frac{s}{1+s^4}+\frac{s}{(1+s^2)\log(2t/\lambda)}\right)ds
    \lesssim\frac{1+\log y}{1+y^2}+\frac{\log(t/\lambda y)+1}{\log(t/
    \lambda)}.
    \end{aligned}
\end{equation*}
and
\begin{equation*}
    \begin{aligned}
        \left|\frac{dS}{dy}(y)\right|=&\left|\Gamma' (y)\int_y^{\infty}F(s)\Lambda W(s)s^3ds-\Lambda W'(y)\int_0^yF(s)\Gamma(s)s^3ds\right|\\
        \lesssim&\frac{1}{y}\int_y^{\infty}\left(\frac{s^3\chi_{s\leq \frac{2t}{\lambda}}}{(1+s^2)(1+s^4)}+\frac{s^3\chi_{s\leq \frac{2t}{\lambda}}}{(1+s^2)^2\log(2t/\lambda)}\right)ds\\
        &+\frac{1}{1+y^3}\int_1^y\left(\frac{s^3}{1+s^4}+\frac{s^3}{(1+s^2)\log(2t/\lambda)}\right)ds\\
    &+\frac{1}{1+y^3}\int_0^1\left(\frac{s}{1+s^4}+\frac{s}{(1+s^2)\log(2t/\lambda)}\right)ds
    \lesssim\frac{1+\log y}{1+y^3}+\frac{\log(t/\lambda y)+1}{\log(t/
    \lambda)y}
    \end{aligned}
\end{equation*}
Observe that when $1\leq y \leq 2t/\lambda$, the time derivative of right-hand side of the equation for $S$ yields 
\begin{equation*}
\left(\frac{|\lambda'|}{\lambda}+\frac{1}{t}\right)\left(  (\underline{\Lambda}\Lambda W)\chi'\left(\frac{\lambda \cdot}{t}\right)+\frac{\Lambda W}{\log (2t/\lambda)}\chi'\left(\frac{\lambda \cdot}{t}\right)\right)+ \frac{\Lambda W \chi_t(\lambda \cdot)}{\log(2t/\lambda)^2}\left(\frac{|\lambda'|}{\lambda}+\frac{1}{t}\right).
\end{equation*}
With this, $\frac{dS}{dt}$ can be estimated by
\begin{equation*}
    \begin{aligned}
        \left|\frac{dS}{dt}(y)\right|\lesssim&\left|\Gamma (y)\int_y^{\infty}\frac{dF}{dt}(s)\Lambda W(s)s^3ds-\Lambda W(y)\int_0^y\frac{dF}{dt}(s)\Gamma(s)s^3ds\right|\\
        \lesssim&\left(\frac{|\lambda'|}{\lambda}+\frac{1}{t}\right)\Bigg[\int_y^{\infty}\left(\frac{s^3\chi'(\frac{\lambda s}{t})}{(1+s^2)(1+s^4)}+\frac{s^3\chi'(\frac{\lambda s}{t})}{(1+s^2)^2\log(2t/\lambda)}+\frac{s^3\chi(\frac{\lambda s}{t})}{(1+s^2)^2\log(2t/\lambda)^2}\right)ds\\
        &+\frac{1}{1+y^2}\int_1^y\left(\frac{s^3\chi'(\frac{\lambda s}{t})}{1+s^4}+\frac{s^3\chi'(\frac{\lambda s}{t})}{(1+s^2)\log(2t/\lambda)}+\frac{s^3\chi(\frac{\lambda s}{t})}{(1+s^2)\log(2t/\lambda)^2}\right)ds\\
    &+\frac{1}{1+y^2}\int_0^1\left(\frac{s\chi(\frac{\lambda s}{t})}{(1+s^2)\log(2t/\lambda)^2}\right)ds\Bigg]
    \lesssim\left(\frac{|\lambda'|}{\lambda}+\frac{1}{t}\right)\left(\frac{1}{\log(t/\lambda)}+\frac{\log(t/\lambda y)}{\log(t/\lambda)^2}+\frac{1}{1+y^2}\right).
    \end{aligned}
\end{equation*}
Again, $\frac{d^2S}{dydt}(y)$ satisfies
\begin{equation*}
   \begin{aligned}
        \left|\frac{d^2S}{dydt}(y)\right|\lesssim&\left|\Gamma' (y)\int_y^{\infty}\frac{dF}{dt}(s)\Lambda W(s)s^3ds-\Lambda W'(y)\int_0^y\frac{dF}{dt}(s)\Gamma(s)s^3ds\right|\\
        \lesssim&\left(\frac{|\lambda'|}{\lambda}+\frac{1}{t}\right)\Bigg[\frac{1}{y}\int_y^{\infty}\left(\frac{s^3\chi'(\frac{\lambda s}{t})}{(1+s^2)(1+s^4)}+\frac{s^3\chi'(\frac{\lambda s}{t})}{(1+s^2)^2\log(2t/\lambda)}+\frac{s^3\chi(\frac{\lambda s}{t})}{(1+s^2)^2\log(2t/\lambda)^2}\right)ds\\
        &+\frac{1}{1+y^3}\int_1^y\left(\frac{s^3\chi'(\frac{\lambda s}{t})}{1+s^4}+\frac{s^3\chi'(\frac{\lambda s}{t})}{(1+s^2)\log(2t/\lambda)}+\frac{s^3\chi(\frac{\lambda s}{t})}{(1+s^2)\log(2t/\lambda)^2}\right)ds\\
    &+\frac{1}{1+y^3}\int_0^1\left(\frac{s\chi(\frac{\lambda s}{t})}{(1+s^2)\log(2t/\lambda)^2}\right)ds\Bigg]
    \lesssim\left(\frac{|\lambda'|}{\lambda}+\frac{1}{t}\right)\left(\frac{1}{y\log(t/\lambda)}+\frac{\log(t/\lambda y)}{y\log(t/\lambda)^2}+\frac{1}{1+y^3}\right).
    \end{aligned}   
\end{equation*}

\end{proof}

\section{Proof of Lemma \ref{property of Ak}.}\label{Proof of Lemma property of Ak}
This section is devoted to proving the property of operator $\underline{A}_k$. For reader's convenience, we restated the lemma before giving its proof.

\begin{lemma}
   For any $k=1,...,N$, the operators $A_k$ and $\underline{A}_k$ satisfy the following properties.

 (\rmnum{1}) The families
\begin{equation*}
    \begin{aligned}
        \left\{A_k:\lambda_k>0,y_k\in \mathbb R^4\right\}\text{ , }&\left\{\underline{A}_k:\lambda_k>0,y_k\in \mathbb R^4\right\}\\
        \left\{\lambda_k\partial_{\lambda_k}A_k:\lambda_k>0,y_k\in \mathbb R^4\right\}\text{ , }&\left\{\lambda_k\partial_{\lambda_k}\underline{A}_k:\lambda_k>0,y_k\in \mathbb R^4\right\}\\
        \left\{\lambda_k\partial_{y_k}A_k:\lambda_k>0,y_k\in \mathbb R^4\right\}\text{ and }&\left\{\lambda_k\partial_{y_k}\underline{A}_k:\lambda_k>0,y_k\in \mathbb R^4\right\}
    \end{aligned}
\end{equation*}
are bounded in $\ML(\dot{H}^1,L^2)$ and  $\ML(L^2,L^{\frac{4}{3}})$ with norms depending on $q$. Moreover, for any $h\in \dot{H}^1\cap \dot{H}^2$,
\begin{equation}\label{Lk 2 bounded for Ak of appendix}
    \left\|\underline{A}_kh\right\|_{\dot{H}^1}+\left\|\lambda_k\partial_{\lambda_k}\underline{A}_kh\right\|_{\dot{H}^1}+\left\|\lambda_k\partial_{y_k}\underline{A}_kh\right\|_{\dot{H}^1}\lesssim \Tilde{R}\left\|\ML_k h\right\|_{L^2}+\frac{1}{\lambda_k}\left\|h\right\|_{\dot{H}^1}.
\end{equation}

(\rmnum{2}) For any $g,h\in \dot{H}^1\cap \dot{H}^2$,
\begin{equation}\label{Ak key property 1 in app}
  \lelan A_kh,f(h+g)-f(h)-f'(h)g \rilan=-\lelan A_k g,f(h+g)-f(h) \rilan. 
\end{equation}

(\rmnum{3}) For any $h\in \dot{H}^1\cap \dot{H}^2$, $g\in\dot{H}^1\cap L^2$, it holds
\begin{equation}\label{Ah delta g in app}
    \left|\lelan \underline{A}_k h,\Delta g \rilan\right|\lesssim \Tilde{R}\left\|\ML_{k}h\right\|_{L^2}\left\|g\right\|_{\dot{H}^1}.
\end{equation}

(\rmnum{4}) For any $\eta>0$, choosing $\epsilon>0$ small enough in Lemma \ref{function q}, it holds for all $h\in \dot{H}^1\cap \dot{H}^2$,
\begin{equation}\label{Ak key property 3 in app}
    \lelan \underline{A}_kh,\Delta h \rilan\leq \frac{\eta}{\lambda_k}\left\|h\right\|_{\dot{H}^1}^2-\frac{1}{\lambda_k}\int_{|x-\boldsymbol{y}_k|\leq R}|\nabla h(x)|^2dx,
\end{equation}
\begin{equation}\label{Ak key property 2 in app}
    \lelan [\Delta,\underline{A}_k]h,\Delta h \rilan\geq  \frac{1}{\lambda_k}\int_{|x-\boldsymbol{y}_k|\leq \lambda_k R}(\Delta h(x))^2dx-\frac{2\epsilon}{\lambda_k}\left\|\ML_k h\right\|_{L^2}^2.
\end{equation}

(\rmnum{5}) For any \(\eta>0\), after choosing \(\epsilon>0\) sufficiently small and then \(R>0\) suitably in Lemma \ref{function q}, it holds
\begin{equation}\label{A and Lambda W nabla W in app}
    \left\|\underline{A}\Lambda W\right\|_{L^2(|x|\geq R)}\leq \left(\eta+\frac{1}{R}\right)\text{ , }\left\|\underline{A}\nabla W\right\|_{L^2(|x|\geq R)}\leq \frac{1}{R}.
\end{equation}
\end{lemma}
\begin{proof}
\textbf{Step 1. Proof of (\rmnum{1}) and (\rmnum{2}).} The boundedness in (\rmnum{1}) is derived from the definition of $q$ and the scaling invariance of the norm. For (\ref{Lk 2 bounded for Ak of appendix}), we first consider $\left\|\underline{A}_kh\right\|_{\dot{H}^1}$. From the definition of the operator $\underline{A}_k$ we have
\begin{equation*}
  \left\|\underline{A}_kh\right\|_{\dot{H}^1}\lesssim{\frac{1}{\lambda_k}\left\|\Delta q\left(\frac{x-\boldsymbol{y}_k}{\lambda_k}\right)h(x)\right\|_{\dot{H}^1}}+\left\|\nabla q\left(\frac{x-\boldsymbol{y}_k}{\lambda_k}\right)\cdot \nabla h(x)\right\|_{\dot{H}^1}
\end{equation*}
For the first term, recall that $|\Delta q|\leq 4$ when $|x|\leq \Tilde{R}$ and property (6) of Lemma \ref{function q}, it holds 
\begin{equation*}
  \frac{1}{\lambda_k}\left\|\Delta q\left(\frac{x-\boldsymbol{y}_k}{\lambda_k}\right)h(x)\right\|_{\dot{H}^1}\lesssim \frac{1}{\lambda_k}\left\|\nabla h(x)\right\|_{L^2}+\frac{1}{\lambda_k^2}\left\|\frac{\epsilon \lambda_kh}{(1+|\log (x-\boldsymbol{y}_k/\lambda_k)|)|x-\boldsymbol{y}_k|}\right\|_{L^2} \lesssim \frac{\left\|h\right\|_{\dot{H}^1}}{\lambda_k}, 
\end{equation*}
where Hardy is applied in the last inequality. For the second term, a direct computation gives
\begin{equation*}
    \begin{aligned}
   \left\|\nabla q\left(\frac{x-\boldsymbol{y}_k}{\lambda_k}\right)\cdot \nabla h(x)\right\|_{\dot{H}^1}\leq&\sum_{1\leq i,j\leq 4}\left(\int_{\mathbb R^4}\left(\frac{1}{\lambda_k}\partial_i\partial_jq\left(\frac{x-\boldsymbol{y}_k}{\lambda_k}\right)\partial_jh\right)^2\right)^{1/2}
   \\&+\sum_{1\leq i,j\leq 4}\left(\int_{\mathbb R^4}\left(\partial_jq\left(\frac{x-\boldsymbol{y}_k}{\lambda_k}\right)\partial_i\partial_jh\right)^2\right)^{1/2}.
    \end{aligned}
\end{equation*}
From (7) of Lemma \ref{function q} and the expression of $q$ when $x\geq \Tilde{R}$, we have
\begin{equation*}
   \sum_{1\leq i,j\leq 4}\left(\int_{\mathbb R^4}\left(\frac{1}{\lambda_k}\partial_i\partial_jq\left(\frac{x-\boldsymbol{y}_k}{\lambda_k}\right)\partial_jh\right)^2\right)^{1/2}\leq \frac{\left\|h\right\|_{\dot{H}^1}}{\lambda_k}, 
\end{equation*}
and
\begin{equation*}
 \sum_{1\leq i,j\leq 4}\left(\int_{\mathbb R^4}\left(\partial_jq\left(\frac{x-\boldsymbol{y}_k}{\lambda_k}\right)\partial_i\partial_jh\right)^2\right)^{1/2}\leq \Tilde{R} \left\|h\right\|_{\dot{H}^2}.   
\end{equation*}
Therefore, using the coercivity estimate we arrive at the estimate
\begin{equation*}
\left\|\underline{A}_kh\right\|_{\dot{H}^1}\lesssim \Tilde{R}\left\|\ML_k h\right\|_{L^2}+\frac{1}{\lambda_k}\left\|h\right\|_{\dot{H}^1}.
\end{equation*}
 For $\lambda_k\partial_{\lambda_k}\underline{A}_kh$ and $\lambda_k\partial_{y_k}\underline{A}_k$, we compute
\begin{equation*}
\begin{aligned}
    \lambda_k\partial_{\lambda_k}\underline{A}_k= &-\frac{1}{2\lambda_k}\Delta q\left(\frac{x-\boldsymbol{y}_k}{\lambda_k}\right)-\frac{1}{2\lambda_k}\frac{x-\boldsymbol{y}_k}{\lambda_k}\cdot\nabla \Delta q\left(\frac{x-\boldsymbol{y}_k}{\lambda_k}\right)\\
    &-\frac{x-\boldsymbol{y}_k}{\lambda_k}\cdot\nabla^2q\left(\frac{x-\boldsymbol{y}_k}{\lambda_k}\right)\cdot\nabla,
\end{aligned}
\end{equation*}
and 
\begin{equation*}
    \lambda_k\partial_{y_k}\underline{A}_k=-\frac{1}{2\lambda_k}\nabla \Delta q\left(\frac{x-\boldsymbol{y}_k}{\lambda_k}\right)-\nabla ^2q\left(\frac{x-\boldsymbol{y}_k}{\lambda_k}\right)\cdot \nabla.
\end{equation*}
Then similarly, using Lemma \ref{function q}, (\ref{Lk 2 bounded for Ak of appendix}) holds and the proof of (\rmnum{1}) is completed. Next, the property (\rmnum{2}) can be proved with the same approach in \cite{Jtb}. 

\textbf{Step 2. Proof of (\rmnum{3}).} From the scaling and translation invariance, it suffices for us to consider the case $\lambda_k=1$ and $\boldsymbol{y}_k=0$. Plugging the definition of $\underline{A}$ into the left hand side of (\ref{Ah delta g in app}) we obtain 
\begin{equation*}
    \begin{aligned}
         \left|\lelan \underline{A} h,\Delta g \rilan\right|\leq \left|\int_{\mathbb R^4}\frac{1}{2}\Delta q \left(x\right)h(x)\Delta g(x)dx\right|+\left|\int_{\mathbb R^4}\nabla q\left({x}\right)\cdot \nabla h(x)\Delta g(x) \right|.
    \end{aligned}
\end{equation*}
Integrating by parts twice, the first term is controlled by
\begin{equation*}
    \begin{aligned}
     &\left|\int_{\mathbb R^4}\frac{1}{2}\Delta q \left(x\right)h(x)\Delta g(x)dx\right|\\
     \lesssim& \left|\int_{\mathbb R^4}\Delta^2 q \left(x\right)h(x) g(x)dx\right|+\left|\int_{\mathbb R^4}\nabla\Delta q \left(x\right)\nabla h(x) g(x)dx\right|+ \left|\int_{\mathbb R^4}\Delta q \left(x\right)\Delta h(x) \nabla g(x)dx\right|.  
    \end{aligned}
\end{equation*}
Using (5), (6) and (7) of Lemma \ref{function q} and the  $q$ in the exterior region $|x|\geq \Tilde{R}$, these terms are separately bounded by 
\begin{equation*}
    \begin{aligned}
    \left|\int_{\mathbb R^4}\Delta^2 q \left(x\right)h(x) g(x)dx\right|\leq&
     \left|\int_{|x|\leq \Tilde{R}}\frac{\epsilon h(x)}{|x|^2(1+|\log |x|)} g(x)dx\right|+ \left|\int_{|x|\geq \Tilde{R}}\frac{h(x)}{|x|^5} g(x)dx\right|\\
     \lesssim & (\Tilde{R}+{\Tilde{R}}^{-1}) \left\|\ML_{k}h\right\|_{L^2}\left\|g\right\|_{\dot{H}^1};
    \end{aligned}
\end{equation*}
\begin{equation*}
   \begin{aligned}
      \left|\int_{\mathbb R^4}\nabla\Delta q \left(x\right)\nabla h(x) g(x)dx\right|\lesssim& \ \left|\int_{|x|\leq \Tilde{R}}\frac{\epsilon \nabla h(x)}{x(1+|\log |x||)} g(x)dx\right|+\left|\int_{|x|\geq \Tilde{R}}\frac{\nabla h(x)}{|x|^3}g(x)dx\right|\\
      \lesssim&(\Tilde{R}+{\Tilde{R}}^{-1}) \left\|\ML_{k}h\right\|_{L^2}\left\|g\right\|_{\dot{H}^1};
   \end{aligned} 
\end{equation*}
and
\begin{equation*}
   \begin{aligned}
     \left|\int_{\mathbb R^4}\Delta q \left(x\right)\Delta h(x) \nabla g(x)dx\right|\lesssim& \left|\int_
     {|x|\leq \Tilde{R}}\Delta h(x)\nabla g(x)dx\right|+\left|\int_
     {|x|\geq \Tilde{R}}\frac{\Delta h(x)}{|x|^2}\nabla g(x)dx\right|
     \lesssim \left\|\ML_{k}h\right\|_{L^2}\left\|g\right\|_{\dot{H}^1}.
   \end{aligned} 
\end{equation*}
Combining these together we have
\begin{equation*}
   \left|\int_{\mathbb R^4}\frac{1}{2}\Delta q \left(x\right)h(x)\Delta g(x)dx\right|\lesssim  \Tilde{R}\left\|\ML_{k}h\right\|_{L^2}\left\|g\right\|_{\dot{H}^1}.
\end{equation*}
With the same strategy, we can also prove 
\begin{equation*}
   \left|\int_{\mathbb R^4}\nabla q\left(x\right)\cdot \nabla h(x)\Delta g(x) \right|\lesssim  \Tilde{R}\left\|\ML_{k}h\right\|_{L^2}\left\|g\right\|_{\dot{H}^1}.
\end{equation*}
Therefore, (\ref{Ah delta g in app}) holds. 

\textbf{Step 3. Proof of (\rmnum{4}).}
By the definition of $\underline{A}_k$ and integration by parts, the left-hand side of (\ref{Ak key property 3 in app}) is expanded as 
\begin{equation*}
\begin{aligned}
   \lelan \underline{A}_kh,\Delta h\rilan=&\int\frac{1}{2\lambda_k}\Delta q \left(\frac{x-\boldsymbol{y}_k}{\lambda_k}\right)h(x)\Delta h(x)+\nabla q\left(\frac{x-\boldsymbol{y}_k}{\lambda_k}\right)\cdot \nabla h(x)\Delta h(x)dx\\
   =&\int-\frac{1}{4\lambda_k^3}\Delta^2 q\left(\frac{x-\boldsymbol{y}_k}{\lambda_k}\right)h^2(x)-\sum_{i,j=1}^4\frac{1}{\lambda_k}\partial_{ij}q\left(\frac{x-\boldsymbol{y}_k}{\lambda_k}\right)\partial_ih(x)\partial_jh(x)dx.
\end{aligned}
\end{equation*}
Since $q$ is a radial function, let $r=|x|$, it holds 
\begin{equation*}
    \begin{aligned}
  &-\int\sum_{i,j=1}^4\frac{1}{\lambda_k}\partial_{ij}q\left(\frac{x-\boldsymbol{y}_k}{\lambda_k}\right)\partial_ih(x)\partial_jh(x)dx\\
  =&-\frac{1}{\lambda_k}\int\frac{q'(r)}{r}|\nabla h|^2+\left(q''(r)-\frac{q'(r)}{r}\right)(\partial_r h)^2dx.
    \end{aligned}
\end{equation*}
Combining these with property (1), (3), (5) of Lemma \ref{function q} completes the proof of (\ref{Ak key property 3 in app}). Now we turn to the estimate (\ref{Ak key property 2 in app}). The definition of $\underline{A}_k$ yields
\begin{equation*}
    \begin{aligned}
        \lelan\left[\Delta , \underline{A}_k\right]h,\Delta h \rilan
        =&\int_{\mathbb R^4}\frac{1}{2}\frac{1}{\lambda_k^3}\Delta^2q(\frac{x-\boldsymbol{y}_k}{\lambda_k})h(x)\Delta h(x)dx+\int_{\mathbb R^4}\frac{2}{\lambda_k^2}\nabla \Delta q(\frac{x-\boldsymbol{y}_k}{\lambda_k})\cdot \nabla h(x)\Delta h(x)dx\\
        &+\int_{\mathbb R^4}\frac{1}{\lambda_k}\sum_{i,j=1}^4\partial_{ij}q(\frac{x-\boldsymbol{y}_k}{\lambda_k})\partial_{ij}h(x)\Delta h(x)dx:=\Rmnum{1}+\Rmnum{2}+\Rmnum{3}.
    \end{aligned}
\end{equation*}
For $\Rmnum{1}$, using integration by part we get
\begin{equation*}
    \begin{aligned}
        \Rmnum{1}=&-\int_{\mathbb R^4}\frac{1}{2}\frac{1}{\lambda_k^3}\Delta^2q(\frac{x-\boldsymbol{y}_k}{\lambda_k})\nabla h(x)\nabla h(x)dx-\int_{\mathbb R^4}\frac{1}{2}\frac{1}{\lambda_k^3}\nabla\Delta^2q(\frac{x-\boldsymbol{y}_k}{\lambda_k})\nabla \frac{h(x)^2}{2}dx\\
        =&-\int_{\mathbb R^4}\frac{1}{2}\frac{1}{\lambda_k^3}\Delta^2q(\frac{x-\boldsymbol{y}_k}{\lambda_k})|\nabla h(x)|^2dx+\int_{\mathbb R^4}\frac{1}{4}\frac{1}{\lambda_k^3}\Delta^3q(\frac{x-\boldsymbol{y}_k}{\lambda_k}){h(x)^2}dx.
    \end{aligned}
\end{equation*}
From property (4) and (5) of Lemma \ref{function q} , these two terms are below bounded by
\begin{equation*}
  -\int_{\mathbb R^4}\frac{1}{2}\frac{1}{\lambda_k^3}\Delta^2q(\frac{x-\boldsymbol{y}_k}{\lambda_k})|\nabla h(x)|^2dx\gtrsim -\frac{1}{\lambda_k}\int_{\mathbb R^4}  \frac{\epsilon|\nabla h(x)|^2}{|x|^2(1+|\log |x||)}dx,
\end{equation*}
and 
\begin{equation*}
 \int_{\mathbb R^4}\frac{1}{4}\frac{1}{\lambda_k^3}\Delta^3q(\frac{x-\boldsymbol{y}_k}{\lambda_k}){h(x)^2}dx\gtrsim -\frac{1}{\lambda_k}\int_{\mathbb R^4}\frac{\epsilon h(x)^2}{|x|^4(1+|\log |x||)^2}dx.   
\end{equation*}
Combining with the coercivity proposition \ref{Coercivity estimate with LG LG} we obtain
\begin{equation}\label{app B est of R1}
    \Rmnum{1}\gtrsim -\frac{\epsilon}{\lambda_k} \left\|\ML_k h\right\|_{L^2}^2.
\end{equation}
Then for $\Rmnum{2}$, applying property (6) of lemma \ref{function q} and Cauchy's inequality we have
\begin{equation*}
   \int_{\mathbb R^4}\frac{2}{\lambda_k^2}\nabla \Delta q(\frac{x-\boldsymbol{y}_k}{\lambda_k})\cdot \nabla h(x)\Delta h(x)dx\gtrsim -\frac{\epsilon}{\lambda_k}\left\|\frac{ \nabla h}{|x|}\right\|_{L^2}\left\|\Delta h\right\|_{L^2}.
\end{equation*}
Thus, again from coercivity proposition, it holds
\begin{equation}\label{app B est of R2}
    \Rmnum{2}\gtrsim -\frac{\epsilon}{\lambda_k} \left\|\ML_k h\right\|_{L^2}^2.
\end{equation}
Finally, we consider \Rmnum{3}. Since $q$ is a radial function, let $r:=|x-\boldsymbol{y}_k|/\lambda_k$, \Rmnum{3} can be rewritten as
\begin{equation*}
    \begin{aligned}
       &\int_{\mathbb R^4}\frac{1}{\lambda_k}\sum_{i,j=1}^4\partial_{ij}q(\frac{x-\boldsymbol{y}_k}{\lambda_k})\partial_{ij}h(x)\Delta h(x)dx\\
       =&\frac{1}{\lambda_k}\int_{\mathbb R^4}\frac{q'(r)}{r}(\Delta h(x))^2+\left(q''(r)-\frac{q'(r)}{r}\right)(\partial_{rr}h(x) \Delta h(x))dx.
    \end{aligned}
\end{equation*}
Recall that $q(x)=\frac{|x|^2}{2}$ when $|x|\leq R$ and property (3) from lemma \ref{function q}, the lower bound of $\Rmnum{3}$
\begin{equation}\label{app B est of R3}
    \Rmnum{3}\geq \frac{1}{\lambda_k}\int_{|x-\boldsymbol{y}_k|\leq \lambda_k R}(\Delta h(x))^2dx-\frac{2\epsilon}{\lambda_k}\left\|\ML_k h\right\|_{L^2}^2
\end{equation}
 holds. Subsequently, combining (\ref{app B est of R1}), (\ref{app B est of R2}) and (\ref{app B est of R3}) we conclude 
(\ref{Ak key property 2 in app}).

\textbf{Step 4. Proof of (\rmnum{5}).} Finally, we consider the estimate (\ref{A and Lambda W nabla W in app}). Due to the asymptotic form of \(q\) in the exterior region $|x|\geq \tilde{R}$, it suffices to estimate on $[R,\tilde{R}]$. Recall the notation in the proof of Lemma \ref{function q}, we have
\begin{equation*}
    \begin{aligned}
        \underline{A}_k \Lambda W=&\frac{1}{2}\Delta q(r)\Lambda W(r)+ q(r)'\partial_rW(r)\\
        =&\frac{p^{(2)}(\log r)+2p^{(1)}(\log r)}{2r^2}\Lambda W(r)+\frac{p^{(1)}(\log r)}{r}\partial_r W(r).
    \end{aligned}
\end{equation*}
Recall that $p$ is given by
\begin{equation*}
      p(x):=c_1(x)+c_2(x)x+c_3(x)e^{2x}+c_4(x)xe^{2x}+c_5(x)e^{-2x}+c_6(x)e^{-x}.
\end{equation*}
The functions $c_j(x)$ here satisfy
\begin{equation*}
\begin{aligned}
   |c_1(x)|\lesssim \alpha_1 e^{2x}(1+x)^{-1}&\text{ , }|c_2(x)|\lesssim \alpha_1 e^{2x}(1+x)^{-2}, \\
   |c_3(x)|\leq C&\text{ , }|c_4(x)|\lesssim \frac{\alpha_1}{x}, \\
   |c_5(x)|\lesssim \alpha_1 e^{4x}(1+x)^{-2}&\text{ , }|c_6(x)|\lesssim \alpha_1 e^{3x}(1+x)^{-2}.
\end{aligned} 
\end{equation*}
All these estimates have already been established in Lemma \ref{function q}. 
Plugging these back into the expression of $\underline{A}_k \Lambda W$, we obtain
\begin{equation*}
 \begin{aligned}
  \underline{A}_k \Lambda W(r)=2(c_3(\log r)+c_4(\log r)\log r)\underline{\Lambda}\Lambda W(r)+O\left(\frac{\alpha_1}{r^2\log r}\right).   
 \end{aligned}   
\end{equation*}
The decay of $\underline{\Lambda}\Lambda W$ then implies 
\begin{equation*}
\begin{aligned}
     \left\|\underline{A}_k \Lambda W\right\|_{L^2(R\leq |x|\leq \tilde{R})}\lesssim& \left\|\underline{\Lambda}\Lambda W\right\|_{L^2(R\leq |x|\leq \tilde{R})}+\alpha_1\left(\int_{R}^{\tilde{R}}\frac{1}{r^4(\log r)^2}r^3dr\right)^{\frac{1}{2}}\\
     \lesssim& \frac{1}{R^2}+\alpha_1 \sqrt{\frac{1}{\log R}-\frac{1}{\log \tilde{R}} }.
\end{aligned}  
\end{equation*}
As a result, choosing $\epsilon$ sufficiently small, we have
\begin{equation*}
    \left\|\underline{A}\Lambda W\right\|_{L^2(|x|\geq R)}\leq \left(\eta+\frac{1}{R}\right).
\end{equation*}
The estimate for $\nabla W$ is easier. Notice that $\nabla W(x)\sim |x|^{-3}$ and $|\nabla q(x)|\leq |x|$, $|\Delta q(x)|\leq C$, the exterior estimate follows directly. 
\end{proof}


\begin{thebibliography}{80}
\bibitem{Aubin}
T. Aubin,
Problèmes isopérimétriques et espaces de Sobolev,
J. Differential Geometry 11 (1976), no. 4, 573--598.

\bibitem{BeceanuCenterStable}
M. Beceanu,
A centre-stable manifold for the energy-critical wave equation in
\(\mathbb R^3\) in the symmetric setting,
J. Hyperbolic Differ. Equ. 11 (2014), no. 4, 821--852.

\bibitem{CDKM 2022} C. Collot, T. Duyckaerts, C. Kenig, and F. Merle. Soliton resolution for the radial quadratic wave equation in six space dimensions. {Viet. J. Math.}{52} (2024), 735–773 .

\bibitem{CDKM 2022.1} C. Collot, T. Duyckaerts, C. Kenig, and F. Merle. On channels of energy for the radial linearised energy critical wave equation in the degenerate case. {Int. Math. Res. Not. IMRN} {24} (2023), 21015–21067.

\bibitem{CDKM 2023} C. Collot, T. Duyckaerts, C. Kenig, and F. Merle. On classification of non-radiative solutions for various energy-critical wave equations. {Adv. Math.} {434} (2023), 91.

\bibitem{CombetMartel}V. Combet, Y. Martel. Construction of the multi-bubble solutions for the critical gKdV equation, SIAM J. Math. Anal. 50 (4) (2018) 3715-3790.

\bibitem{CortazarDelPinoMusso}
C. Cortazar, M. del Pino, and M. Musso,
Green's function and infinite-time bubbling in the critical nonlinear heat equation,
J. Eur. Math. Soc. (JEMS) 22 (2020), no. 1, 283--344.
   
\bibitem{CMM}R. Côte, Y. Martel, F. Merle, Construction of multi-soliton solutions for the $L^2$-supercritical gKdV and NLS equations, Rev. Mat. Iberoam. 27(1) (2011) 273-302.

\bibitem{CoteMunoz}
R. Côte and C. Muñoz,
Multi-solitons for nonlinear Klein--Gordon equations,
Forum Math. Sigma 2 (2014), e15, 38 pp.

\bibitem{DelPinoMussoWei}
M. del Pino, M. Musso, and J. Wei,
Type II blow-up in the 5-dimensional energy critical heat equation,
Acta Math. Sin. (Engl. Ser.) 35 (2019), no. 6, 1027--1042.

\bibitem{DelPinoMussoWei3DInfiniteHeat}
M. del Pino, M. Musso, and J. Wei,
Infinite time blow-up for the 3-dimensional energy critical heat equation,
Anal. PDE 13 (2020), no. 1, 215--274.

\bibitem{DonningerKreiger}R. Donninger, J. Krieger, Nonscattering solutions and blowup at infinity for the critical wave equation, Math. Ann. 357 (1) (2013) 89-163.
   
\bibitem{DuyckaertsJiaKenigMerle}
T. Duyckaerts, H. Jia, C. E. Kenig, and F. Merle,
Soliton resolution along a sequence of times for the focusing energy critical
wave equation,
Geom. Funct. Anal. 27 (2017), no. 4, 798--862.

\bibitem{DuyckaertsKenigMerleUniversality}
T. Duyckaerts, C. E. Kenig, and F. Merle, Universality of blow-up profile for small radial type II blow-up solutions of the energy-critical wave equation, J. Eur. Math. Soc. 13 (3) (2011) 533-599. 

\bibitem{DuyckaertsKenigMartelMerle} T. Duyckaerts, C. E. Kenig, Y. Martel, and F. Merle. Soliton resolution for critical co-rotational wave maps and radial cubic wave equation. {Commun. Math. Phys.} {391} (2022), no. 2, 779-871.

\bibitem{DuyckaertsKenigMerleClassification}
T. Duyckaerts, C. E. Kenig, and F. Merle, Classification of radial solutions of the focusing, energy-critical wave equation, Camb. J. Math. 1 (1) (2013) 75-144.

\bibitem{DuyckaertsKenigMerleSolitonResolution}
T. Duyckaerts, C. E. Kenig, and F. Merle,
Soliton resolution for the radial critical wave equation in all odd space
dimensions,
Acta Math. 230 (2023), no. 1, 1--92.

\bibitem{DM}T. Duyckaerts, F. Merle, Dynamics of threshold solutions for energy-critical wave equation. Int. Math. Res . Pap. IMRP, 2008.

\bibitem{HR}M. Hillairet, P. Rapha\"{e}l, Smooth type II blow-up solutions to the four-dimensional energy-critical wave equation, Anal. PDE 5 (2012), no. 4, 777-829.

\bibitem{HwangKimWaveMapsBubbleTower}
S. Hwang and K. Kim,
Construction of infinite time bubble tower solutions to critical wave maps equation, 
arXiv:2603.01793.

\bibitem{J2bNLS}J. Jendrej, Construction of two-bubble solutions for the energy-critical NLS. Anal. PDE, 10(8): 1923-1959, 2017.

\bibitem{Jtype2d5}J. Jendrej, Construction of type II blow-up solutions for the energy-critical wave equation in dimension 5, J. Funct. Anal. 272 (3) (2017) 866-917.

\bibitem{JendrejNonexistenceOppositeSigns}
J. Jendrej,
Nonexistence of radial two-bubbles with opposite signs for the energy-critical
wave equation,
Ann. Sc. Norm. Super. Pisa Cl. Sci. (5) 18 (2018), no. 2, 735--778.

\bibitem{Jtb}J. Jendrej, Construction of two-bubble solutions for energy-critical wave equations, Am. J. Math. 141 (1) (2019) 55-118.

\bibitem{JendrejKriegerWaveMapsBubbleTree}
J. Jendrej and J. Krieger,
Concentric bubbles concentrating in finite time for the energy critical wave maps equation,
arXiv:2501.08396.

\bibitem{JendrejLawrieExpansion}
J. Jendrej and A. Lawrie,
An asymptotic expansion of two-bubble wave maps,
Anal. PDE 15 (2022), no. 2, 327--403.

\bibitem{JendrejLawrieUniqueness}
J. Jendrej and A. Lawrie,
Uniqueness of two-bubble wave maps in high equivariance classes,
Comm. Pure Appl. Math. 75 (2022), no. 12, 2630--2685.

\bibitem{JL}J. Jendrej and A. Lawrie. Soliton resolution for the energy-critical nonlinear wave equation in the radial case. {Ann. PDE} {9} (2023) article no. 18 .
   
\bibitem{JM}J. Jendrej, Y. Martel, Construction of multi-bubble solutions for the energy-critical wave equation in dimension 5. J. Math. Pures Appl. 139 (2020), no. 9, 317–355.

\bibitem{JendrejZhangZhao} J. Jendrej, C. Zhang, L. Zhao, Rigidity of the multi-bubble solutions to the energy critical wave equation in dimension five. arXiv:2605.27994.

\bibitem{Kadar3DMultiSolitons}
I. Kadar,
Construction of multi-soliton solutions for the energy critical wave equation
in dimension 3,
preprint, arXiv:2409.05267, 2024.

\bibitem{KadarNoQuantization}
I. Kadar,
Smooth finite time singularity formation without quantization,
preprint, arXiv:2603.14985, 2026.

\bibitem{KenigMerle08}
C. E. Kenig and F. Merle,
Global well-posedness, scattering and blow-up for the energy-critical focusing
non-linear wave equation,
Acta Math. 201 (2008), no. 2, 147--212.

\bibitem{KenigMendelson}
C. E. Kenig and D. Mendelson,
The focusing energy-critical nonlinear wave equation with random initial data,
Int. Math. Res. Not. IMRN 2021, no. 19, 14508--14615.

\bibitem{KNSGlobal}
J. Krieger, K. Nakanishi, and W. Schlag, Global dynamics away from the ground state for the energy-critical nonlinear wave equation, Am. J. Math. 135 (4) (2013) 935-965.

\bibitem{KNSCenterStable}
J. Krieger, K. Nakanishi, and W. Schlag,
Center-stable manifold of the ground state in the energy space for the critical
wave equation,
Math. Ann. 361 (2015), no. 1-2, 1--50.

\bibitem{KriegerPalaciosWaveMapsBubbleTrees}
J. Krieger and J. M. Palacios,
Long finite time bubble trees for two co-rotational wave maps,
arXiv:2602.22825.

\bibitem{KriegerSchlagStableManifold}
J. Krieger and W. Schlag,
On the focusing critical semi-linear wave equation,
Amer. J. Math. 129 (2007), no. 3, 843--913.

\bibitem{KriegerSchlag}
J. Krieger and W. Schlag,
Full range of blow up exponents for the quintic wave equation in three
dimensions,
J. Math. Pures Appl. (9) 101 (2014), no. 6, 873--900.

\bibitem{KriegerSchlagTataruWavemap}
J. Krieger, W. Schlag, and D. Tataru,
Renormalization and blow up for charge one equivariant critical wave maps. Invent. Math., 171(3): 543-615, 2008.

\bibitem{KriegerSchlagTataruYangmils}
J. Krieger, W. Schlag, and D. Tataru,
Renormalization and blow up for the critical Yang-Mills problem. Adv. Math., 221(5): 1445-1521, 2009.

\bibitem{KriegerSchlagTataru}
J. Krieger, W. Schlag, and D. Tataru,
Slow blow-up solutions for the \(H^1(\mathbb R^3)\) critical focusing
semilinear wave equation,
Duke Math. J. 147 (2009), no. 1, 1--53.

\bibitem{Mnsgkdv}Y. Martel, Asymptotic N-soliton-like solutions of the subcritical and critical generalized Korteweg-de Vries equations, Am. J. Math. 127 (5) (2005) 1103-1140.

\bibitem{MartelMerle5D}
Y. Martel and F. Merle,
Construction of multi-soliton solutions for the energy-critical wave equation
in dimension 5,
Arch. Ration. Mech. Anal. 222 (2016), no. 3, 1113--1160.


\bibitem{MartelMerleInelasticity}
Y. Martel and F. Merle,
Inelasticity of soliton collisions for the 5D energy critical wave equation,
Invent. Math. 214 (2018), no. 3, 1267--1363.

\bibitem{MartelMerleAnyParameters}
Y. Martel and F. Merle,
Existence of multi-solitons with any parameters for the 5D energy critical wave
equation,
arXiv:2510.19609.


\bibitem{Mkps}F. Merle, Construction of solutions with exactly k blow-up points for the Sch\"{o}dinger equation with critical nonlinearity, Commun. Math. Phys. 129 (2) (1990) 223-240.

\bibitem{RaphaelRodnianski}
P. Raphaël and I. Rodnianski,
Stable blow up dynamics for the critical co-rotational wave maps and
equivariant Yang--Mills problems,
Publ. Math. Inst. Hautes Études Sci. 115 (2012), 1--122.

\bibitem{RS}P. Rapha\"{e}l, J. Szeftel, Existence and uniqueness of minimal blow-up solutions to an inhomogeneous mass critical NLS, J. Am. Math. Soc. 24 (2) (2011) 471-546.

\bibitem{Schweyer}
R. Schweyer,
Type II blow-up for the four dimensional energy critical semi linear heat equation,
J. Funct. Anal. 263 (2012), no. 12, 3922--3983.

\bibitem{Talenti}
G. Talenti,
Best constant in Sobolev inequality,
Ann. Mat. Pura Appl. (4) 110 (1976), 353--372.

\bibitem{wazewski}
        T.~Wa\.zewski.
        \newblock Sur un principe topologique de l'examen de l'allure asymptotique des
int\'egrales des \'equations diff\'erentielles ordinaires.
\newblock {\em Ann. Soc. Polon. Math.}, 20:279--313, 1947.


\end{thebibliography}
\end{document}